\documentclass[10pt,a4paper]{article}
\usepackage{latexsym,amsthm,amssymb,amscd,amsmath}
\usepackage[small,nohug,heads=vee]{diagrams}
\diagramstyle[labelstyle=\scriptstyle]
\newcounter{alphthm}
\newcommand{\im}{\mathrm Im}
\theoremstyle{plain}
\newtheorem{theorem}{Theorem}[section]
\newtheorem{lemma}[theorem]{Lemma}
\newtheorem{proposition}[theorem]{Proposition}
\newtheorem{cor}[theorem]{Corollary}
\theoremstyle{definition}
\newtheorem{defn}[theorem]{Definition}
\newtheorem{rem}[theorem]{Remark}

\newcommand{\be}{\begin{equation}}
\newcommand{\ee}{\end{equation}}
\newcommand{\ben}{\begin{enumerate}}
\newcommand{\een}{\end{enumerate}}

\begin{document}
\title{Models of Finsler Geometry on Lie algebroids}
\author{E. Peyghan}
\maketitle

\maketitle
\begin{abstract}
Horizontal endomorphisms, almost complex structures, vertical, horizontal and complete lifts on prolongation of a Lie algebroid are considered. Then using exact sequences, semisprays are constructed. Moreover, important geometrical objects such as classical distinguished connections, torsions and partial curvatures are studied on prolongation of Lie algebroids. Considering pullback bundle, covariant derivatives are scrutinized based on anchor map. Several Finsler geometry models on Lie algebroid structures, will organized via recent arguments. Finally, it will be overhaulled some special Finsler Lie algebroid spaces.
\end{abstract}
\tableofcontents
\newpage
\section{Introduction}
The notion of Lie algebroids was first introduced by Pradines \cite{pradines}. Researching on this field, is continuum by mathematicians with various purposes up to now. Lie algebroids are studeid pure or in relation with other subjects \cite{A, A1, Ar, Ar1, BP, CLMM, M1, V2, V3}. Precisely, a Lie algebroid is a vector bundle with the property that its sections involve a real Lie algebra. Each section is anchored on a vector field, by means of a linear bundle map named as anchor map, which is further supposed to induce a Lie algebra homomorphism. Specially  when the base manifold $M$ is a point, a Lie algebroid reduces to a Lie algebra. The most simple examples of  Lie algebroids are the zero bundle over $M$ which is denoted by $M$ and tangent bundle over $M$ with identity as anchor map which is denoted by $TM$. Then the tangent bundle is a special case of Lie algebroid structure. Therefore a Lie algebroid is a generalization of a Lie algebra and vector bundle.

The Lie algebroid is a good extension of tangent bundle, since the homomorphism property of the anchor map grants the basic notions of tangent bundle to the vector bundle ( for example the torsion notion that has not a good meaning in vector bundles). However, some theorems on tangent bundle do not work here as we shall see. Indeed, any generalization need to using anchor map to reconstruct the most useful objects to compare geometric structures at different points of a manifold on the vector bundle, namely; covariant derivation and connection concepts . In \cite{f}, one can see how they have generalized on Lie algebroids  by a worthwhile work as the first step. Another acute extension is curvature. In this way, a good conceptualization is \cite{bl}. In \cite{bou}, Riemannian Lie algebroids have introduced and basic facts like Levi-Civita connection, Riemannian metric and curvature, sectional curvature, geodesics and integrability
have studied.

 Recently, Lie algebroids are important issues in physics and mechanics since the extension of Lagrangian and Hamiltonian systems to their entity \cite{ar, LMM, M, mestdag, W} and catching the poisson  structure \cite{popescu0}. They are wrestled with nonsmooth optimization \cite{popescu} and studied on Banach vector bundles \cite{mi}. They have such a flexibility  that holonomy of orbit foliation carried on them \cite{f}. Thus Lie algebroids are strong assorted structures to assemble the Physics and mechanics notions on them. For a good details about penetration of Lie algebroids, see \cite{V}.

The aim of this paper is rebuild the Finsler geometry concepts on Lie algebroid structures. For instance, this matter is discussed in \cite{V1, popescu} of course. Finsler geometry is a generalization of Riemannian geometry such that interfering of direction and position duplicates the degree of freedom in view of configuration. Variety of tensors in Finsler geometry is more than Riemannian case. One can study with more details about Finsler geometry, say in \cite{bcs, B, V4}.

 In Finsler geometry there are two approaches. The first is global make up, and the second is localization. Indeed, non of them have any advantage to the other until one would like use them as gadgets to derive the conclusion or to receive the target. For example, when we want to see the anchor role precisely, it prefered the locally approach shall be applied and when we want to have a boxed and index-free formula to categorize the results, we choose the global viewpoint. Accordingly, we tried to designate the approaches in the sense of the case.

 The paper is organized as follows. In Section 2, we recall differential, contraction and Lie differential operators on Lie algebroids and we study the relation between these operators. We also mention the generalized Fr\"{o}licher-Nijenhuis bracket on Lie algebroids. Then vertical and complete lifts on a Lie algebroid is considered and some important properties of these objects is studied. In Section 3, building on the notion of prolongation $\pounds^\pi E\rightarrow E$ of the Lie algebroid $E\rightarrow M$, we present vertical and complete lifts of a section of $E$ to $\pounds^\pi E$. But the major concept of this section is to construct the vertical endomorphism on a special exact sequence. Then Liouville section and semispray on $\pounds^\pi E$ will be introduced. Moreover, some interesting results on them will be gathered. The aim of Section 4 is to make up pre-curvature concepts like torsion and tension by decomposing $\pounds^\pi E$ using horizontal endomorphisms. The almost complex structure on  $\pounds^\pi E$ and Berwald endomorphism will be introduced  in this section also. Finally, it will be shown that how sections on $E$ lift into the horizontal space of $\pounds^\pi E$ using the horizontal endomorphism. In Section 5, distinguished connections on $\pounds^\pi E$ are introduced and torsion and curvature tensor fields of these connections are considered. In particular, we introduce Berwald-type and Yano-type connections on $\pounds^\pi E$ as two important classes of distinguished connections. In Section 6, we construct $\rho_\pounds$-covariant derivatives
in $\pi^*\pi$ as generalization of covariant derivative in $\pi^*\pi$ to $\pounds^\pi E$. Moreover, Berwald and Yano derivatives as two important classes of $\rho_\pounds$-covariant derivatives in $\pi^*\pi$ are introduced in this section. Section 7 is busy with Finsler algebroids and related materials. Important endomorphisms like Conservative and Barthel, Cartan tensor and some distinguished connections like  Berwald, Cartan, Chern-Rund and Hashiguchi are studied by Szilasi and his collaborators from a special point view based on pullback bundle \cite{SS, SS1, S, SG, ST}. In this section we construct them on Finsler algebroids and obtain some results on these concepts. In section 8, $h$-basic distinguished connections are introduced on Finsler algeboids. Specially, Ichijy\={o} connection that is a special $h$-basic distinguished connection is more studied. Generalized Berwald Lie algebroids are presented next. The section will ended by Wagner-Ichijy\={o} connection that is a special case of Ichijy\={o}'s one.

\section{Basic concepts on Lie algebroids }
Let $E$ be a vector bundle of rank $n$ over a manifold $M$ of dimension $m$ and $\pi:E\rightarrow M$ be
the vector bundle projection. Denote by $\Gamma(E)$ the $C^\infty(M)$-module of sections of $\pi:E\rightarrow M$.
A \textit{Lie algebroid structure} $([., .]_E, \rho)$ on $E$ is a Lie bracket $[., .]_E$ on the space $\Gamma(E)$ and a bundle
map $\rho:E\rightarrow TM$, called the anchor map, such that if we also denote by $\rho:\Gamma(E)\rightarrow\chi(M)$
the homomorphism of $C^\infty(M)$-modules induced by the anchor map then
\[
[X, fY]_E=f[X, Y]_E+\rho(X)(f)Y,\ \ \ \forall X, Y\in\Gamma(E),\ \forall f\in C^\infty(M).
\]
Moreover, we have the relations
\begin{equation}\label{revise}
[\rho(X), \rho(Y)]=\rho([X, Y]_E),
\end{equation}
and
\begin{equation}\label{revise1}
[X, [Y, Z]_E]_E+[Y, [Z, X]_E]_E+[Z, [X, Y]_E]_E=0.
\end{equation}
Then triple $(E, [., .]_E, \rho)$ is called a \textit{Lie algebroid} over $M$.

Trivial examples of Lie algebroids are real Lie algebras of finite dimension, the tangent
bundle $TM$ of an arbitrary manifold $M$ and an integrable distribution of $TM$.

If $(E, [., .]_E, \rho)$ is a Lie algebroid over $M$, then the anchor map $\rho:\Gamma(E)\rightarrow\chi(M)$ is a
homomorphism between the Lie algebras $(\Gamma(E), [., .]_E)$ and $(\chi(M), [·, ·])$.

On Lie algebroids $(E, [., .]_E, \rho)$, we define the differential of $E$, $d^E:\Gamma(\wedge^kE^*)\rightarrow\Gamma(\wedge^{k+1}E^*)$, as follows
\begin{align*}
d^E\mu(X_0,\ldots, X_k)&=\sum_{i=0}^k(-1)^i\rho(X_i)(\mu(X_0,\ldots, \hat{X_i}, \ldots, X_k))\\
&\ \ \ +\sum_{i<j}(-1)^{i+j}\mu([X_i, X_j]_E, X_0, \ldots, \hat{X_i}, \ldots, \hat{X_j}, \ldots, X_k),
\end{align*}
for $\mu\in\Gamma(\wedge^kE^*)$ and $X_0, \ldots, X_k\in\Gamma(E)$. In particular, if $f\in\Gamma(\wedge^0 E^*)=C^\infty(M)$ we have $d^Ef(X)=\rho(X)f$. Using the above equation it follows that $(d^E)^2=0$.

If we take local coordinates $(x^i)$ on $M$ and a local basis $\{e_\alpha\}$ of sections of $E$, then we
have the corresponding local coordinates $(\textbf{x}^i, \textbf{y}^\alpha)$ on $E$, where $\textbf{x}^i=x^i\circ\pi$ and $\textbf{y}^\alpha(u)$ is the $\alpha$-th coordinate of $u\in E$ in the given basis. Such coordinates determine local functions $\rho^i_\alpha$, $L^\gamma_{\alpha\beta}$ on $M$ which contain the local information of the Lie algebroid structure, and accordingly they are called the structure functions of the Lie algebroid. They are given by
\[
\rho(e_\alpha)=\rho^i_\alpha\frac{\partial}{\partial x^i},\ \ \ [e_\alpha, e_\beta]_E=L^\gamma_{\alpha\beta}e_\gamma.
\]
Using (\ref{revise}) and (\ref{revise1}), these functions should satisfy the following relations
\begin{equation}\label{2}
(i)\ \rho^j_\alpha\frac{\partial\rho^i_\beta}{\partial x^j}-\rho^j_\beta\frac{\partial\rho^i_\alpha}{\partial x^j}=\rho^i_\gamma L_{\alpha\beta}^\gamma,\ \ \ (ii)\ \sum_{(\alpha, \beta, \gamma)}[\rho^i_\alpha\frac{\partial L^\nu_{\beta\gamma}}{\partial x^i}+L^\nu_{\alpha\mu}L^\mu_{\beta\gamma}]=0,
\end{equation}
which are usually called the structure equations. We have also,
\begin{equation}\label{1}
d^Ef=\frac{\partial f}{\partial x^i}\rho^i_\alpha e^\alpha, \ \ \ \forall f\in C^\infty(M),
\end{equation}
where $\{e^\alpha\}$ is the dual basis of $\{e_\alpha\}$. On the other hand, if $\omega\in\Gamma(E^*)$ and $\omega=\omega_\gamma e^\gamma$ it follows
\[
d^E\omega=(\frac{\partial\omega_\gamma}{\partial x^i}\rho^i_\beta-\frac{1}{2}\omega_\alpha L^\alpha_{\beta\gamma})e^\beta\wedge e^\gamma.
\]
In particular,
\[
d^Ex^i=\rho^i_\alpha e^\alpha,\ \ \ \ d^Ee^\alpha=-\frac{1}{2}L^\alpha_{\beta\gamma}e^\beta\wedge e^\gamma.
\]
A section $\omega$ of $E^*$ also defines a function $\hat{\omega}$ on $E$ by means of
\[
\hat{\omega}(u)=<\omega_m, u>,\ \ \ \forall u\in E_m.
\]
If $\omega=\omega_\alpha e^\alpha$, then the linear function $\hat{\omega}$ is
\[
\hat{\omega}(x, y)=\omega_\alpha \textbf{y}^\alpha.
\]
In particular, using (\ref{1}) we have
\[
\widehat{d^Ef}=\frac{\partial f}{\partial x^i}\rho^i_\alpha \textbf{y}^\alpha.
\]
\subsection{Generalized Fr\"{o}licher-Nijenhuis bracket}
For $X\in\Gamma(\wedge^kE)$, the contraction $i_X:\Gamma(\wedge^pE^*)\rightarrow\Gamma(\wedge^{p-k}E^*)$ is defined in standard way and the Lie differential operator $\pounds_X^E:\Gamma(\wedge^pE^*)\rightarrow\Gamma(\wedge^{p-k+1}E^*)$ is defined by
\[
\pounds_X^E=i_X\circ d^E-(-1)^kd^E\circ i_X.
\]
Note that if $E=TM$ and $X\in\Gamma(E)=\chi(M)$, then $d^{TM}$ and $\pounds_X^{TM}$ are the usual differential
and the usual Lie derivative with respect to $X$, respectively.

Let $K\in\Gamma(\wedge^kE^*\otimes E)$, then the contraction
\[
i_K:\Gamma(\wedge^n E^*)\rightarrow\Gamma(\wedge^{n+k-1} E^*),
\]
is defined in the natural way. In particular, for simple tensor $K=\mu\otimes X$, where $\mu\in\Gamma(\wedge^k E^*)$, $X\in\Gamma(E)$, we set
\[
i_K\nu=\mu\wedge i_X\nu.
\]
The corresponding Lie differential is defined by the formula
\[
\pounds^E_K=i_K\circ d^E+(-1)^kd^E\circ i_K,
\]
and, in particular,
\[
\pounds^E_{\mu\otimes X}=\mu\wedge\pounds^E_X+(-1)^kd^E\mu\wedge i_X.
\]
The contraction $i_K$ can be extended to an operator
\[
i_K:\Gamma(\wedge^n E^*\otimes E)\rightarrow\Gamma(\wedge^{n+k-1} E^*\otimes E),
\]
by the formula $i_K(\mu\otimes X)=i_K(\mu)\otimes X$. The following theorem contains a list of well-known formulas \cite{GU}:
\begin{theorem}\label{Best}
Let $\mu\in\Gamma(\wedge^kE^*)$, $\nu\in\Gamma(E^*)$ and $X, Y\in\Gamma(E)$. Then we have

(1)\ $d^E\circ d^E=0$,

(2)\ $d^E(\mu\wedge\nu)=d^E\mu\wedge\nu+(-1)^k\mu\wedge d^E\nu$,

(3)\ $i_X(\mu\wedge\nu)=i_X\mu\wedge\nu+(-1)^k\mu\wedge i_X\nu$,

(4)\ $\pounds_X^E(\mu\wedge\nu)=\pounds_X^E\mu\wedge\nu+(-1)^k\mu\wedge \pounds_X^E\nu$,

(5)\ $\pounds^E_X\circ\pounds^E_Y-\pounds^E_Y\circ\pounds^E_X=\pounds^E_{[X, Y]_E}$,

(6)\ $\pounds^E_X\circ i_Y-i_Y\circ\pounds^E_X=i_{[X, Y]_E}$.
\end{theorem}
The \textit{generalized Fr\"{o}licher-Nijenhuis bracket} is defined for simple tensors  $\mu\otimes X\in\Gamma(\wedge^k E^*\otimes E)$ and $\nu\otimes Y\in\Gamma(\wedge^l E^*\otimes E)$ by the formula
\begin{equation}\label{F1}
[\mu\otimes X, \nu\otimes Y]^{F-N}=(\pounds_{\mu\otimes X}\nu)\otimes Y-(-1)^{kl}(\pounds_{\nu\otimes Y}\mu)\otimes X+\mu\wedge\nu\otimes [X, Y]_E.
\end{equation}
Moreover, for $K\in\Gamma(\wedge^k E^*\otimes E)$, $L\in\Gamma(\wedge^l E^*\otimes E)$ and $N\in\Gamma(\wedge^n E^*\otimes E)$ we have
\begin{equation}\label{F2}
\pounds^E_{[K, L]^{F-N}}=\pounds^E_K\circ\pounds^E_L-(-1)^{kl}\pounds_L^E\circ\pounds^E_K,
\end{equation}
\begin{align}\label{F20}
(-1)^{kn}[K, [L, N]^{F-N}]^{F-N}&+(-1)^{lk}[L, [N, K]^{F-N}]^{F-N}\nonumber\\
&+(-1)^{nl}[N, [K, L]^{F-N}]^{F-N}=0.
\end{align}
From (\ref{F1}) and (\ref{F2}) we get
\begin{align}
[K, Y]^{F-N}(X)&=[K(X), Y]_E-K[X, Y]_E,\label{F3}\\
[K, L]^{F-N}(X, Y)&=[K(X), L(Y)]_E+[L(X), K(Y)]_E+(K\circ L+L\circ K)[X, Y]_E\nonumber\\
&\ \ -K[X, L(Y)]_E-K[L(X), Y]_E-L[X, K(Y)]_E\nonumber\\
&\ \ -L[K(X), Y]_E,\label{F4}
\end{align}
where $K\in\Gamma(E^*\otimes E)$, $L\in\Gamma(E^*\otimes E)$ and $X, Y\in\Gamma(E)$. (see \cite{GU}).
\subsection{Vertical and complete lifts on Lie algebroids}
For a function $f$ on $M$ one defines its vertical lift $f^\vee$ on $E$ by $f^\vee(u)=f(\pi(u))$ for $u\in E$.
Now, let $X$ be a section of $E$. Then, we can consider the vertical lift of $X$ as the vector field
on $E$ given by $X^\vee(u)=X(\pi(u))^\vee_u$, $u\in E$, where ${}^\vee_u: E_{\pi(u)}\rightarrow T_u (E_{\pi(u)})$ is the canonical
isomorphism between the vector spaces $E_{\pi(u)}$ and $T_u (E_{\pi(u)})$.
\begin{lemma}\label{alan}
Let $\{e_\alpha\}$ be a basis of sections of $E$. Then we have
\[
e_\alpha^\vee=\frac{\partial}{\partial\textbf{y}^\alpha}.
\]
\end{lemma}
\begin{proof}
We have
\begin{align*}
d\textbf{y}^\alpha(e_\beta^\vee(u))&=d\textbf{y}^\alpha(\frac{d}{dt}|_{t=0}(u+te_\beta))=\frac{d}{dt}|_{t=0}(\textbf{y}^\alpha (u+te_\beta))\\
&=\frac{d}{dt}|_{t=0}(\textbf{y}^\alpha(u)+t\delta^\alpha_\beta)=\delta^\alpha_\beta.
\end{align*}
\end{proof}
From the above lemma we result that if $X=X^\alpha e_\alpha\in\Gamma(E)$, then the vertical lift $X^\vee$ has the locally expression
\[
X^\vee=(X^\alpha\circ\pi)\frac{\partial}{\partial\textbf{y}^\alpha}.
\]
Using the locally expression of $X^\vee$ we can deduce
\begin{lemma}
If $X$, $Y$ are sections of $E$ and $f\in C^\infty(M)$, then
\[
(X+Y)^\vee=X^\vee+Y^\vee,\ \ \ (fX)^\vee=f^\vee X^\vee,\ \ \ X^\vee f^\vee=0.
\]
\end{lemma}
The complete lift of a smooth function $f\in C^{\infty}(M)$ into $C^{\infty}(E)$ is the smooth function
\[
f^c:E\longrightarrow\mathbb{R},\ \ \ \ f^c(u)=d^Ef(u)=\rho(u)f.
\]
In the local basis we have
\begin{align*}
f^c(u)&=f^c(u^\alpha e_\alpha)=\rho(u^\alpha e_\alpha)(f)=u^\alpha\rho(e_\alpha)(f)=u^\alpha\rho^i_\alpha\frac{\partial f}{\partial x^i}\\
&=(\textbf{y}^\alpha((\rho^i_\alpha\frac{\partial f}{\partial x^i})\circ\pi))(u),
\end{align*}
i.e.,
\begin{equation}\label{cf}
f^c|_{\pi^{-1}(U)}=\textbf{y}^\alpha((\rho^i_\alpha\frac{\partial f}{\partial x^i})\circ\pi).
\end{equation}
\begin{lemma}
If $X$ is a section on $E$ and $f, g\in C^\infty(M)$, then
\[
(i)\ (f+g)^c=f^c+g^c,\ \ \ (ii)\ (fg)^c=f^cg^\vee+f^\vee g^c,\ \ \ (iii)\ X^\vee f^c=(\rho(X)f)^\vee.
\]
\end{lemma}
\begin{proof}
The proof of (i) is obvious. Thus we only prove (ii) and (iii). Using the definition of $f^c$ we get
\begin{align*}
(fg)^c(u)&=\rho(u)(fg)=(\rho(u)f)(g\circ\pi)(u)+(f\circ\pi)(u)(\rho(u)g)\\
&=f^c(u)g^\vee(u)+f^\vee(u)g^c(u).
\end{align*}
So we have (ii). Using the locally expressions, we obtain
\begin{align*}
(X^\vee f^c)(u)&=[(X^\alpha\circ\pi)\frac{\partial}{\partial\textbf{y}^\alpha}(\textbf{y}^\beta((\rho^i_\beta\frac{\partial f}{\partial x^i})\circ\pi))](u)=[(X^\alpha\rho^i_\alpha\frac{\partial f}{\partial x^i})\circ\pi](u)\\
&=((\rho(X)f)\circ\pi)(u)=(\rho(X)f)^\vee(u).
\end{align*}
\end{proof}
Let $X$ be a section on $E$. Then there exist a unique vector field $X^c$ on $E$, the complete lift of $X$, satisfying the following conditions:

\textbf{i}) $X^c$ is $\pi$-projectable on $\rho(X)$,

\textbf{ii}) $X^c(\hat{\alpha})=\widehat{\pounds_{X}^{E}\alpha}$,\\
where $\alpha \in \Gamma(E^*)$. It is known that $X^c$ has the following coordinate expression\cite{GU}, \cite{GU1}:
\begin{equation}\label{esa}
X^c=\{(X^\alpha \rho_{\alpha}^{i})\circ \pi\} \frac{\partial}{\partial \textbf{x}^i}+\textbf{y}^\beta \{(\rho_{\beta}^{j}\frac{\partial X^\alpha}{\partial x^j}-X^\gamma L^\alpha_{\gamma\beta})\circ\pi\} \frac{\partial}{\partial \textbf{y}^\alpha}.
\end{equation}
\begin{lemma}
Let $X$ be a section of $E$. Then
\[
X^cf^c=(\rho(X)f)^c,\ \ \ \forall f\in C^\infty(M).
\]
\end{lemma}
\begin{proof}
Using (\ref{cf}) we get
\begin{equation}\label{esa1}
(\rho(X)f)^c=(X^\alpha\rho^i_\alpha\frac{\partial f}{\partial x^i})^c=\textbf{y}^\beta\{(\rho^j_\beta\frac{\partial}{\partial x^j}(X^\alpha\rho^i_\alpha\frac{\partial f}{\partial x^i}))\circ\pi\}.
\end{equation}
Again (\ref{cf}) and (\ref{esa}) give us
\begin{align*}
X^cf^c&=\{ (X^\alpha \rho_{\alpha}^{i})\circ \pi\}\frac{\partial}{\partial \textbf{x}^i}\{ \textbf{y}^\gamma((\rho _{\gamma}^{j}\frac{\partial f}{\partial x^j})\circ \pi) \}\\
&\ \ \ + \textbf{y}^\beta \{(\rho_{\beta}^{j}\frac{\partial X^\alpha}{\partial x^j}- X^{\gamma}L_{\gamma \beta}^{\alpha})\circ \pi \}((\rho_{\alpha}^{i}\frac{\partial f}{\partial x^i})\circ \pi).
\end{align*}
It is easy to see that $\pi_*(\frac{\partial}{\partial \textbf{x}^i})=\frac{\partial}{\partial x^i}$ which gives us $\frac{\partial}{\partial \textbf{x}^i}(f \circ \pi)=\frac{\partial f}{\partial x^i}\circ \pi$ for all $f \in C^\infty (M)$. Thus from the above equation one can deduce the following
\begin{align*}
X^cf^c&=\textbf{y}^\beta\{ \big(X^\alpha \rho_{\alpha}^{i}\rho_{\beta}^{j}\frac{\partial^2f}{\partial x^i \partial x^j}+X^\alpha \frac{\partial f}{\partial x^i}(\rho_{\alpha}^{j}\frac{\partial \rho_{\beta}^{i}}{\partial x^j}-\rho_{\gamma}^{i}L_{\alpha\beta}^{\gamma})\\
&\ \ \ +\rho_{\alpha}^{i}\rho_{\beta}^{j}\frac{\partial f}{\partial x^i}\frac{\partial X^\alpha}{\partial x^j}\big)\circ \pi \}.
\end{align*}
Using $(i)$ of (\ref{2}), the above relation and (\ref{esa1}) yields
\begin{equation}
X^cf^c=\textbf{y}^\beta\{ \big(\rho_{\beta}^{j}\frac{\partial}{\partial x^j}(X^\alpha \rho_{\alpha}^{i}\frac{\partial f}{\partial x^i}) \big)\circ\pi \}=(\rho(X)f)^c.
\end{equation}
\end{proof}
\begin{cor}
Let $X$ be a section of $E$. Then
\[
X^cf^\vee=(\rho(X)f)^\vee,\ \ \ \forall f\in C^\infty(M).
\]
\end{cor}
\begin{proof}
Using the above lemma, we obtain
\[
\frac{1}{2}X^c(f^2)^c=X^c(f^cf^\vee)=(X^cf^c)f^\vee+f^c(X^cf^\vee)=(\rho(X)f)^cf^\vee+f^c(X^cf^\vee).
\]
In other hand, we deduce
\[
\frac{1}{2}X^c(f^2)^c=\frac{1}{2}(\rho(X)f^2)^c=(f\rho(X)f)^c=f^c(\rho(X)f)^\vee+f^\vee(\rho(X)f)^c.
\]
The above equations give us $X^cf^\vee=(\rho(X)f)^\vee$.
\end{proof}
Using the locally expressions of vertical and complete lifts we have
\begin{lemma}
If $X$ and $Y$ are sections of $E$, then
\[
[X^c, Y^c]=[X, Y]^c_E,\ \ \ [X^c, Y^\vee]=[X, Y]^\vee_E,\ \ \ [X^\vee, Y^\vee]=0.
\]
\end{lemma}
\section{The Prolongation of a Lie algebroid}
In this section we will recall the notion of the prolongation of a Lie algebroid and we will consider a Lie algebroid structure on it. We also study the vertical and complete lifts on the prolongation of a Lie algebroid.

Let $\pounds^\pi E$ be the subset of $E\times TE$ defined by $\pounds^\pi E=\{(u, z)\in E\times TE|\rho(u)=\pi_*(z)\}$ and
denote by $\pi_\pounds: \pounds^\pi E\rightarrow E$ the mapping given by $\pi_\pounds(u, z)=\pi_E(z)$, where $\pi_E: TE\rightarrow E$ is the natural projection. Then $(\pounds^\pi E, \pi_\pounds, E)$ is a vector bundle over $E$ of rank $2n$. Indeed, the total space of the prolongation is the total space of the pull-back of $\pi_*:TE\rightarrow TM$ by the anchor map $\rho$.

We introduce the vertical subbundle
\[
v\pounds^\pi E=\ker\tau_\pounds=\{(u, z)\in \pounds^\pi E|\tau_\pounds(u, z)=0\},
\]
where $\tau_\pounds:\pounds^\pi E\rightarrow E$ is the projection onto the first factor, i.e., $\tau_\pounds(u, z)=u$. Therefore an element of $v\pounds^\pi E$ is of the form $(0,z)\in E\times TE$ such that $\pi_*(z)=0$ which is called vertical. Since $\pi_*(z)=0$ and $\ker\pi_*=vE$ ($\pi_*:TE\rightarrow TM$), then we deduce that if $(0, z)$ is vertical then $z$ is a vertical vector on $E$.

For local basis $\{e_\alpha\}$ of sections of $E$ and coordinates $(\textbf{x}^i, \textbf{y}^\alpha)$ on $E$, we have local  coordinates $(\textbf{x}^i, \textbf{y}^\alpha, k^\alpha, z^\alpha)$ on $\pounds^\pi E$ given as follows. If $(u, z)$ is an element of $\pounds^\pi E$, then by using $\rho(u)=\pi_*(z)$, $z$ has the form
\[
z=((\rho^i_\alpha u^\alpha)\circ\pi){\frac{\partial}{\partial\textbf{x}^i}}|_v+z^\alpha{\frac{\partial}{\partial\textbf{y}^\alpha}}|_v,\ \ \ z\in T_vE.
\]
The local basis $\{\mathcal{X}_\alpha, \mathcal{V}_\alpha\}$ of sections of $\pounds^\pi E$ associated to the coordinate system is given by
\begin{equation}\label{revise2}
\mathcal{X}_\alpha(v)=(e_\alpha(\pi(v)), (\rho^i_\alpha\circ\pi)\frac{\partial}{\partial\textbf{x}^i}|_v),\ \ \ \mathcal{V}_\alpha(v)=(0, \frac{\partial}{\partial\textbf{y}^\alpha}|v).
\end{equation}
If $V$ is a section of $\pounds^\pi E$ wchich in coordinates writes
\[
V(x, y)=(\textbf{x}^i, \textbf{y}^\alpha, Z^\alpha(x, y), V^\alpha(x, y)),
\]
then the expression of $V$ in terms of base $\{\mathcal{X}_\alpha, \mathcal{V}_\alpha\}$ is \cite{M}
\[
V=Z^\alpha\mathcal{X}_\alpha+V^\alpha\mathcal{V}_\alpha.
\]
We may introduce the vertical lift $X^V$ and the complete lift $X^C$ of a section $X\in\Gamma(E)$ as the sections of $\pounds^\pi E\rightarrow E$ given by
\[
X^V(u)=(0, X^\vee(u)),\ \ \ X^C(u)=(X(\pi(u)), X^c(u)),\ \ \ u\in E.
\]
Using the coordinate expressions of $X^\vee$ and $X^c$, the coordinate expressions of $X^V$ and $X^C$ as follows:
\begin{equation}\label{esi}
X^V=(X^\alpha\circ\pi)\mathcal{V}_\alpha,\ \ \ X^C=(X^\alpha\circ\pi)\mathcal{X}_\alpha+\textbf{y}^\beta[(\rho^j_\beta\frac{\partial X^\alpha}{\partial x^j}-X^\gamma L^\alpha_{\gamma\beta})\circ\pi]\mathcal{V}_\alpha,
\end{equation}
where $X=X^\alpha e_\alpha\in\Gamma(E)$. In particular we have
\begin{equation}\label{2base}
e_\alpha^V=\mathcal{V}_\alpha.
\end{equation}
Here, we consider the anchor map $\rho_\pounds:\pounds^\pi E\rightarrow TE$ defined by $\rho_\pounds(u, z)=z$ and the bracket $[., .]_\pounds$ satisfying the relations
\begin{equation}\label{emroz}
[X^V, Y^V]_\pounds=0,\ \ \ [X^V, Y^C]_\pounds=[X, Y]_E^V,\ \ \ [X^C, Y^C]_\pounds=[X, Y]^C_E,
\end{equation}
for $X, Y\in\Gamma(E)$. Then  this vector bundle $(\pounds^\pi E, \pi_\pounds, E)$ is a Lie algebroid with structure $([., .]_\pounds, \rho_\pounds)$.

Using (\ref{emroz}) we can deduce the following
\begin{lemma}
The Lie brackets of basis $\{\mathcal{X}_\alpha, \mathcal{V}_\alpha\}$ are
\[
[\mathcal{X}_\alpha, \mathcal{X}_\beta]_\pounds=(L^\gamma_{\alpha\beta}\circ\pi)\mathcal{X}_\gamma,\ \ \ [\mathcal{X}_\alpha, \mathcal{V}_\beta]_\pounds=0,\ \ \ [\mathcal{V}_\alpha, \mathcal{V}_\beta]_\pounds=0.
\]
\end{lemma}
\subsection{A setting for semispray on $\pounds^\pi E$}
A section of $\pi$ along smooth map $f:N\rightarrow M$ is a smooth map $\sigma:N\rightarrow E$ such that $\pi\circ\sigma=f$. The set of sections of $\pi$ along $f$ will be denoted by $\Gamma_f(\pi)$. Then there is a canonical isomorphism between $\Gamma(f^*\pi)$ and $\Gamma_f(\pi)$ (see \cite{S}). Now we consider pullback bundle $\pi^*\pi=(\pi^*E, pr_1, E)$ of vector bundle $(E, \pi, M)$, where
\[
\pi^*E:=E\times_ME:=\{(u, v)\in E\times E|\pi(u)=\pi(v)\},
\]
and $pr_1$ is the projection map onto the first component. The fibres of $\pi^*\pi$ are the $n$-dimensional real vector spaces
$$\lbrace u\rbrace\times E_{\pi(u)}\cong E_{\pi(u)}.$$
Therefore any section in $\Gamma(\pi^*\pi)$ is of the form
$$\bar{X}: u\in E\rightarrow \bar{X}(u)=(u,\b{X}(u)),$$
where $\b{X}:E\rightarrow E$ is a smooth map such that $\pi\circ\b{X}=\pi$. In these terms, the map
\[
\bar{X}\in\Gamma(\pi^*\pi)\rightarrow\b{X}\in\Gamma_\pi(\pi),
\]
is an isomorphism of $C^\infty(E)$-modules. Therefore we have
\[
\Gamma(\pi^*\pi)\cong\Gamma_\pi(\pi).
\]
In $\Gamma(\pi^*\pi)$, there is a distinguished section
\begin{equation}\label{delta}
\delta:u\in E\rightarrow\delta(u)=(u, u)\in \pi^*E,
\end{equation}
that called \textit{the canonical section along $\pi$}. This section corresponds to the
identity map $1_E$ under the isomorphism $\Gamma_\pi(\pi)\cong\Gamma(\pi^*\pi)$.

For any section $X$ on $E$, the map
\[
\widehat{X}:E\rightarrow\pi^*E,
\]
defined by $\widehat{X}(u)=(u, X\circ\pi(u))$ is a section of $\pi^*\pi$ , called the lift of X into $\Gamma(\pi^*\pi)$. $\widehat{X}$ may be identified with the map $X\circ\pi:E\rightarrow E$. It is easy to see that, $\{\widehat{X}|X\in\Gamma(E)\}$ generates locally the $C^\infty(E)$-module $\Gamma(\pi^*\pi)$.

We consider the following sequence
\begin{equation}\label{exact se}
0 \longrightarrow \pi^*(E)\stackrel{i}{\rightarrow}\pounds^\pi E\stackrel{j}{\rightarrow}\pi^*(E)\longrightarrow0,
\end{equation}
with $j(u, z)=(\pi_E(z), Id(u))=(v, u)$, $z\in T_vE$, and $i(u, v)=(0, v^\vee_u)$ where $v^\vee_u:C^\infty(E)\rightarrow\mathbb{R}$ is defined by $v^\vee_u(F)=\frac{d}{dt}|_{t=0}F(u+tv)$. Indeed we have $v^\vee_u=\frac{d}{dt}|_{t=0}(u+tv)$. Function $J=i\circ j:\pounds^\pi E\rightarrow\pounds^\pi E$ is called the {\it {vertical
endomorphism}} ({\it{almost tangent structure}}) of $\pounds^\pi E$.

From the definitions of $i$, $j$ and $J$ we get
\[
\im J=\im i=v\pounds^\pi E,\ \ \ \ker J=\ker j=v\pounds^\pi E,\ \ \ J\circ J=0.
\]
Moreover, $i$ is injective and $j$ is surjective. Therefore the sequence given by (\ref{exact se}) is exact sequence.
\begin{lemma}
Let $J$ be the vertical endomorphism of $\pounds^\pi E$. Then
\begin{equation}\label{vertical}
J(\mathcal{X}_\alpha)=\mathcal{V}_\alpha,\ \ \ J(\mathcal{V}_\alpha)=0.
\end{equation}
\end{lemma}
\begin{proof}
The definition of $J$ implies
\begin{align*}
J(\mathcal{X}_\alpha(v))&=i\circ j(e_\alpha(\pi(v)), (\rho^i_\alpha\circ\pi)\frac{\partial}{\partial\textbf{x}^i}|_v)=i(v, e_\alpha(\pi(v)))=(0, e_\alpha(\pi(v))^\vee_v)\\
&=(0, \frac{\partial}{\partial\textbf{y}^\alpha}|_v)=\mathcal{V}_\alpha(v).
\end{align*}
We also deduce
\[
J(\mathcal{V}_\alpha(v))=i\circ j(0, \frac{\partial}{\partial\textbf{y}^\alpha}|_v)=i(v, 0)=(0, 0).
\]
\end{proof}
\begin{cor}
Let $\{\mathcal{X}^\alpha, \mathcal{V}^\alpha\}$ be the corresponding dual basis of $\{\mathcal{X}_\alpha, \mathcal{V}_\alpha\}$. Then
\begin{equation}\label{vertical00}
J=\mathcal{V}_\alpha\otimes\mathcal{X}^\alpha.
\end{equation}
\end{cor}
Using the above corollary and (\ref{esi}) we obtain
\[
J(X^V)=0,\ \ \ J(X^C)=(X^\alpha\circ\pi)\mathcal{V}_\alpha=X^V.
\]
\begin{defn}
Let $\delta$ be the canonical section along $\pi$ given by (\ref{delta}). Then section $C$ given by
\[
C:=i\circ\delta,
\]
is called Liouville or Euler section.
\end{defn}
Using the definition of Liouville section we have
\begin{align*}
C(u)&=(i\circ\delta)(u)=i(u, u)=(0, u^\vee_u)=(0, (u^\alpha\circ\pi)\frac{\partial}{\partial\textbf{y}^\alpha})\\
&=(0, \textbf{y}^\alpha(u)\frac{\partial}{\partial\textbf{y}^\alpha})=(0, \textbf{y}^\alpha\frac{\partial}{\partial\textbf{y}^\alpha})(u),
\end{align*}
where $u=u^\alpha e_\alpha\in\Gamma(E)$. Therefore, the Liouville section $C$ has the coordinate expression
\begin{equation}\label{Liouville}
C=\textbf{y}^\alpha\mathcal{V}_\alpha,
\end{equation}
with respect to $\{\mathcal{X}_\alpha, \mathcal{V}_\alpha\}$. It is easy to prove the following
\begin{lemma}
Let $X$ be  asection of $E$. Then we have
\begin{equation}\label{Liover}
(i)\ [J, C]^{F-N}_\pounds=J,\ \ \ (ii)\ [X^V, C]_\pounds=X^V,\ \ \ (iii)\ JC=0.
\end{equation}
\end{lemma}
\begin{defn}
Section $\widetilde{X}$ of  vector bundle $(\pounds^\pi E, \pi_\pounds, E)$ is said to be homogenous of degree $r$, where $r$ is an integer, if $[C, \widetilde{X}]_\pounds=(r-1)\widetilde{X}$. Moreover, $\widetilde{f}\in C^\infty(E)$ is said to be homogenous of degree $r$ if $\pounds^\pounds_C\widetilde{f}=\rho_\pounds(C)(\widetilde{f})=r\widetilde{f}$.
\end{defn}
Now, let $\widetilde{X}=\widetilde{X}^\alpha\mathcal{X}_\alpha+\widetilde{Y}^{\alpha}\mathcal{V}_\alpha$. Then we obtain
\[
[C, \widetilde{X}]_\pounds=\textbf{y}^\alpha\frac{\partial \widetilde{X}^\beta}{\partial \textbf{y}^\alpha}\mathcal{X}_\beta+(\textbf{y}^\alpha\frac{\partial \widetilde{Y}^{\beta}}{\partial \textbf{y}^\alpha}-\widetilde{Y}^{\beta})\mathcal{V}_\beta.
\]
Thus $[C, \widetilde{X}]_\pounds=(r-1)\widetilde{X}$ if and only if
\begin{equation}\label{lh}
\textbf{y}^\alpha\frac{\partial \widetilde{X}^\beta}{\partial \textbf{y}^\alpha}=(r-1)\widetilde{X}^\beta,\ \ \ \textbf{y}^\alpha\frac{\partial \widetilde{Y}^{\beta}}{\partial \textbf{y}^\alpha}=r\widetilde{Y}^{\beta}.
\end{equation}
Therefore we have
\begin{lemma}\label{lemma0}
Section $\widetilde{X}=\widetilde{X}^\alpha\mathcal{X}_\alpha+\widetilde{Y}^{\alpha}\mathcal{V}_\alpha$ of $\pounds^\pi E$ is homogenous of degree $r$ if and only if (\ref{lh}) holds.
\end{lemma}
Now, let $\widetilde{f}\in C^\infty(E)$ be homogenous of degree 1. Then we have
\[
\pounds^\pounds_C\widetilde{f}=\rho_\pounds(C)\widetilde{f}=r\widetilde{f}.
\]
The above equation gives us
\[
\textbf{y}^\alpha\rho_\pounds(\mathcal{V}_\alpha)\widetilde{f}=\textbf{y}^\alpha\frac{\partial\widetilde{f}}{\partial\textbf{y}^\alpha}=r\widetilde{f}.
\]
Therefore we have
\begin{lemma}\label{lemmahom}
Real valued smooth function $\widetilde{f}$ on $E$ is homogenous of degree $r$  if and only if $\textbf{y}^\alpha\frac{\partial\widetilde{f}}{\partial\textbf{y}^\alpha}=r\widetilde{f}$.
\end{lemma}
\begin{defn}
A section $S$ of the vector bundle $(\pounds^\pi E, \pi_\pounds, E)$ is said to be a semispray if it satisfies
the condition $J(S)=C$. Moreover if $S$ is homogenous of degree 2, i.e., $[C, S]_\pounds=S$, then we call it spray.
\end{defn}
Let $S=A^\alpha\mathcal{X}_\alpha+S^\alpha\mathcal{V}_\alpha$ be a semispray on $\pounds^\pi E$. Then by using (\ref{vertical}) and (\ref{Liouville}) we deduce $A^\alpha=\textbf{y}^\alpha$. Therefore semispray $S$ has the following coordinate expression:
\begin{equation}\label{semispray}
S=\textbf{y}^\alpha\mathcal{X}_\alpha+S^\alpha\mathcal{V}_\alpha.
\end{equation}
Moreover, from the above lemma we deduce that  $S$ is a spray if and only if
\begin{equation}
2S^\beta=\textbf{y}^\alpha\frac{\partial S^\beta}{\partial\textbf{y}^\alpha}.
\end{equation}
Using (\ref{cf}) and (\ref{semispray}), it is easy to see that
\begin{equation}\label{Best2}
\rho_\pounds(S)(f^v)=f^c.
\end{equation}
\begin{lemma}
Let $S_1$ be a spray on $\pounds^\pi E$ and $\widetilde{f}:E\rightarrow\mathbb{R}$ be a smooth function on $E-\{0\}$. Then $S_2=S_1+\widetilde{f}C$ is a spray on $\pounds^\pi E$ if and only if $\widetilde{f}$ is homogenous of degree 1.
\end{lemma}
\begin{proof}
Let $\widetilde{f}$ be a homogenous function of of degree 1. Then we have $\rho_\pounds(C)\widetilde{f}=\widetilde{f}$. In other hand, since $S_1$ is a spray on $\pounds^\pi E$ then we have $JS_1=C$ and $[C, S_1]_\pounds=S_1$. Therefore
\[
JS_2=JS_1+\widetilde{f}JC=C,
\]
and
\[
[C, S_2]_\pounds=[C, S_1+\widetilde{f}C]_\pounds=[C, S_1]_\pounds+[C, \widetilde{f}C]_\pounds=S_1+(\rho_\pounds(C)\widetilde{f})C=S_1+\widetilde{f}C=S_2.
\]
Thus $S_2$ is a spray on $\pounds^\pi E$. Conversely, let $S_2$ be a spray on $\pounds^\pi E$. Then we have
\[
S_1+\widetilde{f}C=S_2=[C, S_2]_\pounds=S_1+(\rho_\pounds(C)\widetilde{f})C.
\]
Thus we get $\rho_\pounds(C)\widetilde{f}=\widetilde{f}$, i.e., $C$ is homogenous of degree 1.
\end{proof}
The spray $S_2$ given in the above lemma is said to be {\it projective change of $S_1$ by $\widetilde{f}$}.
\begin{defn}
A Lie symmetry of semispray $S$ is a section $X$ of $E$ such that $[S, X^C]_\pounds=0$. Moreover a dynamical symmetry of semispray $S$ is a section $\widetilde{X}$ of $\pounds^\pi E$ such that $[S, \widetilde{X}]_\pounds=0$.
\end{defn}
\begin{proposition}
A section $X=X^\alpha e_\alpha$ of $E$ is a Lie symmetry of $S$ if and only if
\[
\textbf{y}^\beta \textbf{y}^\lambda(\rho^i_\lambda\circ\pi)\frac{\partial(X^\alpha_{|_\beta}\circ\pi)}{\partial\textbf{x}^i}-((X^\lambda\rho^i_\lambda)\circ\pi)\frac{\partial S^\alpha}{\partial\textbf{x}^i}+S^\lambda(X^\alpha_{|_\lambda}\circ\pi)-y^\beta(X^\lambda_{|_\beta}\circ\pi)\frac{\partial S^\alpha}{\partial\textbf{y}^\lambda}=0,
\]
where $X^\alpha_{|_\beta}:=\rho^j_\beta\frac{\partial X^\alpha}{\partial x^j}-X^\gamma L^\alpha_{\gamma\beta}$.
\end{proposition}
\begin{proof}
Using (\ref{esi}) and (\ref{semispray}) we obtain
\begin{align*}
[S, X^C]_\pounds&=[\textbf{y}^\alpha\mathcal{X}_\alpha+S^\alpha\mathcal{V}_\alpha, (X^\lambda\circ\pi)\mathcal{X}_\lambda+\textbf{y}^\beta(X^\lambda_{|_\beta}\circ\pi)\mathcal{V}_\lambda]_\pounds\\
&=\{\textbf{y}^\lambda\rho^i_\lambda\frac{\partial(X^\alpha\circ\pi)}{\partial\textbf{x}^i}+\textbf{y}^\sigma((X^\lambda L^\alpha_{\sigma\lambda})\circ\pi)-\textbf{y}^\beta(X^\alpha_{|_\beta}\circ\pi)\}\mathcal{X}_\alpha\\
&\ \ +\{\textbf{y}^\beta \textbf{y}^\lambda(\rho^i_\lambda\circ\pi)\frac{\partial(X^\alpha_{|_\beta}\circ\pi)}{\partial\textbf{x}^i}-((X^\lambda\rho^i_\lambda)\circ\pi)\frac{\partial S^\alpha}{\partial\textbf{x}^i}+S^\lambda(X^\alpha_{|_\lambda}\circ\pi)\\
&\ \ -y^\beta(X^\lambda_{|_\beta}\circ\pi)\frac{\partial S^\alpha}{\partial\textbf{y}^\lambda}\}\mathcal{V}_\alpha.
\end{align*}
Using direct calculation we deduce that the coefficient of $\mathcal{X}_\alpha$ vanishes. Therefore $[S, X^C]_\pounds=0$ if and only if the coefficient of $\mathcal{V}_\alpha$ is zero.
\end{proof}
\begin{proposition}
A section $\widetilde{X}=\widetilde{X}^\alpha\mathcal{X}_\alpha+\widetilde{Y}^\alpha\mathcal{V}_\alpha$ of $\pounds^\pi E$ is dynamical symmetry of $S$ if and only if
\[
\widetilde{X}^\alpha_{||}=\widetilde{Y}^\alpha,\ \ \ \textbf{y}^\beta\rho^i_\beta\frac{\partial \widetilde{Y}^\alpha}{\partial\textbf{x}^i}-\widetilde{X}^\beta\rho^i_\beta\frac{\partial S^\alpha}{\partial\textbf{x}^i}+S^\beta\frac{\partial \widetilde{Y}^\alpha}{\partial\textbf{y}^\beta}-\widetilde{Y}^\beta\frac{\partial S^\alpha}{\partial\textbf{y}^\beta}=0,
\]
where $\widetilde{X}^\alpha_{||}:=\rho_\pounds(S)(\widetilde{X}^\alpha)+\widetilde{X}^\beta\textbf{y}^\gamma L^\alpha_{\gamma\beta}=\textbf{y}^\beta\rho^i_\beta\frac{\partial \widetilde{X}^\alpha}{\partial\textbf{x}^i}+S^\beta\frac{\partial \widetilde{X}^\alpha}{\partial\textbf{y}^\beta}+\widetilde{X}^\beta\textbf{y}^\gamma L^\alpha_{\gamma\beta}$.
\end{proposition}
\begin{proof}
Using (\ref{semispray}) we obtain
\begin{align*}
[S, \widetilde{X}]_\pounds&=[\textbf{y}^\alpha\mathcal{X}_\alpha+S^\alpha\mathcal{V}_\alpha, \widetilde{X}^\beta\mathcal{X}_\beta+\widetilde{Y}^\beta\mathcal{V}_\beta]_\pounds\\
&=\{\textbf{y}^\beta\rho^i_\beta\frac{\partial \widetilde{X}^\alpha}{\partial\textbf{x}^i}+S^\beta\frac{\partial \widetilde{X}^\alpha}{\partial\textbf{y}^\beta}+\widetilde{X}^\beta\textbf{y}^\gamma L^\alpha_{\gamma\beta}-\widetilde{Y}_\alpha\}\mathcal{X}_\alpha\\
&\ \ +\{\textbf{y}^\beta\rho^i_\beta\frac{\partial \widetilde{Y}^\alpha}{\partial\textbf{x}^i}-\widetilde{X}^\beta\rho^i_\beta\frac{\partial S^\alpha}{\partial\textbf{x}^i}+S^\beta\frac{\partial \widetilde{Y}^\alpha}{\partial\textbf{y}^\beta}-\widetilde{Y}^\beta\frac{\partial S^\alpha}{\partial\textbf{y}^\beta}\}\mathcal{V}_\alpha.
\end{align*}
Therefore $[S, \widetilde{X}]_\pounds=0$ if and only if the coefficients of $\mathcal{X}_\alpha$ and $\mathcal{V}_\alpha$ are zero.
\end{proof}
\section{Horizontal lift on $\pounds^\pi E$}
In this section we introduce horizontal endomorphisms to decompose  $\pounds^\pi E$ to horizontal and vertical subbundles. Then we consider torsions, tension and curvature of a horizontal endomorphism. Moreover, some new results are obtained on horizontal endomorphisms and using a horizontal endomorphism, the horizontal lift of a section of $E$ to $\pounds^\pi E$ is constructed.
\subsection{Horizontal endomorphism}
\begin{defn}
A function $h:\pounds^\pi E\rightarrow \pounds^\pi E$ is called a horizontal endomorphism if $h\circ h=h$, $\ker h=v\pounds^\pi E$ and $h$ is smooth on $\stackrel{\circ}{\pounds^\pi E}=\pounds^\pi E-\{0\}$. Also, $v:=Id-h$ is called vertical projector associated to $h$.
\end{defn}
Setting $h\pounds^\pi E:=\im h$ we have the following splitting for $\pounds^\pi E$:
\begin{equation}\label{good}
\pounds^\pi E=v\pounds^\pi E\oplus h\pounds^\pi E.
\end{equation}
From the above relation we deduce $\im v=v\pounds^\pi E$. Thus, using $\ker J=v\pounds^\pi E$, we obtain
\[
0=Jv\widetilde{X}=J(\widetilde{X}-h\widetilde{X})=J\widetilde{X}-Jh\widetilde{X},\ \ \ \widetilde{X}\in\pounds^\pi E.
\]
Also, from the definition of the horizontal endomorphism we have
\[
\ker h=\im J=\ker J=\im v=v\pounds^\pi E.
\]
\begin{equation}\label{Jh}
(i)\ hJ=hv=Jv=0,\ \ (ii)\ v\circ v=v,\ \ (iii)\ vh=0,\ \ (iv)\ Jh=J=vJ.
\end{equation}
{\textbf{Coordinate expression of $h$.}} Let
\[
h=(\mathcal{A}^\alpha_\beta\mathcal{X}_\alpha+\mathcal{B}^\alpha_\beta\mathcal{V}_\alpha)\otimes\mathcal{X}^\beta
+(\mathcal{C}^\alpha_\beta\mathcal{X}_\alpha+\mathcal{D}^\alpha_\beta\mathcal{V}_\alpha)\otimes\mathcal{V}^\beta.
\]
Then using $\ker h=v\pounds^\pi E$, we have
\[
0=h(\mathcal{V}_\gamma)=(\mathcal{C}^\alpha_\gamma\mathcal{X}_\alpha+\mathcal{D}^\alpha_\gamma\mathcal{V}_\alpha).
\]
Therefore $\mathcal{C}^\alpha_\gamma=\mathcal{D}^\alpha_\gamma=0$. Also (iv) of (\ref{Jh}) yields
\[
\mathcal{V}_\gamma=J(\mathcal{X}_\gamma)=Jh(\mathcal{X}_\gamma)
=J(\mathcal{A}^\alpha_\gamma\mathcal{X}_\alpha+\mathcal{B}^\alpha_\gamma\mathcal{V}_\alpha)=\mathcal{A}^\alpha_\gamma\mathcal{V}_\alpha.
\]
Therefore we have $\mathcal{A}^\alpha_\gamma=\delta^\alpha_\gamma$. Hence $h$ has the following locally expression:
\begin{equation}\label{horizontal end}
h=(\mathcal{X}_\beta+\mathcal{B}^\alpha_\beta\mathcal{V}_\alpha)\otimes\mathcal{X}^\beta.
\end{equation}
\begin{defn}
For $k\in\mathbb{N}$, $K\in\Gamma(\wedge^kE^*\otimes E)$ is called semibasic if
\[
JoK=0,\ \ \ \ i_{JX}K=0,\ \ \forall X\in\Gamma(E).
\]
\end{defn}
\begin{defn}
Let $h$ be a horizontal endomorphism on $\pounds^\pi E$. Then $H=[h, C]^{F-N}_\pounds:\pounds^\pi E\rightarrow\pounds^\pi E$ is called the tension of $h$, where $[h, C]^{F-N}_\pounds$ is the generalized Fr\"{o}licher-Nijenhuis bracket on $\pounds^\pi E$. If $H=0$, then $h$ is called homogeneous.
\end{defn}
Using (\ref{F3}), (\ref{Liouville}) and (\ref{horizontal end}), we obtain
\begin{align*}
H(\mathcal{X}_\lambda)&=[h, C]^{F-N}_\pounds(\mathcal{X}_\lambda)=[h(\mathcal{X}_\lambda), C]_\pounds-h[\mathcal{X}_\lambda, C]_\pounds\\
&=\mathcal{B}^\alpha_\lambda\rho_\pounds(\mathcal{V}_\alpha)(\textbf{y}^\gamma)\mathcal{V}_\gamma-\textbf{y}^\gamma\rho_\pounds(\mathcal{V}_\gamma)(\mathcal{B}^\alpha_\lambda)\mathcal{V}_\alpha\\
&=(\mathcal{B}^\alpha_\lambda-\textbf{y}^\gamma\frac{\partial\mathcal{B}^\alpha_\lambda}{\partial\textbf{y}^\gamma})\mathcal{V}_\alpha,\\
H(\mathcal{V}_\lambda)&=[h, C]^{F-N}_\pounds(\mathcal{V}_\lambda)=[h(\mathcal{V}_\lambda), C]_\pounds-h[\mathcal{V}_\lambda, C]_\pounds=0.
\end{align*}
Using the above equations $H$ has the coordinate expression
\begin{equation}\label{tension}
H=(\mathcal{B}^\alpha_\beta-\textbf{y}^\gamma\frac{\partial\mathcal{B}^\alpha_\beta}{\partial\textbf{y}^\gamma})\mathcal{V}_\alpha\otimes\mathcal{X}^\beta.
\end{equation}
Since $J(\mathcal{V}_\alpha)=0$, then we obtain $J\circ H=0$ and $i_{J\widetilde{X}}H=0$, where $\widetilde{X}\in\Gamma(\pounds^\pi E)$. Therefore $H$ is semibasic.

From (\ref{tension}) we have
\begin{lemma}
The horizontal endomorphism $h$ is homogeneous if and if
\[
\mathcal{B}^\alpha_\beta=\textbf{y}^\gamma\frac{\partial\mathcal{B}^\alpha_\beta}{\partial\textbf{y}^\gamma}.
\]
\end{lemma}
\begin{defn}
Let $h$ be a horizontal endomorphism on $\pounds^\pi E$. Then $t=[J, h]^{F-N}_\pounds\in\Gamma(\pounds^\pi E)$ is called the weak torsion of $h$.
\end{defn}
Using (\ref{F4}), (\ref{vertical}), (\ref{horizontal end}) and (i), (iv) of (\ref{Jh}) we obtain
\begin{align*}
t(\mathcal{X}_\alpha, \mathcal{X}_\beta)&=[J(\mathcal{X}_\alpha), h(\mathcal{X}_\beta)]_\pounds+[h(\mathcal{X}_\alpha), J(\mathcal{X}_\beta)]_\pounds+J[\mathcal{X}_\alpha, \mathcal{X}_\beta]_\pounds-J[\mathcal{X}_\alpha, h(\mathcal{X}_\beta)]_\pounds\\
&\ \ -J[h(\mathcal{X}_\alpha), \mathcal{X}_\beta]_\pounds-h[\mathcal{X}_\alpha, J(\mathcal{X}_\beta)]_\pounds-h[J(\mathcal{X}_\alpha), \mathcal{X}_\beta]_\pounds\\
&=\rho_\pounds(\mathcal{V}_\alpha)(\mathcal{B}^\gamma_\beta)\mathcal{V}_\gamma-\rho_\pounds(\mathcal{V}_\beta)(\mathcal{B}^\gamma_\alpha)\mathcal{V}_\gamma
-(L^\gamma_{\alpha\beta}\circ\pi)\mathcal{V}_\gamma\\
&=[\frac{\partial\mathcal{B}^\gamma_\beta}{\partial\textbf{y}^\alpha}-\frac{\partial\mathcal{B}^\gamma_\alpha}{\partial\textbf{y}^\beta}
-(L^\gamma_{\alpha\beta}\circ\pi)]\mathcal{V}_\gamma,
\end{align*}
and
\[
t(\mathcal{X}_\alpha, \mathcal{V}_\beta)=t(\mathcal{V}_\alpha, \mathcal{V}_\beta)=0.
\]
Therefore we have
\begin{equation}\label{wt}
t=\frac{1}{2}t^\gamma_{\alpha\beta}\mathcal{X}^\alpha\wedge\mathcal{X}^\beta\otimes\mathcal{V}_\gamma,
\end{equation}
where
\begin{equation}\label{wt1}
t^\gamma_{\alpha\beta}:=\frac{\partial\mathcal{B}^\gamma_\beta}{\partial\textbf{y}^\alpha}-\frac{\partial\mathcal{B}^\gamma_\alpha}{\partial\textbf{y}^\beta}
-(L^\gamma_{\alpha\beta}\circ\pi).
\end{equation}
\begin{lemma}\label{torsion lemma}
The weak torsion $t$ is semibasic.
\end{lemma}
\begin{proof}
Since $J(\mathcal{V}_\alpha)=0$, then we deduce $J\circ t=0$. Also we have $i_{J\widetilde{X}}(\mathcal{X}^\alpha)=0$, for each $\widetilde{X}\in\Gamma(\pounds^\pi E)$. Therefore we obtain
\[
i_{J\widetilde{X}}t=\dfrac{1}{2}t_{\alpha\beta}^{\gamma}i_{J\widetilde{X}}(\mathcal{X}^\alpha\wedge\mathcal{X}^\beta)\otimes\mathcal{V}_\gamma
=\dfrac{1}{2}t_{\alpha\beta}^{\gamma}(i_{J\widetilde{X}}(\mathcal{X}^\alpha)\wedge\mathcal{X}^\beta-\mathcal{X}^\alpha\wedge i_{J\widetilde{X}}(\mathcal{X}^\beta))\otimes\mathcal{V}_\gamma=0.
\]
Therefore $t$ is semibasic.
\end{proof}
\begin{defn}
The strong torsion of $h$ is defined by $T=i_St+H$.
\end{defn}
\begin{lemma}
The strong torsion $T$ has the following coordinate expression:
\begin{equation}\label{st}
T=(\mathcal{B}^\alpha_\beta-\textbf{y}^\gamma\frac{\partial\mathcal{B}^\alpha_\gamma}{\partial\textbf{y}^\beta}
-\textbf{y}^\gamma(L^\alpha_{\gamma\beta}\circ\pi))\mathcal{V}_\alpha\otimes\mathcal{X}^\beta.
\end{equation}
\end{lemma}
\begin{proof}
Using (\ref{tension}) we get
\begin{equation}\label{st1}
T(\mathcal{X}_\lambda)=(i_St)(\mathcal{X}_\lambda)+H(\mathcal{X}_\lambda)=(i_St)(\mathcal{X}_\lambda)+(\mathcal{B}^\gamma_\lambda
-\textbf{y}^\alpha\frac{\partial\mathcal{B}^\gamma_\lambda}{\partial\textbf{y}^\alpha})\mathcal{V}_\gamma.
\end{equation}
But using (\ref{semispray}) and (\ref{wt}) we obtain
\begin{equation}\label{4.5}
i_St=\frac{1}{2}t^\gamma_{\alpha\beta}(\textbf{y}^\alpha\mathcal{X}^\beta-\textbf{y}^\beta\mathcal{X}^\alpha)\otimes\mathcal{V}_\gamma.
\end{equation}
Thus
\[
(i_St)(\mathcal{X}_\lambda)=\textbf{y}^\alpha t^\gamma_{\alpha\lambda}\mathcal{V}_\gamma=
\textbf{y}^\alpha(\frac{\partial\mathcal{B}^\gamma_\lambda}{\partial\textbf{y}^\alpha}-\frac{\partial\mathcal{B}^\gamma_\alpha}{\partial\textbf{y}^\lambda}
-(L^\gamma_{\alpha\lambda}\circ\pi))\mathcal{V}_\gamma.
\]
Setting the above equation in (\ref{st1}) we obtain (\ref{st}).
\end{proof}
It is easy to see that $J\circ T=0$ and $i_{J\widetilde{X}}T=0$, for each $\widetilde{X}\in\Gamma(\pounds^\pi E)$. Thus $T$ is semibasic.
\begin{defn}
The curvature of a horizontal endomorphism $h$ is defined by $\Omega=-N_h$, where $N_h$ is the Nijenhuis tensor of $h$ given by
\[
N_h(\widetilde{X}, \widetilde{Y})=[h\widetilde{X}, h\widetilde{Y}]-h[h\widetilde{X}, \widetilde{Y}]-h[\widetilde{X}, h\widetilde{Y}]+h[\widetilde{X}, \widetilde{Y}],\ \ \ \forall \widetilde{X}, \widetilde{Y}\in\Gamma(\pounds^\pi E).
\]
\end{defn}
\begin{lemma}
For sections $\widetilde{X}$ and $\widetilde{Y}$ of $\pounds^\pi E$ we have
\begin{equation}\label{cur1}
\Omega(\widetilde{X}, \widetilde{Y})=\Omega(h\widetilde{X}, h\widetilde{Y})=-v[h\widetilde{X}, h\widetilde{Y}]_\pounds.
\end{equation}
\end{lemma}
\begin{proof}
At first it is easy to check that $[v\widetilde{X}, v\widetilde{Y}]_\pounds\in v\pounds^\pi E$. Thus using $\ker h=v\pounds^\pi E$ and $hv=0$ we get
\[
\Omega(v\widetilde{X}, v\widetilde{Y})=-N_h(v\widetilde{X}, v\widetilde{Y})=-h[v\widetilde{X}, v\widetilde{Y}]_\pounds=0.
\]
Also it is easy to that $N_h(h\widetilde{X}, v\widetilde{Y})=0$ and consequently $\Omega(h\widetilde{X}, v\widetilde{Y})=0$. Therefore we obtain
\begin{align*}
\Omega(\widetilde{X}, \widetilde{Y})&=\Omega(h\widetilde{X}+v\widetilde{X}, h\widetilde{Y}+v\widetilde{Y})=\Omega(h\widetilde{X}, h\widetilde{Y})\\
&=-[h\widetilde{X}, h\widetilde{Y}]_\pounds+h[h\widetilde{X}, h\widetilde{Y}]_\pounds+h[h\widetilde{X}, h\widetilde{Y}]_\pounds-h[h\widetilde{X}, h\widetilde{Y}]_\pounds\\
&=-v[h\widetilde{X}, h\widetilde{Y}]_\pounds.
\end{align*}
\end{proof}
\begin{proposition}
The curvature $\Omega$ has the following coordinate expression:
\begin{equation}\label{curv000}
\Omega=-\frac{1}{2}R^\gamma_{\alpha\beta}\mathcal{X}^\alpha\wedge\mathcal{X}^\beta\otimes\mathcal{V}_\gamma,
\end{equation}
where
\begin{equation}\label{curv0}
R^\gamma_{\alpha\beta}=(\rho_\alpha^i\circ\pi)\frac{\partial{\mathcal{B}^\gamma_\beta}}{\partial{\textbf{x}^i}}
-(\rho_\beta^i\circ\pi)\frac{\partial{\mathcal{B}^\gamma_\alpha}}
{\partial{\textbf{x}^i}}+\mathcal{B}^\lambda_{\alpha}\frac{\partial{B^\gamma_\beta}}
{\partial{\textbf{y}^\lambda}}
-\mathcal{B}^\lambda_{\beta}\frac{\partial{\mathcal{B}^\gamma_\alpha}}{\partial{\textbf{y}^\lambda}}
+(L^\lambda_{\beta\alpha}\circ\pi)\mathcal{B}^\gamma_\lambda.
\end{equation}
\end{proposition}
\begin{proof}
Using (\ref{cur1}) we have
\begin{align}\label{cur}
\Omega(\mathcal{X}_{\alpha}, \mathcal{X}_{\beta})&=-v[h\mathcal{X}_\alpha, h\mathcal{X}_\beta]_\pounds
=-v[\mathcal{X}_{\alpha}+\mathcal{B}^{\lambda}_{\alpha}\mathcal{V}_\lambda,\mathcal{X}_{\beta}+\mathcal{B}^{\gamma}_{\beta}\mathcal{V}_\gamma]_\pounds\nonumber\\
&=-v\Big((L^\gamma_{\alpha\beta}\circ\pi)\mathcal{X}_\gamma+\rho_\pounds(\mathcal{X}_\alpha)(\mathcal{B}^\gamma_\beta)\mathcal{V}_\gamma
-\rho_\pounds(\mathcal{X}_\beta) (\mathcal{B}^\lambda_\alpha)\mathcal{V}_\lambda\nonumber\\
&\ \ \ +\mathcal{B}^\lambda_\alpha\rho_\pounds(\mathcal{V}_\lambda)(\mathcal{B}^\gamma_\beta)\mathcal{V}_\gamma
-\mathcal{B}^\gamma_\beta\rho_\pounds(\mathcal{V}_\gamma)(\mathcal{B}^\lambda_\alpha)\mathcal{V}_\lambda\Big)\nonumber\\
&=-\Big((\rho_\alpha^i\circ\pi)\frac{\partial{\mathcal{B}^\gamma_\beta}}
{\partial{\textbf{x}^i}}
-(\rho_\beta^i\circ\pi)\frac{\partial{\mathcal{B}^\gamma_\alpha}}{\partial{\textbf{x}^i}}+\mathcal{B}^\lambda_{\alpha}\frac{\partial{\mathcal{B}^\gamma_\beta}}{\partial{\textbf{y}^\lambda}}
-\mathcal{B}^\lambda_{\beta}\frac{\partial{\mathcal{B}^\gamma_\alpha}}
{\partial{\textbf{y}^\lambda}}\Big)v\mathcal{V}_\gamma\nonumber\\
&\ \ \ \ -(L^\gamma_{\alpha\beta}\circ\pi)v\mathcal{X}_\gamma.
\end{align}
Using $v=Id-h$ and (\ref{horizontal end}) we deduce that
\[
v\mathcal{V}_\alpha=\mathcal{V}_\alpha,\ \ \ v\mathcal{X}_\alpha=-\mathcal{B}^\beta_\alpha\mathcal{V}_\beta.
\]
Plugging the above equation into (\ref{cur}) yields
\begin{align*}
\Omega(\mathcal{X}_{\alpha}, \mathcal{X}_{\beta})&= \Big((L^\lambda_{\alpha\beta}\circ\pi)\mathcal{B}^\gamma_\lambda-(\rho_\alpha^i\circ\pi)\frac{\partial{\mathcal{B}^\gamma_\beta}}
{\partial{\textbf{x}^i}}
+(\rho_\beta^i\circ\pi)\frac{\partial{\mathcal{B}^\gamma_\alpha}}{\partial{\textbf{x}^i}}\\
&\ \ \ -\mathcal{B}^\lambda_{\alpha}\frac{\partial{\mathcal{B}^\gamma_\beta}}{\partial{\textbf{y}^\lambda}}
+\mathcal{B}^\lambda_{\beta}\frac{\partial{\mathcal{B}^\gamma_\alpha}}
{\partial{\textbf{y}^\lambda}}\Big)\mathcal{V}_\gamma=-R^\gamma_{\alpha\beta}\mathcal{V}_\gamma.
\end{align*}
Similarly, we have
\[
\Omega(\mathcal{X}_{\alpha}, \mathcal{V}_{\beta})=\Omega(\mathcal{V}_{\alpha}, \mathcal{V}_{\beta})=0.
\]
\end{proof}
Similar to the proof of Lemma \ref{torsion lemma}, we can prove the following
\begin{lemma}
The curvature $\Omega$ of horizontal endomorphism $h$ is semibasic.
\end{lemma}
\begin{proposition}
Let the horizontal endomorphism $h$ be given on $\pounds^\pi E$. If $S$ is an
arbitrary semispray of $\pounds^\pi E$, then $\bar{S}=hS$ is also a semispray of $\pounds^\pi E$ which does not
depend on the choice of $S$ . $\bar{S}$ is called the semispray associated to $h$.
\end{proposition}
\begin{proof}
Since $Jh=J$ then we have
\[
J\bar{S}=J(hS)=Jh(S)=JS=C.
\]
Thus $S'$ is a semispray. Now let $S'$ be an another semispray of $\pounds^\pi E$. Then we have
\[
J(S-S')=JS-JS'=C-C=0.
\]
Thus $S-S'\in\ker J=v\pounds^\pi E$, which gives us $0=h(S-S')=hS-hS'$, i.e., $hS=hS'$.
\end{proof}
\begin{proposition}
If the horizontal endomorphism $h$ is homogeneous, then the semispray associated to $h$ is spray.
\end{proposition}
\begin{proof}
Let $S$ be a semispray of $\pounds^\pi E$. Since $h$ is homogeneous, then we have
\begin{align}\label{tension0}
0=H(S)&=[h, C]^{F-N}_\pounds(S)=[hS, C]_\pounds -h[S, C]_\pounds\nonumber\\
&=[hS, C]_\pounds -h([S, C]_\pounds +S)+hS.
\end{align}
But we can obtain
\[
[S, C]_\pounds +S=(2S^\alpha-\textbf{y}^\beta\frac{\partial S^\alpha}{\partial\textbf{y}^\beta})\mathcal{V}_\alpha,
\]
and consequently $h([S, C]_\pounds +S)=0$. Plugging this equation into (\ref{tension0}) implies $[C, hS]_\pounds=hS$, i.e., $hS$ is a spray of $\pounds^\pi E$.
\end{proof}
\begin{lemma}
If $h_1$ and $h_2$ are horizontal endomorphisms on $\pounds^\pi E$, then $h_1-h_2\in v\pounds^\pi E$. Moreover
\begin{equation}\label{tension2}
J[(h_1-h_2)(\widetilde{X}), S]_\pounds=(h_1-h_2)(\widetilde{X}),\ \ \ \forall\widetilde{X}\in\Gamma(\pounds^\pi E).
\end{equation}
\end{lemma}
\begin{proof}
From (\ref{horizontal end}) we have
\[
h_1=(\mathcal{X}_\beta+{}^{h_1}\mathcal{B}^\alpha_\beta\mathcal{V}_\alpha)\otimes\mathcal{X}^\beta, \ \ h_2=(\mathcal{X}_\beta+{}^{h_2}\mathcal{B}^\alpha_\beta\mathcal{V}_\alpha)\otimes\mathcal{X}^\beta.
\]
Thus
\[
h_1-h_2=({}^{h_1}\mathcal{B}^\alpha_\beta-{}^{h_2}\mathcal{B}^\alpha_\beta)\mathcal{V}_\alpha\otimes\mathcal{X}^\beta.
\]
Now let $\widetilde{X}=\widetilde{X}^\beta\mathcal{X}_\beta+\widetilde{Y}^\beta\mathcal{V}_\beta\in\Gamma(\pounds^\pi E)$. Then we obtain
\[
(h_1-h_2)(\widetilde{X})=\widetilde{X}^\beta({}^{h_1}\mathcal{B}^\alpha_\beta
-{}^{h_2}\mathcal{B}^\alpha_\beta)\mathcal{V}_\alpha\in v\pounds^\pi E.
\]
Now, we prove the second part of the lemma. The above equation implies that
\begin{align*}
J[(h_1-h_2)(\widetilde{X}), S]_\pounds&=J[\widetilde{X}^\beta({}^{h_1}\mathcal{B}^\alpha_\beta
-{}^{h_2}\mathcal{B}^\alpha_\beta)\mathcal{V}_\alpha, \textbf{y}^\gamma\mathcal{X}_\gamma+S^\gamma\mathcal{V}_\gamma]_\pounds\\
&=\widetilde{X}^\beta({}^{h_1}\mathcal{B}^\alpha_\beta
-{}^{h_2}\mathcal{B}^\alpha_\beta)J(\mathcal{X}_\alpha)=\widetilde{X}^\beta({}^{h_1}\mathcal{B}^\alpha_\beta
-{}^{h_2}\mathcal{B}^\alpha_\beta)\mathcal{V}_\alpha\\
&=(h_1-h_2)(\widetilde{X}).
\end{align*}
\end{proof}
\begin{theorem}\label{mainth00}
If $h_1$ and $h_2$ are horizontal endomorphisms with same associated semispray and strong torsion, then $h_1=h_2$.
\end{theorem}
\begin{proof}
Let $K=h_1-h_2$. Since $Jh_1=Jh_2=J$ and $h_1J=h_2J=0$, then we obtain $J\circ K=0$ and $i_{J\widetilde{X}}K=K(J\widetilde{X})=0$, for each $\widetilde{X}\in\Gamma(\pounds^\pi E)$. Thus $K$ is a semibasic. since $h_1$ and $h_2$ have the same associated semisprays, then $h_1S=h_2S$, and consequently $KS=0$. But we have
\[
t_2=[J, h_2]^{F-N}_\pounds=[J, h_1]^{F-N}_\pounds+[J, K]^{F-N}_\pounds=t_1+[J, K]^{F-N}_\pounds.
\]
Similarly we obtain
\[
H_2=H_1+[K, C]^{F-N}_\pounds.
\]
The above equations give us
\[
T_2=i_St_2+H_2=T_1+i_S[J, K]^{F-N}_\pounds+[K, C]^{F-N}_\pounds.
\]
Since $T_1=T_2$, then from the above equation we deduce
\begin{equation}\label{tension1}
i_S[J, K]^{F-N}_\pounds(\widetilde{X})=-[K, C]^{F-N}_\pounds(\widetilde{X}),\ \ \ \forall\widetilde{X}\in\Gamma(\pounds^\pi E).
\end{equation}
Since $J\circ K=K\circ J=KS=0$ and $JS=C$, then using (\ref{F4}) we get
\begin{align*}
i_S[J, K]^{F-N}_\pounds(\widetilde{X})&=[J, K]^{F-N}_\pounds(S, \widetilde{X})\\
&=[C, K\widetilde{X}]_\pounds-J[S, K\widetilde{X}]_\pounds\\
&\ \ -K[S, J\widetilde{X}]_\pounds-K[JS, \widetilde{X}]_\pounds.
\end{align*}
Setting the above equation in (\ref{tension1}) and using (\ref{F3}) imply that
\[
J[S, K\widetilde{X}]_\pounds=K[J\widetilde{X}, S]_\pounds.
\]
Using (\ref{tension2}) and the above equation we obtain
\[
-K\widetilde{X}=J[S, K\widetilde{X}]_\pounds=K[J\widetilde{X}, S]_\pounds=K([J\widetilde{X}, S]_\pounds-\widetilde{X})+K\widetilde{X}.
\]
It is easy to see that $[J\widetilde{X}, S]_\pounds-\widetilde{X}\in v\pounds^\pi E$ and $v\pounds^\pi E\subset\ker K$. Thus $K([J\widetilde{X}, S]_\pounds-\widetilde{X})=0$. Therefore the above equation gives us $K\widetilde{X}=0$ and consequently $h_1=h_2$.
\end{proof}
\subsection{Almost complex structure on $\pounds^\pi E$}
Let $S$ be the semispray associated to $h$. We consider the map $F:\pounds^\pi E\rightarrow\pounds^\pi E$ given by $F:=h[S, h]^{F-N}_\pounds-J$. Since $J^2=0$ and $Jh=J$, then we have
\begin{equation}\label{al}
F^2=(h[S, h]^{F-N}_\pounds-J)^2=(h[S, h]^{F-N}_\pounds)^2-J[S, h]^{F-N}_\pounds-h[S, h]^{F-N}_\pounds J.
\end{equation}
But we have
\[
h[S, h]^{F-N}_\pounds\widetilde{X}=h(h[\widetilde{X}, S]_\pounds-[h\widetilde{X}, S]_\pounds)=h[\widetilde{X}-h\widetilde{X}, S]_\pounds=h[v\widetilde{X}, S]_\pounds.
\]
Therefore
\begin{equation}\label{al1}
(h[S, h]^{F-N}_\pounds)^2\widetilde{X}=h[v(h[v\widetilde{X}, S]_\pounds), S]_\pounds=0,
\end{equation}
because $vh=0$. In other hand, by a direct computation, we get
\begin{equation}\label{al2}
J[S, h]^{F-N}_\pounds\widetilde{X}+h[S, h]^{F-N}_\pounds J\widetilde{X}-\widetilde{X}=(J[v\widetilde{X}, S]_\pounds-v\widetilde{X})+(h[J\widetilde{X}, S]_\pounds-h\widetilde{X}).
\end{equation}
But we have
\begin{align}\label{al3}
J[v\widetilde{X}, S]_\pounds&=J[(\widetilde{Y}^\alpha-\widetilde{X}^\gamma\mathcal{B}^\alpha_\gamma)\mathcal{V}_\alpha, \textbf{y}^\beta\mathcal{X}_\beta+S^\beta\mathcal{V}_\beta]_\pounds\nonumber\\
&=(\widetilde{Y}^\alpha-\widetilde{X}^\gamma\mathcal{B}^\alpha_\gamma)\rho_\pounds(\mathcal{V}_\alpha)(\textbf{y}^\beta)J(\mathcal{X}_\beta)\nonumber\\
&=(\widetilde{Y}^\alpha-\widetilde{X}^\gamma\mathcal{B}^\alpha_\gamma)\mathcal{V}_\alpha=v\widetilde{X},
\end{align}
where $\widetilde{X}=\widetilde{X}^\alpha\mathcal{X}_\alpha+\widetilde{Y}^\alpha\mathcal{V}_\alpha$. Also, we can obtain $[J\widetilde{X}, S]_\pounds-\widetilde{X}\in v\pounds^\pi E$. Thus
\begin{equation}\label{al4}
h[J\widetilde{X}, S]_\pounds-h\widetilde{X}=0.
\end{equation}
Setting (\ref{al3}) and (\ref{al4}) in (\ref{al2}) give us
\begin{equation}\label{al5}
J[S, h]^{F-N}_\pounds\widetilde{X}+h[S, h]^{F-N}_\pounds J\widetilde{X}=\widetilde{X}.
\end{equation}
Plugging (\ref{al1}) and (\ref{al5}) into (\ref{al}) yield $F^2=-1_{\pounds^\pi E}$. Thus $F$ is an almost complex structure on $\pounds^\pi E$ which is called {\it the almost complex structure induced by $h$}.
\begin{lemma}\label{IM}
The following relations hold
\begin{equation}\label{IM0}
(i)\ F\circ J=h, \ (ii)\ F\circ h=-J,\ (iii)\ J\circ F=v,\ (iv)\ F\circ v=h\circ F.
\end{equation}
\end{lemma}
\begin{proof}
Since $J^2=0$, then we have
\[
F\circ J=(h[S, h]^{F-N}_\pounds-J)\circ J=h[S, h]^{F-N}_\pounds J.
\]
But we have
\[
h[S, h]^{F-N}_\pounds J\widetilde{X}-h\widetilde{X}=h(h[J\widetilde{X}, S]_\pounds-[hJ\widetilde{X}, S]_\pounds)-h\widetilde{X}=h([J\widetilde{X}, S]_\pounds-\widetilde{X})=0,
\]
because $[J\widetilde{X}, S]_\pounds-\widetilde{X}\in v\pounds^\pi E$. Therefore $F\circ J=h$. Now we prove the secon equation. Since $Jh=J$, then we have
\[
F\circ h=(h[S, h]^{F-N}_\pounds-J)\circ h=h[S, h]^{F-N}_\pounds h-J.
\]
But we have
\begin{equation}\label{Fardin}
h[S, h]^{F-N}_\pounds h\widetilde{X}=h(h[h\widetilde{X}, S]_\pounds-[h^2\widetilde{X}, S]_\pounds)=h^2[h\widetilde{X}, S]_\pounds-h[h\widetilde{X}, S]_\pounds=0.
\end{equation}
Therefore $F\circ h=-J$. Using the definition of $F$ we deduce
\[
J\circ F=J(h[S, h]^{F-N}_\pounds-J)=Jh[S, h]^{F-N}_\pounds=J[S, h]^{F-N}_\pounds.
\]
But using (\ref{al3}) we get
\begin{align*}
J[S, h]^{F-N}_\pounds\widetilde{X}&=J(h[\widetilde{X}, S]_\pounds-[h\widetilde{X}, S]_\pounds)=J[\widetilde{X}, S]_\pounds-J[h\widetilde{X}, S]_\pounds\\
&=J[v\widetilde{X}, S]_\pounds=v\widetilde{X}.
\end{align*}
Therefore $J\circ F=v$. To prove the last equation we have
\begin{align*}
(h\circ F)\widetilde{X}&=h(h[S, h]^{F-N}_\pounds-J)\widetilde{X}=h[S, h]^{F-N}_\pounds\widetilde{X}\\
&=h[S, h]^{F-N}_\pounds h\widetilde{X}+h[S, h]^{F-N}_\pounds v\widetilde{X},
\end{align*}
where $\widetilde{X}\in\Gamma(\pounds^\pi E)$. Setting (\ref{Fardin}) in the above equation implies that
\[
(h\circ F)\widetilde{X}=h[S, h]^{F-N}_\pounds v\widetilde{X}=(h[S, h]^{F-N}_\pounds-J)v\widetilde{X}=(F\circ v)\widetilde{X}.
\]
\end{proof}
Let $F=h[S, h]^{F-N}_\pounds-J$ be the almost complex structure induced by $h$. Since $S$ is the semispray associated to $h$, then we have $S=hS'$, where $S'$ is a semispray of $\pounds^\pi E$. Using (\ref{vertical}), (\ref{semispray}) and (\ref{horizontal end}) we obtain
\[
F(\mathcal{X}_\alpha)=-\mathcal{B}^\gamma_\alpha(\mathcal{X}_\gamma+\mathcal{B}^\beta_\gamma\mathcal{V}_\beta)-\mathcal{V}
_\alpha,\ \ \ F(\mathcal{V}_\alpha)=\mathcal{X}_\alpha+\mathcal{B}^\beta_\alpha\mathcal{V}_\beta.
\]
Therefore $F$ has the following coordinate expression
\begin{equation}\label{complex}
F=-(\mathcal{B}^\gamma_\alpha(\mathcal{X}_\gamma+\mathcal{B}^\beta_\gamma\mathcal{V}_\beta)+\mathcal{V}
_\alpha)\otimes\mathcal{X}^\alpha+(\mathcal{X}_\alpha+\mathcal{B}^\beta_\alpha\mathcal{V}_\beta)\otimes\mathcal{V}^\alpha.
\end{equation}
\begin{proposition}
Let $h$ be a horizontal endomorphism on $\pounds^\pi E$ and $j:\pounds^\pi E\rightarrow E\times_ME$ be the map introduced in (\ref{exact se}). Then we have
\begin{equation}\label{split}
j\circ h=j.
\end{equation}
\end{proposition}
\begin{proof}
Since $\im v=\ker j=v\pounds^\pi E$, then
\[
j\circ h=j\circ(Id-v)=j-jv=j.
\]
\end{proof}
Let $\mathcal{H}:=F\circ i:E\times_ME\rightarrow \pounds^\pi E$. Then using $(i)$ of Lemma \ref{IM} and (\ref{split})
\[
j\circ \mathcal{H}\circ j=j\circ F\circ i\circ j= j\circ F\circ J= j\circ h=j.
\]
Since $j$ is surjective, then the above equation gives us $j\circ \mathcal{H}=1_{E\times_ME}$. Therefore $\mathcal{H}$ is a right splitting of (exacts). We call $\mathcal{H}$ the horizontal map for
$\pounds^\pi E$ associated to $h$. \\
Now we consider
\[
\mathcal{V}:=j\circ F:\pounds^\pi E\rightarrow E\times_ME.
\]
Then we have
\[
\mathcal{V}\circ i=j\circ F \circ i=j\circ \mathcal{H}= 1_{E\times_ME}.
\]
Therefore $\mathcal{V}$ is a left splitting of (exacts), which is called the vertical map for $\pounds^\pi E$ associated to $h$.
\begin{cor}
The following sequence is a double short exact sequence
\begin{diagram}[heads=LaTeX]
0 &\pile{\rTo\\  \lTo} &\pi^*E&\pile{\rTo^i\\  \lTo_{\mathcal{V}}} &\pounds^\pi E&\pile{\rTo^j\\  \lTo_{\mathcal{H}}} &\pi^*E&\pile{\rTo\\  \lTo} &0
\end{diagram}
\end{cor}
\begin{proof}
We obtain
\[
\mathcal{V}\circ\mathcal{H}=(j\circ F)\circ(F\circ i)=j\circ(-1_{\pounds^\pi E})\circ i=-j\circ i=0.
\]
Thus $\im \mathcal{H}=\ker\mathcal{V}$. Moreover $\mathcal{V}$ is surjective, because $j$ is surjective. Similarly, since $i$ is injective, then $\mathcal{H}$ is injective. These complete the proof.
\end{proof}
Using (i) and (iii) of Lemma \ref{IM} we can get
\begin{equation}
(i)\ h=\mathcal{H}\circ j,\ \ \ \ \ \ (ii)\ v=i\circ\mathcal{V}.
\end{equation}
\subsection{Berwald endomorphism}
Let $S$ be a semispray on $\pounds^\pi E$. We consider the map $h_S:\pounds^\pi E\rightarrow\pounds^\pi E$ given by $h_S:=\frac{1}{2}(1_{\pounds^\pi E}+[J, S]^{F-N}_\pounds)$. Using (\ref{vertical}) and (\ref{semispray}) we can obtain
\[
h_S(\mathcal{X}_\alpha)=\mathcal{X}_\alpha+\frac{1}{2}(\frac{\partial S^\gamma}{\partial \textbf{y}^\alpha}-\textbf{y}^\beta (L^\gamma_{\alpha\beta}\circ\pi))\mathcal{V}_\gamma,\ \ \ \ h_S(\mathcal{V}_\alpha)=0.
\]
Therefore $h_S$ has the coordinate expression
\begin{equation}\label{Ber000}
h_S=(\mathcal{X}_\alpha+\mathcal{B}^\gamma_\alpha\mathcal{V}_\gamma)\otimes\mathcal{X}^\alpha,
\end{equation}
where
\begin{equation}\label{Ber}
\mathcal{B}^\gamma_\alpha=\frac{1}{2}(\frac{\partial S^\gamma}{\partial \textbf{y}^\alpha}-\textbf{y}^\beta (L^\gamma_{\alpha\beta}\circ\pi)).
\end{equation}
Now one can easily check that $h_S\circ h_S=h_S$, $J\circ h_S=J$, $h_S\circ J=0$ and consequently $\ker h_S =\ker J=v\pounds^\pi E$. Therefore $h_S$ is a horizontal endomorphism on $\pounds^\pi E$ called {\it horizontal endomorphism generated by semispray $S$}.
\begin{theorem}\label{4.1}
The horizontal endomorphism generated by semispray $S$ is torsion free. Moreover, we have $H_S=\frac{1}{2}[[C, S]_\pounds-S, J]^{F-N}_\pounds$, where $H_S$ is the tension of $h_S$.
\end{theorem}
\begin{proof}
Let $t_S$ be the weak torsion of $h_S$. Then using (\ref{F20}) we have
\begin{align*}
t_S&=[J, h_S]^{F-N}_\pounds=\frac{1}{2}[J, [J, S]^{F-N}_\pounds]^{F-N}_\pounds=\frac{1}{2}[J, [S, J]^{F-N}_\pounds]^{F-N}_\pounds\\
&\ \ \ -\frac{1}{2}[S, [J, J]^{F-N}_\pounds]^{F-N}_\pounds=-\frac{1}{2}[J, [J, S]^{F-N}_\pounds]^{F-N}_\pounds=-t_S.
\end{align*}
Therefore $t_B=0$. (\ref{F20}) and (i) of (\ref{Liover}) give us
\begin{align*}
&\frac{1}{2}[[C, S]_\pounds-S, J]^{F-N}_\pounds=\frac{1}{2}([[C, S]_\pounds, J]^{F-N}_\pounds-[S, J]^{F-N}_\pounds)\\
&=\frac{1}{2}([[J, S]^{F-N}_\pounds, C]^{F-N}_\pounds-[[J, C]^{F-N}_\pounds, S]^{F-N}_\pounds-[S, J]^{F-N}_\pounds)\\
&=\frac{1}{2}([[J, S]^{F-N}_\pounds, C]^{F-N}_\pounds-[J, S]^{F-N}_\pounds-[S, J]^{F-N}_\pounds)\\
&=\frac{1}{2}[[J, S]^{F-N}_\pounds, C]^{F-N}_\pounds=[h_S, C]^{F-N}_\pounds=H_S.
\end{align*}
\end{proof}
Using (\ref{tension}) and (\ref{Ber}) we deduce that $H_B$ has the following coordinate expression:
\[
H_S=\frac{1}{2}(\frac{\partial S^\alpha}{\partial \textbf{y}^\beta}-\textbf{y}^\gamma\frac{\partial^2 S^\alpha}{\partial \textbf{y}^\gamma\partial \textbf{y}^\beta})\mathcal{V}_\alpha\otimes\mathcal{X}^\beta.
\]
\begin{lemma}\label{4.2}
Let $h_S$ be the horizontal endomorphism generated by semispray $S$. Then the semispray associated by $h_S$ is $\frac{1}{2}(S+[C, S]_\pounds)$.
\end{lemma}
\begin{proof}
We have
\begin{align*}
h_SS&=\frac{1}{2}(1_{\pounds^\pi E}+[J, S]^{F-N}_\pounds)S=\frac{1}{2}(S+[J, S]^{F-N}_\pounds S)\\
&=\frac{1}{2}(S+[JS, S]_\pounds-J[S, S]_\pounds)=\frac{1}{2}(S+[C, S]_\pounds).
\end{align*}
\end{proof}
If $h_S$ is the horizontal endomorphism generated by spray $S$, then from Theorem \ref{4.1} and Lemma \ref{4.2} it is easy to see that $h_SS=S$ and $H_S=0$. Thus we have
\begin{cor}\label{4.3}
Let $h_S$ be the horizontal endomorphism generated by spray $S$. Then the spray associated by $h_S$ is $S$. Moreover $h_S$ is homogenous.
\end{cor}
\begin{defn}
The horizontal endomorphism generated by an spray is called Berwald endomorphism.
\end{defn}
\begin{theorem}\label{mainth}
Let $h$ be a homogenous horizontal endomorphism on $\pounds^\pi E$ and $S$ be the semispray associated to $h$. Then we have
\[
h_S=h-\frac{1}{2}i_St,
\]
where $t$ is the weak torsion of $h$ and $h_S$ is the horizontal endomorphism generated by $S$.
\end{theorem}
\begin{proof}
Since $h$ is homogenous, then $S$ is spray. Therefore $h_S$ is the Berwald endomorphism and consequently from Lemma \ref{4.2} and Corollary \ref{4.3} we deduce $h_S$ is homogenous and $h_SS=S$. Also since $h=(\mathcal{X}_\beta+\mathcal{B}^\alpha_\beta\mathcal{V}_\alpha)\otimes\mathcal{X}^\beta$ is homogenous and $hS=S$, we obtain
\begin{equation}
(i)\ \mathcal{B}^\alpha_\beta=\textbf{y}^\gamma\frac{\partial \mathcal{B}^\alpha_\beta}{\partial \textbf{y}^\gamma},\ \ \ \    (ii)\ S^\alpha=\textbf{y}^\beta \mathcal{B}^\alpha_\beta.\label{4.4}
\end{equation}
From $(ii)$ of (\ref{4.4}) we get
\begin{equation}\label{4.6}
\textbf{y}^\alpha \frac{\partial \mathcal{B}^\gamma_\alpha}{\partial \textbf{y}^\beta}=\frac{\partial S^\alpha}{\partial \textbf{y}^\beta}-\mathcal{B}^\gamma_\beta.
\end{equation}
Using (\ref{4.5}) and (\ref{4.6}) we obtain
\[
h(\mathcal{X_\beta})-\frac{1}{2}(i_St)(\mathcal{X_\beta})=\mathcal{X_\beta}+\{\mathcal{B}^\gamma_\beta-\frac{1}{2}\textbf{y}^\alpha\frac{\partial \mathcal{B}^\gamma_\beta}{\partial \textbf{y}^\alpha}+\frac{1}{2}\textbf{y}^\alpha\frac{\partial \mathcal{B}^\gamma_\alpha}{\partial \textbf{y}^\beta}+\frac{1}{2}\textbf{y}^\alpha (L^\gamma_{\alpha\beta}\circ \pi)\}\mathcal{V_\gamma}.
\]
Setting $(i)$ of (\ref{4.4}) in the above equation gives us
\[
h(\mathcal{X_\beta})-\frac{1}{2}(i_St)(\mathcal{X_\beta})=\mathcal{X_\beta}+\{\frac{1}{2}\mathcal{B}^\gamma_\beta+\frac{1}{2}\textbf{y}^\alpha\frac{\partial \mathcal{B}^\gamma_\alpha}{\partial \textbf{y}^\beta}+\frac{1}{2}\textbf{y}^\alpha (L^\gamma_{\alpha\beta}\circ \pi)\}\mathcal{V_\gamma}.
\]
Plugging (\ref{4.6}) into the above equation implies that
\[
h(\mathcal{X_\beta})-\frac{1}{2}(i_St)(\mathcal{X_\beta})=\mathcal{X_\beta}+\frac{1}{2}\{\frac{\partial S^\gamma}{\partial \textbf{y}^\beta}+\textbf{y}^\alpha (L^\gamma_{\alpha\beta}\circ \pi)\}\mathcal{V_\gamma}=h_S(\mathcal{X_B}).
\]
Similarly, we obtain
\[
h(\mathcal{V_\beta})-\frac{1}{2}(i_St)(\mathcal{V_\beta})=h_S(\mathcal{V_B}).
\]
\end{proof}
\subsection{Horizontal lift}
Let $h$ be a horizontal endomorphism on $\pounds^\pi E$. We consider the map
\[
X\in\Gamma(E)\rightarrow X^h:=hX^C\in h\pounds^\pi E,
\]
and we call it \textit{horizontal lift by $h$}. If $X=X^\alpha e_\alpha$, then we have
\begin{equation}
X^h=(X^\alpha\circ\pi)(\mathcal{X}_\alpha+\mathcal{B}^\beta_\alpha\mathcal{V}_\beta).
\end{equation}
\begin{lemma}
Let $h$ be a horizontal endomorphism on $\pounds^\pi E$ and $X, Y\in\Gamma(E)$. Then
\begin{equation}\label{three}
(i)\ JX^h=X^V,\ \ (ii)\ h[X^h, Y^h]_\pounds=[X, Y]_E^h,\ \ (iii)\ [X, Y]^V_E=J[X^h, Y^h]_\pounds.
\end{equation}
\end{lemma}
\begin{proof}
We have
\[
JX^h=JhX^C=JX^C=X^V.
\]
Thus (i) is proved. Now let $X=X^\alpha e_\alpha$ and $Y=Y^\beta e_\beta$. Then by a direct calculation we get
\begin{equation}\label{50}
[X, Y]_E=(X^\alpha\rho^i_\alpha \frac{\partial Y^\gamma}{\partial x^i}-Y^\beta\rho^i_\beta\frac{\partial X^\gamma}{\partial x^i}+X^\alpha Y^\beta L^\gamma_{\alpha\beta})e_\gamma,
\end{equation}
and
\begin{align}\label{51}
[X^h, Y^h]_\pounds&=\Big((X^\alpha\rho^i_\alpha \frac{\partial Y^\gamma}{\partial x^i}-Y^\beta\rho^i_\beta\frac{\partial X^\gamma}{\partial x^i}+X^\alpha Y^\beta L^\gamma_{\alpha\beta})\circ \pi\Big)\mathcal{X_\gamma}\nonumber\\
&\ \ +\Big(((X^\alpha\rho^i_\alpha)\circ\pi)\frac{\partial}{\partial\textbf{x}^i}((Y^\beta\circ\pi)\mathcal{B}^\gamma_\beta)
-((Y^\alpha\rho^i_\alpha)\circ\pi)\frac{\partial}{\partial\textbf{x}^i}((X^\beta\circ\pi)\mathcal{B}^\gamma_\beta)\nonumber\\
&\ \ +((X^\alpha Y^\beta)\circ\pi)\mathcal{B}^\lambda_\alpha\frac{\partial \mathcal{B}^\gamma_\beta}{\partial\textbf{y}^\lambda}-((X^\alpha Y^\beta)\circ\pi)\mathcal{B}^\lambda_\beta\frac{\partial \mathcal{B}^\gamma_\alpha}{\partial\textbf{y}^\lambda}\Big)\mathcal{V}_\gamma.
\end{align}
Therefore
\begin{align*}
[X, Y]^h_E&=\Big((X^\alpha\rho^i_\alpha \frac{\partial Y^\gamma}{\partial x^i}-Y^\beta\rho^i_\beta\frac{\partial X^\gamma}{\partial x^i}+X^\alpha Y^\beta L^\gamma_{\alpha\beta})\circ \pi\Big) (\mathcal{X_\gamma}+B_\gamma^\lambda \mathcal{V_\lambda})\\
&=h[X^h, Y^h]_\pounds.
\end{align*}
Thus we have (ii). Also, using (\ref{50}) and (\ref{51}) we obtain (iii) as follows
\[
[X, Y]^v_E=\Big((X^\alpha\rho^i_\alpha \frac{\partial Y^\gamma}{\partial x^i}-Y^\beta\rho^i_\beta\frac{\partial X^\gamma}{\partial x^i}+X^\alpha Y^\beta L^\gamma_{\alpha\beta})\circ\pi\Big)\mathcal{V}_\gamma=J[X^h, Y^h]_\pounds.
\]
\end{proof}
\begin{lemma}
Let $h$ be a horizontal endomorphism on $\pounds^\pi E$ and $X, Y\in\Gamma(E)$. Then
\begin{equation}\label{leila1}
t(X^h, Y^h)=[X^h, Y^V]_\pounds-[Y^h, X^V]_\pounds-[X, Y]_E^V.
\end{equation}
\end{lemma}
\begin{proof}
Using the definition of the weak torsion we have
\begin{align*}
t(X^h, Y^h)&=[J, h]^{F-N}_\pounds(X^h, Y^h)=[JX^h, hY^h]_\pounds+[hX^h, JY^h]_\pounds\\
&\ \ \ +Jh[X^h, Y^h]_\pounds+hJ[X^h, Y^h]_\pounds-J[X^h, hY^h]_\pounds\\
&\ \ \ -J[hX^h, Y^h]_\pounds-h[X^h, JY^h]_\pounds-h[JX^h, Y^h]_\pounds.
\end{align*}
Using $JX^h=X^h$, (i) of (\ref{three}) and (i), (iv) of (\ref{Jh}) in the above equation, we get
\begin{align}\label{leila}
t(X^h, Y^h)&=[X^V, Y^h]_\pounds+[X^h, Y^V]_\pounds-J[X^h, Y^h]_\pounds\nonumber\\
&\ \ \ -h[X^h, Y^V]_\pounds-h[X^V, Y^h]_\pounds.
\end{align}
But we can obtain
\[
[X^h, Y^V]_\pounds=\Big((X^\alpha\rho^i_\alpha\frac{\partial Y^\beta}{\partial x^i})\circ\pi-((X^\alpha Y^\gamma)\circ\pi)\frac{\partial \mathcal{B}^\beta_\alpha}{\partial\textbf{y}^\gamma}\Big)\mathcal{V}_\beta.
\]
Therefore $h[X^h, Y^V]_\pounds=0$. Similarly we have $h[X^V, Y^h]_\pounds=0$. Setting these equation and (iii) of (\ref{three}) in (\ref{leila}) we obtain (\ref{leila1}).
\end{proof}
\begin{proposition}
If $h$ and $\bar{h}$ are homogenous horizontal endomorphisms on $\pounds^\pi E$ such that
\begin{equation}\label{lili}
[X^h, Y^V]_\pounds=[X^{\bar{h}}, Y^V]_\pounds,\ \ \ \forall X, Y\in\Gamma(E),
\end{equation}
then $h=\bar{h}$.
\end{proposition}
\begin{proof}
Let $h=(\mathcal{X}_\alpha+\mathcal{B}^\beta_\alpha\mathcal{V}_\beta)\otimes\mathcal{X}^\beta$ and $\bar{h}=(\mathcal{X}_\alpha+\mathcal{B}^\beta_\alpha\mathcal{V}_\beta)\otimes\mathcal{X}^\beta$. Since $h$ and $\bar{h}$ are homogenous, then we have
\begin{equation}\label{lili1}
\mathcal{B}^\beta_\alpha=\textbf{y}^\gamma\frac{\partial \mathcal{B}_\alpha^\beta}{\partial \textbf{y}^\gamma},\ \ \ \mathcal{B}^\beta_\alpha=\textbf{y}^\gamma\frac{\partial \bar{\mathcal{B}}_\alpha^\beta}{\partial \textbf{y}^\gamma}.
\end{equation}
Setting $X=e_\alpha$ and $Y=e_\beta$ in (\ref{lili}), we have $[e_\alpha^h, e_\beta^v]_\pounds=[e_\alpha^{\bar{h}}, e_\beta^v]_\pounds$. This equation gives us
\[
\frac{\partial \mathcal{B}_\alpha^\beta}{\partial \textbf{y}^\gamma}=\frac{\partial \bar{\mathcal{B}}_\alpha^\beta}{\partial \textbf{y}^\gamma}.
\]
Contracting the above equation by $\textbf{y}^\gamma$ and using (\ref{lili1}) we deduce $\mathcal{B}^\beta_\alpha=\mathcal{B}^\beta_\alpha$ and consequently $h=\bar{h}$.
\end{proof}
We set $\delta_\alpha=e_\alpha^h$. Then we have $\delta_\alpha=\mathcal{X}_\alpha+\mathcal{B}^\beta_\alpha\mathcal{V}_\beta=h(\mathcal{X}_\alpha)$. It is easy to see that $h\delta_\alpha=\delta_\alpha$, $v\delta_\alpha=0$ and
\begin{equation}\label{rho}
\rho_\pounds(\delta_\alpha)=(\rho^i_\alpha\circ\pi)\frac{\partial }{\partial \textbf{x}^i}+\mathcal{B}^\gamma_\alpha\frac{\partial}{\partial \textbf{y}^\gamma}.
\end{equation}
Moreover,  $\{\delta_\alpha\}$ generate a basis of $h\pounds^\pi E$ and the frame $\{\delta_\alpha, \mathcal{V}_\alpha\}$ is a local basis of $\pounds^\pi E$ adapted to splitting (\ref{good}) which is called adapted basis. The dual adapted basis is $\{\mathcal{X}^\alpha, \delta\mathcal{V}^\alpha\}$, where
\[
\delta\mathcal{V}^\alpha=\mathcal{V}^\alpha-\mathcal{B}^\alpha_\beta\mathcal{X}^\beta.
\]
\begin{proposition}
The Lie brakhets of the adapted basis $\{\delta_\alpha, \mathcal{V}_\alpha\}$  are
\begin{equation}\label{LieB}
[\delta_\alpha, \delta_\beta]_\pounds=(L^\gamma_{\alpha\beta}\circ\pi)\delta_\gamma+R^\gamma_{\alpha\beta}\mathcal{V}_\gamma,\ \ \ [\delta_\alpha, \mathcal{V}_\beta]_\pounds=-\frac{\partial \mathcal{B}^\gamma_\alpha}{\partial\textbf{y}^\beta}\mathcal{V}_\gamma,\ \ \ [\mathcal{V}_\alpha, \mathcal{V}_\beta]_\pounds=0,
\end{equation}
where $R^\gamma_{\alpha\beta}$ is geven by (\ref{curv0}).
\end{proposition}
Using (\ref{horizontal end}) and (\ref{complex}), $h$ and $F$ have the following coordinate expressions with respect to adapted basis
\begin{equation}\label{hF}
(i)\ h=\delta_\alpha\otimes\mathcal{X}^\alpha,\ \ \ \ F=-\mathcal{V}_\alpha\otimes\mathcal{X}^\alpha+\delta_\alpha\otimes\delta\mathcal{V}^\alpha.
\end{equation}
 \section{Distinguished connections on Lie algebroids}
 This section is appertained to constructing distinguished connections on Lie algebroids. Intrinsic $v$-connections and Berwald-type and Yano-type connections are also studied. Ultimately, The Douglas tensor of a Berwald endomorphism based on Yano connection is introduced.

A linear connection on a Lie algebroid $(E, [, ]_E, ó\rho)$ is a map
\[
D:\Gamma(E)\times\Gamma(E)\rightarrow\Gamma(E)
\]
which satisfies the rules
\begin{align*}
D_{fX+Y}Z&=fD_XY+D_YZ,\\
D_X(fY+Z)&=(\rho(X)f)Y+fD_XY+D_XZ,
\end{align*}
for any function $f\in C^{\infty}(M)$ and $X, Y, Z\in \Gamma(E)$.
\begin{defn}
Let $D$ be a linear connection on $\pounds^\pi E$ and $h$ be a horizontal endomorphism on $\pounds^\pi E$. Then $(D, h)$ is called a distinguished connection (or d-connection) on $\pounds^\pi E$, if

i) $D$ is reducible, i.e., $Dh=0$,

ii) $D$ is almost complex, i.e., $DF=0$,\\
where $F$ is the almost complex structure associated by $h$.
\end{defn}
\begin{lemma}\label{6.2}
If $D$ is reducible respect to $h$, then we have
\begin{equation}
(i)\ D_{\widetilde{X}}h\widetilde{Y}=hD_{\widetilde{X}}\widetilde{Y}\in h\pounds^\pi E,\ \ \ (ii)\ D_{\widetilde{X}}v\widetilde{Y}=vD_{\widetilde{X}}\widetilde{Y}\in v\pounds^\pi E,
\end{equation}
where $\widetilde{X}$ and $\widetilde{Y}$ are sections of $\pounds^\pi E$.
\end{lemma}
\begin{proof}
Since $Dh=0$, then we have
\[
0=Dh(\widetilde{X}, \widetilde{Y})=D_{\widetilde{X}}h\widetilde{Y}-hD_{\widetilde{X}}\widetilde{Y},
\]
which gives us (i). Similarly we can prove (ii).
\end{proof}
Since $\im h=h\pounds^\pi E$ and $\im v=v\pounds^\pi E$, then we have
\begin{cor}
If $\widetilde{Y}$ and $\widetilde{Z}$ are sections of $v\pounds^\pi E$ and $h\pounds^\pi E$, respectively, then we have $D_{\widetilde{X}}\widetilde{Y}\in v\pounds^\pi E$ and $D_{\widetilde{X}}\widetilde{Z}\in h\pounds^\pi E$.
\end{cor}
\begin{lemma}
If the linear connection $D$ is almost complex on $\pounds^\pi E$, then $D$ is determined on $\pounds^\pi E\times v\pounds^\pi E$, completely.
\end{lemma}
\begin{proof}
From $DF=0$, we deduce $D_{\widetilde{X}}F\widetilde{Y}=FD_{\widetilde{X}}\widetilde{Y}$, for all $\widetilde{X}, \widetilde{Y}\in\Gamma(\pounds^\pi E)$. Thus we have
\begin{align}
D_{v\widetilde{X}}h\widetilde{Y}&=D_{v\widetilde{X}}FJ\widetilde{Y}=FD_{v\widetilde{X}}J\widetilde{Y},\label{conn0}\\
D_{h\widetilde{X}}h\widetilde{Y}&=D_{h\widetilde{X}}FJ\widetilde{Y}=FD_{h\widetilde{X}}J\widetilde{Y}.\label{conn00}
\end{align}
\end{proof}
\begin{lemma}
If $(D, h)$ is a d-connection, then $DJ=0$.
\end{lemma}
\begin{proof}
Let $\widetilde{X}$ and $\widetilde{Y}$ be sections of $\pounds^\pi E$ and $v\pounds^\pi E$, respectively. Then from the above lemma we have $D_{\widetilde{X}}\widetilde{Y}\in\Gamma(v\pounds^\pi E)$. Thus, since $\im J=v\pounds^\pi E$, then we have $J\widetilde{Y}=0$ and $JD_{\widetilde{X}}\widetilde{Y}=0$. Therefore we obtain
\[
DJ(\widetilde{Y}, \widetilde{X})=D_{\widetilde{X}}J\widetilde{Y}-JD_{\widetilde{X}}\widetilde{Y}=0.
\]
\end{proof}
Using (ii) of Lemma \ref{6.2}, we have $D_{\delta_\alpha}\mathcal{V}_\beta\in v\pounds^\pi E$ and $D_{\mathcal{V}_\alpha}\mathcal{V}_\beta\in v\pounds^\pi E$. Thus these have the following coordinate expressions
\begin{equation}\label{conn}
D_{\delta_\alpha}\mathcal{V}_\beta=F^\gamma_{\alpha\beta}\mathcal{V}_\gamma,\ \ \
D_{\mathcal{V}_\alpha}\mathcal{V}_\beta=C^\gamma_{\alpha\beta}\mathcal{V}_\gamma.
\end{equation}
From (\ref{hF}), (\ref{conn00}) and the above equation we obtain
\begin{equation}\label{conn1}
D_{\delta_\alpha}\delta_\beta=D_{\delta_\alpha}h\delta_\beta=FD_{\delta_\alpha}J\delta_\beta=FD_{\delta_\alpha}
\mathcal{V}_\beta=F^\gamma_{\alpha\beta}\delta_\gamma.
\end{equation}
Similarly (\ref{hF}), (\ref{conn0}) and (\ref{conn}) imply that
\begin{equation}\label{conn2}
D_{\mathcal{V}_\alpha}\delta_\beta=C^\gamma_{\alpha\beta}\delta_\gamma.
\end{equation}
\begin{defn}
Let $(D, h)$ be a d-connection. Then
\begin{equation*}
\left\{
\begin{array}{cc}
D^h:\Gamma(\pounds^\pi E)\times\Gamma(\pounds^\pi E)\rightarrow\Gamma(\pounds^\pi E)\\
\hspace{1.5cm}(\widetilde{X}, \widetilde{Y})\mapsto D^h_{\widetilde{X}}\widetilde{Y}:=D_{h\widetilde{X}}\widetilde{Y}
\end{array}
\right.
\end{equation*}
and
\begin{equation*}
\left\{
\begin{array}{cc}
D^v:\Gamma(\pounds^\pi E)\times\Gamma(\pounds^\pi E)\rightarrow\Gamma(\pounds^\pi E)\\
\hspace{1.5cm}(\widetilde{X}, \widetilde{Y})\mapsto D^v_{\widetilde{X}}\widetilde{Y}:=D_{v\widetilde{X}}\widetilde{Y}
\end{array}
\right.
\end{equation*}
are called $h$-covariant derivative and $v$-covariant derivative, respectively. Moreover,
\begin{equation}\label{h*}
\left\{
\begin{array}{cc}
h^*(DC):\Gamma(\pounds^\pi E)\rightarrow\Gamma(\pounds^\pi E)\\
\widetilde{X}\mapsto DC(h\widetilde{X}):=D_{h\widetilde{X}}C
\end{array}
\right.
\end{equation}
and
\begin{equation*}
\left\{
\begin{array}{cc}
v^*(DC):\Gamma(\pounds^\pi E)\rightarrow\Gamma(\pounds^\pi E)\\
\widetilde{X}\mapsto DC(v\widetilde{X}):=D_{v\widetilde{X}}C
\end{array}
\right.
\end{equation*}
are called $h$-deflection and $v$-deflection of $(D, h)$, respectively.
\end{defn}
Using (\ref{conn}), (\ref{conn1}) and (\ref{conn2}) we get
\begin{equation}\label{h-cov}
D^h_{\delta_\alpha}\delta_\beta=F^\gamma_{\alpha\beta}\delta_\gamma,\ \ D^h_{\delta_\alpha}\mathcal{V}_\beta=F^\gamma_{\alpha\beta}\mathcal{V}_\gamma,\ \ D^h_{\mathcal{V}_\alpha}\delta_\beta=D^h_{\mathcal{V}_\alpha}\mathcal{V}_\beta=0.
\end{equation}
Similarly we obtain
\begin{equation}\label{v-cov}
D^v_{\mathcal{V}_\alpha}\delta_\beta=C^\gamma_{\alpha\beta}\delta_\gamma,\ \ D^v_{\mathcal{V}_\alpha}\mathcal{V}_\beta=C^\gamma_{\alpha\beta}\mathcal{V}_\gamma,\ \ D^v_{\delta_\alpha}\delta_\beta=D^v_{\delta_\alpha}\mathcal{V}_\beta=0.
\end{equation}
Using (\ref{h-cov}) and (\ref{v-cov}) we deduce $D=D^h+D^v$. (\ref{conn}), (\ref{conn1}) and (\ref{conn2}) give us
\begin{align*}
h^*(DC)(\delta_\alpha)&=D_{h\delta_\alpha}C=D_{\delta_\alpha}(\textbf{y}^\beta\mathcal{V}_\beta)=\rho(\delta_\alpha)(\textbf{y}^\beta)\mathcal{V}_\beta+\textbf{y}^\beta D_{\delta_\alpha}\mathcal{V}_\beta\\
&=(\mathcal{B}^\gamma_\alpha+\textbf{y}^\beta F^\gamma_{\alpha\beta})\mathcal{V}_\gamma,
\end{align*}
and $h^*(DC)(\mathcal{V}_\alpha)=0$. Therefore $h^*(DC)$ has the following coordinate expression:
\begin{equation}\label{h-def}
h^*(DC)=(\mathcal{B}^\gamma_\alpha+\textbf{y}^\beta F^\gamma_{\alpha\beta})\mathcal{V}_\gamma\otimes\mathcal{X}^\alpha.
\end{equation}
Similarly, we can see that $v^*(DC)$ has the following coordinate expression:
\begin{equation}\label{h-def}
v^*(DC)=(\delta^\gamma_\alpha+\textbf{y}^\beta C^\gamma_{\alpha\beta})\mathcal{V}_\gamma\otimes\delta\mathcal{V}^\alpha,
\end{equation}
where $\delta^\gamma_\alpha$ is the Kronicher symble.
\begin{theorem}
Let $(D, h)$ be a d-connection on $\pounds^\pi E$. Then the torsion tensor field $T$ of $D$ determined by the following, completely:
\begin{align}
A(\widetilde{X}, \widetilde{Y}):&=hT(h\widetilde{X}, h\widetilde{Y})=D_{h\widetilde{X}}h\widetilde{Y}-D_{h\widetilde{Y}}h\widetilde{X}-h[h\widetilde{X}, h\widetilde{Y}]_\pounds,\label{TD1}\\
B(\widetilde{X}, \widetilde{Y}):&=hT(h\widetilde{X}, J\widetilde{Y})=-D_{J\widetilde{Y}}h\widetilde{X}-h[h\widetilde{X}, J\widetilde{Y}]_\pounds,\label{TD2}\\
R^1(\widetilde{X}, \widetilde{Y}):&=vT(h\widetilde{X}, h\widetilde{Y})=-v[h\widetilde{X}, h\widetilde{Y}]_\pounds,\label{TD3}\\
P^1(\widetilde{X}, \widetilde{Y}):&=vT(h\widetilde{X}, J\widetilde{Y})=D_{h\widetilde{X}}J\widetilde{Y}-v[h\widetilde{X}, J\widetilde{Y}]_\pounds,\label{TD4}\\
S^1(\widetilde{X}, \widetilde{Y}):&=vT(J\widetilde{X}, J\widetilde{Y})=D_{J\widetilde{X}}J\widetilde{Y}-D_{J\widetilde{Y}}J\widetilde{X}-v[J\widetilde{X}, J\widetilde{Y}]_\pounds,\label{TD5}
\end{align}
where $A$, $B$, $R^1$, $P^1$ and $R^1$ are called $h$- horizontal, $h$- mixed, $v$- horizontal, $v$- mixed and $v$- vertical torsion, respectively.
\end{theorem}
\begin{proof}
We have
\begin{align*}
hT(\widetilde{X}, \widetilde{Y})&=hT(h\widetilde{X}, h\widetilde{Y})+hT(h\widetilde{X}, v\widetilde{Y})+hT(v\widetilde{X}, h\widetilde{Y})+hT(v\widetilde{X}, v\widetilde{Y})\\
&=A(\widetilde{X}, \widetilde{Y})+B(\widetilde{X}, \widetilde{Y})-B(\widetilde{Y}, \widetilde{X})+hT(v\widetilde{X}, v\widetilde{Y}).
\end{align*}
It is easy to check that $T(v\widetilde{X}, v\widetilde{Y})\in v\pounds^\pi E$ and consequently $hT(v\widetilde{X}, v\widetilde{Y})=0$. Therefore we obtain
\begin{equation}\label{TD6}
hT(\widetilde{X}, \widetilde{Y})=A(\widetilde{X}, \widetilde{Y})+B(\widetilde{X}, \widetilde{Y})-B(\widetilde{Y}, \widetilde{X}).
\end{equation}
Similarly we get
\begin{equation}\label{TD7}
vT(\widetilde{X}, \widetilde{Y})=P^1(\widetilde{X}, \widetilde{Y})+R^1(\widetilde{X}, \widetilde{Y})+S^1(\widetilde{Y}, \widetilde{X})-P^1(\widetilde{Y}, \widetilde{X}).
\end{equation}
Summing (\ref{TD6}) and (\ref{TD7}) we conclude that the torsion $T$ of $D$ completely determined by (\ref{TD1})-(\ref{TD5}).
\end{proof}
It is easy to check that the components of the torsion tensor field have the following coordinate expressions:
\begin{equation}\label{very im2}
\left\{
\begin{array}{cc}
A=T^\gamma_{\alpha\beta}\delta_\gamma\otimes\mathcal{X}^\alpha\otimes\mathcal{X}^\beta,\ \ B=-C^\gamma_{\alpha\beta}\delta_\gamma\otimes\mathcal{X}^\alpha\otimes\mathcal{X}^\beta,\\
R^1=-R^\gamma_{\alpha\beta}\mathcal{V}_\gamma\otimes\mathcal{X}^\alpha\otimes\mathcal{X}^\beta,\ \
P^1=P^\gamma_{\alpha\beta}\mathcal{V}_\gamma\otimes\mathcal{X}^\alpha\otimes\mathcal{X}^\beta,\\ \hspace{-4.4cm}Q^1=S^\gamma_{\alpha\beta}\mathcal{V}_\gamma\otimes\mathcal{X}^\alpha\otimes\mathcal{X}^\beta,
\end{array}
\right.
\end{equation}
where
\begin{equation}\label{very im3}
(i)\ T^\gamma_{\alpha\beta}=F^\gamma_{\alpha\beta}-F^\gamma_{\beta\alpha}-(L^\gamma_{\alpha\beta}\circ\pi),\ (ii)\ P^\gamma_{\alpha\beta}=F^\gamma_{\alpha\beta}+\frac{\partial \mathcal{B}^\gamma_\alpha}{\partial\textbf{y}^\beta},\ (iii)\ S^\gamma_{\alpha\beta}=C^\gamma_{\alpha\beta}-C^\gamma_{\beta\alpha}.
\end{equation}
\begin{theorem}
Let $(D, h)$ be a d-connection on $\pounds^\pi E$. Then the curvature tensor field $K$ of $D$ completely determined by the following
\begin{align*}
(i)\ \ R(\widetilde{X}, \widetilde{Y})\widetilde{Z}:&=K(h\widetilde{X}, h\widetilde{Y})J\widetilde{Z},\\
(ii)\ \ P(\widetilde{X}, \widetilde{Y})\widetilde{Z}:&=K(h\widetilde{X}, J\widetilde{Y})J\widetilde{Z},\\
(iii)\ \  Q(\widetilde{X}, \widetilde{Y})\widetilde{Z}:&=K(J\widetilde{X}, J\widetilde{Y})J\widetilde{Z}.
\end{align*}
$R$, $P$ and $Q$ are called horizontal, mixed and vertical curvature, respectively.
\end{theorem}
\begin{proof}
Since $D$ is a d-connection, then we have
\[
D_{\widetilde{X}}J\widetilde{Y}=JD_{\widetilde{X}}\widetilde{Y},\ \ D_{\widetilde{X}}F\widetilde{Y}=FD_{\widetilde{X}}\widetilde{Y}.
\]
From the above relation we get
\begin{align}\label{*}
JK(\widetilde{X},\widetilde{Y})\widetilde{Z}&=K(\widetilde{X},\widetilde{Y})J\widetilde{Z},\nonumber\\ FK(\widetilde{X},\widetilde{Y})\widetilde{Z}&=K(\widetilde{X},\widetilde{Y})F\widetilde{Z},\\ JFK(\widetilde{X},\widetilde{Y})\widetilde{Z}&=K(\widetilde{X},\widetilde{Y})JF\widetilde{Z}\nonumber.
\end{align}
Therefore using $(i)$, $(iii)$ of (\ref{IM0}) we obtain
\begin{align*}
hK(\widetilde{X},\widetilde{Y})\widetilde{Z}&=FJK(\widetilde{X},\widetilde{Y})\widetilde{Z}=FK(\widetilde{X},\widetilde{Y})J\widetilde{Z}=FK(h\widetilde{X},h\widetilde{Y})J\widetilde{Z}\\
&\ +FK(h\widetilde{X},v\widetilde{Y})\widetilde{Z}+FK(v\widetilde{X},v\widetilde{Y})J\widetilde{Z}+FK(v\widetilde{X},h\widetilde{Y})J\widetilde{Z}\\
&=FR(\widetilde{X}, \widetilde{Y})\widetilde{Z}+FP(\widetilde{X}, F\widetilde{Y})\widetilde{Z}+FQ(F\widetilde{X}, F\widetilde{Y})\widetilde{Z}-FP(\widetilde{Y}, F\widetilde{X})\widetilde{Z}.
\end{align*}
Similarly, using (\ref{*}) we deduce
\[
vK(\widetilde{X},\widetilde{Y})\widetilde{Z}=R(\widetilde{X},\widetilde{Y})F\widetilde{Z}+P(\widetilde{X},F\widetilde{Y})F\widetilde{Z}+Q(F\widetilde{X},F\widetilde{Y})F\widetilde{Z}-P(\widetilde{Y},F\widetilde{X})F\widetilde{Z}.
\]
Summing two above equation we derive that $K$ completely determined by $R$, $P$ and $Q$.
\end{proof}
By a direct calculation, we can see that the horizontal, mixed and vertical curvature, have the following coordinate expressions:
\begin{align*}
R&=R_{\alpha\beta\gamma}^{\ \ \ \ \lambda}\ \mathcal{V}_\lambda\otimes\mathcal{X}^\alpha\otimes\mathcal{X}^\beta\otimes\mathcal{X}^\gamma,\\
P&=P_{\alpha\beta\gamma}^{\ \ \ \ \lambda}\ \mathcal{V}_\lambda\otimes\mathcal{X}^\alpha\otimes\mathcal{X}^\beta\otimes\mathcal{X}^\gamma,\\
Q&=S_{\alpha\beta\gamma}^{\ \ \ \ \lambda}\ \mathcal{V}_\lambda\otimes\mathcal{X}^\alpha\otimes\mathcal{X}^\beta\otimes\mathcal{X}^\gamma.
\end{align*}
where
\begin{align}
R^{\ \ \ \ \lambda}_{\alpha\beta\gamma}&=(\rho^i_\alpha\circ \pi)\frac{\partial F^\lambda_{\beta\gamma}}{\partial \textbf{x}^i}+\mathcal{B}^\mu_\alpha\frac{\partial F^\lambda_{\beta\gamma}}{\partial \textbf{y}^\mu}-(\rho^i_\beta\circ \pi)\frac{\partial F^\lambda_{\alpha\gamma}}{\partial \textbf{x}^i}-\mathcal{B}^\mu_\beta\frac{\partial F^\lambda_{\alpha\gamma}}{\partial \textbf{y}^\mu}+F^\mu_{\beta\gamma}F^\lambda_{\alpha\mu}\nonumber\\
&\ -F^\mu_{\alpha\gamma}F^\lambda_{\beta\mu}-(L^\mu_{\alpha\beta}\circ \pi)F^\lambda_{\mu\gamma}-R^{\ \ \mu}_{\alpha\beta}C^\lambda_{\mu\gamma},\label{hh}\\
P^{\ \ \ \ \lambda}_{\alpha\beta\gamma}&=(\rho^i_\alpha\circ \pi)\frac{\partial C^\lambda_{\beta\gamma}}{\partial \textbf{x}^i}+\mathcal{B}^\mu_\alpha\frac{\partial C^\lambda_{\beta\gamma}}{\partial \textbf{y}^\mu}+C^\mu_{\beta\gamma}F^\lambda_{\alpha\mu}-\frac{\partial F^\lambda_{\alpha\gamma}}{\partial \textbf{y}^\beta}-F^\mu_{\alpha\gamma}C^\lambda_{\beta\mu}+\frac{\partial \mathcal{B}^\mu_{\alpha}}{\partial \textbf{y}^\beta}C^\lambda_{\mu\gamma},\label{hv}\\
S^{\ \ \ \ \lambda}_{\alpha\beta\gamma}&=\frac{\partial C^\lambda_{\beta\gamma}}{\partial \textbf{y}^\alpha}+C^\mu_{\beta\gamma}C^\lambda_{\alpha\mu}-\frac{\partial C^\lambda_{\alpha\gamma}}{\partial \textbf{y}^\beta}-C^\mu_{\alpha\gamma}C^\lambda_{\beta\mu}.\label{vv}
\end{align}
\begin{defn}
Let $(D, h)$ be a d-connection on $\pounds^\pi E$. Then the tensor field
\begin{equation*}
\left\{
\begin{array}{cc}
P_{ric}:\Gamma(\pounds^\pi E)\times\Gamma(\pounds^\pi E) \rightarrow C^\infty(E),\\
(\widetilde{X},\widetilde{Y}) \rightarrow tr[F\circ(\widetilde{Z}\rightarrow P(\widetilde{Y},\widetilde{Z})\widetilde{X})],
\end{array}
\right.
\end{equation*}
is called mixed Ricci tensor of d-connection $(D, h)$, where $F$ is the almost complex structure associated to $h$.
\end{defn}
By a direct calculation we can see that the mixed Ricci tensor of $(D, h)$ has the following coordinate expression
\[
P_{ric}=P_{\alpha\beta}\mathcal{X}^\alpha\otimes\mathcal{X}^\beta,
\]
where $P_{\alpha\beta}=P_{\alpha\beta\gamma}^{\ \ \ \ \beta}$.
 \subsection{Intrinsic v-connections}
\begin{defn}
The canonical map
\begin{equation*}
\left\{
\begin{array}{cc}
\stackrel{i}{D}:\Gamma(\pounds^\pi E)\times\Gamma(\pounds^\pi E) \rightarrow\Gamma(\pounds^\pi E),\\
(J\widetilde{X},J\widetilde{Y}) \rightarrow D^i_{J\widetilde{X}}J\widetilde{Y}:=[J, J\widetilde{Y}]^{F-N}_{\pounds}\widetilde{X},
\end{array}
\right.
\end{equation*}
is called intrinsic or the flat $v$-connection in $v\pounds^\pi E$.
\end{defn}
\begin{lemma}
Let $\widetilde{X}$ and $\widetilde{Y}$ be two section of $\pounds^\pi E$. Then we have
\[
\stackrel{i}{D}_{J\widetilde{X}}J\widetilde{Y}:=J[J\widetilde{X}, \widetilde{Y}]_{\pounds},\ \ \stackrel{i}{D}_{v\widetilde{X}}J\widetilde{Y}:=J[v\widetilde{X}, \widetilde{Y}]_{\pounds}.
\]
\end{lemma}
\begin{proof}
From $N_J=0$, we obtain
\[
[J\widetilde{X},J\widetilde{Y}]_{\pounds}-J[\widetilde{X},J\widetilde{Y}]_{\pounds}-J[J\widetilde{X}, \widetilde{Y}]_{\pounds}=0.
\]
Therefore we get
\[
\stackrel{i}{D}_{J\widetilde{X}}J\widetilde{Y}:=[J, J\widetilde{Y}]^{F-N}_{\pounds}\widetilde{X}=[J\widetilde{X},J\widetilde{Y}]_{\pounds}-J[\widetilde{X},J\widetilde{Y}]_{\pounds}=J[J\widetilde{X}, \widetilde{Y}]_{\pounds}.
\]
Also since $v=J\circ F$, then the above equation gives us
\[
\stackrel{i}{D}_{v\widetilde{X}}J\widetilde{Y}=\stackrel{i}{D}_{JF\widetilde{X}}J\widetilde{Y}=J[JF\widetilde{X}, \widetilde{Y}]_{\pounds}=J[v\widetilde{X}, \widetilde{Y}]_{\pounds}.
\]
\end{proof}
Let $D^i$ be the intrinsic $v$-connection. We consider the map
\[
\stackrel{pi}{D}:\Gamma(v\pounds^\pi E)\times\Gamma(\pounds^\pi E) \rightarrow\Gamma(v\pounds^\pi E)
\]
defined by
\[
\stackrel{pi}{D}_{J\widetilde{X}}J\widetilde{Y}=\stackrel{i}{D}_{J\widetilde{X}}J\widetilde{Y},\ \ \stackrel{pi}{D}_{J\widetilde{X}}h\widetilde{Y}=F\stackrel{i}{D}_{J\widetilde{X}}J\widetilde{Y}.
\]
It is easy to see that
\begin{equation}\label{m1}
\stackrel{pi}{D}_{J\widetilde{X}}J\widetilde{Y}=J[J\widetilde{X}, \widetilde{Y}]_{\pounds},\ \ \stackrel{pi}{D}_{J\widetilde{X}}h\widetilde{Y}=h[J\widetilde{X}, \widetilde{Y}]_{\pounds}.
\end{equation}
\begin{theorem}
Let $(D,h)$ be a d-connection on $\pounds^\pi E$ and $\stackrel{pi}{D}$ be given by (\ref{m1}). If $\widetilde{D}$ is the map
\begin{equation}\label{m2}
\left\{
\begin{array}{cc}
\widetilde{D}:\Gamma(\pounds^\pi E)\times\Gamma(\pounds^\pi E) \rightarrow\Gamma(\pounds^\pi E),\\
(\widetilde{X},\widetilde{Y}) \rightarrow \widetilde{D}_{\widetilde{X}}\widetilde{Y}:=D_{h\widetilde{X}}\widetilde{Y}+\stackrel{pi}{D}_{v\widetilde{X}}\widetilde{Y},
\end{array}
\right.
\end{equation}
then $(\widetilde{D},h)$ is a d-connection on $\pounds^\pi E$, which is called the d-connection associated to $(D,h)$.
\end{theorem}
\begin{proof}
At first we show that $\widetilde{D}$ is a linear connection on $\pounds^\pi E$. Let $f\in C^\infty(M)$. Then we have $D_{h\widetilde{X}}f\widetilde{Y}=\rho_\pounds(h\widetilde{X})(f)\widetilde{Y}+fD_{h\widetilde{X}}\widetilde{Y}$, because $D$ is a linear connection. Direct calculations give us
\begin{align*}
\stackrel{pi}{D}_{v\widetilde{X}}fY&=\stackrel{pi}{D}_{v\widetilde{X}}hf\widetilde{Y}+\stackrel{pi}{D}_{v\widetilde{X}}vf\widetilde{Y}
=\stackrel{pi}{D}_{v\widetilde{X}}hf\widetilde{Y}
+\stackrel{pi}{D}_{v\widetilde{X}}fJF\widetilde{Y}\\
&=h[v\widetilde{X},f\widetilde{Y}]_{\pounds}+J[v\widetilde{X},fF\widetilde{Y}]_{\pounds}
=h\{\rho_{\pounds}(v\widetilde{X})(f)\widetilde{Y}+f[v\widetilde{X},\widetilde{Y}]_{\pounds}\}\\
&\ \ \ \ +J\{\rho_{\pounds}(v\widetilde{X})(f)F\widetilde{Y}+f[v\widetilde{X},F\widetilde{Y}]_{\pounds}\}=\rho_{\pounds}(v\widetilde{X})(f)h\widetilde{Y}
+fh[v\widetilde{X},\widetilde{Y}]_{\pounds}\\
&\ \ \ \ +\rho_{\pounds}(v\widetilde{X})(f)v\widetilde{Y}+fJ[v\widetilde{X},F\widetilde{Y}]_{\pounds}=\rho_{\pounds}(v\widetilde{X})(f)(\widetilde{Y})
+f\stackrel{pi}{D}_{v\widetilde{X}}\widetilde{Y}.
\end{align*}
Therefore we have
\[
\widetilde{D}_{\widetilde{X}}f\widetilde{Y}=\rho_{\pounds}(h\widetilde{X})(f)Y+\rho_{\pounds}(v\widetilde{X})(f)Y+fD_{h\widetilde{X}}\widetilde{Y}
+f\stackrel{pi}{D}_{v\widetilde{X}}\widetilde{Y}=\rho_{\pounds}(\widetilde{X})(f)\widetilde{Y}+f\widetilde{D}_{\widetilde{X}}\widetilde{Y}.
\]
Similarly we can prove
\[
\widetilde{D}_{\widetilde{X}}(\widetilde{Y}+\widetilde{Z})=\widetilde{D}_{\widetilde{X}}\widetilde{Y}+\widetilde{D}_{\widetilde{Y}}\widetilde{Z},\ \ \ \widetilde{D}_{f\widetilde{X}+\widetilde{Y}}\widetilde{Z}=f\widetilde{D}_{\widetilde{X}}\widetilde{Z}+\widetilde{D}_{\widetilde{Y}}\widetilde{Z}.
\]
Thus $\widetilde{D}$ is a linear connection on $\pounds^\pi E$. Now, we show that $\widetilde{D}$ is reducible. Since $D$ is reducible, then we have $Dh=0$. So,
\begin{align*}
(\widetilde{D}_{\widetilde{X}}h)(\widetilde{Y})&=\widetilde{D}_{\widetilde{X}}h\widetilde{Y}-h\widetilde{D}_{\widetilde{X}}\widetilde{Y}
={D}_{h\widetilde{X}}h\widetilde{Y}+\stackrel{pi}{D}_{v\widetilde{X}}h\widetilde{Y}
-h{D}_{h\widetilde{X}}\widetilde{Y}-h\stackrel{pi}{D}_{v\widetilde{X}}\widetilde{Y}\\
&=({D}_{h\widetilde{X}}h)(\widetilde{Y})+\stackrel{pi}{D}_{v\widetilde{X}}h\widetilde{Y}-h{D}^{pi}_{v\widetilde{X}}h\widetilde{Y}
-h\stackrel{pi}{D}_{v\widetilde{X}}v\widetilde{Y}\\
&=v\stackrel{pi}{D}_{v\widetilde{X}}h\widetilde{Y}-h\stackrel{pi}{D}_{v\widetilde{X}}v\widetilde{Y}\\
&=vh[v\widetilde{X}, \widetilde{Y}]_\pounds+hJ[v\widetilde{X}, F\widetilde{Y}]_\pounds=0.
\end{align*}
Similarly, we can show that $\widetilde{D}F=0$, i.e., $\widetilde{D}$ is an almost complex connection. Therefore $(\widetilde{D}, h)$ is a d-connection on $\pounds^\pi E$.
\end{proof}
Let $\widetilde{X}=\widetilde{X}^\alpha\delta_\alpha+\widetilde{X}^{\bar\alpha}\mathcal{V}_\alpha$ and $Y=\widetilde{Y}^\beta\delta_\beta+\widetilde{Y}^{\bar\beta}\mathcal{V}_\beta$ are sections of $\pounds^\pi E$ and $(F_{\alpha\beta}^\gamma, C_{\alpha\beta}^\gamma)$ are the local coefficients of d-connection $D$. Using (\ref{rho}), (\ref{m1}) and (\ref{m2}) we deduce the following coordinate expression for $\widetilde{D}$:
\begin{align}
\widetilde{D}_{\widetilde{X}}\widetilde{Y}&=\Big(\widetilde{X}^\alpha(\rho^i_\alpha\circ\pi)\frac{\partial \widetilde{Y}^\beta}{\partial\textbf{x}^i}+\widetilde{X}^\alpha \mathcal{B}^\gamma_\alpha\frac{\partial \widetilde{Y}^\beta}{\partial\textbf{y}^\gamma}+\widetilde{X}^\alpha \widetilde{Y}^\gamma F_{\alpha\gamma}^\beta+{\widetilde{X}}^{\bar\alpha}\frac{\partial \widetilde{Y}^\beta}{\partial\textbf{y}^{\alpha}}\Big)\delta_\beta\nonumber\\
&\ \ \ +\Big(\widetilde{X}^\alpha(\rho^i_\alpha\circ\pi)\frac{\partial {\widetilde{Y}}^{\bar\beta}}{\partial\textbf{x}^i}+\widetilde{X}^\alpha \mathcal{B}^\gamma_\alpha\frac{\partial \widetilde{Y}^{\bar\beta}}{\partial\textbf{y}^\gamma}+\widetilde{X}^\alpha {\widetilde{Y}}^{\bar\gamma } F_{\alpha\gamma}^\beta+{\widetilde{X}}^{\bar\alpha}\frac{\partial {\widetilde{Y}}^{\bar\beta}}{\partial\textbf{y}^\alpha}\Big)\mathcal{V}_\beta.
\end{align}
If we denote the local coefficients of d-connection $\widetilde{D}$ by $(\widetilde{F}_{\alpha\beta}^\gamma, \widetilde{C}_{\alpha\beta}^\gamma)$, then from the above equation we conclude $\widetilde{F}_{\alpha\beta}^\gamma=F_{\alpha\beta}^\gamma$ and $\widetilde{C}_{\alpha\beta}^\gamma=0$. Therefore using (\ref{hh}), (\ref{hv}) and (\ref{vv}) we derive that
\begin{align}
\widetilde{R}^{\ \ \ \ \lambda}_{\alpha\beta\gamma}&=(\rho^i_\alpha\circ \pi)\frac{\partial F^\lambda_{\beta\gamma}}{\partial \textbf{x}^i}+\mathcal{B}^\mu_\alpha\frac{\partial F^\lambda_{\beta\gamma}}{\partial \textbf{y}^\mu}-(\rho^i_\beta\circ \pi)\frac{\partial F^\lambda_{\alpha\gamma}}{\partial \textbf{x}^i}-\mathcal{B}^\mu_\beta\frac{\partial F^\lambda_{\alpha\gamma}}{\partial \textbf{y}^\mu}+F^\mu_{\beta\gamma}F^\lambda_{\alpha\mu}\nonumber\\
&\ -F^\mu_{\alpha\gamma}F^\lambda_{\beta\mu}-(L^\mu_{\alpha\beta}\circ \pi)F^\lambda_{\mu\gamma},\nonumber\\
\widetilde{P}^{\ \ \ \ \lambda}_{\alpha\beta\gamma}&=-\frac{\partial F^\lambda_{\alpha\gamma}}{\partial \textbf{y}^\beta},\ \ \ \ \
\widetilde{S}^{\ \ \ \ \lambda}_{\alpha\beta\gamma}=0,\label{niaz}
\end{align}
where $\widetilde{R}^{\ \ \ \ \lambda}_{\alpha\beta\gamma}$, $\widetilde{P}^{\ \ \ \ \lambda}_{\alpha\beta\gamma}$ and $\widetilde{S}^{\ \ \ \ \lambda}_{\alpha\beta\gamma}$ are the coefficients of the horizontal, mixed and vertical curvatures of d-connection $(\widetilde{D}, h)$, respectively. Therefore the vertical curvature of d-connection $\widetilde{D}$ is vanished. Also, it is easy to see that
\begin{proposition}
The mixed curvature $\widetilde{P}$ of $\widetilde{D}$ satisfies
\[
\widetilde{P}(X^C, Y^C)Z^C=-[J, D_{X^h} Z^V]^{F-N}_\pounds Y^C.
\]
\end{proposition}
\subsection{Berwald-type connection}
Let $h$ be a horizontal endomorphism on $\pounds^\pi E$. Then the map
\begin{equation*}
\left\{
\begin{array}{cc}
\stackrel{\text{\begin{tiny}B\end{tiny}}}{D}:\Gamma(\pounds^\pi E)\times\Gamma(\pounds^\pi E) \rightarrow\Gamma(\pounds^\pi E),\\
(\widetilde{X},\widetilde{Y}) \rightarrow \stackrel{\text{\begin{tiny}B\end{tiny}}}{D}_{\widetilde{X}}\widetilde{Y},
\end{array}
\right.
\end{equation*}
defined by
\begin{equation*}
\stackrel{\text{\begin{tiny}B\end{tiny}}}{D}_{\widetilde{X}}\widetilde{Y}:=hF[h\widetilde{X},J\widetilde{Y}]_\pounds+v[h\widetilde{X},v\widetilde{Y}]_\pounds
+h[v\widetilde{X},\widetilde{Y}]_\pounds+J[v\widetilde{X},F\widetilde{Y}]_\pounds,
\end{equation*}
is a linear connection on $\pounds^\pi E$. Similar to $\widetilde{D}$, we can prove that $\stackrel{\text{\begin{tiny}B\end{tiny}}}{D}\!\!h=\stackrel{\text{\begin{tiny}B\end{tiny}}}{D}\!\!F=0$. Therefore  $(\stackrel{\text{\begin{tiny}B\end{tiny}}}{D}, h)$ is a d-connection, which is called the \textit{Berwald-type} connection. If, in particular, $h$ is a Berwald endomorphism, then we call $(\stackrel{\text{\begin{tiny}B\end{tiny}}}{D}, h)$ a \textit{Berwald connection}.

It is easy to see that
\begin{equation}\label{4ta}
\left\{
\begin{array}{cc}
\stackrel{\text{\begin{tiny}B\end{tiny}}}{D}_{\delta_\alpha}\delta_\beta=-\frac{\partial \mathcal{B}^\gamma_\alpha}{\partial y^\beta}\delta_\gamma,\ \ \ \ \ \stackrel{\text{\begin{tiny}B\end{tiny}}}{D}_{v_\alpha}v_\beta=0,\\
\stackrel{\text{\begin{tiny}B\end{tiny}}}{D}_{\delta_\alpha}v_\beta=-\frac{\partial \mathcal{B}^\gamma_\alpha}{\partial y^\beta}v_\gamma,\ \ \ \ \ \stackrel{\text{\begin{tiny}B\end{tiny}}}{D}_{v_\alpha}\delta_\beta=0.
\end{array}
\right.
\end{equation}
If we denote the local coefficients of Berwald connection ${\stackrel{\text{\begin{tiny}B\end{tiny}}}{D}}$ by $({\stackrel{\text{\begin{tiny}B\end{tiny}}}{F}}_{\alpha\beta}^\gamma, {\stackrel{\text{\begin{tiny}B\end{tiny}}}{C}}_{\alpha\beta}^\gamma)$, then from the above equation we conclude ${\stackrel{\text{\begin{tiny}B\end{tiny}}}{F}}_{\alpha\beta}^\gamma=-\frac{\partial \mathcal{B}^\gamma_\alpha}{\partial y^\beta}$ and ${\stackrel{\text{\begin{tiny}B\end{tiny}}}{C}}_{\alpha\beta}^\gamma=0$.
Therefore using (\ref{hh}), (\ref{hv}) and (\ref{vv}) we derive that
\begin{align}
{\stackrel{\text{\begin{tiny}B\end{tiny}}}{R}}^{\ \ \ \ \lambda}_{\alpha\beta\gamma}&=-(\rho^i_\alpha\circ \pi)\frac{\partial^2 \mathcal{B}^\lambda_{\beta}}{\partial \textbf{x}^i\partial \textbf{y}^\gamma}-\mathcal{B}^\mu_\alpha\frac{\partial^2 \mathcal{B}^\lambda_{\beta}}{\partial \textbf{y}^\mu\partial \textbf{y}^\gamma}+(\rho^i_\beta\circ \pi)\frac{\partial^2 \mathcal{B}^\lambda_{\alpha}}{\partial \textbf{x}^i\partial \textbf{y}^\gamma}+\mathcal{B}^\mu_\beta\frac{\partial^2 \mathcal{B}^\lambda_{\alpha}}{\partial \textbf{y}^\mu\partial \textbf{y}^\gamma}\nonumber\\
&\ \ \ +\frac{\partial \mathcal{B}^\mu_{\beta}}{\partial \textbf{y}^\gamma}\frac{\partial \mathcal{B}^\lambda_{\alpha}}{\partial \textbf{y}^\mu}-\frac{\partial \mathcal{B}^\mu_{\alpha}}{\partial \textbf{y}^\gamma}\frac{\partial \mathcal{B}^\lambda_{\beta}}{\partial \textbf{y}^\mu}+(L^\mu_{\alpha\beta}\circ \pi)\frac{\partial \mathcal{B}^\lambda_{\mu}}{\partial \textbf{y}^\gamma},\label{RB}\\
{\stackrel{\text{\begin{tiny}B\end{tiny}}}{P}}^{\ \ \ \ \lambda}_{\alpha\beta\gamma}&=\frac{\partial^2 \mathcal{B}^\lambda_{\alpha}}{\partial \textbf{y}^\beta\partial \textbf{y}^\gamma},\label{mixed}\\
{\stackrel{\text{\begin{tiny}B\end{tiny}}}{S}}^{\ \ \ \ \lambda}_{\alpha\beta\gamma}&=0,\label{QB}
\end{align}
where ${\stackrel{\text{\begin{tiny}B\end{tiny}}}{R}}^{\ \ \ \ \lambda}_{\alpha\beta\gamma}$, ${\stackrel{\text{\begin{tiny}B\end{tiny}}}{P}}^{\ \ \ \ \lambda}_{\alpha\beta\gamma}$ and ${\stackrel{\text{\begin{tiny}B\end{tiny}}}{S}}^{\ \ \ \ \lambda}_{\alpha\beta\gamma}$ are the coefficients of the horizontal, mixed and vertical curvatures of d-connection $(\stackrel{\text{\begin{tiny}B\end{tiny}}}{D}, h)$, respectively. Therefore the vertical curvature of d-connection $\stackrel{\text{\begin{tiny}B\end{tiny}}}{D}$ vanishes.
\begin{proposition}
Let $(\stackrel{\text{\begin{tiny}B\end{tiny}}}{D}, h)$ be the Berwald-type connection. Then

(i)\ The h-deflection of $(\stackrel{\text{\begin{tiny}B\end{tiny}}}{D}, h)$ coincides with the tension of $h$.

(ii)\ The torsion tensor field $\stackrel{\text{\begin{tiny}B\end{tiny}}}{T}$ of $\stackrel{\text{\begin{tiny}B\end{tiny}}}{D}$ can be represented in the form
\begin{equation}
\stackrel{\text{\begin{tiny}B\end{tiny}}}{T}=F\circ t+\Omega,
\end{equation}
where $t$ and $\Omega$ are the weak torsion and the curvature of $h$.
\end{proposition}
\begin{proof}
(i)\ Let $\widetilde{X}=\widetilde{X}^\alpha\delta_\alpha+\widetilde{X}^{\bar\alpha}\mathcal{V}_\alpha$ be a section of $\pounds^\pi E$. Then using (\ref{Liouville}), (\ref{tension}), (\ref{rho}) and (\ref{h*}) we have
\[
h^*(DC)(\widetilde{X})=D_{h\widetilde{X}}C=D_{\widetilde{X}^\alpha\delta_\alpha}(\textbf{y}^\beta\mathcal{V}_\beta)=\widetilde{X}^\alpha(\mathcal{B}^\gamma_\alpha
-\textbf{y}^\beta\frac{\partial \mathcal{B}^\gamma_\alpha}{\partial \textbf{y}^\beta})\mathcal{V}_\gamma=H(\widetilde{X}).
\]
(ii) Using (\ref{wt}), (\ref{wt1}), (\ref{curv000}), (\ref{curv0}) and (\ref{4ta}) we obtain
\begin{align*}
\stackrel{\text{\begin{tiny}B\end{tiny}}}{T}(\delta_\alpha, \delta_\beta)&=(\frac{\partial\mathcal{B}^\gamma_\beta}{\partial\textbf{y}^\alpha}-\frac{\partial\mathcal{B}^\gamma_\alpha}{\partial\textbf{y}^\beta}
-(L^\gamma_{\alpha\beta}\circ\pi))\delta_\gamma-R_{\alpha\beta}^{\gamma}\mathcal{V_\gamma}\\
&=Ft(\delta_\alpha, \delta_\beta)+\Omega(\delta_\alpha, \delta_\beta),\\
\stackrel{\text{\begin{tiny}B\end{tiny}}}{T}(\delta_\alpha, \mathcal{V}_\beta)&=0=Ft(\delta_\alpha, \mathcal{V}_\beta)+\Omega(\delta_\alpha, \mathcal{V}_\beta),\\
\stackrel{\text{\begin{tiny}B\end{tiny}}}{T}(\mathcal{V}_\alpha, \mathcal{V}_\beta)&=0=Ft(\mathcal{V}_\alpha, \mathcal{V}_\beta)+\Omega(\mathcal{V}_\alpha, \mathcal{V}_\beta).
\end{align*}
\end{proof}
Similar to Lemma \ref{lemma0} we can prove
\begin{lemma}\label{important}
A section $\widetilde{X}=\widetilde{X}^\alpha\delta_\alpha+\widetilde{X}^{\bar\alpha}\mathcal{V}_\alpha$ is homogenous of degree $r$ if and only if
\[
\textbf{y}^\alpha\frac{\partial \widetilde{X}^\beta}{\partial \textbf{y}^\alpha}=(r-1)\widetilde{X}^\beta,\ \ \ \textbf{y}^\alpha\frac{\partial \widetilde{X}^{\bar\beta}}{\partial \textbf{y}^\alpha}+\widetilde{X}^\gamma(\textbf{y}^\alpha\frac{\partial \mathcal{B}^\beta_\gamma}{\partial \textbf{y}^\alpha}-\mathcal{B}^\beta_\gamma)=r\widetilde{X}^{\bar\beta}.
\]
\end{lemma}
\begin{proposition}
The mixed curvature $\stackrel{\text{\begin{tiny}B\end{tiny}}}{P}$ of Berwald-type connection $\stackrel{\text{\begin{tiny}B\end{tiny}}}{D}$ is symmetric with respect to last two variables. Moreover, if $h$ is torsion free, then $\stackrel{\text{\begin{tiny}B\end{tiny}}}{P}$ is symmetric with respect to all variables.
\end{proposition}
\begin{proof}
Equation (\ref{mixed}) told us that ${\stackrel{\text{\begin{tiny}B\end{tiny}}}{P}}^{\ \ \ \ \lambda}_{\alpha\beta\gamma}$ is symmetric with respect to last two indices. Therefore $\stackrel{\text{\begin{tiny}B\end{tiny}}}{P}$ is symmetric with respect to last two variables. Now, let $h$ be torsion free. Then using (\ref{wt1}) we obtain
\[
{\stackrel{\text{\begin{tiny}B\end{tiny}}}{P}}^{\ \ \ \ \lambda}_{\alpha\beta\gamma}=\frac{\partial^2 \mathcal{B}^\lambda_{\alpha}}{\partial \textbf{y}^\beta\partial \textbf{y}^\gamma}=\frac{\partial}{\partial \textbf{y}^\gamma}(\frac{\partial \mathcal{B}^\lambda_{\alpha}}{\partial \textbf{y}^\beta})=\frac{\partial}{\partial \textbf{y}^\gamma}(\frac{\partial \mathcal{B}^\lambda_{\beta}}{\partial \textbf{y}^\alpha}+(L^\lambda_{\beta\alpha}\circ\pi))=\frac{\partial^2 \mathcal{B}^\lambda_{\beta}}{\partial \textbf{y}^\alpha\partial \textbf{y}^\gamma}={\stackrel{\text{\begin{tiny}B\end{tiny}}}{P}}^{\ \ \ \ \lambda}_{\beta\alpha\gamma}.
\]
Similarly, we can prove ${\stackrel{\text{\begin{tiny}B\end{tiny}}}{P}}^{\ \ \ \ \lambda}_{\alpha\beta\gamma}={\stackrel{\text{\begin{tiny}B\end{tiny}}}{P}}^{\ \ \ \ \lambda}_{\gamma\beta\alpha}$.
\end{proof}
\begin{proposition}\label{Berwald}
Let $h$ be a homogenous horizontal endomorphism on $\pounds^\pi E$. Then the mixed curvature $\stackrel{\text{\begin{tiny}B\end{tiny}}}{P}$ of $(\stackrel{\text{\begin{tiny}B\end{tiny}}}{D},h)$ is homogenous of degree $-1$. Moreover if the weak torsion of $h$ is zero, then for any semispray $S$ we have $i_S\!\!\stackrel{\text{\begin{tiny}B\end{tiny}}}{P}=0$.
\end{proposition}
\begin{proof}
To proof the first part of proposition, using the above lemma, we must show $y^\mu\frac{\partial \stackrel{\text{\begin{tiny}B\end{tiny}}}{P}^{\ \ \ \lambda}_{\alpha\beta\gamma}}{\partial y^\mu}=-\stackrel{\text{\begin{tiny}B\end{tiny}}}{P}^{\ \ \ \lambda}_{\alpha\beta\gamma}$. Since $h$ is homogenous, then we have $y^\beta\frac{\partial \mathcal{B}^\lambda_\alpha}{\partial y^\beta}=\mathcal{B}^\lambda_\alpha$. Differentiating with respect to $y^\gamma$ we obtain
\begin{equation}\label{0}
y^\beta\frac{\partial^2 \mathcal{B}^\lambda_\alpha}{\partial y^\beta\partial y^\gamma}=0.
\end{equation}
Differentiating (\ref{0}) with respect to $y^\mu$ gives us
\begin{equation}\label{00}
y^\beta\frac{\partial^3 \mathcal{B}^\lambda_\alpha}{\partial y^\mu\partial y^\beta\partial y^\gamma}=-\frac{\partial^2 \mathcal{B}^\lambda_\alpha}{\partial y^\mu\partial y^\gamma}.
\end{equation}
Therefore we have
\[
y^\mu\frac{\partial \stackrel{\text{\begin{tiny}B\end{tiny}}}{P}^{\ \ \ \lambda}_{\alpha\beta\gamma}}{\partial y^\mu}=y^\beta\frac{\partial^3 \mathcal{B}^\lambda_\alpha}{\partial y^\mu\partial y^\beta\partial y^\gamma}=-\frac{\partial^2 \mathcal{B}^\lambda_\alpha}{\partial y^\mu\partial y^\gamma}=-\stackrel{\text{\begin{tiny}B\end{tiny}}}{P}^{\ \ \ \lambda}_{\alpha\beta\gamma}.
\]
Now, we proof the second part of assertion. From the above Proposition, we deduce that $\stackrel{\text{\begin{tiny}B\end{tiny}}}{P}$ is symmetric with respect to all variables. Thus we have $(i_S\!\!\stackrel{\text{\begin{tiny}B\end{tiny}}}{P})(\widetilde{X},\widetilde{Y})=\stackrel{\text{\begin{tiny}B\end{tiny}}}{P}(\widetilde{X},S)\widetilde{Y}$. Thus using (\ref{mixed}) and (\ref{0}) we get
\begin{equation}\label{*1}
(i_S\stackrel{\text{\begin{tiny}B\end{tiny}}}{P})(\widetilde{X},\widetilde{Y})=\widetilde{X}^\alpha y^\beta \widetilde{Y}^\gamma \stackrel{\text{\begin{tiny}B\end{tiny}}}{P}^{\ \ \ \lambda}_{\alpha\beta\gamma}\mathcal{V_\lambda}=\widetilde{X}^\alpha y^\beta \widetilde{Y}^\gamma\frac{\partial^2 \mathcal{B}^\lambda_\alpha}{\partial \textbf{y}^\beta \textbf{y}^\gamma}=0,
\end{equation}
where $\widetilde{X}=\widetilde{X}^\alpha\delta_\alpha+\widetilde{X}^{\bar{\alpha}}\mathcal{V_\alpha}$, $\widetilde{Y}=\widetilde{Y}^\beta\delta_\beta+\widetilde{Y}^{\bar{\beta}}\mathcal{V_\beta}$.
\end{proof}
\begin{proposition}
The mixed curvature $\stackrel{\text{\begin{tiny}B\end{tiny}}}{P}$ of Berwald-type connection $D^0$ satisfies
\[
\stackrel{\text{\begin{tiny}B\end{tiny}}}{P}(X^C, Y^C)Z^C=[[X^h, Y^V]_\pounds, Z^V]_\pounds.
\]
\end{proposition}
\begin{proof}

Let $X=X^\alpha e_\alpha$, $Y=Y^\beta e_\beta$ and $Z=Z^\gamma e_\gamma$ are sections of $E$. Then we can obtain
\[
\stackrel{\text{\begin{tiny}B\end{tiny}}}{P}(X^C, Y^C)Z^C=((X^\alpha Y^\beta Z^\gamma)\circ\pi)\frac{\partial^2 \mathcal{B}^\lambda_\alpha}{\partial\textbf{y}^\beta\partial\textbf{y}^\gamma}\mathcal{V}_\lambda=[[X^h, Y^V]_\pounds, Z^V]_\pounds.
\]
\end{proof}
\begin{proposition}
Let $h$ be a homogenous horizontal endomorphism on $\pounds^\pi E$. The mixed Ricci tensor $\stackrel{\text{\begin{tiny}B\end{tiny}}}{P}_{ric}$ of Berwald-type connection $(\stackrel{\text{\begin{tiny}B\end{tiny}}}{D},h)$ is homogenous of degree $-1$. Moreover, we have
\[
\pounds_C^\pounds \stackrel{\text{\begin{tiny}B\end{tiny}}}{P}_{ric}=\stackrel{\text{\begin{tiny}B\end{tiny}}}{D}_C\stackrel{\text{\begin{tiny}B\end{tiny}}}{P}_{ric}
=-\stackrel{\text{\begin{tiny}B\end{tiny}}}{P}_{ric}.
\]
\end{proposition}
\begin{proof}
Using (\ref{mixed}) and (\ref{00}) we have
\[
\textbf{y}^\lambda\frac{\partial \stackrel{\text{\begin{tiny}B\end{tiny}}}{P}_{\alpha\gamma}}{\partial \textbf{y}^\lambda}=\textbf{y}^\lambda\frac{\partial {\stackrel{\text{\begin{tiny}B\end{tiny}}}{P}}_{\alpha\beta\gamma}^{\ \ \ \ \beta}}{\partial \textbf{y}^\lambda}=\textbf{y}^\lambda\frac{\partial^3 \mathcal{B}^\beta_\alpha}{\partial \textbf{y}^\lambda\partial \textbf{y}^\beta\partial \textbf{y}^\gamma}=-\frac{\partial^2 \mathcal{B}^\beta_\alpha}{\partial \textbf{y}^\beta\partial \textbf{y}^\gamma}=-\stackrel{\text{\begin{tiny}B\end{tiny}}}{P}_{\alpha\gamma}.
\]
Thus from Lemma \ref{important}, we deduce $-\stackrel{\text{\begin{tiny}B\end{tiny}}}{P}_{ric}$ is homogenous of degree $-1$. Also, using (\ref{4ta}) we get $\stackrel{\text{\begin{tiny}B\end{tiny}}}{D}_C\stackrel{\text{\begin{tiny}B\end{tiny}}}{P}_{ric}(\delta_\alpha,\mathcal{V_\beta})
=\stackrel{\text{\begin{tiny}B\end{tiny}}}{D}_C\stackrel{\text{\begin{tiny}B\end{tiny}}}{P}_{ric}(\mathcal{V_\alpha},\mathcal{V_\beta})=0$ and
\[
(\stackrel{\text{\begin{tiny}B\end{tiny}}}{D}_C\stackrel{\text{\begin{tiny}B\end{tiny}}}{P}_{ric})(\delta_\alpha,\delta_\beta)
=\stackrel{\text{\begin{tiny}B\end{tiny}}}{D}_C\stackrel{\text{\begin{tiny}B\end{tiny}}}{P}_{ric}(\delta_\alpha,\delta_\beta)=\textbf{y}^\gamma\frac{\partial \stackrel{\text{\begin{tiny}B\end{tiny}}}{P}_{\alpha\beta}}{\partial \textbf{y}^\gamma}=-\stackrel{\text{\begin{tiny}B\end{tiny}}}{P}_{\alpha\beta}.
\]
Therefore we deduce $\stackrel{\text{\begin{tiny}B\end{tiny}}}{D}_C\stackrel{\text{\begin{tiny}B\end{tiny}}}{P}_{ric}=-\stackrel{\text{\begin{tiny}B\end{tiny}}}{P}_{ric}$. Similarly we have $(\pounds_C^\pounds \stackrel{\text{\begin{tiny}B\end{tiny}}}{P}_{ric})(\delta_\alpha,\mathcal{V_\beta})=(\pounds_C^\pounds \stackrel{\text{\begin{tiny}B\end{tiny}}}{P}_{ric})(\mathcal{V_\alpha},\mathcal{V_\beta})=0$ and
\[
(\pounds_C^\pounds \stackrel{\text{\begin{tiny}B\end{tiny}}}{P}_{ric})(\delta_\alpha,\delta_\beta)=C(\stackrel{\text{\begin{tiny}B\end{tiny}}}{P}_{ric}(\delta_\alpha,\delta_\beta))
=\textbf{y}^\gamma\frac{\partial \stackrel{\text{\begin{tiny}B\end{tiny}}}{P}_{\alpha\beta}}{\partial \textbf{y}^\gamma}=-\stackrel{\text{\begin{tiny}B\end{tiny}}}{P}_{ric}.
\]
Therefore $\pounds_C^\pounds \stackrel{\text{\begin{tiny}B\end{tiny}}}{P}_{ric}=-\stackrel{\text{\begin{tiny}B\end{tiny}}}{P}_{ric}$.

\end{proof}
\begin{proposition}
Let $(\stackrel{\text{\begin{tiny}B\end{tiny}}}{D}, h)$ be a Berwald-type d-connection and $K\in\Gamma(\wedge^kE^*\otimes E)$ be a semibasic. Then
\[
\stackrel{\text{\begin{tiny}B\end{tiny}}}{D}_{X^V}K=\pounds^\pounds_{X^V}K,\ \ \ \ \forall X\in\Gamma(E).
\]
\end{proposition}
\begin{proof}
\begin{align*}
(\pounds^\pounds_{X^V}K)(\delta_{\alpha_1},\ldots, \delta_{\alpha_k})&=(\pounds^\pounds_{X^V}K)(e_{\alpha_1}^h,\ldots,e_{\alpha_k}^h)\\
&=[X^V, K(e_{\alpha_1}^h,\ldots,e_{\alpha_k}^h)]_\pounds-\sum_i K(e_{\alpha_1}^h,\ldots,[X^V, e_{\alpha_i}^h]_\pounds,\ldots,e_{\alpha_k}^h)\\
&=[X^V, K(e_{\alpha_1}^h,\ldots,e_{\alpha_k}^h)]_\pounds=-[JFK(e_{\alpha_1}^h,\ldots,e_{\alpha_k}^h),X^V]_\pounds\\
&=[J,X^V]^{F-N}_\pounds(FK(X_1^h,\ldots,X_s^h))-J[FK(X_1^h,\ldots,X_s^h),X^V]_\pounds\\
&=J[JX^C,FK(e_{\alpha_1}^h,\ldots,e_{\alpha_k}^h)]_\pounds=\stackrel{\text{\begin{tiny}B\end{tiny}}}{D}_{X^V}K(e_{\alpha_1}^h,\ldots,e_{\alpha_k}^h)\\
&=(\stackrel{\text{\begin{tiny}B\end{tiny}}}{D}_{X^V}K)(e_{\alpha_1}^h,\ldots,e_{\alpha_k}^h)+\sum_iK(e_{\alpha_1}^h,\ldots,\stackrel{\text{\begin{tiny}B\end{tiny}}}{D}_{X^V}e_{\alpha_i}^h,\ldots,e_{\alpha_k}^h)\\
&=(\stackrel{\text{\begin{tiny}B\end{tiny}}}{D}_{X^V}K)(e_{\alpha_1}^h,\ldots,e_{\alpha_k}^h)\\
&=(\stackrel{\text{\begin{tiny}B\end{tiny}}}{D}_{X^V}K)(\delta_{\alpha_1},\ldots, \delta_{\alpha_k}).
\end{align*}
\end{proof}
\subsection{Yano-type connection}
Let $h$ be a horizontal endomorphism on  $\pounds^\pi E$ with associated almost complex structure $F$ and $\omega\in\Gamma(\wedge^2(\pounds^\pi E)^*)$ be a symmetric tensor, satisfying the condition
\begin{equation}
i_S\omega=0,
\end{equation}
where $S$ is an arbitrary semispray on $\pounds^\pi E$. We define the mapping
\[
D:\Gamma(\pounds^\pi E)\times\Gamma(\pounds^\pi E)\rightarrow\Gamma(\pounds^\pi E),
\]
by the following rules:
\begin{align}
D_{v\widetilde{X}}v\widetilde{Y}&=J[v\widetilde{X}, F\widetilde{Y}]_\pounds=\stackrel{\text{\begin{tiny}B\end{tiny}}}{D}_{v\widetilde{X}}v\widetilde{Y},\label{Y1}\\
D_{h\widetilde{X}}v\widetilde{Y}&=v[h\widetilde{X}, v\widetilde{Y}]_\pounds+\omega(\widetilde{X}, F\widetilde{Y})\widetilde{U}=\stackrel{\text{\begin{tiny}B\end{tiny}}}{D}_{h\widetilde{X}}v\widetilde{Y}+\omega(\widetilde{X}, F\widetilde{Y})\widetilde{U},\label{Y2}\\
D_{v\widetilde{X}}h\widetilde{Y}&=h[v\widetilde{X}, \widetilde{Y}]_\pounds=\stackrel{\text{\begin{tiny}B\end{tiny}}}{D}_{v\widetilde{X}}h\widetilde{Y},\label{Y3}\\
D_{h\widetilde{X}}h\widetilde{Y}&=hF[h\widetilde{X}, J\widetilde{Y}]_\pounds+\omega(\widetilde{X}, \widetilde{Y})F\widetilde{U}=\stackrel{\text{\begin{tiny}B\end{tiny}}}{D}_{h\widetilde{X}}h\widetilde{Y}+\omega(\widetilde{X}, \widetilde{Y})F\widetilde{U},\label{Y4}
\end{align}
where $\widetilde{U}$ is a nonzero section of $v\pounds^\pi E$.
\begin{align}
D_{\widetilde{X}}\widetilde{Y}&=hF[h\widetilde{X}, J\widetilde{Y}]_\pounds+v[h\widetilde{X}, v\widetilde{Y}]_\pounds+h[v\widetilde{X}, \widetilde{Y}]_\pounds+J[v\widetilde{X}, F\widetilde{Y}]_\pounds\nonumber\\
&\ \ \ +\omega(\widetilde{X}, \widetilde{Y})F\widetilde{U}+\omega(\widetilde{X}, F\widetilde{Y})\widetilde{U}.
\end{align}
It is easy to see that $(D, h)$ is a d-connection on $\pounds^\pi E$. In the coordinate expression we have
\begin{equation}\label{4ta0}
\left\{
\begin{array}{cc}
D_{\delta_\alpha}\delta_\beta=(\omega_{\alpha\beta}\widetilde{U}^{\bar{\gamma}}-\frac{\partial \mathcal{B}^\gamma_\alpha}{\partial y^\beta})\delta_\gamma,\ \ \ \ \ D_{\mathcal{V}_\alpha}\mathcal{V}_\beta=0,\\
D_{\delta_\alpha}\mathcal{V}_\beta=(\omega_{\alpha\beta}\widetilde{U}^{\bar{\gamma}}-\frac{\partial \mathcal{B}^\gamma_\alpha}{\partial y^\beta})\mathcal{V}_\gamma,\ \ \ \ \ D_{\mathcal{V}_\alpha}\delta_\beta=0,
\end{array}
\right.
\end{equation}
and consequently
\begin{align*}
D_{\widetilde{X}}\widetilde{Y}&=\Big(\widetilde{X}^\alpha\{(\rho^i_\alpha\circ \pi)\frac{\partial \widetilde{Y}^\gamma}{\partial \textbf{x}^i}+\mathcal{B}^\lambda_\alpha\frac{\partial \widetilde{Y}^\gamma}{\partial \textbf{y}^\lambda}\}+\widetilde{X}^\alpha \widetilde{Y}^\beta\{-\frac{\partial \mathcal{B}^\gamma_\alpha}{\partial \textbf{y}^\beta}+\omega_{\alpha\beta} \widetilde{U}^{\bar{\gamma}}\}+\widetilde{X}^{\bar{\alpha}}\frac{\partial \widetilde{Y}^\gamma}{\partial \textbf{y}^\alpha}\Big)\delta_\gamma\\
&\ +\Big(\widetilde{X}^\alpha\{(\rho^i_\alpha\circ \pi)\frac{\partial \widetilde{Y}^{\bar{\gamma}}}{\partial \textbf{x}^i}+\mathcal{B}^\lambda_\alpha\frac{\partial \widetilde{Y}^{\bar{\gamma}}}{\partial \textbf{y}^\lambda}\}+\widetilde{X}^\alpha \widetilde{Y}^{\bar{\beta}}\{-\frac{\partial \mathcal{B}^\gamma_\alpha}{\partial \textbf{y}^\beta}+\omega_{\alpha\beta} \widetilde{U}^{\bar{\gamma}}\}+\widetilde{X}^{\bar{\alpha}}\frac{\partial \widetilde{Y}^{\bar{\gamma}}}{\partial \textbf{y}^\alpha}\Big)\mathcal{V}_\gamma,
\end{align*}
where $\widetilde{X}=\widetilde{X}^\alpha\delta_\alpha+\widetilde{X}^{\bar{\alpha}}\mathcal{V}_\alpha$, $\widetilde{Y}=\widetilde{y}^\beta\delta_\beta+\widetilde{Y}^{\bar{\beta}}\mathcal{V}_\beta$, $\widetilde{U}=\widetilde{U}^{\bar{\gamma}}\mathcal{V}_\gamma$ and $\omega_{\alpha\beta}=\omega(\delta_\alpha, \delta_\beta)$.
\begin{rem}\label{remark}
From (\ref{Y2}) and (\ref{Y4}) we deduce that $\omega(h\widetilde{X}, v\widetilde{Y})\widetilde{U}=\omega(v\widetilde{X}, v\widetilde{Y})\widetilde{U}=0$. Therefore we have $\omega(\delta_\alpha, \mathcal{V}_\beta)\widetilde{U}=\omega(\mathcal{V}_\alpha, \mathcal{V}_\beta)\widetilde{U}=0$.
\end{rem}
\begin{theorem}\label{above}
Let $(D, h)$ be the d-connection given by (\ref{Y1})-(\ref{Y4}). Then

(i)\ the $v$-mixed torsion of $D$ is $P^1=\omega\otimes \widetilde{U}$,

(ii)\ the $h$-mixed torsion $B$ of $D$ vanishes.\\
Moreover, if

(iii)\ the $h$-deflection of $(D, h)$ vanishes,

(iv)\ the $h$-horizontal torsion of $D$ vanishes,\\
then the horizontal endomorphism $h$ is homogeneous and torsion free.
\end{theorem}
\begin{proof}
Using Remark (\ref{remark}) and (\ref{4ta0}) we get
\begin{align*}
P^1(\delta_\alpha, \delta_\beta)&=\omega_{\alpha\beta}\widetilde{U}^{\bar\gamma}\mathcal{V}_\gamma=\omega(\delta_\alpha, \delta_\beta)\widetilde{U},\\
P^1(\delta_\alpha, \mathcal{V}_\beta)&=0=\omega(\delta_\alpha, \mathcal{V}_\beta)\widetilde{U},\\
P^1(\mathcal{V}_\alpha, \mathcal{V}_\beta)&=0=\omega(\mathcal{V}_\alpha, \mathcal{V}_\beta)\widetilde{U}.
\end{align*}
Therefore we have (i). Also, using (\ref{4ta0}) we deduce $B(\delta_\alpha, \delta_\beta)=0$ that gives us (ii). Now let (iii) and (iv) hold. (iii) gives us
\[
\mathcal{B}^\gamma_\alpha+y^\beta\omega_{\alpha\beta}\widetilde{U}^{\bar\gamma}-y^\beta\frac{\partial \mathcal{B}^\gamma_\alpha}{\partial y^\beta}=0.
\]
But from the condition $i_S\omega=0$ we derive that $y^\beta\omega_{\alpha\beta}=0$. Setting this in the above equation we have $\mathcal{B}^\gamma_\alpha=y^\beta\frac{\partial \mathcal{B}^\gamma_\alpha}{\partial y^\beta}$, i.e., $h$ is homogenous. (iv) gives us
\[
0=\frac{\partial\mathcal{B}^\gamma_\beta}{\partial\textbf{y}^\alpha}-\frac{\partial\mathcal{B}^\gamma_\alpha}{\partial\textbf{y}^\beta}
-(L^\gamma_{\alpha\beta}\circ\pi)=t^\gamma_{\alpha\beta}.
\]
Therefore $h$ is torsion free.
\end{proof}
Here, let $h$ be a homogenous and torsion free horizontal endomorphism on $\pounds^\pi E$ and $\stackrel{\text{\begin{tiny}B\end{tiny}}}{P}_{ric}$ be the mixed Ricci tensor of the Berwald-type connection $(\stackrel{\text{\begin{tiny}B\end{tiny}}}{D}, h)$. From Proposition \ref{Berwald} we can deduce that $i_S\!\!\stackrel{\text{\begin{tiny}B\end{tiny}}}{P}_{ric}=0$. Replacing $\omega$ and $\widetilde{U}$ in (\ref{Y1})-(\ref{Y4}) by $\frac{1}{n+1}\stackrel{\text{\begin{tiny}B\end{tiny}}}{P}_{ric}$ and Liouville section $C$, respectively, where $n=rank E$, we have the following d-connection
\begin{align}
\stackrel{\text{\begin{tiny}Y\end{tiny}}}{D}_{v\widetilde{X}}v\widetilde{Y}&=\stackrel{\text{\begin{tiny}B\end{tiny}}}{D}_{v\widetilde{X}}v\widetilde{Y},\label{Y01}\\
\stackrel{\text{\begin{tiny}Y\end{tiny}}}{D}_{h\widetilde{X}}v\widetilde{Y}&=\stackrel{\text{\begin{tiny}B\end{tiny}}}{D}_{h\widetilde{X}}v\widetilde{Y}
+\frac{1}{n+1}\stackrel{\text{\begin{tiny}B\end{tiny}}}{P}_{ric}(\widetilde{X}, F\widetilde{Y})C,\label{Y02}\\
\stackrel{\text{\begin{tiny}Y\end{tiny}}}{D}_{v\widetilde{X}}h\widetilde{Y}&=\stackrel{\text{\begin{tiny}B\end{tiny}}}{D}_{v\widetilde{X}}h\widetilde{Y},\label{Y03}\\
\stackrel{\text{\begin{tiny}Y\end{tiny}}}{D}_{h\widetilde{X}}h\widetilde{Y}&=\stackrel{\text{\begin{tiny}B\end{tiny}}}{D}_{h\widetilde{X}}h\widetilde{Y}
+\frac{1}{n+1}\stackrel{\text{\begin{tiny}B\end{tiny}}}{P}_{ric}(\widetilde{X}, \widetilde{Y})FC.\label{Y04}
\end{align}
This d-connection is said to be the \textit{Yano-type connection induced by $h$}. If, in particular,
$h$ is a Berwald endomorphism, then we call it a \textit{Yano connection}. From (\ref{4ta0}) we derive that the Yano-type connection has the following coordinate expression:
\begin{equation}\label{4ta1}
\left\{
\begin{array}{cc}
\stackrel{\text{\begin{tiny}Y\end{tiny}}}{D}_{\delta_\alpha}\delta_\beta=(\frac{1}{n+1}\frac{\partial^2\mathcal{B}^\lambda_\alpha}{\partial y^\lambda\partial y^\beta}y^\gamma-\frac{\partial \mathcal{B}^\gamma_\alpha}{\partial y^\beta})\delta_\gamma,\ \ \ \ \ \stackrel{\text{\begin{tiny}Y\end{tiny}}}{D}_{\mathcal{V}_\alpha}\mathcal{V}_\beta=0,\\
\stackrel{\text{\begin{tiny}Y\end{tiny}}}{D}_{\delta_\alpha}\mathcal{V}_\beta=(\frac{1}{n+1}\frac{\partial^2\mathcal{B}^\lambda_\alpha}{\partial y^\lambda\partial y^\beta}y^\gamma-\frac{\partial \mathcal{B}^\gamma_\alpha}{\partial y^\beta})\mathcal{V}_\gamma,\ \ \ \ \ \stackrel{\text{\begin{tiny}Y\end{tiny}}}{D}_{\mathcal{V}_\alpha}\delta_\beta=0.
\end{array}
\right.
\end{equation}
If we denote the local coefficients of Yano-type connection $D$ by $(\stackrel{\text{\begin{tiny}Y\end{tiny}}}{F}_{\alpha\beta}^\gamma, \stackrel{\text{\begin{tiny}Y\end{tiny}}}{C}_{\alpha\beta}^\gamma)$, then from the above equation we conclude $\stackrel{\text{\begin{tiny}Y\end{tiny}}}{F}_{\alpha\beta}^\gamma=\frac{1}{n+1}\frac{\partial^2\mathcal{B}^\lambda_\alpha}{\partial y^\lambda\partial y^\beta}y^\gamma-\frac{\partial \mathcal{B}^\gamma_\alpha}{\partial y^\beta}$ and $\stackrel{\text{\begin{tiny}Y\end{tiny}}}{C}_{\alpha\beta}^\gamma=0$.
Therefore using (\ref{hh}), (\ref{hv}) and (\ref{vv}) we derive that
\begin{align}
\stackrel{\text{\begin{tiny}Y\end{tiny}}}{R}^{\ \ \ \ \lambda}_{\alpha\beta\gamma}&=(\rho^i_\alpha\circ \pi)(\frac{1}{n+1}\frac{\partial^3 \mathcal{B}^\mu_\beta}{\partial \textbf{x}^i\partial \textbf{y}^\mu\partial \textbf{y}^\gamma}\textbf{y}^\lambda-\frac{\partial^2 \mathcal{B}^\lambda_{\beta}}{\partial \textbf{x}^i\partial \textbf{y}^\gamma})-\frac{1}{n+1}\frac{\partial^2 \mathcal{B}^\nu_{\beta}}{\partial \textbf{y}^\nu\partial \textbf{y}^\gamma}\frac{\partial \mathcal{B}^\lambda_{\alpha}}{\partial \textbf{y}^\mu}\textbf{y}^\mu\nonumber\\
&\ \ \ +\mathcal{B}^\mu_\alpha(\frac{1}{n+1}\frac{\partial^3 \mathcal{B}^\nu_\beta}{\partial \textbf{y}^\mu\partial \textbf{y}^\nu\partial \textbf{y}^\gamma}\textbf{y}^\lambda-\frac{\partial^2 \mathcal{B}^\lambda_{\beta}}{\partial \textbf{y}^\mu\partial \textbf{y}^\gamma})\nonumber+\frac{1}{n+1}\frac{\partial^2 \mathcal{B}^\nu_{\alpha}}{\partial \textbf{y}^\nu\partial \textbf{y}^\gamma}\frac{\partial \mathcal{B}^\lambda_{\beta}}{\partial \textbf{y}^\mu}\textbf{y}^\mu\\
&\ \ \ +(\rho^i_\beta\circ \pi)(\frac{\partial^2 \mathcal{B}^\lambda_{\alpha}}{\partial \textbf{x}^i\partial \textbf{y}^\gamma}-\frac{1}{n+1}\frac{\partial^3 \mathcal{B}^\mu_\alpha}{\partial \textbf{x}^i\partial \textbf{y}^\mu\partial \textbf{y}^\gamma}\textbf{y}^\lambda)-\frac{1}{n+1}\frac{\partial^2 \mathcal{B}^\nu_{\alpha}}{\partial \textbf{y}^\nu\partial \textbf{y}^\mu}\frac{\partial \mathcal{B}^\mu_{\beta}}{\partial \textbf{y}^\gamma}\textbf{y}^\lambda\nonumber\\
&\ \ \ +\mathcal{B}^\mu_\beta(\frac{\partial^2 \mathcal{B}^\lambda_{\alpha}}{\partial \textbf{y}^\mu\partial \textbf{y}^\gamma}-\frac{1}{n+1}\frac{\partial^3 \mathcal{B}^\nu_\alpha}{\partial \textbf{y}^\mu\partial \textbf{y}^\nu\partial \textbf{y}^\gamma}\textbf{y}^\lambda)+\frac{1}{n+1}\frac{\partial^2 \mathcal{B}^\nu_{\beta}}{\partial \textbf{y}^\nu\partial \textbf{y}^\mu}\frac{\partial \mathcal{B}^\mu_{\alpha}}{\partial \textbf{y}^\gamma}\textbf{y}^\lambda\nonumber\\
&\ \ \ +\frac{\partial \mathcal{B}^\mu_{\beta}}{\partial \textbf{y}^\gamma}\frac{\partial \mathcal{B}^\lambda_{\alpha}}{\partial \textbf{y}^\mu}+(\frac{1}{n+1})^2(\frac{\partial^2 \mathcal{B}^\nu_\beta}{\partial \textbf{y}^\nu\partial \textbf{y}^\gamma}\frac{\partial^2 \mathcal{B}^\sigma_\alpha}{\partial \textbf{y}^\sigma\partial \textbf{y}^\mu}-\frac{\partial^2 \mathcal{B}^\nu_\alpha}{\partial \textbf{y}^\nu\partial \textbf{y}^\gamma}\frac{\partial^2 \mathcal{B}^\sigma_\beta}{\partial \textbf{y}^\sigma\partial \textbf{y}^\mu})\textbf{y}^\mu\textbf{y}^\lambda\nonumber\\
&\ \ \ -\frac{\partial \mathcal{B}^\mu_{\alpha}}{\partial \textbf{y}^\gamma}\frac{\partial \mathcal{B}^\lambda_{\beta}}{\partial \textbf{y}^\mu}+(L^\mu_{\alpha\beta}\circ \pi)(\frac{\partial \mathcal{B}^\lambda_{\mu}}{\partial \textbf{y}^\gamma}-\frac{1}{n+1}\frac{\partial^2 \mathcal{B}^\nu_\mu}{\partial \textbf{y}^\nu\partial \textbf{y}^\gamma}\textbf{y}^\lambda),\\
\stackrel{\text{\begin{tiny}Y\end{tiny}}}{P}^{\ \ \ \ \lambda}_{\alpha\beta\gamma}&=\frac{\partial^2 \mathcal{B}^\lambda_{\alpha}}{\partial \textbf{y}^\beta\partial \textbf{y}^\gamma}-\frac{1}{n+1}(\frac{\partial^2 \mathcal{B}^\mu_{\alpha}}{\partial \textbf{y}^\mu\partial \textbf{y}^\gamma}\delta^\lambda_\beta+\frac{\partial^3 \mathcal{B}^\mu_{\alpha}}{\partial\textbf{y}^\beta\partial \textbf{y}^\mu\partial \textbf{y}^\gamma}\textbf{y}^\lambda),\label{mixed00}\\
\stackrel{\text{\begin{tiny}Y\end{tiny}}}{S}^{\ \ \ \ \lambda}_{\alpha\beta\gamma}&=0,
\end{align}
where $\stackrel{\text{\begin{tiny}Y\end{tiny}}}{R}^{\ \ \ \ \lambda}_{\alpha\beta\gamma}$, $\stackrel{\text{\begin{tiny}Y\end{tiny}}}{P}^{\ \ \ \ \lambda}_{\alpha\beta\gamma}$ and $\stackrel{\text{\begin{tiny}Y\end{tiny}}}{S}^{\ \ \ \ \lambda}_{\alpha\beta\gamma}$ are the coefficients of the horizontal, mixed and vertical curvatures of Yano-type connection $(\stackrel{\text{\begin{tiny}Y\end{tiny}}}{D}, h)$, respectively. Therefore the vertical curvature of d-connection $\stackrel{\text{\begin{tiny}Y\end{tiny}}}{D}$ is vanished.

From theorem \ref{above} we have
\begin{cor}
Let $(\stackrel{\text{\begin{tiny}Y\end{tiny}}}{D}, h)$ be the Yano-type d-connection. Then

(i)\ the $v$-mixed torsion of $\stackrel{\text{\begin{tiny}Y\end{tiny}}}{D}$ is $\stackrel{\text{\begin{tiny}Y\end{tiny}}}{P^1}=\frac{1}{n+1}(\stackrel{\text{\begin{tiny}B\end{tiny}}}{P}_{ric}\otimes C)$,

(ii)\ the $h$-mixed torsion $\stackrel{\text{\begin{tiny}Y\end{tiny}}}{B}$ of $\stackrel{\text{\begin{tiny}Y\end{tiny}}}{D}$ vanishes.
\end{cor}
\begin{proposition}
Let $h$ be a torsion free, homogeneous horizontal endomorphism. If $(\stackrel{\text{\begin{tiny}B\end{tiny}}}{D}, h)$ and $(D, h)$ are the induced Berwald-type and Yano-type connections with mixed
curvature and mixed Ricci tensors $\stackrel{\text{\begin{tiny}B\end{tiny}}}{P}$, $\stackrel{\text{\begin{tiny}Y\end{tiny}}}{P}$ and $\stackrel{\text{\begin{tiny}B\end{tiny}}}{P}_{ric}$, $\stackrel{\text{\begin{tiny}Y\end{tiny}}}{P}_{ric}$, respectively, then we have
\begin{align}
\stackrel{\text{\begin{tiny}Y\end{tiny}}}{P}&=\stackrel{\text{\begin{tiny}B\end{tiny}}}{P}-\frac{1}{n+1}D_J\stackrel{\text{\begin{tiny}B\end{tiny}}}{P}_{ric}\otimes C-\frac{1}{n+1}\stackrel{\text{\begin{tiny}B\end{tiny}}}{P}_{ric}\otimes J,\label{mr}\\
\stackrel{\text{\begin{tiny}Y\end{tiny}}}{P}_{ric}&=\frac{2}{n+1}\stackrel{\text{\begin{tiny}B\end{tiny}}}{P}_{ric}\label{mr0},
\end{align}
where $n=rank E$.
\end{proposition}
\begin{proof}
Using (\ref{mixed}) we can obtain
\begin{align}\label{mix}
(\stackrel{\text{\begin{tiny}B\end{tiny}}}{P}&\ -\frac{1}{n+1}D_J\stackrel{\text{\begin{tiny}B\end{tiny}}}{P}_{ric}\otimes C-\frac{1}{n+1}\stackrel{\text{\begin{tiny}B\end{tiny}}}{P}_{ric}\otimes J)(\delta_\alpha, \delta_\beta, \delta_\gamma)\nonumber\\
&=(\frac{\partial^2 \mathcal{B}^\lambda_{\alpha}}{\partial \textbf{y}^\beta\partial \textbf{y}^\gamma}-\frac{1}{n+1}\frac{\partial^2 \mathcal{B}^\mu_{\alpha}}{\partial \textbf{y}^\mu\partial \textbf{y}^\gamma}\delta^\lambda_\beta
-\frac{1}{n+1}\frac{\partial^3 \mathcal{B}^\mu_{\beta}}{\partial\textbf{y}^\alpha\partial \textbf{y}^\mu\partial \textbf{y}^\gamma}\textbf{y}^\lambda)\mathcal{V}_\lambda.
\end{align}
Since $h$ is torsion free, then we get
\[
\frac{\partial^3 \mathcal{B}^\mu_{\beta}}{\partial\textbf{y}^\alpha\partial \textbf{y}^\mu\partial \textbf{y}^\gamma}=\frac{\partial^3 \mathcal{B}^\mu_{\alpha}}{\partial\textbf{y}^\beta\partial \textbf{y}^\mu\partial \textbf{y}^\gamma}.
\]
Setting the above equation in (\ref{mixed}) and using (\ref{mixed00}) we deduce (\ref{mr}). To prove the (\ref{mr0}) we let $\stackrel{\text{\begin{tiny}B\end{tiny}}}{P}_{\alpha\gamma}$ and $\stackrel{\text{\begin{tiny}Y\end{tiny}}}{P}_{\alpha\gamma}$ be the coefficients of mixed Ricci tensors of Berwald-type and Yano-type connections, respectively. Then using (\ref{mixed00}) we obtain
\[
\stackrel{\text{\begin{tiny}Y\end{tiny}}}{P}_{\alpha\gamma}=\stackrel{\text{\begin{tiny}Y\end{tiny}}}{P}^{\ \ \ \ \lambda}_{\alpha\lambda\gamma}=\frac{1}{n+1}(\frac{\partial^2 \mathcal{B}^\lambda_{\alpha}}{\partial \textbf{y}^\lambda\partial \textbf{y}^\gamma}-\frac{\partial^3 \mathcal{B}^\mu_{\alpha}}{\partial\textbf{y}^\lambda\partial \textbf{y}^\mu\partial \textbf{y}^\gamma}\textbf{y}^\lambda),
\]
where $n=rank E$. Since $h$ is homogenous, then we have (\ref{00}). Setting (\ref{00}) in the above equation and using (\ref{mixed}) we get
\[
\stackrel{\text{\begin{tiny}Y\end{tiny}}}{P}_{\alpha\gamma}=\stackrel{\text{\begin{tiny}Y\end{tiny}}}{P}^{\ \ \ \ \lambda}_{\alpha\lambda\gamma}=\frac{2}{n+1}\frac{\partial^2 \mathcal{B}^\lambda_{\alpha}}{\partial \textbf{y}^\lambda\partial \textbf{y}^\gamma}=\frac{2}{n+1}\stackrel{\text{\begin{tiny}B\end{tiny}}}{P}_{\alpha\gamma}.
\]
\end{proof}
\subsubsection{The Douglas tensor of a Berwald endomorphism}
Let $h$ be a Berwald endomorphism on the manifold $\pounds^\pi E$. If $(\stackrel{\text{\begin{tiny}Y\end{tiny}}}{D},h)$ is the
Yano connection induced by $h$ and $\stackrel{\text{\begin{tiny}Y\end{tiny}}}{P}$ is the mixed curvature of $\stackrel{\text{\begin{tiny}Y\end{tiny}}}{D}$, then the tensor
\[
\textbf{D}=\stackrel{\text{\begin{tiny}Y\end{tiny}}}{P}-\frac{1}{2}(\stackrel{\text{\begin{tiny}Y\end{tiny}}}{P}_{ric}\otimes J+J\otimes \stackrel{\text{\begin{tiny}Y\end{tiny}}}{P}_{ric}),
\]
is said to be the Douglas tensor of the Berwald endomorphism. Using (\ref{vertical00}) and (\ref{mixed00}), the Douglas tensor $\textbf{D}$ has the following coordinate expression:
\begin{equation}\label{D1}
\textbf{D}=D_{\alpha\beta\gamma}^\lambda\mathcal{V}_\lambda \otimes\mathcal{X^\alpha}\otimes\mathcal{X^\beta}\otimes\mathcal{X^\gamma},
\end{equation}
where
\begin{equation}\label{Dg}
D_{\alpha\beta\gamma}^\lambda=\frac{\partial^2 \mathcal{B}^\lambda_{\alpha}}{\partial \textbf{y}^\beta\partial \textbf{y}^\gamma}-\frac{1}{n+1}(\frac{\partial^2 \mathcal{B}^\mu_{\alpha}}{\partial \textbf{y}^\mu\partial \textbf{y}^\gamma}\delta^\lambda_\beta+\frac{\partial^3 \mathcal{B}^\mu_{\beta}}{\partial \textbf{y}^\alpha\partial \textbf{y}^\mu\partial \textbf{y}^\gamma}y^\lambda+\frac{\partial^2 \mathcal{B}^\mu_{\alpha}}{\partial \textbf{y}^\mu\partial \textbf{y}^\beta}\delta^\lambda_\gamma+\frac{\partial^2 \mathcal{B}^\mu_{\beta}}{\partial \textbf{y}^\mu\partial \textbf{y}^\gamma}\delta^\lambda_\alpha).
\end{equation}
(\ref{D1}) told us that $\textbf{D}$ is semibasic. Moreover, since the Berwald endomorphism is homogenous and torsion free, then  from the above equation we deduce $D_{\alpha\beta\gamma}^\lambda=D_{\beta\alpha\gamma}^\lambda=D_{\gamma\beta\alpha}^\lambda$, i.e., $\textbf{D}$ is symmetric.
\begin{proposition}
Let $\textbf{D}$ be the Douglas tensor of a Berwald endomorphism. Then
$i_S\textbf{D}=0$ and $\textbf{D}_{ric}=0$.

\end{proposition}
\begin{proof}
Let $\widetilde{X}=\widetilde{X}^\alpha\delta_\alpha+\widetilde{X}^{\bar\alpha}\mathcal{V}_\alpha$ and  $\widetilde{Y}=\widetilde{Y}^\gamma\delta_\gamma+\widetilde{Y}^{\bar\gamma}\mathcal{V}_\gamma$. Since $\textbf{D}$ is symmetric then using (\ref{D1}) we get
\[
(i_S\textbf{D})(\widetilde{X},\widetilde{Y})=\textbf{D}(\widetilde{X}, S)\widetilde{Y}=\textbf{y}^\beta\widetilde{X}^\alpha\widetilde{Y}^\gamma D_{\alpha\beta\gamma}^\lambda.
\]
But using (\ref{0}) and (\ref{00}) we deduce  $\textbf{y}^\beta D_{\alpha\beta\gamma}^\lambda=0$. Therefore we have $i_S\textbf{D}=0$. Now we prove the second part of assertion. It is easy to see that
\[
\textbf{D}_{ric}=D_{\alpha\gamma}\mathcal{X^\alpha}\otimes\mathcal{X^\gamma},
\]
where $D_{\alpha\gamma}=D_{\alpha\lambda\gamma}^\lambda$. But using (\ref{00}) and (\ref{Dg}) we deduce $D_{\alpha\gamma}=0$ and consequently $\textbf{D}_{ric}=0$.
\end{proof}
\begin{theorem}
The Douglas tensor of a Berwald endomorphism is invariant under the projective changes of the associated spray.
\end{theorem}
\begin{proof}
Let $h$ be a Berwald endomorphism on $\pounds^\pi E$ with associated spray $S$ and $\textbf{D}$ be the
Douglas tensor of $h$. Also, let $\bar{S}$ be the projective change of $S$ by $\widetilde{f}$. Then $\bar{S}$ generates a Berwald endomorphism $\bar{h}$. Denote by $\bar{\textbf{D}}$ the Douglas tensor of $\bar{h}$. If $S=\textbf{y}^\alpha\mathcal{X}_\alpha+S^\alpha\mathcal{V}_\alpha$ and $\bar{S}=\textbf{y}^\alpha\mathcal{X}_\alpha+{\bar{S}}^\alpha\mathcal{V}_\alpha$, then $\bar{S}=S+\widetilde{f}C$ gives us
\begin{equation}\label{pchange}
{\bar{S}}^\alpha=S^\alpha+y^\alpha\widetilde{f}.
\end{equation}
From (\ref{Ber000}) and (\ref{Ber}), $h$ and $\bar{h}$ have the following coordinate expressions:
\begin{equation}
h=(\mathcal{X}_\alpha+\mathcal{B}^\gamma_\alpha\mathcal{V}_\gamma)\otimes\mathcal{X}^\alpha,\ \ \ \bar{h}=(\mathcal{X}_\alpha+{\bar{\mathcal{B}}}^\gamma_\alpha\mathcal{V}_\gamma)\otimes\mathcal{X}^\alpha,
\end{equation}
where
\begin{equation}\label{pchange1}
\mathcal{B}^\gamma_\alpha=\frac{1}{2}(\frac{\partial S^\gamma}{\partial \textbf{y}^\alpha}-\textbf{y}^\beta (L^\gamma_{\alpha\beta}\circ\pi)),\ \ \ {\bar{\mathcal{B}}}^\gamma_\alpha=\frac{1}{2}(\frac{\partial {\bar S}^\gamma}{\partial \textbf{y}^\alpha}-\textbf{y}^\beta (L^\gamma_{\alpha\beta}\circ\pi)).
\end{equation}
Using (\ref{pchange}) and (\ref{pchange1}) we get
\begin{equation}\label{pchange2}
{\bar{\mathcal{B}}}^\gamma_\alpha=\mathcal{B}^\gamma_\alpha+\widetilde{f}^\gamma_\alpha,
\end{equation}
where
\[
{\widetilde{f}}^\gamma_\alpha=\frac{1}{2}(\widetilde{f}\delta^\gamma_\alpha+\textbf{y}^\gamma\frac{\partial{\widetilde{f}}}{\partial \textbf{y}^\alpha}).
\]
If we denote by $D^\lambda_{\alpha\beta\gamma}$ and ${\bar{D}}^\lambda_{\alpha\beta\gamma}$ the coefficients of $\textbf{D}$ and $\bar{\textbf{D}}$, respectively, then using (\ref{Dg}) and (\ref{pchange2}) we get
\begin{align}\label{Dchange}
{\bar{D}}_{\alpha\beta\gamma}^\lambda&=D_{\alpha\beta\gamma}^\lambda+\frac{\partial^2 {\widetilde{f}}^\lambda_{\alpha}}{\partial \textbf{y}^\beta\partial \textbf{y}^\gamma}-\frac{1}{n+1}(\frac{\partial^2 {\widetilde{f}}^\mu_{\alpha}}{\partial \textbf{y}^\mu\partial \textbf{y}^\gamma}\delta^\lambda_\beta+\frac{\partial^3 {\widetilde{f}}^\mu_{\beta}}{\partial \textbf{y}^\alpha\partial \textbf{y}^\mu\partial \textbf{y}^\gamma}\textbf{y}^\lambda\nonumber\\
&\ \ \ +\frac{\partial^2 {\widetilde{f}}^\mu_{\alpha}}{\partial \textbf{y}^\mu\partial \textbf{y}^\beta}\delta^\lambda_\gamma+\frac{\partial^2 {\widetilde{f}}^\mu_{\beta}}{\partial \textbf{y}^\mu\partial \textbf{y}^\gamma}\delta^\lambda_\alpha).
\end{align}
Since $\widetilde{f}$ is homogenous of degree 1, then we can obtain
\[
\frac{\partial^3 \widetilde{f}}{\partial \textbf{y}^\beta\partial \textbf{y}^\gamma\partial \textbf{y}^\alpha}\textbf{y}^\beta=-\frac{\partial^2 \widetilde{f}}{\partial \textbf{y}^\gamma\partial \textbf{y}^\alpha}.
\]
The above equation and direct calculation give us
\begin{align}
\frac{\partial^2 {\widetilde{f}}^\lambda_{\alpha}}{\partial \textbf{y}^\beta\partial \textbf{y}^\gamma}&=\frac{1}{2}\Big(\frac{\partial^2 \widetilde{f}}{\partial \textbf{y}^\beta\partial \textbf{y}^\gamma}\delta^\lambda_\alpha+\frac{\partial^2 \widetilde{f}}{\partial \textbf{y}^\gamma\partial \textbf{y}^\alpha}\delta^\lambda_\beta+\frac{\partial^2 \widetilde{f}}{\partial \textbf{y}^\beta\partial \textbf{y}^\alpha}\delta^\lambda_\gamma+\frac{\partial^3 \widetilde{f}}{\partial \textbf{y}^\beta\partial \textbf{y}^\gamma\partial \textbf{y}^\alpha}\textbf{y}^\lambda\Big), \label{Dchange1}\\
\frac{\partial^2 {\widetilde{f}}^\mu_{\alpha}}{\partial \textbf{y}^\mu\partial \textbf{y}^\gamma}&=\frac{1}{2}(n+1)\frac{\partial^2 \widetilde{f}}{\partial \textbf{y}^\gamma\partial \textbf{y}^\alpha},\label{Dchange2}\\
\frac{\partial^2 {\widetilde{f}}^\mu_{\alpha}}{\partial \textbf{y}^\mu\partial \textbf{y}^\beta}&=\frac{1}{2}(n+1)\frac{\partial^2 \widetilde{f}}{\partial \textbf{y}^\beta\partial \textbf{y}^\alpha},\label{Dchange3}\\
\frac{\partial^2 {\widetilde{f}}^\mu_{\beta}}{\partial \textbf{y}^\mu\partial \textbf{y}^\gamma}&=\frac{1}{2}(n+1)\frac{\partial^2 \widetilde{f}}{\partial \textbf{y}^\gamma\partial \textbf{y}^\beta},\label{Dchange4}\\
\frac{\partial^3 {\widetilde{f}}^\mu_{\beta}}{\partial \textbf{y}^\alpha\partial \textbf{y}^\mu\partial \textbf{y}^\gamma}&=\frac{1}{2}(n+1)\frac{\partial^3 \widetilde{f}}{\partial \textbf{y}^\alpha\partial \textbf{y}^\gamma\partial \textbf{y}^\beta}.\label{Dchange5}
\end{align}
Setting (\ref{Dchange1})-(\ref{Dchange5}) in (\ref{Dchange}) we obtain ${\bar{D}}_{\alpha\beta\gamma}^\lambda=D_{\alpha\beta\gamma}^\lambda$, i.e., $\bar{\textbf{D}}=\textbf{D}$.
\end{proof}
\section{$\rho_\pounds$-covariant derivatives in $\pi^*\pi$}
In this section, we investigate geometric properties of $\rho_\pounds$-covariant derivatives in $\pi^*\pi$ like torsion and partial curvature. Results are in a deep relation with Berwald derivative.

We can deduce the following double-exact short sequence from the double-exact short sequenc (\ref{exact se})
\begin{diagram}[heads=LaTeX]
0 &\pile{\rTo\\  \lTo} &\Gamma(\pi^*\pi)&\pile{\rTo^{\bar i}\\  \lTo_{\bar{\mathcal{V}}}} &\Gamma(\pounds^\pi E)&\pile{\rTo^{\bar j}\\  \lTo_{\bar{\mathcal{H}}}} &\Gamma(\pi^*\pi)&\pile{\rTo\\  \lTo} &0,
\end{diagram}
such that for every $\bar{X}\in\Gamma(\pi^*\pi)$ and $\xi\in\Gamma(\pounds^\pi E)$ the followings hold
\begin{equation}\label{2base0}
\bar{i}(\bar{X}):=i\circ \bar{X}, \qquad \bar{j}(\xi):=j\circ \xi,
\qquad\bar{\mathcal{H}}(\bar{X}):=\mathcal{H}\circ \bar{X}, \qquad \bar{\mathcal{V}}(\xi):=\mathcal{V}\circ \xi.
\end{equation}
\begin{proposition}\label{YB}
Let $X$ belongs to $\Gamma(E)$. Then we have the followings
$$
(i)\ \bar{i}(\widehat{X})=X^V,\qquad (ii)\ \bar{j}(X^V)=0;\qquad (iii)\ \bar{j}(X^C)=\widehat{X},$$
$$(iv)\ \bar{\mathcal{H}}(\widehat{X})=X^h,\qquad
(v)\ \bar{\mathcal{V}}(X^V)=\widehat{X},\qquad
(vi)\ \bar{\mathcal{V}}(X^h)=0.
$$
\end{proposition}
\begin{proof}
Let $u\in E$. Then we have
\begin{align*}
\bar{i}(\widehat{X})(u)&=i\circ\widehat{X}=i(u, X(\pi(u)))=(0, X(\pi(u))^{\vee}_{u})\\
& =(0, X^{\vee}(u))=X^V(u),
\end{align*}
 that gives us the first one. The second one is obvious. For the thirst, since $J(X^C)=X^V$, then we have $i\circ j(X^C)=X^V=i(\widehat{X})$. Because $i$ is injective, $j(X^C)=\widehat{X}$ and consequently $\bar{j}(X^C)=\widehat{X}$. For the forth, we can deduce
 $$\bar{\mathcal{H}}(\widehat{X})=\mathcal{H}\circ\widehat{X}=F\circ i\circ\widehat{X}=F\circ X^V=F\circ J(X^C)=h(X^C)=X^h.$$
 Using (\ref{split}), the fifth equation proves as follows
 $$\bar{\mathcal{V}}(X^V)=j\circ F\circ X^V=j\circ F\circ J(X^C)=j\circ h(X^C)=j\circ X^C=\widehat{X}.$$
 The last one obvious.
\end{proof}
\begin{rem}
The mapping $\bar{i}$ is an isomorphism between $\Gamma(\pi^*\pi)$ and $\Gamma(v\pounds^\pi E)$. Thus every section of $v(\pounds^\pi E)$ can be shown like $\bar{i}{\widetilde{X}}$ where $\widetilde{X}\in \Gamma(\pi^*\pi)$. Moreover, since $\bar{j}$ is surjective, then every member of $\Gamma(\pi^*\pi)$ has the format $\bar{j}(\xi)$, where $\xi\in\Gamma(\pounds^\pi(E))$.
\end{rem}
\begin{defn}\label{YB2}
Operator $\nabla^v$  with  properties
\begin{enumerate}
\item[(i)]
$\nabla^{v}_{\bar{X}}\widetilde{f}:=\rho_\pounds(\bar{i}\bar{X})\widetilde{f}$,
\item[(ii)]
$\nabla^{v}_{\bar{X}}\bar{Y}:=\bar{j}[\bar{i}\bar{X},\bar{\mathcal{H}}\bar{Y}]_\pounds$,
\item[(iii)]
$(\nabla^{v}_{\bar{X}}\bar{\alpha})(\bar{Y}):=\rho_{\pounds}(\bar{i}\bar{X})(\bar{\alpha}(\bar{Y}))-\bar{\alpha}(\nabla^{v}_{\bar{X}}\bar{Y})$,
\end{enumerate}
is called the \textit{canonical $v$-covariant differential}, where $\widetilde{f}\in C^{\infty}(E)$, $\bar{X}, \bar{Y}\in\Gamma(\pi^*\pi)$, $\bar{\alpha}\in \Omega^1(\pi)$.
\end{defn}
\begin{rem}
The second condition of the above definition is independent of choosing $\bar{\mathcal{H}}$. Indeed since  $\bar{j}$ is surjective, there is some $\widetilde{Y}\in \Gamma(\pounds^\pi E)$, such that $\bar{Y}=\bar{j}\widetilde{Y}$. Thus
$$\nabla^{v}_{\bar{X}}\bar{j}\widetilde{Y}=\bar{j}[\bar{i}\bar{X}, \bar{\mathcal{H}}\circ \bar{j}\widetilde{Y}]_\pounds=j[\bar{i}\bar{X},h\widetilde{Y}]_\pounds.$$
But $[\bar{i}\bar{X},v\widetilde{Y}]_\pounds$ is vertical. Therefore
$$\nabla^{v}_{\bar{X}}\bar{j}\widetilde{Y}=\bar{j}[\bar{i}\bar{X},\widetilde{Y}]_\pounds.$$
\end{rem}
Let $\bar{A}\in\mathcal{T}^{k}_{l}(\pi)$. Then we define
\begin{align*}
(\nabla^{v}_{\bar{X}}\bar{A})(\bar{\alpha_1}, \bar{\alpha_2},..., \bar{\alpha_k}, \bar{X_1}, \bar{X_2},..., \bar{X_l})&:=\rho_{\pounds}(\bar{i}\bar{X})(\bar{\alpha_1}, \bar{\alpha_2},..., \bar{\alpha_k}, \bar{X_1}, \bar{X_2},..., \bar{X_l})\\
&-\sum_{i=1}^{k}\bar{A}(\bar{\alpha_1},...,\nabla^{v}_{\bar{X}}\bar{\alpha_i},..., \bar{\alpha_k}, \bar{X_1}, \bar{X_2},..., \bar{X_l})\\
&-\sum_{i=1}^{l}\bar{A}(\bar{\alpha_1},\bar{\alpha_2},..., \bar{\alpha_k}, \bar{X_1},...,\nabla^{v}_{\bar{X}}\bar{X_i},..., \bar{X_l}).
\end{align*}
Moreover, for $\bar{A}\in\mathcal{T}^{k}_{l}(\pi)$ tensor field $\nabla^{v}\bar{A}\in \mathcal{T}^{k}_{l+1}(\pi)$ is defined by the following rule
$$(\nabla^{v}\bar{A})(\bar{X},\bar{\alpha_1}, \bar{\alpha_2},..., \bar{\alpha_k}, \bar{X_1}, \bar{X_2},..., \bar{X_l}):=(\nabla^{v}_{\bar{X}}\bar{A})(\bar{\alpha_1}, \bar{\alpha_2},..., \bar{\alpha_k}, \bar{X_1}, \bar{X_2},..., \bar{X_l}).$$
\begin{defn}
Let $\widetilde{f}$ be a smooth function on $E$. Then tensor field
\[
\nabla^v\nabla^v\widetilde{f}:=\nabla^v(\nabla^v\widetilde{f})\in\mathcal{T}^0_2(\pi),
\]
\end{defn}
is said to be hession of $\widetilde{f}$.
\begin{proposition}
Function $\widetilde{f}\in C^\infty(E)$ is homogenous of degree 1 if and only if $\nabla^v_\delta\widetilde{f}=\widetilde{f}$.
\end{proposition}
\begin{proof}
Let $\widetilde{f}$ be a homogenous function of degree 1 on $E$. Then we have $\rho_\pounds(C)\widetilde{f}=\widetilde{f}$. Thus
\[
\nabla^v_\delta\widetilde{f}=\rho_\pounds(\bar{i}\delta)\widetilde{f}=\rho_\pounds(i\circ\delta)\widetilde{f}
=\rho_\pounds(C)\widetilde{f}=\widetilde{f}.
\]
From the above equation, also we can deduce the convers of assertion.
\end{proof}
\begin{proposition}
Let $X$ and $Y$ be sections of $E$ and $\widetilde{f}\in C^\infty(E)$. Then
\begin{equation}
\nabla^v\nabla^v\widetilde{f}(\widehat{X}, \widehat{Y})=\rho_\pounds(X^V)(\rho_\pounds(Y^V)\widetilde{f}).
\end{equation}
Moreover, the hessian of $\widetilde{f}$ is symmetric.
\end{proposition}
\begin{proof}
Using the definition of hessian of $\widetilde{f}$, (i) of proposition \ref{YB} and (iii) of definition \ref{YB2} we get
\begin{align}\label{YB1}
\nabla^v\nabla^v\widetilde{f}(\widehat{X}, \widehat{Y})&=(\nabla^v_{\widehat{X}}(\nabla^v\widetilde{f}))(\widehat{Y})=\rho_\pounds(\bar{i}\widehat{X})((\nabla^v\widetilde{f})(\widehat{Y}))-\nabla^v\widetilde{f}(\nabla_{\widehat{X}}\widehat{Y})\nonumber\\
&=\rho_\pounds(\bar{i}\widehat{X})(\rho_\pounds(\bar{i}\widehat{Y})\widetilde{f})-\rho_\pounds(\bar{i}(\nabla_{\widehat{X}}\widehat{Y}))\widetilde{f}\nonumber\\
&=\rho_\pounds(X^V)(\rho_\pounds(Y^V)\widetilde{f})-\rho_\pounds(\bar{i}(\nabla_{\widehat{X}}\widehat{Y}))\widetilde{f}.
\end{align}
But using (i), (ii) and (iv) of proposition \ref{YB} we deduce
\[
\nabla^v_{\widehat{X}}\widehat{Y}=\bar{j}[\bar{i}\widehat{X}, \bar{\mathcal{H}}\widehat{Y}]_\pounds=\bar{j}[X^V, Y^h]_\pounds=0,
\]
because $[X^V, Y^h]_\pounds\in\Gamma(v\pounds^\pi E)$. Plugging the above equation into (\ref{YB1}) implies the first part of assertion. Now, we prove the second part of assertion. Since $[X^V, Y^V]_\pounds=0$, then using the first part of assertion we get
\begin{align*}
\nabla^v\nabla^v\widetilde{f}(\widehat{X}, \widehat{Y})&=\rho_\pounds(X^V)(\rho_\pounds(Y^V)\widetilde{f})=\rho_\pounds([X^V, Y^V]_\pounds)(\widetilde{f})\\
&\ \ \ +\rho_\pounds(Y^V)(\rho_\pounds(X^V)\widetilde{f})\\
&=\rho_\pounds(Y^V)(\rho_\pounds(X^V)\widetilde{f})\\
&=\nabla^v\nabla^v\widetilde{f}(\widehat{Y}, \widehat{X}).
\end{align*}
\end{proof}
\begin{proposition}
Let $\widetilde{f}\in C^\infty(E)$ be a homogenous function of degree 1. Then
\[
\nabla^v_\delta(\nabla^v\nabla^v\widetilde{f})=-\nabla^v\nabla^v\widetilde{f}.
\]
\end{proposition}
\begin{proof}
Setting $\bar{A}=\nabla^v\nabla^v\widetilde{f}$, we must show $\nabla^v_\delta\bar{A}=-\bar{A}$. Let $X$ and $Y$ be sections of $E$. Then we have
\begin{equation}\label{**}
(\nabla^v_\delta\bar{A})(\widehat{X}, \widehat{Y})=\rho_\pounds(\bar{i}\delta)\bar{A}(\widehat{X}, \widehat{Y})-\bar{A}(\nabla^v_\delta\widehat{X}, \widehat{Y})-\bar{A}(\widehat{X}, \nabla^v_\delta\widehat{Y}).
\end{equation}
But using (ii) of definition \ref{YB2}, we deduce
\[
\nabla^v_\delta\widehat{X}=\bar{j}[\bar{i}\delta, \bar{H}\widehat{Y}]_\pounds=\bar{j}[C, Y^h]_\pounds=0,
\]
because $[C, Y^h]_\pounds\in\Gamma(v\pounds^\pi E)$. Similarly we have $\nabla^v_\delta\widehat{Y}=0$. Therefore (\ref{**}) reduce to the following
\begin{equation}\label{***}
(\nabla^v_\delta\bar{A})(\widehat{X}, \widehat{Y})=\rho_\pounds(C)\bar{A}(\widehat{X}, \widehat{Y})=\rho_\pounds(C)\Big(\rho_\pounds(X^V)(\rho_\pounds(Y^V)\widetilde{f})\Big).
\end{equation}
In other hand, using (ii) of (\ref{Liover}) we get
\begin{align*}
\bar{A}(\hat{X}, \hat{Y})&=\rho_\pounds(X^V)(\rho_\pounds(Y^V)\widetilde{f})=\rho_\pounds([X^V, C]_\pounds)(\rho_\pounds(Y^V)\widetilde{f})\nonumber\\
&=[\rho_\pounds(X^V), \rho_\pounds(C)](\rho_\pounds(Y^V)\widetilde{f})=\rho_\pounds(X^V)\Big(\rho_\pounds(C)(\rho_\pounds(Y^V)\widetilde{f})\Big)\nonumber\\
&\ \ \ -\rho_\pounds(C)\Big(\rho_\pounds(X^V)(\rho_\pounds(Y^V)\widetilde{f})\Big)=\rho_\pounds(X^V)\Big([\rho_\pounds(C), \rho_\pounds(Y^V)]\widetilde{f}\nonumber\\
&\ \ \ +\rho_\pounds(Y^V)(\rho_\pounds(C)\widetilde{f})\Big)-\rho_\pounds(C)\Big(\rho_\pounds(X^V)(\rho_\pounds(Y^V)\widetilde{f})\Big)\nonumber\\
&=\rho_\pounds(X^V)\Big(\rho_\pounds[C, Y^V]_\pounds\widetilde{f}+\rho_\pounds(Y^V)(\rho_L(C)\widetilde{f})\Big)\nonumber\\
&\ \ \ -\rho_\pounds(C)\Big(\rho_\pounds(X^V)(\rho_\pounds(Y^V)\widetilde{f})\Big).
\end{align*}
Since $\widetilde{f}$ is homogenous of degree 1, then we have $\rho_\pounds(C)\widetilde{f}=\widetilde{f}$. Setting this in the above equation and using (ii) of (\ref{Liover}) we get
\begin{equation}\label{****}
\bar{A}(\hat{X}, \hat{Y})=-\rho_\pounds(C)\Big(\rho_\pounds(X^V)(\rho_\pounds(Y^V)\widetilde{f})\Big).
\end{equation}
From (\ref{***}) and (\ref{****}) we have the assertion.
\end{proof}
\begin{defn}\label{Hi2}
Let $h$ be a horizontal endomorphism and $\bar{\mathcal{H}}$ be a horizontal map of $\pi$ associated to $h$. Operator $\nabla^h$  with  properties
\begin{enumerate}
\item[(i)]
$\nabla^{h}_{\bar{X}}\widetilde{f}:=\rho_\pounds(\bar{\mathcal{H}}\bar{X})\widetilde{f}$,
\item[(ii)]
$\nabla^{h}_{\bar{X}}\bar{Y}:=\bar{\mathcal{V}}[\bar{\mathcal{H}}\bar{X},\bar{i}\bar{Y}]_\pounds$,
\item[(iii)]
$(\nabla^{h}_{\bar{X}}\bar{\alpha})(\bar{Y}):=\rho_{\pounds}(\bar{\mathcal{H}}\bar{X})(\bar{\alpha}(\bar{Y}))-\bar{\alpha}(\nabla^{h}_{\bar{X}}\bar{Y})$,
\end{enumerate}
is called the \textit{canonical $h$-covariant differential}, where $\widetilde{f}\in C^{\infty}(E)$, $\bar{X}, \bar{Y}\in\Gamma(\pi^*\pi)$, $\bar{\alpha}\in \Omega^1(\pi)$.
\begin{lemma}
Let $H$ be the tension of $h$ and $\widetilde{X}$ be a section of $\pounds^\pi E$. Then
\begin{equation}\label{hdel}
(\nabla^h\delta)(\bar{j}\widetilde{X})=\bar{\mathcal{V}}H(\widetilde{X}).
\end{equation}
\end{lemma}
\begin{proof}
Using (ii) of the above definition we get
\[
(\nabla^h\delta)(\bar{j}\widetilde{X})=\nabla^h_{\bar{j}\widetilde{X}}\delta=\bar{\mathcal{V}}[\bar{\mathcal{H}}\bar{j}\widetilde{X}, \bar{i}\delta]_\pounds=\bar{\mathcal{V}}[h\widetilde{X}, C]_\pounds=\bar{\mathcal{V}}[h, C]^{F-N}_\pounds(\widetilde{X})=\bar{\mathcal{V}}H(\widetilde{X}).
\]
\end{proof}
Since $\bar{i}\bar{\mathcal{V}}=v$, then (\ref{hdel}) gives us
\begin{equation}\label{hdel1}
\bar{i}(\nabla^h\delta)(\bar{j}\widetilde{X})=vH(\widetilde{X})=H(\widetilde{X}).
\end{equation}
By reason of the above relation, the (1, 1) tensor field $\bar{H}=\nabla^h\delta$ is called the tension of the horizontal map $\bar{\mathcal{H}}$. Indeed, we have
\begin{equation}\label{Hi1}
\bar{H}(\bar{X})=\bar{\mathcal{V}}[\bar{\mathcal{H}}\bar{X}, C]_\pounds,\ \ \ \forall \bar{X}\in\Gamma(\pi^*\pi).
\end{equation}
\end{defn}
Let $\bar{A}\in\mathcal{T}^{k}_{l}(\pi)$. Then we define
\begin{align*}
(\nabla^{h}_{\bar{X}}\bar{A})(\bar{\alpha_1}, \bar{\alpha_2},..., \bar{\alpha_k}, \bar{X_1}, \bar{X_2},..., \bar{X_l})&:=\rho_{\pounds}(\bar{\mathcal{H}}\bar{X})(\bar{\alpha_1}, \bar{\alpha_2},..., \bar{\alpha_k}, \bar{X_1}, \bar{X_2},..., \bar{X_l})\\
&-\sum_{i=1}^{k}\bar{A}(\bar{\alpha_1},...,\nabla^{h}_{\bar{X}}\bar{\alpha_i},..., \bar{\alpha_k}, \bar{X_1}, \bar{X_2},..., \bar{X_l})\\
&-\sum_{i=1}^{l}\bar{A}(\bar{\alpha_1},\bar{\alpha_2},..., \bar{\alpha_k}, \bar{X_1},...,\nabla^{h}_{\bar{X}}\bar{X_i},..., \bar{X_l}).
\end{align*}
Moreover, For $\bar{A}\in\mathcal{T}^{k}_{l}(\pi)$ tensor field $\nabla^{h}\bar{A}\in \mathcal{T}^{k}_{l+1}(\pi)$ is defined by the following rule
$$(\nabla^{h}\bar{A})(\bar{X},\bar{\alpha_1}, \bar{\alpha_2},..., \bar{\alpha_k}, \bar{X_1}, \bar{X_2},..., \bar{X_l}):=(\nabla^{h}_{\bar{X}}\bar{A})(\bar{\alpha_1}, \bar{\alpha_2},..., \bar{\alpha_k}, \bar{X_1}, \bar{X_2},..., \bar{X_l}).$$
Now, we consider map
\begin{equation}\left\{
\begin{array}{cc}
D: \Gamma(\pounds^\pi E)\times\Gamma(\pi^*\pi)\longrightarrow\Gamma(\pi^*\pi),\\
(\widetilde{X},\bar{Y})\longmapsto D_{\widetilde{X}}\bar{Y},
\end{array}\right.
\end{equation}
satisfies
\begin{enumerate}
\item[(i)] $D_{f\widetilde{X}+\widetilde{Y}}\bar{Z}=fD_{\widetilde{X}}\bar{Z}+D_{\widetilde{Y}}\bar{Z,}$
\item[(ii)] $D_{\widetilde{X}}\widetilde{f}\bar{Z}=\widetilde{f}D_{\widetilde{X}}\bar{Z}+\rho_{\pounds}(\widetilde{X})(\widetilde{f})\bar{Z,}$
\item[(iii)]$D_{\widetilde{X}}(\bar{Z}+\bar{W})=D_{\widetilde{X}}\bar{Z}+D_{\widetilde{X}}\bar{W}.$
\end{enumerate}
We call this map a \textit{$\rho_\pounds$-covariant derivative in $\Gamma(\pi^*\pi)$}.
\begin{theorem}
Let $h$ be a horizontal endomorphism and $\bar{\mathcal{H}}$ be a horizontal map of $\pi$ associated to $h$. Then
\[
\nabla: \Gamma(\pounds^\pi E)\times\Gamma(\pi^*\pi)\longrightarrow\Gamma(\pi^*\pi),
\]
given by
\begin{equation}\label{Beri}
\nabla_{\widetilde{X}}\bar{Y}:=\nabla^v_{\bar{\mathcal{V}}\widetilde{X}}\bar{Y}+\nabla^h_{\bar{j}\widetilde{X}}\bar{Y},
\end{equation}
is a $\rho_\pounds$-covariant derivative in $\Gamma(\pi^*\pi)$, where $\widetilde{X}\in\Gamma(\pounds^\pi E)$ and $\bar{Y}\in \Gamma(\pi^*\pi)$.
\end{theorem}
\begin{proof}
Let $\widetilde{f}\in C^\infty(E)$. Then we have
\begin{align*}
\nabla_{\widetilde{X}} \widetilde{f}\bar{Y}&=\nabla^v_{\bar{\mathcal{V}}\widetilde{X}}\widetilde{f}\bar{Y}+\nabla^h_{\bar{j}\widetilde{X}}\widetilde{f}\bar{Y}
=\rho_\pounds(\bar{i}\bar{\mathcal{V}}\widetilde{X})\widetilde{f}+\rho_\pounds(\bar{\mathcal{H}}\bar{j}\widetilde{X})\widetilde{f}
+\widetilde{f}\nabla^v_{\bar{\mathcal{V}}\widetilde{X}}\bar{Y}+\widetilde{f}\nabla^h_{\bar{j}\widetilde{X}}\bar{Y}\nonumber\\
&=\rho_\pounds(\bar{i}\bar{\mathcal{V}}\widetilde{X})\widetilde{f}+\rho_\pounds(\bar{\mathcal{H}}\bar{j}\widetilde{X})\widetilde{f}+\widetilde{f}\nabla_{\widetilde{X}} \bar{Y}.
\end{align*}
It is easy to show that $\bar{i}\bar{\mathcal{V}}\widetilde{X}=v\widetilde{X}$ and $\bar{\mathcal{H}}\bar{j}\widetilde{X}=h\widetilde{X}$. Therefore the above equation gives us
\[
\nabla_{\widetilde{X}} \widetilde{f}\bar{Y}=\rho_\pounds(v\widetilde{X})\widetilde{f}+\rho_\pounds(h\widetilde{X})\widetilde{f}+\widetilde{f}\nabla_{\widetilde{X}} \bar{Y}=\rho_\pounds(\widetilde{X})\widetilde{f}+\widetilde{f}\nabla_{\widetilde{X}} \bar{Y}.
\]
Similarly we can show $\nabla_{\widetilde{X}}(\bar{Y}+\bar{Z})=\nabla_{\widetilde{X}}\bar{Y}+\nabla_{\widetilde{X}}\bar{Z}$ and $\nabla_{\widetilde{f}\widetilde{X}+\widetilde{Y}}\bar{Y}=\widetilde{f}\nabla_{\widetilde{X}}\bar{Z}+\nabla_{\widetilde{Y}}\bar{Z}$. Therefore $\nabla$ is a $\rho_\pounds$-covariant derivative in $\Gamma(\pi^*\pi)$.
\end{proof}
The $\rho_\pounds$-covariant derivative $\nabla$ introduced by the above theorem is called \textit{Berwald derivative generated by $h$}. Indeed the Berwald derivative is as follows:
\begin{equation}
\nabla_{\widetilde{X}}\bar{Y}=\bar{j}[v\widetilde{X}, \bar{\mathcal{H}}\bar{Y}]_\pounds+\bar{\mathcal{V}}[h\widetilde{X}, \bar{i}\bar{Y}]_\pounds,\ \ \ \forall\widetilde{X}\in\Gamma(\pounds^\pi E),\ \ \forall\bar{Y}\in\Gamma(\pi^*\pi).
\end{equation}
Using the above equation we can obtain
\begin{align}
\nabla_{X^V}\widehat{Y}&=0, \hspace{2.1cm}\nabla_{X^h}\widehat{Y}=\bar{\mathcal{V}}[X^h, Y^V]_\pounds,\label{Berder}\\ \nabla_{\bar{i}\bar{X}}\bar{Y}&=\bar{j}[\bar{i}\bar{X},\mathcal{H}\bar{Y}]_\pounds,\ \ \ \ \nabla_{\mathcal{H}\bar{X}}\bar{Y}=\bar{\mathcal{V}}[\mathcal{H}\bar{X},i\bar{Y}]_\pounds.\label{Berder1}
\end{align}
where $X$ and $Y$ are sections of $E$ and $\bar{X}, \bar{Y}\in\Gamma(\pi^*\pi)$.

Now we consider the local basis $\{e_\alpha\}$ of $\Gamma(E)$. Then $\{\widehat{e_\alpha}\}$ is a basis of $\Gamma(\pi^*\pi)$, where $\widehat{e_\alpha}(u)=(u, e_\alpha(\pi(u)))$, for all $u\in E$. Using (\ref{2base}), proposition \ref{YB}, and the definition of $j$, it is easy to check that
\begin{equation}\label{YB4}
\bar{\mathcal{H}}\widehat{e_\alpha}=\delta_\alpha,\ \ \ \bar{i}\widehat{e_\alpha}=\mathcal{V}_\alpha,\ \ \ \bar{j}(\delta_\alpha)=\widehat{e_\alpha},\ \ \ \bar{\mathcal{V}}(\mathcal{V}_\alpha)=\widehat{e_\alpha}.
\end{equation}
Also we can deduce $\bar{\mathcal{V}}(\delta_\alpha)=0$. Therefore using the above equation, (\ref{2base}) and (\ref{Berder}) we obtain
\begin{align*}
\nabla_{\delta_\alpha}\widehat{e_\beta}&=\bar{\mathcal{V}}[\delta_\alpha, e^V_\beta]_\pounds=\bar{\mathcal{V}}[\delta_\alpha, \mathcal{V}_\beta]_\pounds=-\frac{\partial \mathcal{B}^\gamma_\alpha}{\partial\textbf{y}^\beta}\widehat{e_{\ \gamma}},\\
\nabla_{\mathcal{V}_\alpha}\widehat{e_\beta}&=\bar{j}[\mathcal{V}_\alpha, e^h_\beta]_\pounds=\bar{j}[\mathcal{V}_\alpha, \delta_\beta]_\pounds=0,
\end{align*}
and consequently
\begin{equation}\label{YB3}
\nabla_{\widetilde{X}}\bar{Y}=\Big(\widetilde{X}^\alpha((\rho^i_\alpha\circ\pi)\frac{\partial\bar{Y}^\beta}{\partial\textbf{x}^i}
+\mathcal{B}^\gamma_\alpha\frac{\partial\bar{Y}^\beta}{\partial\textbf{y}^\gamma})-\widetilde{X}^\alpha\bar{Y}^\gamma
\frac{\partial \mathcal{B}^\beta_\alpha}{\partial\textbf{y}^\gamma}+\widetilde{X}^{\bar\alpha}\frac{\partial\bar{Y}^\beta}{\partial\textbf{y}^\alpha}\Big)
\widehat{e_\beta},
\end{equation}
where $\widetilde{X}=\widetilde{X}^\alpha\delta_\alpha+\widetilde{X}^{\bar\alpha}\mathcal{V}_\alpha\in\Gamma(\pounds^\pi E)$ and $\bar{Y}=\bar{Y}^\beta \widehat{e_\beta}\in\Gamma(\pi^*\pi)$.
\begin{defn}
A $\rho_\pounds$-covariant derivative operator $D$ in $\Gamma(\pi^*\pi)$ is said to be associated to the horizontal map $\bar{\mathcal{H}}$ if $D\delta=\bar{\mathcal{V}}$.
\end{defn}
\begin{lemma}
Let $\nabla$ be the Berwald derivative induced by $h$. Then
\begin{equation}\label{pr1}
\nabla\delta=\bar{H}\circ\bar{j}+\bar{\mathcal{V}}.
\end{equation}
\end{lemma}
\begin{proof}
Using (ii) of definition \ref{YB2}, (ii) of definition \ref{Hi2} and (\ref{Hi1}) we get
\begin{align}\label{pr2}
(\nabla\delta)(\widetilde{X})=\nabla^v_{\bar{\mathcal{V}}\widetilde{X}}\delta+\nabla^h_{\bar{j}\widetilde{X}}\delta
=\bar{j}[\bar{i}\bar{\mathcal{V}}\widetilde{X}, \bar{\mathcal{H}}\delta]_\pounds+\bar{H}(\bar{j}\widetilde{X})=\bar{j}[v\widetilde{X}, \bar{\mathcal{H}}\delta]_\pounds+\bar{H}(\bar{j}\widetilde{X}).
\end{align}
Now let $\widetilde{X}=\widetilde{X}^\alpha\delta_\alpha+\widetilde{X}^{\bar{\alpha}}\mathcal{V}_\alpha$. It is easy to see that $\delta=\textbf{y}^\alpha\widehat{e_\alpha}$. Then using (\ref{YB4}) we obtain
\[
\bar{j}[v\widetilde{X}, \bar{\mathcal{H}}\delta]_\pounds=\bar{j}[\widetilde{X}^{\bar{\alpha}}\mathcal{V}_\alpha, \textbf{y}^\beta\delta_\beta]_\pounds=\widetilde{X}^{\bar{\alpha}}\bar{j}(\delta_\alpha)=\widetilde{X}^{\bar{\alpha}}\widehat{e_\alpha}
=\bar{\mathcal{V}}(\widetilde{X}).
\]
Setting the above equation in (\ref{pr2}) implies (\ref{pr1}).
\end{proof}
\begin{proposition}
Let $S$ be a spray on $\pounds^\pi E$ and $h$ be the horizontal endomorphism generated by it. If $\bar{\mathcal{H}}$ be the horizontal map generated by $h$ and $\nabla$ be the Berwald derivative induced by $h$, then $\nabla_S\delta=0$.
\end{proposition}
\begin{proof}
From the above lemma we have
\[
\nabla_S\delta=\bar{H}\bar{j}(S)+\bar{\mathcal{V}}(S).
\]
Using (\ref{YB4}) it is easy to see that $\bar{j}S=\textbf{y}^\alpha \widehat{e_\alpha}=\delta$. Thus we have
\begin{equation}\label{111}
\nabla_S\delta=\bar{H}\delta+\bar{\mathcal{V}}(S).
\end{equation}
But (\ref{Hi1}) gives us $\bar{H}\delta=\bar{\mathcal{V}}[\bar{\mathcal{H}}\delta, C]_\pounds$. In other hand, from Corollary \ref{4.3} we have $hS=S$. Therefore we get
\[
S=hS=\bar{\mathcal{H}}\bar{j}S=\bar{\mathcal{H}}\delta,
\]
and consequently $\bar{H}\delta=\bar{\mathcal{V}}[S, C]_\pounds$. Since $S$ is a spray then $[S, C]_\pounds=-S$. Therefore $\bar{H}\delta=-\bar{\mathcal{V}}(S)$. Setting this equation in (\ref{111}) we obtain $\nabla_S\delta=0$.

\end{proof}
\subsection{Torsions and partial curvatures}
Let $D$ be a $\rho_\pounds$-covariant derivative in $\Gamma(\pi^*\pi)$. The $\pi ^*\pi$-valued two-forms
\begin{align*}
T^h(D)(\widetilde{X}, \widetilde{Y})&:=D_{\widetilde{X}}\bar{j}\widetilde{Y}-D_{\widetilde{Y}}\bar{j}\widetilde{X}-\bar{j}[\widetilde{X},\widetilde{Y}]_\pounds,\\ T^v(D)(\widetilde{X}, \widetilde{Y})&:=D_{\widetilde{X}}\bar{\mathcal{V}}\widetilde{Y}-D_{\widetilde{Y}}\bar{\mathcal{V}}\widetilde{X}-\bar{\mathcal{V}}[\widetilde{X},\widetilde{Y}]_\pounds,
\end{align*}
are said to be the horizontal and the vertical torsion of $D$, respectively, where $\widetilde{X}$ and $\widetilde{Y}$ belong to $\Gamma(\pounds^\pi E)$.

Let $\bar{X}, \bar{Y}\in\Gamma(\pi^*\pi)$. The maps $A$ and $B$ given by
\begin{equation}\label{tensor}
A(\bar{X}, \bar{Y}):=T^h(D)(\bar{\mathcal{H}}\bar{X},\bar{\mathcal{H}}\bar{Y}), \ \ \ \ B(\bar{X}, \bar{Y}):=T^h(D)(\bar{\mathcal{H}}\bar{X},\bar{i}\bar{Y}),
\end{equation}
are called the $h$-horizontal and the $h$-mixed torsion of $D$ (with respect to $\bar{\mathcal{H}}$), respectively. $A$ will also be mentioned as the torsion of $D$, while for $B$ we
use the term Finsler torsion as well. $D$ is said to be symmetric if $A=0$ and $B$ is symmetric. The maps $R^1$, $P^1$ and $Q^1$ given by
\begin{align}
R^1(\bar{X},\bar{Y})&:=T^v(D)(\bar{\mathcal{H}}\bar{X},\bar{\mathcal{H}}\bar{Y}), \ \ \ \ P^1(\bar{X},\bar{Y}):=T^v(D)(\bar{\mathcal{H}}\bar{X},\bar{i}\bar{Y}),\label{tensor1}\\
Q^1(\bar{X},\bar{Y})&:=T^v(D)(\bar{i}\bar{X},\bar{i}\bar{Y}), \ \ \ \forall \bar{X}, \bar{Y}\in\Gamma(\pi^*\pi),\label{tensor2}
\end{align}
are called the $v$-horizontal , the $v$-mixed and the $v$-vertical torsion of $D$, respectively.
Using (\ref{Berder1}), (\ref{tensor}), (\ref{tensor1}) and (\ref{tensor2}) we can obtain
\begin{lemma}\label{5ta}
Let $D$ be a $\rho_\pounds$-covariant derivative in $\Gamma(\pi^*\pi)$. Then all of the partial torsions of the $\rho_\pounds$-covariant derivative operator $D$ are tensor fields of type (1, 2) on $\Gamma(\pi^*\pi)$. Moreover, for any vector fields $\bar{X}$, $\bar{Y}$ belong $\Gamma(\pi^*\pi)$ we have
\begin{align*}
A(\bar{X},\bar{Y})&=D_{\bar{\mathcal{H}}\bar{X}}\bar{Y}-D_{\bar{\mathcal{H}}\bar{Y}}\bar{X}-\bar{j}[\bar{\mathcal{H}}\bar{X},\bar{\mathcal{H}}\bar{Y}]_\pounds,\\
B(\bar{X},\bar{Y})&=-D_{\bar{i}\bar{Y}}\bar{X}-\bar{j}[\bar{\mathcal{H}}\bar{X},\bar{i}\bar{Y}]_\pounds=-D_{\bar{i}\bar{Y}}\bar{X}+\nabla_{\bar{i}\bar{Y}}\bar{X},\\
R^1(\bar{X},\bar{Y})&=-\bar{\mathcal{V}}[\bar{\mathcal{H}}\bar{X},\bar{\mathcal{H}}\bar{Y}]_\pounds,\\
P^1(\bar{X},\bar{Y})&=D_{\bar{\mathcal{H}}\bar{X}}\bar{Y}-\bar{\mathcal{V}}[\bar{\mathcal{H}}\bar{X},\bar{i}\bar{Y}]_\pounds
=D_{\bar{\mathcal{H}}\bar{X}}\bar{Y}-\nabla_{\bar{\mathcal{H}}\bar{X}}\bar{Y},\\
Q^1(\bar{X},\bar{Y})&=D_{\bar{i}\bar{X}}\bar{Y}-D_{\bar{i}\bar{Y}}\bar{X}-\bar{\mathcal{V}}[\bar{i}\bar{X},\bar{i}\bar{Y}]_\pounds,
\end{align*}
where $\nabla$ is the Berwald derivative given by (\ref{Beri}).
\end{lemma}
\begin{cor}
A $\rho_\pounds$-covariant derivative in $\Gamma(\pi^*E)$ is the Berwald derivative
induced by a given horizontal endomorphism if and only if, its Finsler torsion and $v$-mixed torsion vanish.
\end{cor}
Using the above lemma we get
\begin{align*}
A(\bar{j}\widetilde{X}, \bar{j}\widetilde{Y})&=D_{h\widetilde{X}}\bar{j}\widetilde{Y}-D_{h\widetilde{Y}}\bar{j}\widetilde{X}-\bar{j}[h\widetilde{X}, h\widetilde{Y}]_\pounds,\\
B(\bar{j}\widetilde{X}, \bar{\mathcal{V}}\widetilde{Y})&=-D_{v\widetilde{Y}}\bar{j}\widetilde{X}-\bar{j}[h\widetilde{X}, v\widetilde{Y}]_\pounds,\\
-B(\bar{j}\widetilde{Y}, \bar{\mathcal{V}}\widetilde{X})&=D_{v\widetilde{X}}\bar{j}\widetilde{Y}+\bar{j}[h\widetilde{Y}, v\widetilde{X}]_\pounds.
\end{align*}
Since $[v\widetilde{X}, v\widetilde{Y}]\in\Gamma(v\pounds^\pi E)$, then $\bar{j}[v\widetilde{X}, v\widetilde{Y}]=0$. Therefore summing the above equations give us
\begin{align*}
A(\bar{j}\widetilde{X}, \bar{j}\widetilde{Y})+B(\bar{j}\widetilde{X}, \bar{\mathcal{V}}\widetilde{Y})-B(\bar{j}\widetilde{Y}, \bar{\mathcal{V}}\widetilde{X})&=D_{\widetilde{X}}\bar{j}\widetilde{Y}-D_{\widetilde{Y}}\bar{j}\widetilde{X}-\bar{j}[\widetilde{X},\widetilde{Y}]_\pounds\\
&=T^h(D)(\widetilde{X}, \widetilde{Y}).
\end{align*}
Thus we have
\begin{lemma}
The horizontal torsion $T^h(D)$ is completely determined by the
torsion $A$ and the Finsler torsion $B$. Indeed, we have
\[
T^h(D)(\widetilde{X}, \widetilde{Y})=A(\bar{j}\widetilde{X}, \bar{j}\widetilde{Y})+B(\bar{j}\widetilde{X}, \bar{\mathcal{V}}\widetilde{Y})-B(\bar{j}\widetilde{Y}, \bar{\mathcal{V}}\widetilde{X}),\ \ \ \forall \widetilde{X}, \widetilde{Y}\in\Gamma(\pounds^\pi E).
\]
\end{lemma}
\begin{lemma}
Let $D$ be a $\rho_\pounds$-covariant derivative in $\Gamma(\pi^*\pi)$. If $D$ is associated to the horizontal map $\bar{\mathcal{H}}$, then for every section $\bar{X}$ of $\pi^*\pi$ we have
\[
B(\delta, \bar{X})=0,\ \ \ P^1(\bar{X}, \delta)=-\bar{H}(\bar{X}).
\]
\end{lemma}
\begin{proof}
Since $D$ is associated to the horizontal map $\bar{\mathcal{H}}$, then $D\delta=\bar{\mathcal{V}}$. Therefore using lemma \ref{5ta} we get
\begin{align}\label{5ta0}
B(\delta, \bar{X})&=-D_{\bar{i}\bar{X}}\delta-\bar{j}[\bar{\mathcal{H}}\delta, \bar{i}\bar{X}]_\pounds=-\bar{\mathcal{V}}(\bar{i}\bar{X})-\bar{j}[\bar{\mathcal{H}}\delta, \bar{i}\bar{X}]_\pounds\nonumber\\
&=-\bar{X}-\bar{j}[\bar{\mathcal{H}}\delta, \bar{i}\bar{X}]_\pounds.
\end{align}
Now let $\bar{X}=\bar{X}^\alpha\widehat{e_\alpha}$. Then we deduce $\bar{i}\bar{X}=\bar{X}^\alpha\mathcal{V}_\alpha$ and consequently
\[
\bar{j}[\bar{\mathcal{H}}\delta, \bar{i}\bar{X}]_\pounds=\bar{j}[\textbf{y}^\alpha\delta_\alpha, \bar{X}^\beta\mathcal{V}_\beta]_\pounds=-\bar{j}(\bar{X}^\alpha\delta_\alpha)=-\bar{X}^\alpha\widehat{e_\alpha}
=-\bar{X}.
\]
Setting the above equation in (\ref{5ta0}) we derive that $B(\delta, \bar{X})=0$. Using (\ref{pr1}) and lemma \ref{5ta} we get
\[
P^1(\bar{X}, \delta)=D_{\bar{\mathcal{H}}\bar{X}}\delta-\nabla_{\bar{\mathcal{H}}\bar{X}}\delta
=\bar{\mathcal{V}}\bar{\mathcal{H}}\bar{X}-(\bar{H}\circ\bar{j}+\bar{\mathcal{V}})(\bar{\mathcal{H}}\bar{X})=-\bar{H}\circ\bar{j}(\bar{\mathcal{H}}\bar{X}).
\]
But we have $\bar{j}\bar{\mathcal{H}}=1_{\Gamma(\pi^*\pi)}$. Therefore the above equation gives us the second part of the assertion.
\end{proof}
\begin{defn}\label{RPQ}
Let $D$ be a $\rho_\pounds$-covariant derivative in $\Gamma(\pi^*\pi)$. Then the
maps $R$, $P$ and $Q$ given by
\begin{align*}
R(\bar{X},\bar{Y})\bar{Z}&:=K^D(\bar{\mathcal{H}}\bar{X},\bar{\mathcal{H}}\bar{Y})\bar{Z},\\
P(\bar{X},\bar{Y})\bar{Z}&:=K^D(\bar{\mathcal{H}}\bar{X},\bar{i}\bar{Y})\bar{Z},\\
Q(\bar{X},\bar{Y})\bar{Z}&:=K^D(\bar{i}\bar{X},\bar{i}\bar{Y})\bar{Z},
\end{align*}
are said to be the horizontal or Riemann curvature, the mixed or Berwald
curvature and the vertical or Berwald-Cartan curvature of $D$ (with respect
to $\bar{\mathcal{H}}$), respectively.
\end{defn}
\begin{lemma}
Let $D$ be a $\rho_\pounds$-covariant derivative in $\Gamma(\pi^*\pi)$. If $D$ is associated to the horizontal map $\bar{\mathcal{H}}$, then we have
\[
R(\bar{X},\bar{Y})\delta=R^1(\bar{X}, \bar{Y}),\ P(\bar{X},\bar{Y})\delta=P^1(\bar{X}, \bar{Y}),\ Q(\bar{X},\bar{Y})\delta=Q^1(\bar{X}, \bar{Y}),
\]
where $\bar{X}, \bar{Y}\in\Gamma(\pi^*\pi)$. Moreover, if the Finsler torsion is symmetric, then $Q(., .)\delta=Q^1=0$.
\end{lemma}
\begin{proof}
Since $D$ is associated to the horizontal map $\bar{\mathcal{H}}$, then $D\delta=\bar{\mathcal{V}}$ and therefore
\[
D_{\bar{\mathcal{H}}\bar{X}}\delta=0,\ \ \ D_{\bar{i}\bar{X}}\delta=\bar{X},\ \ \forall\bar{X}\in\Gamma(\pi^*\pi).
\]
Using the above equations, the proof of the first part of the assertion is obvious. Now we prove the second part. From the first part we have
\[
Q(\bar{X}, \bar{Y})\delta=Q^1(\bar{X}, \bar{Y})=D_{\bar{i}\bar{X}}\bar{Y}-D_{\bar{i}\bar{Y}}\bar{X}-\bar{\mathcal{V}}[\bar{i}\bar{X},\bar{i}\bar{Y}]_\pounds.
\]
Since the Finsler torsion $B$ is symmetric, then
\[
0=B(\bar{X}, \bar{Y})-B(\bar{Y}, \bar{X})=D_{\bar{i}\bar{X}}\bar{Y}-D_{\bar{i}\bar{Y}}\bar{X}-\bar{j}[\bar{\mathcal{H}}\bar{X}, \bar{i}\bar{Y}]_\pounds+\bar{j}[\bar{\mathcal{H}}\bar{Y}, \bar{i}\bar{X}]_\pounds.
\]
Two above equations give us
\[
Q(\bar{X}, \bar{Y})\delta=\bar{j}[\bar{\mathcal{H}}\bar{X}, \bar{i}\bar{Y}]_\pounds-\bar{j}[\bar{\mathcal{H}}\bar{Y}, \bar{i}\bar{X}]_\pounds-\bar{\mathcal{V}}[\bar{i}\bar{X},\bar{i}\bar{Y}]_\pounds.
\]
Since $\bar{j}$ is surjective, then there exist $\widetilde{X}, \widetilde{Y}\in\Gamma(\pounds^\pi E)$ such that $\bar{X}=\bar{j}\widetilde{X}$ and $\bar{Y}=\bar{j}\widetilde{Y}$. Setting these equations in the above equation imply
\[
Q(\bar{j}\widetilde{X}, \bar{j}\widetilde{Y})\delta=\bar{j}[h\widetilde{X}, J\widetilde{Y}]_\pounds-\bar{j}[h\widetilde{Y}, J\widetilde{X}]_\pounds-\bar{\mathcal{V}}[J\widetilde{X}, J\widetilde{Y}]_\pounds,
\]
and consequently
\begin{align*}
\bar{i}(Q(\bar{j}\widetilde{X}, \bar{j}\widetilde{Y})\delta)&=J[h\widetilde{X}, J\widetilde{Y}]_\pounds-J[h\widetilde{Y}, J\widetilde{X}]_\pounds-v[J\widetilde{X}, J\widetilde{Y}]_\pounds\\
&=J[\widetilde{X}, J\widetilde{Y}]_\pounds+J[J\widetilde{X}, \widetilde{Y}]_\pounds-[J\widetilde{X}, J\widetilde{Y}]_\pounds\\
&=-N_J(\widetilde{X}, \widetilde{Y})=0.
\end{align*}
Since $\bar{i}$ is injective, then the above equation gives us $Q(\bar{j}\widetilde{X}, \bar{j}\widetilde{Y})\delta=0$ and therefore $Q(\bar{X}, \bar{Y})\delta=0$.
\end{proof}
Now we denote the torsions and the curvatures of the Berwald derivative $\nabla$, by $\stackrel{\circ}{A}$, $\stackrel{\circ}{B}$, $\stackrel{\circ}{R^1}$, $\stackrel{\circ}{P^1}$, $\stackrel{\circ}{Q^1}$ and $\stackrel{\circ}{R}$, $\stackrel{\circ}{P}$, $\stackrel{\circ}{Q}$, respectively. Using (\ref{YB3}) and Lemma \ref{5ta} it is easy to prove the following
\begin{lemma}
Let $\nabla$ be the Berwald derivative induced by $h$ and $\{e_\alpha\}$ be a basis of $E$. Then
\begin{align}
\stackrel{\circ}{A}&=\frac{1}{2}t^\gamma_{\alpha\beta}\widehat{e^\alpha}\wedge\widehat{e^\beta}\otimes\widehat{\ e_\gamma},\label{151}\\
\stackrel{\circ}{R^1}&=-\frac{1}{2}R^\gamma_{\alpha\beta}\widehat{e^\alpha}\wedge\widehat{e^\beta}\otimes\widehat{\ e_\gamma},\label{152}\\
\stackrel{\circ}{B}&=0,\ \stackrel{\circ}{P^1}=0,\ \stackrel{\circ}{Q^1}=0,
\end{align}
where $\{\widehat{e^\alpha}\}$ is a dual basis of $\{\widehat{e_\alpha}\}$ and $t^\gamma_{\alpha\beta}$ and $R^\gamma_{\alpha\beta}$ are given by (\ref{wt1}) and (\ref{curv0}).
\end{lemma}
Using $\stackrel{\circ}{A}$ and $\stackrel{\circ}{R^1}$ we introduce the following tensor fields:
\begin{equation}\label{153}
\left\{
\begin{array}{cc}
\stackrel{\circ}{A}_\circ:\Gamma(\pounds^\pi E)\times\Gamma(\pounds^\pi E)\rightarrow\Gamma(\pounds^\pi E),\\
\stackrel{\circ}{A}_\circ(\widetilde{X}, \widetilde{Y})=\bar{i}\stackrel{\circ}{A}(\bar{j}\widetilde{X}, \bar{j}\widetilde{Y})
\end{array}
\right.
\end{equation}
\begin{equation}\label{154}
\left\{
\begin{array}{cc}
\stackrel{\circ}{R^1}_\circ:\Gamma(\pounds^\pi E)\times\Gamma(\pounds^\pi E)\rightarrow\Gamma(\pounds^\pi E),\\
\stackrel{\circ}{R^1}_\circ(\widetilde{X}, \widetilde{Y})=\bar{i}\stackrel{\circ}{R^1}(\bar{j}\widetilde{X}, \bar{j}\widetilde{Y})
\end{array}
\right.
\end{equation}
Using (\ref{151}) and (\ref{153}) we can obtain
\[
\stackrel{\circ}{A}_\circ(\delta_\alpha, \delta_\beta)=t^\gamma_{\alpha\beta}\mathcal{V}_\gamma,\ \ \ \stackrel{\circ}{A}_\circ(\mathcal{V}_\alpha, \delta_\beta)=\stackrel{\circ}{A}_\circ(\mathcal{V}_\alpha, \mathcal{V}_\beta)=0.
\]
Therefore from (\ref{wt}) we deduce
\[
\stackrel{\circ}{A}_\circ=\frac{1}{2}t^\gamma_{\alpha\beta}\mathcal{X}^\alpha\wedge\mathcal{X}^\beta\otimes\mathcal{V}_\gamma=t,
\]
where $t$ is the weak torsion of $h$. Similarly using (\ref{152}) and (\ref{154}) we obtain
\[
\stackrel{\circ}{R^1}_\circ=-\frac{1}{2}R^\gamma_{\alpha\beta}\mathcal{X}^\alpha\wedge\mathcal{X}^\beta\otimes\mathcal{V}_\gamma=\Omega,
\]
where $\Omega$ is the curvature of $h$ given in (\ref{curv000}).
\begin{proposition}
Let $\nabla$ be the Berwald derivativ induced by $h$ in $\Gamma(\pi^*\pi)$. Then $\stackrel{\circ}{A}_\circ=t$ and $\stackrel{\circ}{R^1}_\circ=\Omega$, where $t$ and $\Omega$ are weak torsion and curvature of $h$, respectively.
\end{proposition}
Using (\ref{YB4}), (\ref{YB3}) and definition \ref{RPQ} we can deduce
\begin{theorem}
Let $\nabla$ be the Berwald derivative induced by $h$ in $\Gamma(\pi^*\pi)$ and $\{e_\alpha\}$ be a basis of $E$. Then
\begin{align*}
\stackrel{\circ}{R}&=\stackrel{\circ}{R}_{\alpha\beta\gamma}^{\ \ \ \ \lambda}\ \widehat{e_{\:\lambda}}\otimes\widehat{e^\alpha}\otimes\widehat{e^\beta}\otimes\widehat{e^{\:\gamma}},\\
\stackrel{\circ}{P}&=\stackrel{\circ}{P}_{\alpha\beta\gamma}^{\ \ \ \ \lambda}\widehat{e_{\:\lambda}}\otimes\widehat{e^\alpha}\otimes\widehat{e^\beta}\otimes\widehat{e^{\:\gamma}},\\
\stackrel{\circ}{Q}&=\stackrel{\circ}{S}_{\alpha\beta\gamma}^{\ \ \ \ \lambda}\ \widehat{e_{\:\lambda}}\otimes\widehat{e^\alpha}\otimes\widehat{e^\beta}\otimes\widehat{e^{\:\gamma}},
\end{align*}
where
\begin{align}
{\stackrel{\circ}{R}}^{\ \ \ \ \lambda}_{\alpha\beta\gamma}&=-(\rho^i_\alpha\circ \pi)\frac{\partial^2 \mathcal{B}^\lambda_{\beta}}{\partial \textbf{x}^i\partial \textbf{y}^\gamma}-\mathcal{B}^\mu_\alpha\frac{\partial^2 \mathcal{B}^\lambda_{\beta}}{\partial \textbf{y}^\mu\partial \textbf{y}^\gamma}+(\rho^i_\beta\circ \pi)\frac{\partial^2 \mathcal{B}^\lambda_{\alpha}}{\partial \textbf{x}^i\partial \textbf{y}^\gamma}+\mathcal{B}^\mu_\beta\frac{\partial^2 \mathcal{B}^\lambda_{\alpha}}{\partial \textbf{y}^\mu\partial \textbf{y}^\gamma}\nonumber\\
&\ \ \ +\frac{\partial \mathcal{B}^\mu_{\beta}}{\partial \textbf{y}^\gamma}\frac{\partial \mathcal{B}^\lambda_{\alpha}}{\partial \textbf{y}^\mu}-\frac{\partial \mathcal{B}^\mu_{\alpha}}{\partial \textbf{y}^\gamma}\frac{\partial \mathcal{B}^\lambda_{\beta}}{\partial \textbf{y}^\mu}+(L^\mu_{\alpha\beta}\circ \pi)\frac{\partial \mathcal{B}^\lambda_{\mu}}{\partial \textbf{y}^\gamma},\\
{\stackrel{\circ}{P}}^{\ \ \ \ \lambda}_{\alpha\beta\gamma}&=\frac{\partial^2 \mathcal{B}^\lambda_{\alpha}}{\partial \textbf{y}^\beta\partial \textbf{y}^\gamma},\label{mixed000}\\
{\stackrel{\circ}{S}}^{\ \ \ \ \lambda}_{\alpha\beta\gamma}&=0.
\end{align}
\end{theorem}
Using $\stackrel{\circ}{R}$, $\stackrel{\circ}{P}$ and $\stackrel{\circ}{Q}$ we introduce the following tensor fields:
\begin{equation}\label{159}
\left\{
\begin{array}{cc}
\stackrel{\circ}{R}_\circ:\Gamma(\pounds^\pi E)\times\Gamma(\pounds^\pi E)\rightarrow\Gamma(\pounds^\pi E),\\
\stackrel{\circ}{R}_\circ(\widetilde{X}, \widetilde{Y})=\bar{i}\stackrel{\circ}{R}(\bar{j}\widetilde{X}, \bar{j}\widetilde{Y}),
\end{array}
\right.
\end{equation}
\begin{equation}\label{160}
\left\{
\begin{array}{cc}
\stackrel{\circ}{P}_\circ:\Gamma(\pounds^\pi E)\times\Gamma(\pounds^\pi E)\rightarrow\Gamma(\pounds^\pi E),\\
\stackrel{\circ}{P}_\circ(\widetilde{X}, \widetilde{Y})=\bar{i}\stackrel{\circ}{P}(\bar{j}\widetilde{X}, \bar{j}\widetilde{Y}),
\end{array}
\right.
\end{equation}
\begin{equation}\label{161}
\left\{
\begin{array}{cc}
\stackrel{\circ}{Q}_\circ:\Gamma(\pounds^\pi E)\times\Gamma(\pounds^\pi E)\rightarrow\Gamma(\pounds^\pi E),\\
\stackrel{\circ}{Q}_\circ(\widetilde{X}, \widetilde{Y})=\bar{i}\stackrel{\circ}{Q}(\bar{j}\widetilde{X}, \bar{j}\widetilde{Y}).
\end{array}
\right.
\end{equation}
Using the above theorem, (\ref{RB})-(\ref{QB}) and (\ref{159})-(\ref{161}) we derive that
\begin{proposition}
Let $\nabla$ be the Berwald derivative induced by $h$. Then
\[
\stackrel{\circ}{R}_\circ={\stackrel{\text{\begin{tiny}B\end{tiny}}}{R}},\ \ \ \stackrel{\circ}{P}_\circ={\stackrel{\text{\begin{tiny}B\end{tiny}}}{P}},\ \ \ \stackrel{\circ}{Q}_\circ={\stackrel{\text{\begin{tiny}B\end{tiny}}}{Q}},
\]
where ${\stackrel{\text{\begin{tiny}B\end{tiny}}}{R}}$, ${\stackrel{\text{\begin{tiny}B\end{tiny}}}{P}}$ and ${\stackrel{\text{\begin{tiny}B\end{tiny}}}{Q}}$ are the horizontal, mixed and vertical curvatures of Berwald connection $(\stackrel{\text{\begin{tiny}B\end{tiny}}}{D}, h)$, respectively.
\end{proposition}
\begin{proposition}
Let $\nabla$ be the Berwald derivative induced by $h$. Then for sections $X$, $Y$ and $Z$ of $E$ we have
\[
\stackrel{\circ}{P}(\widehat{X}, \widehat{Y})\widehat{Z}=\bar{\mathcal{V}}[[X^h, Y^V]_\pounds, Z^V]_\pounds.
\]
\end{proposition}
\begin{proof}
Let $X=X^\alpha e_\alpha$, $Y=Y^\beta e_\beta$ and $Z=Z^\gamma e_\gamma$ be sections of $E$. Then we have $\widehat{X}=(X^\alpha\circ\pi)\widehat{e}_\alpha$, $\widehat{Y}=(Y^\beta\circ\pi)\widehat{e}_\beta$ and $\widehat{Z}=(Z^\gamma\circ\pi)\widehat{e}_\gamma$. Therefore, (\ref{mixed000}) implies
\[
\stackrel{\circ}{P}(\widehat{X}, \widehat{Y})\widehat{Z}=((X^\alpha Y^\beta Z^\gamma)\circ\pi)\frac{\partial^2 \mathcal{B}^\lambda_\alpha}{\partial\textbf{y}^\beta\partial\textbf{y}^\gamma}\widehat{e}_\lambda.
\]
Similarly we can obtain
\begin{align*}
\bar{\mathcal{V}}[[X^h, Y^V]_\pounds, Z^V]_\pounds&=\bar{\mathcal{V}}[[(X^\alpha\circ\pi)\delta_\alpha, (Y^\beta\circ\pi)\mathcal{V}_\beta]_\pounds, (Z^\gamma\circ\pi)\mathcal{V}_\gamma]_\pounds\\
&=((X^\alpha Y^\beta Z^\gamma)\circ\pi)\frac{\partial^2 \mathcal{B}^\lambda_\alpha}{\partial\textbf{y}^\beta\partial\textbf{y}^\gamma}\widehat{e}_\lambda.
\end{align*}
Two above equations gives us the assertion.
\end{proof}
With the help of the mixed curvature $\stackrel{\circ}{P}$, we define an important change of the Berwald derivative $\nabla$ by the formula
\begin{equation}
D_{\widetilde{X}}\bar{Y}:=\nabla_{\widetilde{X}}\bar{Y}+\frac{1}{n+1}(tr \stackrel{\circ}{P}(\bar{j}\widetilde{X}, \bar{Y}))\delta.
\end{equation}
The covariant derivative operator D so obtained is called the Yano derivative induced
by $\bar{\mathcal{H}}$. Using (\ref{YB3}) and the above equation we get
\begin{align*}
D_{\widetilde{X}}\bar{Y}&=\Big(\widetilde{X}^\alpha((\rho^i_\alpha\circ\pi)\frac{\partial\bar{Y}^\beta}{\partial\textbf{x}^i}
+\mathcal{B}^\gamma_\alpha\frac{\partial\bar{Y}^\beta}{\partial\textbf{y}^\gamma})-\widetilde{X}^\alpha\bar{Y}^\gamma
\frac{\partial \mathcal{B}^\beta_\alpha}{\partial\textbf{y}^\gamma}+\widetilde{X}^{\bar\alpha}\frac{\partial\bar{Y}^\beta}{\partial\textbf{y}^\alpha}
\\
&\ \ \ +\frac{1}{n+1}\widetilde{X}^\alpha\bar{Y}^\gamma \textbf{y}^\beta\frac{\partial^2\mathcal{B}^\lambda_\alpha}{\partial\textbf{y}^\lambda\partial\textbf{y}^\gamma}\Big)
\widehat{e}_\beta,
\end{align*}
where $\widetilde{X}=\widetilde{X}^\alpha\delta_\alpha+\widetilde{X}^{\bar\alpha}\mathcal{V}_\alpha\in\Gamma(\pounds^\pi E)$ and $\bar{Y}=\bar{Y}^\beta \widehat{e}_\beta\in\Gamma(\pi^*\pi)$. In particular case we have
\begin{align*}
D_{\delta_\alpha}\widehat{e}_\beta&=(\frac{1}{n+1}\textbf{y}^\gamma\frac{\partial^2\mathcal{B}^\lambda_\alpha}{\partial\textbf{y}^\lambda\partial\textbf{y}^\beta}-\frac{\partial \mathcal{B}^\gamma_\alpha}{\partial\textbf{y}^\beta})\widehat{e}_\gamma,\\
D_{\mathcal{V}_\alpha}\widehat{e}_\beta&=0.
\end{align*}
\section{Finsler algebroids}
This section is devoted to Finsler algebroids and their outputs. We will derive a pseudo-Riemannian metric from Finsler algebroid. Gradient of  smooth functions on Lie algebroid bundle and their lifts is studied. Special case of horizontal endomorphism named conservative, are visited. Barthel endomorphism on Finsler algebroids is proceeded too. Cartan tensor and some distinguished connections on Finsler algebroids are studied finally.

\begin{defn}
Finsler algebroid $(E, \mathcal{F})$ is a Lie algebroid $\pounds^\pi E$ provided with a fundamental Finsler function
$\mathcal{F}:E\rightarrow \mathbb{R}$ satisfying the conditions:

(i)\ $\mathcal{F}$ is a scalar differentiable function on the manifold $\stackrel{\circ}{E}=E-\{0\}$ and
continuous on the null section of $\pi:E\rightarrow M$,

(ii)\ $\mathcal{F}$ is a positive function and homogeneous of degree 2, i. e., $\pounds^\pounds_C\mathcal{F}=2\mathcal{F}$,

(iii)\ The fundamental form $\omega=d^\pounds d^\pounds_J\mathcal{F}$ is nondegenerate, where
\[
d^\pounds_J\mathcal{F}=i_Jd^\pounds\mathcal{F}=d^\pounds\mathcal{F}\circ J.
\]
\end{defn}
For the basis $\{\mathcal{X}_\alpha, \mathcal{V}_\alpha\}$ of $\Gamma(\pounds^\pi E)$ and the dual basis $\{\mathcal{X}^\alpha, \mathcal{V}^\alpha\}$ of it, we get $d^\pounds_J\mathcal{F}(\mathcal{V}_\alpha)=0$ and $d^\pounds_J\mathcal{F}(\mathcal{X}_\alpha)=\frac{\partial\mathcal{F}}{\partial{\textbf{y}}^\alpha}$. Therefore $d^\pounds_J\mathcal{F}$ has the following coordinate expression:
\begin{equation}\label{Finsler algebroid}
d^\pounds_J\mathcal{F}=\frac{\partial\mathcal{F}}{\partial{\textbf{y}}^\alpha}\mathcal{X}^\alpha.
\end{equation}
\begin{lemma}
The fundamental form $\omega$ of a Finsler algebroid has the following coordinate expression:
\begin{equation}\label{Finsler algebroid1}
\omega=\Big((\rho^i_\alpha\circ\pi)\frac{\partial^2\mathcal{F}}{\partial\textbf{x}^i\partial\textbf{y}^\beta}
-\frac{1}{2}\frac{\partial\mathcal{F}}{\partial{\textbf{y}}^\gamma}(L^\gamma_{\alpha\beta}\circ\pi)\Big)\mathcal{X}^\alpha\wedge\mathcal{X}^\beta
-\frac{\partial^2\mathcal{F}}{\partial\textbf{y}^\alpha\partial\textbf{y}^\beta}\mathcal{X}^\alpha\wedge\mathcal{V}^\beta.
\end{equation}
\end{lemma}
\begin{proof}
Using (\ref{Finsler algebroid}) we have
\begin{equation}\label{Finsler algebroid2}
\omega=d^\pounds d^\pounds_J\mathcal{F}=d^\pounds(\frac{\partial\mathcal{F}}{\partial{\textbf{y}}^\gamma})\wedge\mathcal{X}^\gamma+\frac{\partial\mathcal{F}}{\partial{\textbf{y}}^\gamma}d^\pounds\mathcal{X}^\gamma.
\end{equation}
It is easy to see that $(d^\pounds\mathcal{X}^\gamma)(\mathcal{X}_\alpha, \mathcal{X}_\beta)=-(L^\gamma_{\alpha\beta}\circ\pi)$ and $(d^\pounds\mathcal{X}^\gamma)(\mathcal{X}_\alpha, \mathcal{V}_\beta)=(d^\pounds\mathcal{X}^\gamma)(\mathcal{V}_\alpha, \mathcal{V}_\beta)=0$. Thus we have
\[
d^\pounds\mathcal{X}^\gamma=-\frac{1}{2}(L^\gamma_{\alpha\beta}\circ\pi)\mathcal{X}^\alpha\wedge\mathcal{X}^\beta.
\]
Also it is easy to check that $(d^\pounds(\frac{\partial\mathcal{F}}{\partial{\textbf{y}}^\gamma}))(\mathcal{X}_\beta)=\rho^i_\beta\frac{\partial^2\mathcal{F}}{\partial\textbf{x}^i\partial{\textbf{y}}^\gamma}$ and $(d^\pounds(\frac{\partial\mathcal{F}}{\partial{\textbf{y}}^\gamma}))(\mathcal{V}_\beta)=\frac{\partial^2\mathcal{F}}{\partial\textbf{y}^\beta\partial{\textbf{y}}^\gamma}$. Thus we have
\[
d^\pounds(\frac{\partial\mathcal{F}}{\partial{\textbf{y}}^\gamma})=(\rho^i_\beta\circ\pi)\frac{\partial^2\mathcal{F}}{\partial\textbf{x}^i\partial{\textbf{y}}^\gamma}\mathcal{X}^\beta+\frac{\partial^2\mathcal{F}}{\partial\textbf{y}^\beta\partial{\textbf{y}}^\gamma}\mathcal{V}^\beta.
\]
Setting the above two equations in (\ref{Finsler algebroid2}) imply (\ref{Finsler algebroid1}).
\end{proof}
From (\ref{Finsler algebroid1}) we deduce that the fundamental form $\omega$ is nondegenarate if and only if the symmetric matrix $(\frac{\partial^2\mathcal{F}}{\partial\textbf{y}^\alpha\partial\textbf{y}^\beta})$ is regular.
\begin{proposition}\label{Best1}
For the fundamental form $\omega$ we have the following identities:
\[
(i)\ i_J\omega=0,\ \ \ (ii)\ \pounds^\pounds_C\omega=\omega,\ \ \ (iii)\ i_C\omega=d^\pounds_J\mathcal{F}.
\]
\end{proposition}
\begin{proof}
We have
\[
i_J\omega=i_{\mathcal{X}^\gamma\otimes\mathcal{V}_\gamma}\omega=\mathcal{X}^\gamma\wedge i_{\mathcal{V}_\gamma}\omega.
\]
It is easy to check that $i_{\mathcal{V}_\gamma}\mathcal{X^\alpha}=0$ and $i_{\mathcal{V}_\gamma}\mathcal{V^\alpha}=\delta^\alpha_\gamma$. Therefore from (\ref{Finsler algebroid1}) we get
$
i_{\mathcal{V}_\gamma}\omega=\frac{\partial^2\mathcal{F}}{\partial\textbf{y}^\alpha\partial\textbf{y}^\gamma}\mathcal{X}^\alpha,
$
and consequently
\[
i_J\omega=\frac{\partial^2\mathcal{F}}{\partial\textbf{y}^\alpha\partial\textbf{y}^\gamma}\mathcal{X}^\gamma\wedge\mathcal{X}^\alpha.
\]
It is easy to see that $\frac{\partial^2\mathcal{F}}{\partial\textbf{y}^\alpha\partial\textbf{y}^\gamma}\mathcal{X}^\gamma\wedge\mathcal{X}^\alpha=-\frac{\partial^2\mathcal{F}}{\partial\textbf{y}^\alpha\partial\textbf{y}^\gamma}\mathcal{X}^\gamma\wedge\mathcal{X}^\alpha$. Thus we deduce $i_J\omega=0$. Now we prove (ii). Since $[C, \mathcal{X}_\alpha]=0$, then using (\ref{Finsler algebroid1}) we derive that
\begin{align*}
(\pounds^\pounds_C\omega)(\mathcal{X}_\alpha, \mathcal{X}_\beta)&=\rho_\pounds(C)\Big((\rho^i_\alpha\circ\pi)\frac{\partial^2\mathcal{F}}{\partial\textbf{x}^i\partial\textbf{y}^\beta}
-(\rho^i_\beta\circ\pi)\frac{\partial^2\mathcal{F}}{\partial\textbf{x}^i\partial\textbf{y}^\alpha}
-\frac{\partial\mathcal{F}}{\partial{\textbf{y}}^\gamma}(L^\gamma_{\alpha\beta}\circ\pi)\Big)\\
&=\textbf{y}^\lambda\Big((\rho^i_\alpha\circ\pi)\frac{\partial^3\mathcal{F}}{\partial\textbf{x}^i\partial\textbf{y}^\beta\partial\textbf{y}^\lambda}
-(\rho^i_\beta\circ\pi)\frac{\partial^3\mathcal{F}}{\partial\textbf{x}^i\partial\textbf{y}^\alpha\partial\textbf{y}^\lambda}\\
&\ \ \ -\frac{\partial^2\mathcal{F}}{\partial{\textbf{y}}^\gamma\partial\textbf{y}^\lambda}(L^\gamma_{\alpha\beta}\circ\pi)\Big).
\end{align*}
Since $\mathcal{F}$ is homogenous of degree 2, then we can obtain
\begin{equation}\label{only God}
\frac{\partial\mathcal{F}}{\partial{\textbf{y}}^\gamma}=\textbf{y}^\lambda\frac{\partial^2\mathcal{F}}{\partial{\textbf{y}}^\gamma\partial\textbf{y}^\lambda}.
\end{equation}
Using this equation in the above equation we get
\[
(\pounds^\pounds_C\omega)(\mathcal{X}_\alpha, \mathcal{X}_\beta)=(\rho^i_\alpha\frac{\partial^2\mathcal{F}}{\partial\textbf{x}^i\partial\textbf{y}^\beta}-\rho^i_\beta\frac{\partial^2\mathcal{F}}{\partial\textbf{x}^i\partial\textbf{y}^\alpha}-\frac{\partial\mathcal{F}}{\partial{\textbf{y}}^\gamma}L^\gamma_{\alpha\beta})=\omega(\mathcal{X}_\alpha, \mathcal{X}_\beta).
\]
Similarly, we can obtain
\[
(\pounds^\pounds_C\omega)(\mathcal{X}_\alpha, \mathcal{V}_\beta)=-\frac{\partial^2\mathcal{F}}{\partial\textbf{y}^\alpha\partial\textbf{y}^\beta}=\omega(\mathcal{X}_\alpha, \mathcal{V}_\beta),\ \ \  (\pounds^\pounds_C\omega)(\mathcal{V}_\alpha, \mathcal{V}_\beta)=0=\omega(\mathcal{V}_\alpha, \mathcal{V}_\beta).
\]
Thus we have (ii). It is easy to check that $i_C\mathcal{X}^\gamma=0$ and $i_C\mathcal{V}^\gamma=\textbf{y}^\gamma$. Thus using (\ref{Finsler algebroid1}) and (\ref{only God}) we get
\[
i_C\omega=\textbf{y}^\gamma\frac{\partial^2\mathcal{F}}{\partial\textbf{y}^\alpha\partial\textbf{y}^\gamma}\mathcal{X}^\alpha=\frac{\partial\mathcal{F}}{\partial\textbf{y}^\alpha}\mathcal{X}^\alpha=d^\pounds_J\mathcal{F}.
\]
\end{proof}
\begin{defn}
Let $(E, \mathcal{F})$ be a Finsler algebroid with fundamental form $\omega$. Map
\[
\mathcal{G}:\Gamma(v\!\!\stackrel{\circ}{\pounds^\pi E})\times\Gamma(v\!\!\stackrel{\circ}{\pounds^\pi E})\rightarrow C^\infty(\stackrel{\circ}{\pounds^\pi E}),
\]
defined by $\mathcal{G}(J\widetilde{X}, J\widetilde{Y}):=\omega(J\widetilde{X}, \widetilde{Y})$ is called the vertical metric of Finsle algebroid $(E, \mathcal{F})$.
\end{defn}
\begin{rem}\label{Best4}
It is easy to check that $\mathcal{G}$ is bilinear, symmetric and nondegenerate on $v\!\!\stackrel{\circ}{\pounds^\pi E}$.
\end{rem}
From remark \ref{Best4} we have
\begin{proposition}
Let $h$ be a horizontal endomorphism and $\mathcal{G}$ be the vertical metric of Finsler manifold $(E, \mathcal{F})$. Then
the function $\widetilde{\mathcal{G}}:\Gamma(\stackrel{\circ}{\pounds^\pi E})\times\Gamma(\stackrel{\circ}{\pounds^\pi E})\rightarrow C^\infty(\stackrel{\circ}{\pounds^\pi E})$ given by
\begin{equation}\label{Best5}
\widetilde{\mathcal{G}}(\widetilde{X}, \widetilde{Y}):=\mathcal{G}(J\widetilde{X}, J\widetilde{Y})+\mathcal{G}(v\widetilde{X}, v\widetilde{Y}),\ \ \ \forall \widetilde{X}, \widetilde{Y}\in\Gamma(\stackrel{\circ}{\pounds^\pi E}).
\end{equation}
 is a pseudo-Riemannian metric on $\stackrel{\circ}{\pounds^\pi E}$.
\end{proposition}
The pseudo-Riemannian metric $\widetilde{\mathcal{G}}$ introduced in the above proposition, is called the \textit{prolongation of $\mathcal{G}$ along $h$}.

In the coordinate expression, using (\ref{Finsler algebroid1}) we obtain
\begin{equation}\label{Best6}
\mathcal{G}_{\alpha\beta}:=\mathcal{G}(\mathcal{V}_\alpha, \mathcal{V}_\beta)=\omega(\mathcal{V}_\alpha, \mathcal{X}_\beta)=\frac{\partial^2\mathcal{F}}{\partial\textbf{y}^\alpha\partial\textbf{y}^\beta}.
\end{equation}
Also, using (\ref{Best5}) we can obtain
\[
\widetilde{\mathcal{G}}(\delta_\alpha, \delta_\beta)=\mathcal{G}_{\alpha\beta},\ \ \ \ \ \ \ \ \widetilde{\mathcal{G}}(\delta_\alpha, \mathcal{V}_\beta)=0, \ \ \ \ \ \ \ \ \widetilde{\mathcal{G}}(\mathcal{V}_\alpha, \mathcal{V}_\beta)=\mathcal{G}_{\alpha\beta},
\]
and consequently
\begin{equation}\label{pmetric}
\widetilde{\mathcal{G}}=\mathcal{G}_{\alpha\beta}\mathcal{X}^\alpha\otimes\mathcal{X}^\beta
+\mathcal{G}_{\alpha\beta}\delta\mathcal{V}^\alpha\otimes\delta\mathcal{V}^\beta.
\end{equation}
\begin{proposition}
For metrics $\mathcal{G}$, $\widetilde{\mathcal{G}}$ and sections $X$, $Y$ of $\stackrel{\circ}{E}$, we have
\begin{align}
\widetilde{\mathcal{G}}(X^V, Y^V)&=\mathcal{G}(X^V, Y^V)=\rho_\pounds(X^V)(\rho_\pounds(Y^V)\mathcal{F}),\\
\widetilde{\mathcal{G}}(C, C)&=\mathcal{G}(C, C)=2\mathcal{F}.
\end{align}
\end{proposition}
\begin{proof}
Using (\ref{Best6}) we get
\[
\mathcal{G}(C, C)=\textbf{y}^\alpha\textbf{y}^\beta\frac{\partial^2\mathcal{F}}{\partial\textbf{y}^\alpha\partial\textbf{y}^\beta}.
\]
Since $\mathcal{F}$ is homogenous of degree 2, then we can obtain $\textbf{y}^\alpha\textbf{y}^\beta\frac{\partial^2\mathcal{F}}{\partial\textbf{y}^\alpha\partial\textbf{y}^\beta}=2\mathcal{F}$. Thus we deduce $\mathcal{G}(C, C)=2\mathcal{F}$. Using (iii) of (\ref{Liover}) and (\ref{Best5}) we deduce
\[
\widetilde{\mathcal{G}}(C, C)=\mathcal{G}(JC, JC)+\mathcal{G}(vC, vC)=\mathcal{G}(C, C)=2\mathcal{F}.
\]
Now let $X=X^\alpha e_\alpha$ and $Y=Y^\beta e_\beta$ be sections of $\stackrel{\circ}{E}$. Then we have
\begin{align*}
\mathcal{G}(X^V, Y^V)&=\mathcal{G}((X^\alpha\circ\pi)\mathcal{V}_\alpha, (Y^\beta\circ\pi)\mathcal{V}_\beta)=(X^\alpha\circ\pi)(Y^\beta\circ\pi)\frac{\partial^2\mathcal{F}}{\partial\textbf{y}^\alpha\partial\textbf{y}^\beta}\\
&=\rho_\pounds(X^V)(\rho_\pounds(Y^V)\mathcal{F}).
\end{align*}
Using (\ref{Best5}) and the above equation we can obtain
\[
\widetilde{\mathcal{G}}(X^V, Y^V)=\rho_\pounds(X^V)(\rho_\pounds(Y^V)\mathcal{F}).
\]
\end{proof}
Let $h$ be a horizontal endomorphism on $\pounds^\pi E$ and $\widetilde{\mathcal{G}}$ be a pseudo-Riemannian metric given by (\ref{Best5}). We consider
\[
\mathcal{K}_h(\widetilde{X}, \widetilde{Y})=\widetilde{\mathcal{G}}(\widetilde{X}, J\widetilde{Y})-\widetilde{\mathcal{G}}(J\widetilde{X}, \widetilde{Y}),\ \ \ \forall \widetilde{X}, \widetilde{Y}\in\Gamma(\pounds^\pi E),
\]
and we call it the \textit{K\"{a}hler form with respect to $\widetilde{\mathcal{G}}$}.
\begin{proposition}
We have $\mathcal{K}_h=i_v\omega$.
\end{proposition}
\begin{proof}
Let $\widetilde{X}, \widetilde{Y}\in\Gamma(\pounds^\pi E)$. Then we have
\begin{align*}
(i_v\omega)(\widetilde{X}, \widetilde{Y})&=\omega(v\widetilde{X}, \widetilde{Y})+\omega(\widetilde{X}, v\widetilde{Y})=\omega(v\widetilde{X}, \widetilde{Y})-\omega(v\widetilde{Y}, \widetilde{X})\\
&=\widetilde{\mathcal{G}}(\widetilde{X}, J\widetilde{Y})-\widetilde{\mathcal{G}}(J\widetilde{X}, \widetilde{Y})\\
&=\mathcal{K}_h(\widetilde{X}, \widetilde{Y}).
\end{align*}
\end{proof}
Using (\ref{pmetric}), K\"{a}hler form $\mathcal{K}_h$ has the following coordinate expression with respect to $\{\delta_\alpha, \mathcal{V}_\alpha\}$:
\[
\mathcal{K}_h=\mathcal{G}_{\alpha\beta}\delta\mathcal{V}^\alpha\wedge\mathcal{X}^\beta.
\]
\begin{defn}
Let $(E, \mathcal{F})$ be a Finsler algebroid with fundamental form $\omega$. If $\phi:E\rightarrow\mathbb{R}$ is a smooth function, then the section $\text{grad}\phi\in\Gamma(\pounds^\pi E)$ characterized by
\begin{equation}\label{gradian}
d^\pounds\phi=i_{\text{grad} \phi}\omega,
\end{equation}
is called the gradient of $\phi$.
\end{defn}

\begin{rem}
In the above definition, the nondegeneracy of $\omega$ guarantees the existence and unicity of the gradient section.
\end{rem}
If $\beta$ is a nonzero 1-form on $\pounds^\pi E$, we denote by $\beta^\sharp$ the section corresponding to $\omega$, i.e., $i_{\beta^\sharp}\omega=\beta$. Thus we can introduce the gradient of $\phi$ by $\text{grad} \phi=(d^\pounds\phi)^\sharp$.

Since $\text{grad}\phi\in\Gamma(\pounds^\pi E)$, then we can write it as follow
\begin{equation}\label{gradian0}
\text{grad}\phi=(\text{grad}\phi)^\alpha\mathcal{X}_\alpha+(\text{grad}\phi)^{\bar{\alpha}}\mathcal{V}_\alpha.
\end{equation}
Thus using (\ref{Finsler algebroid1}) and (\ref{gradian}) we get
\[
\frac{\partial\phi}{\partial \textbf{y}^\beta}=(d^\pounds\phi)(\mathcal{V}_\beta)=(i_{\text{grad} \phi}\omega)(\mathcal{V}_\beta)=-(\text{grad}\phi)^\alpha\frac{\partial^2\mathcal{F}}{\partial \textbf{y}^\alpha\partial \textbf{y}^\beta}=-(\text{grad}\phi)^\alpha\mathcal{G}_{\alpha\beta},
\]
which yields
\begin{equation}\label{gradian1}
(\text{grad}\phi)^\alpha=-\mathcal{G}^{\alpha\beta}\frac{\partial\phi}{\partial \textbf{y}^\beta},
\end{equation}
where $(\mathcal{G}^{\alpha\beta})$ is the inverse matric of $(\mathcal{G}_{\alpha\beta})$. Similarly, using (\ref{Finsler algebroid1}), (\ref{gradian}) and the above equation we obtain
\begin{align*}
(\rho^i_\beta\circ\pi)\frac{\partial\phi}{\partial \textbf{x}^i}&=(d^\pounds\phi)(\mathcal{X}_\beta)=(i_{\text{grad} \phi}\omega)(\mathcal{X}_\beta)=-\mathcal{G}^{\alpha\gamma}\frac{\partial\phi}{\partial \textbf{y}^\gamma}\Big((\rho^i_\alpha\circ\pi)\frac{\partial^2\mathcal{F}}{\partial \textbf{x}^i\partial \textbf{y}^\beta}\\
&\ \ \ -(\rho^i_\beta\circ\pi)\frac{\partial^2\mathcal{F}}{\partial \textbf{x}^i\partial \textbf{y}^\alpha}-\frac{\partial\mathcal{F}}{\partial{\textbf{y}}^\gamma}(L^\gamma_{\alpha\beta}\circ\pi)\Big)
+(\text{grad}\phi)^{\bar{\alpha}}\mathcal{G}_{\alpha\beta},
\end{align*}
which gives us
\begin{align}\label{gradian2}
(\text{grad}\phi)^{\bar{\alpha}}&=\mathcal{G}^{\alpha\beta}\Big\{(\rho^i_\beta\circ\pi)\frac{\partial\phi}{\partial \textbf{x}^i}+\mathcal{G}^{\lambda\gamma}\frac{\partial\phi}{\partial \textbf{y}^\gamma}\Big((\rho^i_\lambda\circ\pi)\frac{\partial^2\mathcal{F}}{\partial \textbf{x}^i\partial \textbf{y}^\beta}-(\rho^i_\beta\circ\pi)\frac{\partial^2\mathcal{F}}{\partial \textbf{x}^i\partial \textbf{y}^\lambda}\nonumber\\
&\ \ \ -\frac{\partial\mathcal{F}}{\partial{\textbf{y}}^\gamma}(L^\gamma_{\lambda\beta}\circ\pi)\Big)\Big\}.
\end{align}
Plugging (\ref{gradian1}) and (\ref{gradian2}) into (\ref{gradian0}) imply the following local expression for gradient
\begin{align}\label{gradian3}
\text{grad}\phi&=-\mathcal{G}^{\alpha\beta}\frac{\partial\phi}{\partial \textbf{y}^\beta}\mathcal{X}_\alpha+\mathcal{G}^{\alpha\beta}\Big\{(\rho^i_\beta\circ\pi)\frac{\partial\phi}{\partial \textbf{x}^i}+\mathcal{G}^{\lambda\gamma}\frac{\partial\phi}{\partial \textbf{y}^\gamma}\Big((\rho^i_\lambda\circ\pi)\frac{\partial^2\mathcal{F}}{\partial \textbf{x}^i\partial \textbf{y}^\beta}\nonumber\\
&\ \ \ -(\rho^i_\beta\circ\pi)\frac{\partial^2\mathcal{F}}{\partial \textbf{x}^i\partial \textbf{y}^\lambda}-\frac{\partial\mathcal{F}}{\partial{\textbf{y}}^\gamma}(L^\gamma_{\lambda\beta}\circ\pi)\Big)\Big\}\mathcal{V}_\alpha.
\end{align}
\begin{proposition}
Let $(E, \mathcal{F})$ be a Finsler algebroid and $f\in C^\infty(M)$. Then we have
\[
(i)\ \text{grad} f^\vee\in\Gamma(v\pounds^\pi E),\ (ii)\ [C, \text{grad} f^\vee]_\pounds=-\text{grad} f^\vee,
\ (iii)\ \rho_\pounds(\text{grad} f^\vee)(\mathcal{F})=f^c.
\]
\end{proposition}
\begin{proof}
Since $f^\vee=f\circ\pi$ is a function with respect to $(\textbf{x}^i)$, then we have $\frac{\partial f^\vee}{\partial \textbf{y}^\beta}=0$. Thus from (\ref{gradian3}), we deduce that $\text{grad}f^\vee$ has the following coordinate expression
\begin{equation}\label{gradian4}
\text{grad}f^\vee=\mathcal{G}^{\alpha\beta}(\rho^i_\beta\circ\pi)\frac{\partial(f\circ\pi)}{\partial \textbf{x}^i}\mathcal{V}_\alpha.
\end{equation}
Thus we have (i). The above equation and (\ref{Liouville}) give us
\[
[C, \text{grad} f^\vee]_\pounds=\Big(\textbf{y}^\alpha\frac{\partial\mathcal{G}^{\beta\gamma}}{\partial \textbf{y}^\alpha}(\rho^i_\gamma\circ\pi)\frac{\partial(f\circ\pi)}{\partial \textbf{x}^i}-\mathcal{G}^{\beta\gamma}(\rho^i_\gamma\circ\pi)\frac{\partial(f\circ\pi)}{\partial \textbf{x}^i}\Big)\mathcal{V}_\beta.
\]
But using (\ref{only God}) we can deduce $\frac{\partial\mathcal{G}_{\beta\gamma}}{\partial \textbf{y}^\alpha}=0$ and consequently $\frac{\partial\mathcal{G}^{\beta\gamma}}{\partial \textbf{y}^\alpha}=0$. Setting this equation in the above equation implies
\[
[C, \text{grad} f^\vee]_\pounds=-\mathcal{G}^{\beta\gamma}(\rho^i_\gamma\circ\pi)\frac{\partial(f\circ\pi)}{\partial \textbf{x}^i}\mathcal{V}_\beta=-\text{grad} f^\vee.
\]
Thus we have (ii). To prove (iii), we use (\ref{only God}) and (\ref{gradian4}) as follows
\begin{align*}
\rho_\pounds(\text{grad} f^\vee)(\mathcal{F})&=\mathcal{G}^{\alpha\beta}(\rho^i_\beta\circ\pi)\frac{\partial(f\circ\pi)}{\partial \textbf{x}^i}\rho_\pounds(\mathcal{V}_\alpha)(\mathcal{F})=\mathcal{G}^{\alpha\beta}(\rho^i_\beta\circ\pi)\frac{\partial(f\circ\pi)}{\partial \textbf{x}^i}\frac{\partial\mathcal{F}}{\partial \textbf{y}^\alpha}\\
&=\mathcal{G}^{\alpha\beta}(\rho^i_\beta\circ\pi)\frac{\partial(f\circ\pi)}{\partial \textbf{x}^i}\textbf{y}^\lambda\mathcal{G}_{\alpha\lambda}=\textbf{y}^\beta(\rho^i_\beta\circ\pi)\frac{\partial(f\circ\pi)}{\partial \textbf{x}^i}=f^c.
\end{align*}
\end{proof}
\subsection{Conservative endomorphism on Finsler algebroids}
\begin{defn}
Horizontal endomorphism $h$ on Finsler algebroid $(E, \mathcal{F})$ is called conservative if $d^\pounds_h\mathcal{F}=0$.
\end{defn}
Using (\ref{horizontal end}), it is easy to check that $h$ is conservative if and only if
\begin{equation}\label{cons}
(\rho^i_\alpha\circ\pi)\frac{\partial\mathcal{F}}{\partial\textbf{x}^i}+\mathcal{B}^\beta_\alpha\frac{\partial\mathcal{F}}{\partial\textbf{y}^\beta}=0.
\end{equation}
\begin{proposition}
Let $h$ be a conservative horizontal endomorphism on Finsler algebroid $(E, \mathcal{F})$. Then we have $d^\pounds_H\mathcal{F}=0$, where $H$ is the tension of $h$.
\end{proposition}
\begin{proof}
Using (\ref{tension}) we can obtain $d^\pounds_H\mathcal{F}(\mathcal{V}_\alpha)=0$ and
\begin{equation}\label{cons1}
d^\pounds_H\mathcal{F}(\mathcal{X}_\alpha)=(\mathcal{B}^\beta_\alpha-\textbf{y}^\gamma\frac{\partial \mathcal{B}^\beta_\alpha}{\partial\textbf{y}^\gamma})\frac{\partial \mathcal{F}}{\partial \textbf{y}^\beta}.
\end{equation}
Since $h$ is conservative, then differentiating (\ref{cons}) with respect to $\textbf{y}^\gamma$ we obtain
\begin{equation}\label{cons2}
(\rho^i_\alpha\circ\pi)\frac{\partial^2\mathcal{F}}{\partial \textbf{x}^i\partial \textbf{y}^\gamma}+\frac{\partial \mathcal{B}^\beta_\alpha}{\partial \textbf{y}^\gamma}\frac{\partial \mathcal{F}}{\partial \textbf{y}^\beta}+\mathcal{B}^\beta_\alpha\frac{\partial^2 F}{\partial\textbf{y}^\beta\partial\textbf{y}^\gamma}=0.
\end{equation}
Contracting the above equation by $\textbf{y}^\gamma$ and using homogeneity of $\mathcal{F}$ we get
\begin{equation}\label{cons00}
(\rho^i_\alpha\circ\pi)\frac{\partial\mathcal{F}}{\partial\textbf{x}^i}+\textbf{y}^\gamma\frac{\partial \mathcal{B}^\beta_\alpha}{\partial\textbf{y}^\gamma}\frac{\partial\mathcal{F}}{\partial\textbf{y}^\beta}=0.
\end{equation}
Setting the above equation in (\ref{cons1}) and using (\ref{cons}) we deduce $d^\pounds_H\mathcal{F}(\mathcal{X}_\alpha)=0$. Therefore $d^\pounds_H\mathcal{F}=0$.
\end{proof}
\begin{lemma}
If $\omega$ is the fundamental two-form of Finsler algebroid $(E, \mathcal{F})$
and $h$ is a conservative horizontal endomorphism on $\pounds^\pi E$, then
\[
i_h\omega=\omega+i_td^\pounds\mathcal{F}.
\]
\end{lemma}
\begin{proof}
Since $h$ is conservative, then we have (\ref{cons2}). Thus using (\ref{Finsler algebroid1}) and (\ref{cons2}) we get
\begin{align*}
(i_h\omega)(\mathcal{X_\alpha},\mathcal{X_\beta})&=(\rho^i_\alpha\circ\pi)\frac{\partial^2\mathcal{F}}{{\partial\textbf{x}^i}{\partial\textbf{y}^\beta}}
-(\rho^i_\beta\circ\pi)\frac{\partial^2\mathcal{F}}{{\partial\textbf{x}^i}{\partial\textbf{y}^\alpha}}-2\frac{\partial \mathcal{F}}{{\partial\textbf{y}^\gamma}}(L^\gamma_{\alpha\beta}\circ\pi)\\
&\ \ \ -\frac{\partial \mathcal{B}^\lambda_\alpha}{{\partial\textbf{y}^\beta}}\frac{\partial\mathcal{F}}{{\partial\textbf{y}^\lambda}}+\frac{\partial \mathcal{B}^\lambda_\beta}{{\partial\textbf{y}^\alpha}}\frac{\partial\mathcal{F}}{{\partial\textbf{y}^\lambda}}.
\end{align*}
Also, (\ref{wt}) and (\ref{wt1}) give us
\[
(i_td^\pounds\mathcal{F})(\mathcal{X_\alpha},\mathcal{X_\beta})=\frac{\partial\mathcal{F}}{{\partial\textbf{y}^\gamma}}\Big(\frac{\partial \mathcal{B}^\gamma_\beta}{{\partial\textbf{y}^\alpha}}-\frac{\partial \mathcal{B}^\gamma_\alpha}{{\partial\textbf{y}^\beta}}-(L^\gamma_{\alpha\beta}\circ\pi)\Big).
\]
Two above equations yield
\begin{align*}
(i_h\omega-i_td^\pounds\mathcal{F})(\mathcal{X_\alpha},\mathcal{X_\beta})&=(\rho^i_\alpha\circ\pi)\frac{\partial^2\mathcal{F}}{{\partial\textbf{x}^i}{\partial\textbf{y}^\beta}}
-(\rho^i_\beta\circ\pi)\frac{\partial^2\mathcal{F}}{{\partial\textbf{x}^i}{\partial\textbf{y}^\alpha}}-\frac{\partial \mathcal{F}}{{\partial\textbf{y}^\gamma}}(L^\gamma_{\alpha\beta}\circ\pi)\\
&=\omega(\mathcal{X_\alpha},\mathcal{X_\beta}).
\end{align*}
Similarly we get
\[
(i_h\omega-i_td^\pounds\mathcal{F})(\mathcal{X_\alpha},\mathcal{V_\beta})=(i_h\omega)(\mathcal{X_\alpha},\mathcal{X_\beta})=\omega(h\mathcal{X}_\alpha, \mathcal{V}_\beta)=\omega(\mathcal{X}_\alpha, \mathcal{V}_\beta),
\]
and
\[
(i_h\omega-i_td^\pounds\mathcal{F})(\mathcal{V}_\alpha,\mathcal{V}_\beta)=0=\omega(\mathcal{V}_\alpha, \mathcal{V}_\beta).
\]
\end{proof}
\begin{cor}
If $\omega$ is the fundamental two-form of Finsler algebroid $(E, \mathcal{F})$
and $h$ is a torsion free conservative horizontal endomorphism on $\pounds^\pi E$, then
\[
i_h\omega=\omega.
\]
\end{cor}
On any Finsler algebroid there is a spray $S_\circ:E\rightarrow\pounds^\pi E$, which is uniquely determined on $\stackrel{\circ}{\pounds^\pi E}$ by the formula
\begin{equation}\label{Best3}
i_{S_\circ}\omega=-d^\pounds\mathcal{F}.
\end{equation}
This spray is called the \textit{canonical spray} of the Finsler algebroid.\\

Using (\ref{pmetric}) and the above equation, the canonical spray $S_\circ$ has the coordinate expression $S_\circ=y^\alpha\mathcal{X}_\alpha+S_\circ^\alpha\mathcal{V}_\alpha$, where
\begin{equation}\label{canonical}
S_\circ^\alpha=\mathcal{G}^{\alpha\beta}\Big((\rho^i_\beta\circ\pi)\frac{\partial \mathcal{F}}{\partial \textbf{x}^i}+\textbf{y}^\gamma(\frac{\partial\mathcal{F}}{\partial \textbf{y}^\lambda}(L^\lambda_{\gamma\beta}\circ\pi)-(\rho^i_\gamma\circ\pi)\frac{\partial^2\mathcal{F}}{\partial \textbf{x}^i\partial \textbf{y}^\beta})\Big),
\end{equation}
and $(\mathcal{G}^{\alpha\beta})$ is the inverse matric of $(\mathcal{G}_{\alpha\beta})$.
\begin{proposition}\label{mainpor}
Let $S_\circ$ be the canonical spray and $h$ be a conservative horizontal endomorphism on Finsler algebroid $(E, \mathcal{F})$ with the associated semispray $S$. Then we have
\[
S-S_\circ=(d^\pounds_{i_St}\mathcal{F})^\sharp,
\]
where $i_{(d^\pounds_{i_St}\mathcal{F})^\sharp}\omega=d^\pounds_{i_St}\mathcal{F}$.
\end{proposition}
\begin{proof}
Let $h=(\mathcal{X}_\alpha+\mathcal{B}^\beta_\alpha\mathcal{V}_\beta)\otimes\mathcal{X}^\alpha$, $S=\textbf{y}^\alpha\mathcal{X}_\alpha+S^\alpha\mathcal{V}_\alpha$ and $S_\circ=\textbf{y}^\alpha\mathcal{X}_\alpha+S_\circ^\alpha\mathcal{V}_\alpha$, where $S_\circ^\alpha$ are given by (\ref{canonical}). Since $(i_{\mathcal{V}_\alpha}\omega)(\mathcal{X}_\beta)=\frac{\partial^2\mathcal{F}}{\partial\textbf{y}^\alpha\partial\textbf{y}^\beta}=\mathcal{G}_{\alpha\beta}$ and $(i_{\mathcal{V}_\alpha}\omega)(\mathcal{V}_\beta)=0$, then we have $i_{\mathcal{V}_\alpha}\omega=\mathcal{G}_{\alpha\beta}\mathcal{X}^\beta$. Therefore, using (\ref{canonical}) we get
\begin{align*}
i_{S-S_\circ}\omega=(S-S_\circ)i_{\mathcal{V}_\alpha}\omega&=(S^\alpha\frac{\partial^2\mathcal{F}}{\partial\textbf{y}^\alpha\partial\textbf{y}^\beta}
-(\rho^i_\beta\circ\pi)\frac{\partial\mathcal{F}}{\partial\textbf{x}^i}
-\textbf{y}^\gamma\frac{\partial\mathcal{F}}{\partial\textbf{y}^\lambda}(L^\lambda_{\gamma\beta}\circ\pi)\\
&\ \ \ \ +(\rho^i_\gamma\circ\pi)\textbf{y}^\gamma\frac{\partial^2\mathcal{F}}{\partial\textbf{x}^i\partial\textbf{y}^\beta})\mathcal{X}^\beta.
\end{align*}
From $S=hS_\circ$ we deduce $S^\alpha=\textbf{y}^\gamma \mathcal{B}^\alpha_\gamma$. Setting this in the above equation gives us
\begin{align*}
i_{S-S_\circ}\omega&=(\textbf{y}^\gamma \mathcal{B}^\alpha_\gamma\frac{\partial^2\mathcal{F}}{\partial\textbf{y}^\alpha\partial\textbf{y}^\beta}-(\rho^i_\beta\circ\pi)\frac{\partial\mathcal{F}}{\partial\textbf{x}^i}
-\textbf{y}^\gamma\frac{\partial\mathcal{F}}{\partial\textbf{y}^\lambda}(L^\lambda_{\gamma\beta}\circ\pi)\\
&\ \ \ +(\rho^i_\gamma\circ\pi)\textbf{y}^\gamma\frac{\partial^2\mathcal{F}}{\partial\textbf{x}^i\partial\textbf{y}^\beta})\mathcal{X}^\beta.
\end{align*}
Since $h$ is conservative, then we have (\ref{cons}), (\ref{cons2}) and (\ref{cons00}). Using these equations in the above equation and using (\ref{wt1}) we get
\begin{align*}
i_{S-S_\circ}\omega&=\textbf{y}^\alpha(\frac{\partial \mathcal{B}^\gamma_\beta}{\partial\textbf{y}^\alpha}-\frac{\partial \mathcal{B}^\gamma_\alpha}{\partial\textbf{y}^\beta}-(L^\gamma_{\alpha\beta}\circ\pi))\frac{\partial\mathcal{F}}{\partial\textbf{y}^\gamma}\mathcal{X}^\beta=\textbf{y}^\alpha t^\gamma_{\alpha\beta}\frac{\partial\mathcal{F}}{\partial\textbf{y}^\gamma}\mathcal{X}^\beta\\
&=\frac{1}{2}t^\gamma_{\alpha\beta}\frac{\partial\mathcal{F}}{\partial\textbf{y}^\gamma}(\textbf{y}^\alpha\mathcal{X}^\beta-\textbf{y}^\beta\mathcal{X}^\alpha)=\frac{1}{2}t^\gamma_{\alpha\beta}(\textbf{y}^\alpha\mathcal{X}^\beta-\textbf{y}^\beta\mathcal{X}^\alpha)\rho_\pounds(\mathcal{V}_\gamma)(\mathcal{F})\\
&=\frac{1}{2}t^\gamma_{\alpha\beta}(\textbf{y}^\alpha\mathcal{X}^\beta-\textbf{y}^\beta\mathcal{X}^\alpha)i_{\mathcal{V}_\gamma}d^\pounds\mathcal{F}
=i_{i_St}d^\pounds\mathcal{F}=d^\pounds_{i_St}\mathcal{F}=i_{(d^\pounds_{i_St}\mathcal{F})^\sharp}\omega.
\end{align*}
\end{proof}
\subsubsection{Barthel endomorphism on Finsler algebroids}
Let $S_\circ$ be the canonical spray on Finsler algebroid $(E, \mathcal{F})$. We consider
\[
h_\circ=\frac{1}{2}(1_{\Gamma(\pounds^\pi E)}+[J, S_\circ]^{F-N}_\pounds).
\]
In the coordinate expression, we can obtain
\begin{equation}
h_\circ=\Big(\mathcal{X}_\alpha+\frac{1}{2}(\frac{\partial S^\beta_\circ}{\partial\textbf{y}^\alpha}-\textbf{y}^\gamma (L^\beta_{\alpha\gamma}\circ\pi))\mathcal{V}_\beta\Big)\otimes\mathcal{X}^\alpha.
\end{equation}
From the above equation we deduce $h_\circ^2=h_\circ$ and $\ker h_\circ=v\pounds^\pi E$. Thus $h_\circ$ is a horizontal endomorphism on $\pounds^\pi E$ which is called {\it Barthel endomorphism}. Since $S_0$ is a spray on $(E, \mathcal{F})$, then we can deduce that the Barthel endomorphism is homogenous.
\begin{proposition}\label{12.16}
Let $h$ be a conservative and homogenous horizontal endomorphism and $h_\circ$ be the Barthel endomorphism on Finsle algebroid $(E, \mathcal{F})$. Then we have
\[
h=h_\circ+\frac{1}{2}i_St+\frac{1}{2}[J, (d^\pounds_{i_St}\mathcal{F})^\sharp]_\pounds.
\]
\end{proposition}
\begin{proof}
Let $S$ be the semispray associated to $h$ and $h'$ be the horizontal endomorphism generated by $S$. Then using theorem \ref{mainth} we get
\begin{align*}
h_\circ&=\frac{1}{2}(1_{\Gamma(\pounds^\pi E)}+[J, S_\circ]_\pounds)=\frac{1}{2}(1_{\Gamma(\pounds^\pi E)}+[J, S]_\pounds-[J, (d^\pounds_{i_St}\mathcal{F})^\sharp]_\pounds)\\
&=h'-\frac{1}{2}[J, (d^\pounds_{i_St}\mathcal{F})^\sharp]_\pounds=h-\frac{1}{2}i_St-\frac{1}{2}[J, (d^\pounds_{i_St}\mathcal{F})^\sharp]_\pounds.
\end{align*}
\end{proof}
\begin{theorem}
Barthel endomorphism of Finsler algebroid $(E, \mathcal{F})$ is conservative.
\end{theorem}
\begin{proof}
Using (\ref{cons}) it is sufficient to show that
\begin{equation}\label{only God1}
(\rho^i_\alpha\circ\pi)\frac{\partial\mathcal{F}}{\partial\textbf{x}^i}+\mathcal{B}^\beta_\alpha\frac{\partial\mathcal{F}}{\partial\textbf{y}^\beta}=0,
\end{equation}
where $\mathcal{B}^\beta_\alpha=\frac{1}{2}(\frac{\partial S^\beta_\circ}{\partial\textbf{y}^\alpha}-\textbf{y}^\gamma (L^\beta_{\alpha\gamma}\circ\pi))$ and $S^\beta_\circ$ are given by (\ref{canonical}). Using (\ref{only God}) we deduce
\begin{equation}\label{2eq}
(i)\ \frac{\partial\mathcal{F}}{\partial\textbf{y}^\gamma}=\textbf{y}^\lambda\mathcal{G}_{\gamma\lambda},\ \ \ \ \ \ \ (ii)\ \textbf{y}^\mu\frac{\partial^3\mathcal{F}}{\partial\textbf{y}^\gamma\partial\textbf{y}^\lambda\partial\textbf{y}^\mu}=0.
\end{equation}
Thus using (i) of (\ref{2eq}) we derive that
\begin{equation}\label{+}
\mathcal{B}^\beta_\alpha\frac{\partial\mathcal{F}}{\partial\textbf{y}^\beta}=\frac{1}{2}(\frac{\partial S^\beta_\circ}{\partial\textbf{y}^\alpha}-\textbf{y}^\gamma (L^\beta_{\alpha\gamma}\circ\pi))\textbf{y}^\mu\mathcal{G}_{\mu\beta}.
\end{equation}
Using (\ref{canonical}) we obtain
\begin{align}\label{++}
\frac{\partial S^\beta_\circ}{\partial\textbf{y}^\alpha}\textbf{y}^\mu\mathcal{G}_{\mu\beta}&=\textbf{y}^\mu\mathcal{G}_{\mu\beta}(\frac{\partial\mathcal{G}^{\beta\sigma}}{\partial\textbf{y}^\alpha})\Big((\rho^i_\sigma\circ\pi)\frac{\partial\mathcal{F}}{\partial\textbf{x}^i}+\textbf{y}^\gamma(\frac{\partial\mathcal{F}}{\partial\textbf{y}^\lambda}(L^\lambda_{\gamma\sigma}\circ\pi)\nonumber\\
&\ \ \ -(\rho^i_\gamma\circ\pi)\frac{\partial^2\mathcal{F}}{\partial\textbf{x}^i\partial\textbf{y}^\sigma})\Big)+\textbf{y}^\sigma\Big((\rho^i_\sigma\circ\pi)\frac{\partial^2\mathcal{F}}{\partial\textbf{x}^i\partial\textbf{y}^\alpha}\nonumber\\
&\ \ \ +\frac{\partial\mathcal{F}}{\partial\textbf{y}^\lambda}(L^\lambda_{\alpha\sigma}\circ\pi)-(\rho^i_\alpha\circ\pi)\frac{\partial^2\mathcal{F}}{\partial\textbf{x}^i\partial\textbf{y}^\sigma}+\textbf{y}^\gamma\frac{\partial^2\mathcal{F}}{\partial\textbf{y}^\alpha\partial\textbf{y}^\lambda}(L^\lambda_{\gamma\sigma}\circ\pi)\nonumber\\
&\ \ \ -\textbf{y}^\gamma(\rho^i_\gamma\circ\pi)\frac{\partial^3\mathcal{F}}{\partial\textbf{x}^i\partial\textbf{y}^\alpha\partial\textbf{y}^\sigma}\Big).
\end{align}
But (ii) of (\ref{2eq}) implies
\[
\textbf{y}^\mu\mathcal{G}_{\mu\beta}\frac{\partial\mathcal{G}^{\beta\sigma}}{\partial\textbf{y}^\alpha}=-\textbf{y}^\mu\mathcal{G}^{\beta\sigma}\frac{\partial\mathcal{G}_{\mu\beta}}{\partial\textbf{y}^\alpha}=-\textbf{y}^\mu\mathcal{G}^{\beta\sigma}\frac{\partial^3\mathcal{F}}{\partial\textbf{y}^\alpha\partial\textbf{y}^\mu\partial\textbf{y}^\beta}=0.
\]
Moreover, we have $\textbf{y}^\gamma\textbf{y}^\sigma(L^\lambda_{\gamma\sigma}\circ\pi)=0$, $\textbf{y}^\sigma\frac{\partial^3\mathcal{F}}{\partial\textbf{x}^i\partial\textbf{y}^\alpha\partial\textbf{y}^\sigma}=\frac{\partial^2\mathcal{F}}{\partial\textbf{x}^i\partial\textbf{y}^\alpha}$ and $\textbf{y}^\sigma\frac{\partial^2\mathcal{F}}{\partial\textbf{x}^i\partial\textbf{y}^\sigma}=2\frac{\partial^2\mathcal{F}}{\partial\textbf{x}^i}$, because $\frac{\partial^2\mathcal{F}}{\partial\textbf{x}^i\partial\textbf{y}^\alpha}$ and $\frac{\partial\mathcal{F}}{\partial\textbf{x}^i}$ are homogenous of degree 1 and 2, respectively. Therefore (\ref{++}) reduce to
\[
\frac{\partial S^\beta_\circ}{\partial\textbf{y}^\alpha}\textbf{y}^\mu\mathcal{G}_{\mu\beta}=\textbf{y}^\sigma\frac{\partial\mathcal{F}}{\partial\textbf{y}^\lambda}(L^\lambda_{\alpha\sigma}\circ\pi)-2(\rho^i_\alpha\circ\pi)\frac{\partial\mathcal{F}}{\partial\textbf{x}^i}.
\]
Setting the above equation in (\ref{+}) we deduce $\mathcal{B}^\beta_\alpha\frac{\partial\mathcal{F}}{\partial\textbf{y}^\beta}=-(\rho^i_\alpha\circ\pi)\frac{\partial\mathcal{F}}{\partial\textbf{x}^i}$. Therefore we have (\ref{only God1}).
\end{proof}
\begin{theorem}\label{12.18}
Let $h_1$ and $h_2$ be conservative horizontal endomorphisms on
Finsler algebroid $(E, \mathcal{F})$. If $h_1$ and $h_2$ have common strong torsions, then $h_1=h_2$.
\end{theorem}
\begin{proof}
We denote by $S_1$ and $S_2$ the associated semispray of $h_1$ and $h_2$, respectively and we let $T_1$ and $T_2$ be the strong torsions of $h_1$ and $h_2$, respectively. Then from hypothesis we have $d^\pounds_{h_1}\mathcal{F}=d^\pounds_{h_2}\mathcal{F}=0$ and $T_1=T_2$. Also, from the last equation in the proof of proposition \ref{mainpor}, we deduce
$i_{S_1-S_0}\omega=d^\pounds_{i_{S_1}t_1}\mathcal{F}$, $i_{S_2-S_0}\omega=d^\pounds_{i_{S_2}t_2}\mathcal{F}$ and consequently
\begin{equation}\label{comp}
i_{S_1-S_2}\omega=d^\pounds_{i_{S_1}t_1}\mathcal{F}-d^\pounds_{i_{S_2}t_2}\mathcal{F},
\end{equation}
where $t_1$ and $t_2$ are weak torsions of $h_1$ and $h_2$, respectively. From the definition of strong torsion we have
\[
d^\pounds_{i_{S_1}t_1}\mathcal{F}=d^\pounds_{T_1-H_1}\mathcal{F}=d^\pounds_{T_1}\mathcal{F},
\]
because $d^\pounds_{H_1}\mathcal{F}=0$, where $H_1$ is the tension of $h_1$. Similarly we obtain $d^\pounds_{i_{S_1}t_1}\mathcal{F}=d^\pounds_{T_2}\mathcal{F}$. Setting this equation together the above equation in (\ref{comp}) we deduce $i_{S_1-S_2}\omega=d^\pounds_{T_1}\mathcal{F}-d^\pounds_{T_2}\mathcal{F}=0$. Since $\omega$ is nondegenerate, then this equation gives us $S_1=S_2$ and consequently using theorem \ref{mainth00}, we deduce $h_1=h_2$.
\end{proof}
From the above results we understand that Barthel endomorphism is homogenous, conservative and torsion free. Moreover, since Barthel endomorphism is homogenous and torsion free, then we deduce that the it's strong torsion is zero. Also, from the above theorem we derive that if $h$ is a homogenous, conservative and torsion free horizontal endomorphism then it is coincide with Barthel endomorphism. Thus we have the following
\begin{theorem}
There exists a unique horizontal endomorphism on Finsler algebroid $(E, \mathcal{F})$ such that it is homogenous, conservative and torsion free.
\end{theorem}
\subsection{Cartan tensor on Finsler algebroids}
Here, we consider the tensor
\begin{equation}
\left\{
\begin{array}{cc}
\mathcal{C}:\Gamma(\stackrel{\circ}{\pounds^\pi E})\times\Gamma(\stackrel{\circ}{\pounds^\pi E})\rightarrow\Gamma(\stackrel{\circ}{\pounds^\pi E}),\\
(\widetilde{X}, \widetilde{Y})\rightarrow\mathcal{C}(\widetilde{X}, \widetilde{Y}),
\end{array}
\right.
\end{equation}
on Finsler algebroid $(E, \mathcal{F})$ which  satisfies in
\begin{equation}\label{Cartan}
J\circ\mathcal{C}=0,
\end{equation}
\begin{equation}\label{Cartan1}
\mathcal{G}(\mathcal{C}(\widetilde{X}, \widetilde{Y}), J\widetilde{Z})=\frac{1}{2}(\pounds_{J\widetilde{X}}J^*\mathcal{G})(\widetilde{Y}, \widetilde{Z}),
\end{equation}
where $\widetilde{X}, \widetilde{Y}, \widetilde{Z}\in\Gamma(\stackrel{\circ}{\pounds^\pi E})$ and we call it the \textit{first Cartan tensor}. Also, \textit{the lowered tensor} $\mathcal{C}_\flat$ of $\mathcal{C}$ is defined by
\begin{equation}\label{Cartan3}
\mathcal{C}_\flat(\widetilde{X}, \widetilde{Y}, \widetilde{Z})=\mathcal{G}(\mathcal{C}(\widetilde{X}, \widetilde{Y}), J\widetilde{Z}),\ \ \ \forall\widetilde{X}, \widetilde{Y}, \widetilde{Z}\in\Gamma(\stackrel{\circ}{\pounds^\pi E}).
\end{equation}
(\ref{Cartan}) told us that $\mathcal{C}(\widetilde{X}, \widetilde{Y})$ belongs to $\Gamma(v\!\!\stackrel{\circ}{\pounds^\pi E})$. Also, from (\ref{Cartan1})we deduce that $\mathcal{C}(\mathcal{X}_\alpha, \mathcal{V}_\beta)=\mathcal{C}(\mathcal{V}_\alpha, \mathcal{V}_\beta)=0$ and
\[
\mathcal{C}(\mathcal{X}_\alpha, \mathcal{X}_\beta)=\frac{1}{2}\frac{\partial\mathcal{G}_{\beta\gamma}}{\partial\textbf{y}^\alpha}\mathcal{G}^{\gamma\lambda}\mathcal{V}_\lambda
=\frac{1}{2}\frac{\partial^3\mathcal{F}}{\partial\textbf{y}^\alpha\partial\textbf{y}^\beta\partial\textbf{y}^\gamma}\mathcal{G}^{\gamma\lambda}\mathcal{V}_\lambda.
\]
Therefore the first Cartan tensor has the following coordinate expression:
\begin{equation}\label{Cartan2}
\mathcal{C}=\mathcal{C}_{\alpha\beta}^\gamma\mathcal{X}^\alpha\otimes\mathcal{X}^\beta\otimes\mathcal{V}_\gamma,
\end{equation}
where
\[
\mathcal{C}_{\alpha\beta}^\gamma=\frac{1}{2}\frac{\partial\mathcal{G}_{\beta\lambda}}{\partial\textbf{y}^\alpha}\mathcal{G}^{\gamma\lambda}
=\frac{1}{2}\frac{\partial^3\mathcal{F}}{\partial\textbf{y}^\alpha\partial\textbf{y}^\beta\partial\textbf{y}^\lambda}\mathcal{G}^{\gamma\lambda}.
\]
From (\ref{Cartan2}) and the above equation, we can deduce
\begin{proposition}
The first Cartan tensor is semibasic. Moreover, it and the lowered tensor of it, are symmetric tensors.
\end{proposition}
Using (\ref{Cartan3}) and (\ref{Cartan2}) we can obtain the following coordinate expression for the lowered tensor:
\[
\mathcal{C}_\flat=\mathcal{C}_{\alpha\beta_\gamma}\mathcal{X}^\alpha\otimes\mathcal{X}^\beta\otimes\mathcal{X}^\gamma,
\]
where
\[
\mathcal{C}_{\alpha\beta\gamma}=\mathcal{C}_{\alpha\beta}^\lambda\mathcal{G}_{\gamma\lambda}
=\frac{1}{2}\frac{\partial^3\mathcal{F}}{\partial\textbf{y}^\alpha\partial\textbf{y}^\beta\partial\textbf{y}^\gamma}.
\]
\begin{proposition}
If $S$ is a semispray on $\pounds^\pi E$, then we have $i_S\mathcal{C}=i_S\mathcal{C}_\flat=0$.
\end{proposition}
\begin{proof}
Let $\widetilde{Y}=\widetilde{Y}^\beta\mathcal{X}_\beta+\widetilde{Y}^{\bar{\beta}}\mathcal{V}_\beta$ and $\widetilde{Z}=\widetilde{Z}^\gamma\mathcal{X}_\gamma+\widetilde{Z}^{\bar{\gamma}}\mathcal{V}_\gamma$ be sections of $\stackrel{\circ}{\pounds^\pi E}$. Then using (\ref{2eq}), we have
\[
(i_S\mathcal{C}_\flat)(\widetilde{Y}, \widetilde{Z})=\mathcal{C}_\flat(S, \widetilde{Y}, \widetilde{Z})=\frac{1}{2}\textbf{y}^\alpha\widetilde{Y}^\beta\widetilde{Z}^\gamma
\frac{\partial^3\mathcal{F}}{\partial\textbf{y}^\alpha\partial\textbf{y}^\beta\partial\textbf{y}^\gamma}=0.
\]
Similarly we can prove $i_S\mathcal{C}=0$.
\end{proof}
Now we consider a horizontal endomorphism $h$ on $\pounds^\pi E$, and the prolongation $\widetilde{\mathcal{G}}$ of
the vertical metric $\mathcal{G}$ along $h$. \textit{The second Cartan tensor}
\begin{equation}
\left\{
\begin{array}{cc}
\widetilde{\mathcal{C}}:\Gamma(\stackrel{\circ}{\pounds^\pi E})\times\Gamma(\stackrel{\circ}{\pounds^\pi E})\rightarrow\Gamma(\stackrel{\circ}{\pounds^\pi E}),\\
(\widetilde{X}, \widetilde{Y})\rightarrow\widetilde{\mathcal{C}}(\widetilde{X}, \widetilde{Y}),
\end{array}
\right.
\end{equation}
(belonging to $h$) is defined by the rules
\begin{equation}\label{Cartan4}
J\circ\widetilde{\mathcal{C}}=0,
\end{equation}
\begin{equation}\label{Cartan5}
\widetilde{\mathcal{G}}(\widetilde{\mathcal{C}}(\widetilde{X}, \widetilde{Y}), J\widetilde{Z})=\frac{1}{2}(\pounds_{h\widetilde{X}}\widetilde{\mathcal{G}})(J\widetilde{Y}, J\widetilde{Z}),
\end{equation}
where $\widetilde{X}, \widetilde{Y}, \widetilde{Z}\in\Gamma(\stackrel{\circ}{\pounds^\pi E})$. Also, \textit{the lowered tensor} $\widetilde{\mathcal{C}}_\flat$ of $\widetilde{\mathcal{C}}$ is defined by
\begin{equation}\label{Cartan6}
\widetilde{\mathcal{C}}_\flat(\widetilde{X}, \widetilde{Y}, \widetilde{Z})=\widetilde{\mathcal{G}}(\widetilde{\mathcal{C}}(\widetilde{X}, \widetilde{Y}), J\widetilde{Z}), \ \ \ \forall\widetilde{X}, \widetilde{Y}, \widetilde{Z}\in\Gamma(\stackrel{\circ}{\pounds^\pi E}).
\end{equation}
Similar to the first Cartan tensor, using (\ref{Cartan4}) and (\ref{Cartan5}), we can deduce that the second Cartan tensor has the following coordinate expression:
\begin{equation}\label{Cartan7}
\widetilde{\mathcal{C}}={\widetilde{\mathcal{C}}}^\gamma_{\alpha\beta}\mathcal{X}^\alpha\otimes\mathcal{X}^\beta\otimes\mathcal{V}_\gamma,
\end{equation}
where
\begin{equation}\label{Cartan17}
{\widetilde{\mathcal{C}}}^\gamma_{\alpha\beta}=\frac{1}{2}\Big((\rho^i_\alpha\circ\pi)\frac{\partial\mathcal{G}_{\beta\mu}}{\partial\textbf{x}^i}\mathcal{G}^{\gamma\mu}
+\mathcal{B}^\lambda_\alpha\frac{\partial\mathcal{G}_{\beta\mu}}{\partial\textbf{y}^\lambda}\mathcal{G}^{\gamma\mu}
+\frac{\partial \mathcal{B}^\gamma_\alpha}{\partial\textbf{y}^\beta}
+\frac{\partial \mathcal{B}^\lambda_\alpha}{\partial\textbf{y}^\mu}\mathcal{G}^{\gamma\mu}\mathcal{G}_{\beta\lambda}\Big).
\end{equation}
From (\ref{Cartan7}), it is easy to see that the second Cartan tensor is semibasic. Moreover, (\ref{Cartan6}) and (\ref{Cartan7}) give us
\begin{equation}
\widetilde{\mathcal{C}}_\flat={\widetilde{\mathcal{C}}}_{\alpha\beta\gamma}\mathcal{X}^\alpha\otimes\mathcal{X}^\beta\otimes\mathcal{X}^\gamma,
\end{equation}
where
\begin{equation}\label{Cartan8}
{\widetilde{\mathcal{C}}}_{\alpha\beta\gamma}={\widetilde{\mathcal{C}}}^\lambda_{\alpha\beta}\mathcal{G}_{\lambda\gamma}
=\frac{1}{2}\Big((\rho^i_\alpha\circ\pi)\frac{\partial\mathcal{G}_{\beta\gamma}}{\partial\textbf{x}^i}
+\mathcal{B}^\lambda_\alpha\frac{\partial\mathcal{G}_{\beta\gamma}}{\partial\textbf{y}^\lambda}
+\frac{\partial \mathcal{B}^\lambda_\alpha}{\partial\textbf{y}^\beta}\mathcal{G}_{\lambda\gamma}
+\frac{\partial \mathcal{B}^\lambda_\alpha}{\partial\textbf{y}^\gamma}\mathcal{G}_{\beta\lambda}\Big).
\end{equation}
\begin{proposition}
Let $(E, \mathcal{F})$ be a Finsler algebroid. Then we have
\begin{align}
2\mathcal{C}_\flat(X^C, Y^C, Z^C)&=\rho_\pounds(X^V)(\rho_\pounds(Y^V)(\rho_\pounds(Z^V)\mathcal{F})),\\
2\widetilde{\mathcal{C}}_\flat(X^C, Y^C, Z^C)&=[Y^V, [X^h, Z^V]_\pounds]_\pounds+\rho_\pounds(Y^V)(\rho_\pounds(Z^V)(\rho_\pounds(X^h)\mathcal{F})).\label{2az1}
\end{align}
\end{proposition}
\begin{proof}
Let $X$, $Y$ and $Z$ be sections of $E$. Using the second part of (\ref{esi}) we get
\begin{align*}
2\mathcal{C}_\flat(X^C, Y^C, Z^C)&=2(X^\alpha\circ\pi)(Y^\beta\circ\pi)(Z^\gamma\circ\pi)\mathcal{C}_\flat(\mathcal{X}_\alpha, \mathcal{X}_\beta, \mathcal{X}_\gamma)\\
&=(X^\alpha\circ\pi)(Y^\beta\circ\pi)(Z^\gamma\circ\pi)\frac{\partial^3\mathcal{F}}{\partial\textbf{y}^\alpha\partial\textbf{y}^\beta\partial\textbf{y}^\gamma}\\
&=(X^\alpha\circ\pi)(Y^\beta\circ\pi)(Z^\gamma\circ\pi)\rho_\pounds(\mathcal{V}_\alpha)(\rho_\pounds(\mathcal{V}_\beta)
(\rho_\pounds(\mathcal{V}_\gamma)\mathcal{F}))\\
&=\rho_\pounds(X^V)(\rho_\pounds(Y^V)(\rho_\pounds(Z^V)\mathcal{F})).
\end{align*}
Now we prove (\ref{2az1}). Direct calculation gives us
\begin{align*}
&[Y^V, [X^h, Z^V]_\pounds]_\pounds+\rho_\pounds(Y^V)(\rho_\pounds(Z^V)(\rho_\pounds(X^h)\mathcal{F}))\\
&=(X^\alpha\circ\pi)(Y^\beta\circ\pi)(Z^\gamma\circ\pi)\Big((\rho^i_\alpha\circ\pi)\frac{\partial^3\mathcal{F}}{\partial\textbf{y}^\beta\partial
\textbf{y}^\gamma\partial\textbf{x}^i}
+\frac{\partial \mathcal{B}^\lambda_\alpha}{\partial\textbf{y}^\gamma}
\frac{\partial^2\mathcal{F}}{\partial\textbf{y}^\beta\partial\textbf{y}^\lambda}\\
&\ \ \ +\frac{\partial \mathcal{B}^\lambda_\alpha}{\partial\textbf{y}^\beta}
\frac{\partial^2\mathcal{F}}{\partial\textbf{y}^\gamma\partial\textbf{y}^\lambda}+\mathcal{B}^\lambda_\alpha\frac{\partial^3\mathcal{F}}{\partial\textbf{y}^\beta\partial
\textbf{y}^\gamma\partial\textbf{y}^\lambda}\Big),
\end{align*}
But using (\ref{Cartan8}), we can see that the above equation is equal to $2\widetilde{\mathcal{C}}_\flat(X^C, Y^C, Z^C)$. Thus we have (\ref{2az1}).
\end{proof}
\begin{proposition}
Let $(E, \mathcal{F})$ be a Finsler algebroid. If $h$ is a torsion free and conservative horizontal endomorphism on $\pounds^\pi E$, then the lowered second Cartan tensor is symmetric.
\end{proposition}
\begin{proof}
(\ref{Cartan8}) told us that $\widetilde{C}_{\alpha\beta\gamma}$ is symmetric with respect to last two variables. Thus it is sufficient to prove that $\widetilde{C}_{\alpha\beta\gamma}$ is symmetric with respect to first two variables. Since $h$ is conservative, then using (\ref{only God1}) and (i) of (\ref{2eq}) in (\ref{Cartan8}) we obtain
\[
\widetilde{C}_{\alpha\beta\gamma}=-\frac{1}{2}\textbf{y}^\mu\frac{\partial^2\mathcal{B}^\lambda_\alpha}{\partial\textbf{y}^\beta\partial\textbf{y}^\gamma}
\mathcal{G}_{\lambda\mu}.
\]
Since $h$ is torsion free, then using (\ref{wt1}) we have $\frac{\partial^2\mathcal{B}^\lambda_\alpha}{\partial\textbf{y}^\beta\partial\textbf{y}^\gamma}
=\frac{\partial^2\mathcal{B}^\lambda_\beta}{\partial\textbf{y}^\alpha\partial\textbf{y}^\gamma}$. Setting this equation in the above equation implies $\widetilde{C}_{\alpha\beta\gamma}=\widetilde{C}_{\beta\alpha\gamma}$.
\end{proof}
\subsection{Distinguished connections on Finsler algebroids}
\begin{theorem}
Let $(E, \mathcal{F})$ be a Finsler algebroid and $h$ be a conservative horizontal endomorphism on $\pounds^\pi E$. Then there exists a unique d-connection
$\stackrel{\text{\begin{tiny}BF\end{tiny}}}{D}$ on $(E, \mathcal{F})$ such that the $v$-mixed and $h$-mixed torsions of $\stackrel{\text{\begin{tiny}BF\end{tiny}}}{D}$ are zero.
\end{theorem}
\begin{proof}
Let there exist a d-connection $\stackrel{\text{\begin{tiny}BF\end{tiny}}}{D}$ on $(E, \mathcal{F})$ such that the $v$-mixed and $h$-mixed torsions of are zero. If we denote by $\stackrel{\text{\begin{tiny}BF\end{tiny}}}{P^1}$, the $v$-mixed torsion of $\stackrel{\text{\begin{tiny}BF\end{tiny}}}{D}$, then we have
\begin{align*}
0&=\stackrel{\text{\begin{tiny}BF\end{tiny}}}{P^1}(\widetilde{X}, F\widetilde{Y})=v\stackrel{\text{\begin{tiny}BF\end{tiny}}}{T}(h\widetilde{X}, v\widetilde{Y})=v(\stackrel{\text{\begin{tiny}BF\end{tiny}}}{D}_{h\widetilde{X}}v\widetilde{Y}
-\stackrel{\text{\begin{tiny}BF\end{tiny}}}{D}_{v\widetilde{Y}}h\widetilde{X}-[h\widetilde{X}, v\widetilde{Y}]_\pounds)\nonumber\\
&=\stackrel{\text{\begin{tiny}BF\end{tiny}}}{D}_{h\widetilde{X}}v\widetilde{Y}
-v[h\widetilde{X}, v\widetilde{Y}]_\pounds,
\end{align*}
where $\stackrel{\text{\begin{tiny}BF\end{tiny}}}{T}$ is the torsion of $\stackrel{\text{\begin{tiny}BF\end{tiny}}}{D}$. The above equation gives us
\begin{equation}\label{BF1}
\stackrel{\text{\begin{tiny}BF\end{tiny}}}{D}_{h\widetilde{X}}v\widetilde{Y}
=v[h\widetilde{X}, v\widetilde{Y}]_\pounds.
\end{equation}
Since the $h$-mixed torsion of $\stackrel{\text{\begin{tiny}BF\end{tiny}}}{D}$ is zero, then we have
\begin{align*}
0&=\stackrel{\text{\begin{tiny}BF\end{tiny}}}{B}(\widetilde{Y}, F\widetilde{X})=h\stackrel{\text{\begin{tiny}BF\end{tiny}}}{T}(h\widetilde{Y}, v\widetilde{X})=h(\stackrel{\text{\begin{tiny}BF\end{tiny}}}{D}_{h\widetilde{Y}}v\widetilde{X}
-\stackrel{\text{\begin{tiny}BF\end{tiny}}}{D}_{v\widetilde{X}}h\widetilde{Y}-[h\widetilde{Y}, v\widetilde{X}]_\pounds)\nonumber\\
&=-\stackrel{\text{\begin{tiny}BF\end{tiny}}}{D}_{v\widetilde{X}}h\widetilde{Y}
-h[h\widetilde{Y}, v\widetilde{X}]_\pounds=-\stackrel{\text{\begin{tiny}BF\end{tiny}}}{D}_{v\widetilde{X}}h\widetilde{Y}
-h[\widetilde{Y}, v\widetilde{X}]_\pounds,
\end{align*}
where $\stackrel{\text{\begin{tiny}BF\end{tiny}}}{B}$ is the $h$-mixed torsion of $\stackrel{\text{\begin{tiny}BF\end{tiny}}}{D}$. The above equation gives us
\begin{equation}\label{BF2}
\stackrel{\text{\begin{tiny}BF\end{tiny}}}{D}_{v\widetilde{X}}h\widetilde{Y}
=h[v\widetilde{X}, \widetilde{Y}]_\pounds.
\end{equation}
Since $\stackrel{\text{\begin{tiny}BF\end{tiny}}}{D}$ is d-connection, then using (\ref{BF1}), (iv) of (\ref{Jh}) and (i), (iv) of (\ref{IM0}) we get
\begin{align}\label{BF3}
\stackrel{\text{\begin{tiny}BF\end{tiny}}}{D}_{h\widetilde{X}}h\widetilde{Y}&=F\stackrel{\text{\begin{tiny}BF\end{tiny}}}{D}_{h\widetilde{X}}J\widetilde{Y}
=F\stackrel{\text{\begin{tiny}BF\end{tiny}}}{D}_{h\widetilde{X}}vJ\widetilde{Y}=Fv[h\widetilde{X}, vJ\widetilde{Y}]_\pounds\nonumber\\
&=hF[h\widetilde{X}, J\widetilde{Y}]_\pounds.
\end{align}
Since $\stackrel{\text{\begin{tiny}BF\end{tiny}}}{D}$ is d-connection, then (iii), (iv) of (\ref{IM0}), (ii), (iv) of (\ref{Jh}) and (\ref{BF2}) give us
\begin{align}\label{BF4}
\stackrel{\text{\begin{tiny}BF\end{tiny}}}{D}_{v\widetilde{X}}v\widetilde{Y}&=\stackrel{\text{\begin{tiny}BF\end{tiny}}}{D}_{v\widetilde{X}}v(v\widetilde{Y})
=\stackrel{\text{\begin{tiny}BF\end{tiny}}}{D}_{v\widetilde{X}}J(Fv\widetilde{Y})=J\stackrel{\text{\begin{tiny}BF\end{tiny}}}{D}_{v\widetilde{X}}hF\widetilde{Y}
=Jh[v\widetilde{X}, F\widetilde{Y}]_\pounds\nonumber\\
&=J[v\widetilde{X}, F\widetilde{Y}]_\pounds.
\end{align}
Relations (\ref{BF1})-(\ref{BF4}) prove the existence and uniqueness of $\stackrel{\text{\begin{tiny}BF\end{tiny}}}{D}$.
\end{proof}
Using (\ref{BF1})-(\ref{BF4}), the d-connection $\stackrel{\text{\begin{tiny}BF\end{tiny}}}{D}$ has the following coordinate expression:
\begin{equation}\label{BF5}
\left\{
\begin{array}{cc}
\stackrel{\text{\begin{tiny}BF\end{tiny}}}{D}_{\delta_\alpha}\!\!\delta_\beta=-\frac{\partial \mathcal{B}^\gamma_\alpha}{\partial y^\beta}\delta_\gamma,\ \ \ \ \ \stackrel{\text{\begin{tiny}BF\end{tiny}}}{D}_{v_\alpha}\!\!\mathcal{V}_\beta=0,\\
\stackrel{\text{\begin{tiny}BF\end{tiny}}}{D}_{\delta_\alpha}\!\!\mathcal{V}_\beta=-\frac{\partial \mathcal{B}^\gamma_\alpha}{\partial y^\beta}\mathcal{V}_\gamma,\ \ \ \ \ \stackrel{\text{\begin{tiny}BF\end{tiny}}}{D}_{\mathcal{V}_\alpha}\!\!\delta_\beta=0.
\end{array}
\right.
\end{equation}
\begin{proposition}\label{PBerwald}
Let $(E, \mathcal{F})$ be a Finsler algebroid, $h$ be a conservative horizontal endomorphism on $\pounds^\pi E$ and $\stackrel{\text{\begin{tiny}BF\end{tiny}}}{D}$ be the d-connection given by (\ref{BF5}). If $h$-deflection and $h$-horizontal torsion of $\stackrel{\text{\begin{tiny}BF\end{tiny}}}{D}$ are zero, then $h$ is the Barthel endomorphism.
\end{proposition}
\begin{proof}
It is sufficient to show that $h$ is homogenous and torsion free. Since $h$-deflection of $(\stackrel{\text{\begin{tiny}BF\end{tiny}}}{D}, h)$ is zero, then we have
\[
0=h^*(\stackrel{\text{\begin{tiny}BF\end{tiny}}}{D}\!\!C)(\delta_\alpha)=\stackrel{\text{\begin{tiny}BF\end{tiny}}}{D}_{h\delta_\alpha}\!\!\!(C)
=\stackrel{\text{\begin{tiny}BF\end{tiny}}}{D}_{h\delta_\alpha}\!\!\!(\textbf{y}^\beta\mathcal{V_\beta})
=(\mathcal{B}^\beta_\alpha-\textbf{y}^\lambda\frac{\partial \mathcal{B}^\beta_\alpha}{\partial\textbf{y}^\lambda})\mathcal{V_\beta}.
\]
The above equation shows that $h$ is homogenous. Also, since the $h$-horizontal torsion of $\stackrel{\text{\begin{tiny}BF\end{tiny}}}{D}$ is zero, then we get
\begin{align*}
0&=hT(\delta_\alpha,\delta_\beta)=h(\stackrel{\text{\begin{tiny}BF\end{tiny}}}{D}_{\delta_\alpha}\!\!\!\delta_\beta
-\stackrel{\text{\begin{tiny}BF\end{tiny}}}{D}_{\delta_\beta}\!\!\delta_\alpha-[\delta_\alpha,\delta_\beta]_\pounds)\nonumber\\
&=h\Big((\frac{\partial \mathcal{B}^\gamma_\beta}{\partial\textbf{y}^\alpha}-\frac{\partial \mathcal{B}^\gamma_\alpha}{\partial\textbf{y}^\beta}-(L^\gamma_{\alpha\beta}\circ\pi))\delta_\gamma-R^\gamma_{\alpha\beta}\mathcal{V\gamma}\Big)\\
&=t^\gamma_{\alpha\beta}\delta_\gamma.
\end{align*}
From the above equation we deduce that the weak torsion of $h$ is zero.
\end{proof}
If $h$ is the Barthel endomorphism of Finsler algebroid $(E, \mathcal{F})$, then the d-connection $\stackrel{\text{\begin{tiny}BF\end{tiny}}}{D}$ given in (\ref{BF5}) is called the Berwald connection of  $(E, \mathcal{F})$.
\begin{theorem}\label{TCartan}
Let $(E, \mathcal{F})$ be a Finsler algebroid, $h$ be a torsion free and conservative horizontal endomorphism on $\pounds^\pi E$, $\widetilde{\mathcal{G}}$ be the prolongation of $\mathcal{G}$ along $h$. Then there exists a unique d-connection
$\stackrel{\text{\begin{tiny}C\end{tiny}}}{D}$ on $(E, \mathcal{F})$ such that $\stackrel{\text{\begin{tiny}C\end{tiny}}}{D}$ is metrical, i.e., $\stackrel{\text{\begin{tiny}C\end{tiny}}}{D}\widetilde{\mathcal{G}}=0$ and the $v$-vertical and $h$-horizontal torsions of $\stackrel{\text{\begin{tiny}C\end{tiny}}}{D}$ are zero.
\end{theorem}
\begin{proof}
Let there exist a d-connection $\stackrel{\text{\begin{tiny}C\end{tiny}}}{D}$ such that $\stackrel{\text{\begin{tiny}C\end{tiny}}}{D}$ is metrical and the $v$-vertical and $h$-horizontal torsions of $\stackrel{\text{\begin{tiny}C\end{tiny}}}{D}$ are zero. Since $\stackrel{\text{\begin{tiny}C\end{tiny}}}{D}$ is metrical, then we have
\begin{align}
\rho_\pounds(\delta_\alpha)\widetilde{\mathcal{G}}(\delta_\beta, \delta_\gamma)&=\widetilde{\mathcal{G}}(\stackrel{\text{\begin{tiny}C\end{tiny}}}{D}_{\delta_\alpha}\!\!\delta_\beta, \delta_\gamma)+\widetilde{\mathcal{G}}(\delta_\beta, \stackrel{\text{\begin{tiny}C\end{tiny}}}{D}_{\delta_\alpha}\!\!\delta_\gamma),\label{3ta1}\\
\rho_\pounds(\delta_\beta)\widetilde{\mathcal{G}}(\delta_\gamma, \delta_\alpha)&=\widetilde{\mathcal{G}}(\stackrel{\text{\begin{tiny}C\end{tiny}}}{D}_{\delta_\beta}\!\!\delta_\gamma, \delta_\alpha)+\mathcal{G}(\delta_\gamma, \stackrel{\text{\begin{tiny}C\end{tiny}}}{D}_{\delta_\beta}\!\!\delta_\alpha),\\
-\rho_\pounds(\delta_\gamma)\widetilde{\mathcal{G}}(\delta_\alpha, \delta_\beta)&=-\widetilde{\mathcal{G}}(\stackrel{\text{\begin{tiny}C\end{tiny}}}{D}_{\delta_\gamma}\!\!\delta_\alpha, \delta_\beta)-\widetilde{\mathcal{G}}(\delta_\alpha, \stackrel{\text{\begin{tiny}C\end{tiny}}}{D}_{\delta_\gamma}\!\!\delta_\beta).\label{3ta3}
\end{align}
Since the $h$-horizontal torsion of $\stackrel{\text{\begin{tiny}C\end{tiny}}}{D}$ is zero, then we have
\[
\stackrel{\text{\begin{tiny}C\end{tiny}}}{D}_{\delta_\alpha}\!\!\delta_\beta-\stackrel{\text{\begin{tiny}C\end{tiny}}}{D}_{\delta_\beta}\!\!\delta_\alpha
=[\delta_\alpha, \delta_\beta]_\pounds=(L^\gamma_{\alpha\beta}\circ\pi)\delta_\gamma+R^\gamma_{\alpha\beta}\mathcal{V}_\gamma.
\]
Summing (\ref{3ta1})-(\ref{3ta3}) and using the above equation give us
\begin{align*}
\widetilde{\mathcal{G}}(\stackrel{\text{\begin{tiny}C\end{tiny}}}{D}_{\delta_\alpha}\!\!\delta_\beta, \delta_\gamma)&=\frac{1}{2}\Big((\rho^i_\alpha\circ\pi)\frac{\partial\mathcal{G}_{\beta\gamma}}{\partial\textbf{x}^i}+\mathcal{B}^\lambda_\alpha
\frac{\partial\mathcal{G}_{\beta\gamma}}{\partial\textbf{y}^\lambda}+(\rho^i_\beta\circ\pi)\frac{\partial\mathcal{G}_{\alpha\gamma}}{\partial\textbf{x}^i}
+\mathcal{B}^\lambda_\beta\frac{\partial\mathcal{G}_{\alpha\gamma}}{\partial\textbf{y}^\lambda}\\
&\ \ \ -(\rho^i_\gamma\circ\pi)\frac{\partial\mathcal{G}_{\alpha\beta}}{\partial\textbf{x}^i}-\mathcal{B}^\lambda_\gamma
\frac{\partial\mathcal{G}_{\alpha\beta}}{\partial\textbf{y}^\lambda}-(L^\lambda_{\beta\alpha}\circ\pi)\mathcal{G}_{\lambda\gamma}
-(L^\lambda_{\alpha\gamma}\circ\pi)\mathcal{G}_{\lambda\beta}\\
&\ \ \ -(L^\lambda_{\beta\gamma}\circ\pi)\mathcal{G}_{\alpha\lambda}\Big).
\end{align*}
Since $h$ is torsion free, then using (\ref{wt1}) in the above equation we get
\begin{align}\label{3ta4}
\stackrel{\text{\begin{tiny}C\end{tiny}}}{D}_{\delta_\alpha}\!\!\delta_\beta&=\frac{1}{2}\mathcal{G}^{\mu\gamma}\Big((\rho^i_\alpha\circ\pi)\frac{\partial\mathcal{G}_{\beta\gamma}}{\partial\textbf{x}^i}+\mathcal{B}^\lambda_\alpha
\frac{\partial\mathcal{G}_{\beta\gamma}}{\partial\textbf{y}^\lambda}+(\rho^i_\beta\circ\pi)\frac{\partial\mathcal{G}_{\alpha\gamma}}{\partial\textbf{x}^i}
+\mathcal{B}^\lambda_\beta\frac{\partial\mathcal{G}_{\alpha\gamma}}{\partial\textbf{y}^\lambda}\nonumber\\
&\ \ \ -(\rho^i_\gamma\circ\pi)\frac{\partial\mathcal{G}_{\alpha\beta}}{\partial\textbf{x}^i}-\mathcal{B}^\lambda_\gamma
\frac{\partial\mathcal{G}_{\alpha\beta}}{\partial\textbf{y}^\lambda}-\frac{\partial \mathcal{B}^\lambda_\alpha}{\partial\textbf{y}^\beta}\mathcal{G}_{\lambda\gamma}+
\frac{\partial \mathcal{B}^\lambda_\beta}{\partial\textbf{y}^\alpha}\mathcal{G}_{\lambda\gamma}-\frac{\partial \mathcal{B}^\lambda_\gamma}{\partial\textbf{y}^\alpha}\mathcal{G}_{\lambda\beta}\nonumber\\
&\ \ \ +\frac{\partial \mathcal{B}^\lambda_\alpha}{\partial\textbf{y}^\gamma}\mathcal{G}_{\lambda\beta}-\frac{\partial \mathcal{B}^\lambda_\gamma}{\partial\textbf{y}^\beta}\mathcal{G}_{\alpha\lambda}
+\frac{\partial \mathcal{B}^\lambda_\beta}{\partial\textbf{y}^\gamma}\mathcal{G}_{\alpha\lambda}\Big)\delta_\mu.
\end{align}
Since $h$ is conservative, then we have (\ref{cons2}). Differentiation of this equation with respect to $y$ give us
\begin{align}
&(\rho^i_\beta\circ\pi)\frac{\partial\mathcal{G}_{\gamma\alpha}}{\partial\textbf{x}^i}+\frac{\partial^2\mathcal{B}^\lambda_\beta}{\partial\textbf{y}^\gamma
\partial\textbf{y}^\alpha}\frac{\partial\mathcal{F}}{\partial\textbf{y}^\lambda}+\frac{\partial \mathcal{B}^\lambda_\beta}{\partial\textbf{y}^\gamma}\mathcal{G}_{\lambda\alpha}+\frac{\partial \mathcal{B}^\lambda_\beta}{\partial\textbf{y}^\alpha}\mathcal{G}_{\lambda\gamma}+\mathcal{B}^\lambda_\beta\frac{\partial\mathcal{G}_{\gamma\alpha}}{\partial\textbf{y}^\lambda}
=0,\label{211}\\
&(\rho^i_\gamma\circ\pi)\frac{\partial\mathcal{G}_{\beta\alpha}}{\partial\textbf{x}^i}+\frac{\partial^2\mathcal{B}^\lambda_\gamma}{\partial\textbf{y}^\beta
\partial\textbf{y}^\alpha}\frac{\partial\mathcal{F}}{\partial\textbf{y}^\lambda}+\frac{\partial \mathcal{B}^\lambda_\gamma}{\partial\textbf{y}^\beta}\mathcal{G}_{\lambda\alpha}+\frac{\partial \mathcal{B}^\lambda_\gamma}{\partial\textbf{y}^\alpha}\mathcal{G}_{\lambda\beta}+\mathcal{B}^\lambda_\gamma\frac{\partial\mathcal{G}_{\beta\alpha}}{\partial\textbf{y}^\lambda}=0.
\end{align}
Setting two above equation in (\ref{3ta4}) we obtain
\begin{equation}\label{CF}
\stackrel{\text{\begin{tiny}C\end{tiny}}}{D}_{\delta_\alpha}\!\!\delta_\beta=\frac{1}{2}\mathcal{G}^{\mu\gamma}
\Big((\rho^i_\alpha\circ\pi)\frac{\partial\mathcal{G}_{\beta\gamma}}{\partial\textbf{x}^i}+\mathcal{B}^\lambda_\alpha
\frac{\partial\mathcal{G}_{\beta\gamma}}{\partial\textbf{y}^\lambda}
-\frac{\partial \mathcal{B}^\lambda_\alpha}{\partial\textbf{y}^\beta}\mathcal{G}_{\lambda\gamma}+\frac{\partial \mathcal{B}^\lambda_\alpha}{\partial\textbf{y}^\gamma}\mathcal{G}_{\lambda\beta}\Big)\delta_\mu.
\end{equation}
Since the $v$-horizontal torsion of $\stackrel{\text{\begin{tiny}C\end{tiny}}}{D}$ is zero, then we have
\[
\stackrel{\text{\begin{tiny}C\end{tiny}}}{D}_{\mathcal{V}_\alpha}\!\!\mathcal{V}_\beta-\stackrel{\text{\begin{tiny}C\end{tiny}}}{D}_{\mathcal{V}_\beta}\!\!\mathcal{V}_\alpha
=[\mathcal{V}_\alpha, \mathcal{V}_\beta]_\pounds=0.
\]
If we replace $\delta_\alpha$, $\delta_\beta$, $\delta_\gamma$ by $\mathcal{V}_\alpha$, $\mathcal{V}_\beta$, $\mathcal{V}_\gamma$ in (\ref{3ta1})-(\ref{3ta3}), then summing these equations and using the above equation we get
\[
\widetilde{\mathcal{G}}(\stackrel{\text{\begin{tiny}C\end{tiny}}}{D}_{\mathcal{V}_\alpha}\!\!\mathcal{V}_\beta, \mathcal{V}_\gamma)=\frac{1}{2}\Big(\frac{\partial\mathcal{G}_{\beta\gamma}}{\partial\textbf{y}^\alpha}+\frac{\partial\mathcal{G}_{\alpha\gamma}}{\partial\textbf{y}^\beta}-\frac{\partial\mathcal{G}_{\alpha\beta}}{\partial\textbf{y}^\gamma}\Big)=\frac{1}{2}\frac{\partial\mathcal{G}_{\beta\gamma}}{\partial\textbf{y}^\alpha},
\]
which gives us
\begin{equation}\label{CF1}
\stackrel{\text{\begin{tiny}C\end{tiny}}}{D}_{\mathcal{V}_\alpha}\!\!\mathcal{V}_\beta=\frac{1}{2}\frac{\partial\mathcal{G}_{\beta\gamma}}{\partial\textbf{y}^\alpha}\mathcal{G}^{\gamma\mu}\mathcal{V}_\mu.
\end{equation}
Since $\stackrel{\text{\begin{tiny}C\end{tiny}}}{D}$ is d-connection, then using the above equation we obtain
\[
\stackrel{\text{\begin{tiny}C\end{tiny}}}{D}_{\mathcal{V}_\alpha}\!\!\delta_\beta=\stackrel{\text{\begin{tiny}C\end{tiny}}}{D}_{\mathcal{V}_\alpha}\!\!F\mathcal{V}_\beta=F\stackrel{\text{\begin{tiny}C\end{tiny}}}{D}_{\mathcal{V}_\alpha}\!\!\mathcal{V}_\beta=\frac{1}{2}\frac{\partial\mathcal{G}_{\beta\gamma}}{\partial\textbf{y}^\alpha}\mathcal{G}^{\gamma\mu}F(\mathcal{V}_\mu),
\]
which gives us
\begin{equation}\label{CF2}
\stackrel{\text{\begin{tiny}C\end{tiny}}}{D}_{\mathcal{V}_\alpha}\!\!\delta_\beta=\frac{1}{2}\frac{\partial\mathcal{G}_{\beta\gamma}}{\partial\textbf{y}^\alpha}\mathcal{G}^{\gamma\mu}\delta_\mu.
\end{equation}
Similarly, using (\ref{CF}) we get
\begin{align*}
\stackrel{\text{\begin{tiny}C\end{tiny}}}{D}_{\delta_\alpha}\!\!\mathcal{V}_\beta&=\stackrel{\text{\begin{tiny}C\end{tiny}}}{D}_{\delta_\alpha}\!\!J\delta_\beta=J\stackrel{\text{\begin{tiny}C\end{tiny}}}{D}_{\delta_\alpha}\!\!\delta_\beta\\
&=\frac{1}{2}\mathcal{G}^{\mu\gamma}
\Big((\rho^i_\alpha\circ\pi)\frac{\partial\mathcal{G}_{\beta\gamma}}{\partial\textbf{x}^i}+\mathcal{B}^\lambda_\alpha
\frac{\partial\mathcal{G}_{\beta\gamma}}{\partial\textbf{y}^\lambda}
-\frac{\partial \mathcal{B}^\lambda_\alpha}{\partial\textbf{y}^\beta}\mathcal{G}_{\lambda\gamma}+\frac{\partial \mathcal{B}^\lambda_\alpha}{\partial\textbf{y}^\gamma}\mathcal{G}_{\lambda\beta}\Big)J(\delta_\mu),
\end{align*}
which gives us
\begin{equation}\label{CF3}
\stackrel{\text{\begin{tiny}C\end{tiny}}}{D}_{\delta_\alpha}\!\!\mathcal{V}_\beta=\frac{1}{2}\mathcal{G}^{\mu\gamma}
\Big((\rho^i_\alpha\circ\pi)\frac{\partial\mathcal{G}_{\beta\gamma}}{\partial\textbf{x}^i}+\mathcal{B}^\lambda_\alpha
\frac{\partial\mathcal{G}_{\beta\gamma}}{\partial\textbf{y}^\lambda}
-\frac{\partial \mathcal{B}^\lambda_\alpha}{\partial\textbf{y}^\beta}\mathcal{G}_{\lambda\gamma}+\frac{\partial \mathcal{B}^\lambda_\alpha}{\partial\textbf{y}^\gamma}\mathcal{G}_{\lambda\beta}\Big)\mathcal{V}_\mu.
\end{equation}
Relations (\ref{CF})-(\ref{CF3}) prove the existence and uniqueness of $\stackrel{\text{\begin{tiny}C\end{tiny}}}{D}$.
\end{proof}
\begin{proposition}\label{PCartan}
Let $(E, \mathcal{F})$ be a Finsler algebroid, $h$ be a torsion free and conservative horizontal endomorphism on $\pounds^\pi E$ and $\stackrel{\text{\begin{tiny}C\end{tiny}}}{D}$ be the d-connection given by the above theorem. If $h$-deflection of $\stackrel{\text{\begin{tiny}C\end{tiny}}}{D}$ is zero, then $h$ is the Barthel endomorphism.
\end{proposition}
\begin{proof}
It is sufficient to show that $h$ is homogenous. Since $h$-deflection of $(\stackrel{\text{\begin{tiny}C\end{tiny}}}{D}, h)$ is zero, then using (\ref{2eq}) and (\ref{CF3}) we obtain
\begin{align*}
0&=h^*(\stackrel{\text{\begin{tiny}C\end{tiny}}}{D}\!\!C)(\delta_\alpha)=\stackrel{\text{\begin{tiny}C\end{tiny}}}{D}_{h\delta_\alpha}\!\!\!(C)
=\stackrel{\text{\begin{tiny}C\end{tiny}}}{D}_{\delta_\alpha}\!\!\!(\textbf{y}^\beta\mathcal{V_\beta})\\
&=\frac{1}{2}\mathcal{G}^{\mu\gamma}
\Big((\rho^i_\alpha\circ\pi)\frac{\partial^2\mathcal{F}}{\partial\textbf{x}^i\partial\textbf{y}^\gamma}
-\textbf{y}^\beta\frac{\partial \mathcal{B}^\lambda_\alpha}{\partial\textbf{y}^\beta}\mathcal{G}_{\lambda\gamma}+\frac{\partial \mathcal{B}^\lambda_\alpha}{\partial\textbf{y}^\gamma}\textbf{y}^\beta\mathcal{G}_{\lambda\beta}\Big)\mathcal{V}_\mu+\mathcal{B}^\mu_\alpha\mathcal{V}_\mu.
\end{align*}
Since $h$ is conservative, then we have (\ref{cons2}). Using this equation in the above equation we deduce
\[
0=\frac{1}{2}\mathcal{G}^{\mu\gamma}
\Big(-\mathcal{B}^\beta_\alpha\frac{\partial^2\mathcal{F}}{\partial\textbf{y}^\beta\partial\textbf{y}^\gamma}
-\textbf{y}^\beta\frac{\partial \mathcal{B}^\lambda_\alpha}{\partial\textbf{y}^\beta}\mathcal{G}_{\lambda\gamma}\Big)\mathcal{V}_\mu+\mathcal{B}^\mu_\alpha\mathcal{V}_\mu
=\frac{1}{2}(\mathcal{B}^\mu_\alpha-\textbf{y}^\beta\frac{\partial \mathcal{B}^\mu_\alpha}{\partial\textbf{y}^\beta})\mathcal{V}_\mu.
\]
The above equation shows that $h$ is homogenous.
\end{proof}
If $h$ is the Barthel endomorphism of Finsler algebroid $(E, \mathcal{F})$, then the d-connection $\stackrel{\text{\begin{tiny}C\end{tiny}}}{D}$ given by (\ref{CF})-(\ref{CF3}) is called the Cartan connection of  $(E, \mathcal{F})$.

Using (\ref{hh}), (\ref{hv}), (\ref{vv}) and (\ref{CF})-(\ref{CF3}) we can obtain
\begin{align*}
\stackrel{\text{\begin{tiny}C\end{tiny}}}{R}_{\alpha\beta\gamma}^{\ \ \ \ \lambda}&=-(\rho^i_\alpha\circ\pi)\frac{\partial}{\partial\textbf{x}^i}\Big(\frac{1}{2}\frac{\partial ^2\mathcal{B}^\nu_\beta}{\partial\textbf{y}^\gamma\partial\textbf{y}^\kappa}\frac{\partial \mathcal{F}}{\partial\textbf{y}^\nu}\mathcal{G}^{\lambda\kappa}+\frac{\partial \mathcal{B}^\lambda_\beta}{\partial\textbf{y}^\gamma}\Big)\\
&\ \ \ -\mathcal{B}^\mu_\alpha\frac{\partial}{\partial\textbf{y}^\mu}\Big(\frac{\partial \mathcal{B}^\lambda_\beta}{\partial\textbf{y}^\gamma}+\frac{1}{2}\frac{\partial ^2\mathcal{B}^\nu_\beta}{\partial\textbf{y}^\gamma\partial\textbf{y}^\kappa}\frac{\partial \mathcal{F}}{\partial\textbf{y}^\nu}\mathcal{G}^{\lambda\kappa}\Big)\\
&\ \ \ +(\rho^i_\beta\circ\pi)\frac{\partial}{\partial\textbf{x}^i}\Big(\frac{1}{2}\frac{\partial ^2\mathcal{B}^\nu_\alpha}{\partial\textbf{y}^\gamma\partial\textbf{y}^\kappa}\frac{\partial \mathcal{F}}{\partial\textbf{y}^\nu}\mathcal{G}^{\lambda\kappa}+\frac{\partial \mathcal{B}^\lambda_\alpha}{\partial\textbf{y}^\gamma}\Big)\\
&\ \ \ +\mathcal{B}^\mu_\beta\frac{\partial}{\partial\textbf{y}^\mu}\Big(\frac{1}{2}\frac{\partial ^2\mathcal{B}^\nu_\alpha}{\partial\textbf{y}^\gamma\partial\textbf{y}^\kappa}\frac{\partial \mathcal{F}}{\partial\textbf{y}^\nu}\mathcal{G}^{\lambda\kappa}+\frac{\partial \mathcal{B}^\lambda_\alpha}{\partial\textbf{y}^\gamma}\Big)\\
&\ \ \ +\Big(\frac{1}{2}\frac{\partial ^2\mathcal{B}^\nu_\beta}{\partial\textbf{y}^\gamma\partial\textbf{y}^\kappa}\frac{\partial \mathcal{F}}{\partial\textbf{y}^\nu}\mathcal{G}^{\mu\kappa}+\frac{\partial \mathcal{B}^\mu_\beta}{\partial\textbf{y}^\gamma}\Big)\Big(\frac{1}{2}\frac{\partial ^2\mathcal{B}^\iota_\alpha}{\partial\textbf{y}^\mu\partial\textbf{y}^\sigma}\frac{\partial \mathcal{F}}{\partial\textbf{y}^\iota}\mathcal{G}^{\lambda\sigma}+\frac{\partial \mathcal{B}^\lambda_\alpha}{\partial\textbf{y}^\mu}\Big)\\
&\ \ \ -\Big(\frac{1}{2}\frac{\partial ^2\mathcal{B}^\nu_\alpha}{\partial\textbf{y}^\gamma\partial\textbf{y}^\kappa}\frac{\partial \mathcal{F}}{\partial\textbf{y}^\nu}\mathcal{G}^{\mu\kappa}+\frac{\partial \mathcal{B}^\mu_\alpha}{\partial\textbf{y}^\gamma}\Big)\Big(\frac{1}{2}\frac{\partial ^2\mathcal{B}^\iota_\beta}{\partial\textbf{y}^\mu\partial\textbf{y}^\sigma}\frac{\partial \mathcal{F}}{\partial\textbf{y}^\iota}\mathcal{G}^{\lambda\sigma}+\frac{\partial \mathcal{B}^\lambda_\beta}{\partial\textbf{y}^\mu}\Big)\\
&\ \ \ +(L^\mu_{\alpha\beta}\circ\pi)\Big(\frac{1}{2}\frac{\partial ^2\mathcal{B}^\nu_\mu}{\partial\textbf{y}^\gamma\partial\textbf{y}^\kappa}\frac{\partial \mathcal{F}}{\partial\textbf{y}^\nu}\mathcal{G}^{\lambda\kappa}+\frac{\partial \mathcal{B}^\lambda_\mu}{\partial\textbf{y}^\gamma}\Big)-\frac{1}{2}R^\mu_{\alpha\beta}\frac{\partial \mathcal{G}_{\gamma\kappa}}{\partial\textbf{y}^\mu}\mathcal{G}^{\lambda\kappa},
\end{align*}
\begin{align*}
\stackrel{\text{\begin{tiny}C\end{tiny}}}{P}_{\alpha\beta\gamma}^{\ \ \ \lambda}&=\frac{1}{2}(\rho^i_\alpha\circ\pi)\frac{\partial}{\partial\textbf{x}^i}(\frac{\partial \mathcal{G}_{\gamma\kappa}}{\partial\textbf{y}^\beta}\mathcal{G}^{\lambda\kappa})+\frac{1}{2}\mathcal{B}^\mu_\alpha\frac{\partial}{\partial\textbf{y}^\mu}(\frac{\partial \mathcal{G}_{\gamma\kappa}}{\partial\textbf{y}^\beta}\mathcal{G}^{\lambda\kappa})-\frac{1}{2}\frac{\partial \mathcal{G}_{\gamma\kappa}}{\partial\textbf{y}^\beta}\mathcal{G}^{\mu\kappa}(\frac{\partial \mathcal{B}^\lambda_\alpha}{\partial\textbf{y}^\mu}\\
&\ \ \ +\frac{1}{2}\frac{\partial ^2\mathcal{B}^\nu_\alpha}{\partial\textbf{y}^\mu\partial\textbf{y}^\sigma}\frac{\partial \mathcal{F}}{\partial\textbf{y}^\nu}\mathcal{G}^{\lambda\sigma})+\frac{\partial}{\partial\textbf{y}^\beta}(\frac{1}{2}\frac{\partial ^2\mathcal{B}^\nu_\alpha}{\partial\textbf{y}^\gamma\partial\textbf{y}^\kappa}\frac{\partial \mathcal{F}}{\partial\textbf{y}^\nu}\mathcal{G}^{\lambda\kappa}
+\frac{\partial \mathcal{B}^\lambda_\alpha}{\partial\textbf{y}^\gamma})\\
&\ \ \ +\frac{1}{2}\frac{\partial \mathcal{G}_{\mu\kappa}}{\partial\textbf{y}^\beta}\mathcal{G}^{\lambda\kappa}(\frac{\partial \mathcal{B}^\mu_\alpha}{\partial\textbf{y}^\gamma}
+\frac{1}{2}\frac{\partial ^2\mathcal{B}^\nu_\alpha}{\partial\textbf{y}^\gamma\partial\textbf{y}^\iota}\frac{\partial \mathcal{F}}{\partial\textbf{y}^\nu}\mathcal{G}^{\mu\iota})+\frac{1}{2}
\frac{\partial \mathcal{B}^\mu_\alpha}{\partial\textbf{y}^\beta}\frac{\partial \mathcal{G}_{\gamma\kappa}}{\partial\textbf{y}^\mu}\mathcal{G}^{\lambda\kappa},
\end{align*}
\begin{align*}
\stackrel{\text{\begin{tiny}C\end{tiny}}}{S}_{\alpha\beta\gamma}^{\ \ \ \lambda}&=\frac{1}{2}(\frac{\partial \mathcal{G}_{\gamma\kappa}}{\partial\textbf{y}^\beta}\frac{\partial \mathcal{G}^{\lambda\kappa}}{\partial\textbf{y}^\alpha}-\frac{\partial \mathcal{G}_{\gamma\kappa}}{\partial\textbf{y}^\alpha}\frac{\partial \mathcal{G}^{\lambda\kappa}}{\partial\textbf{y}^\beta})+\frac{1}{4}(\mathcal{G}^{\sigma\mu}\mathcal{G}^{\lambda\kappa}\frac{\partial \mathcal{G}_{\gamma\sigma}}{\partial\textbf{y}^\beta}\frac{\partial \mathcal{G}_{\mu\kappa}}{\partial\textbf{y}^\alpha}\\
&\ \ \ -\mathcal{G}^{\mu\kappa}\mathcal{G}^{\lambda\sigma}\frac{\partial \mathcal{G}_{\gamma\kappa}}{\partial\textbf{y}^\alpha}\frac{\partial \mathcal{G}_{\mu\sigma}}{\partial\textbf{y}^\beta}).
\end{align*}
Let $\widetilde{X}$ and $\widetilde{Y}$ are sections of $\stackrel{\circ}{\pounds^\pi E}$. Then using (\ref{CF})-(\ref{CF3}) we can obtain the following formula for Cartan connection:
\[
\stackrel{\text{\begin{tiny}C\end{tiny}}}{D}_{\widetilde{X}}\widetilde{Y}=\stackrel{\text{\begin{tiny}C\end{tiny}}}{D}_{v\widetilde{X}}v\widetilde{Y}+\stackrel{\text{\begin{tiny}C\end{tiny}}}{D}_{v\widetilde{X}}h\widetilde{Y}+\stackrel{\text{\begin{tiny}C\end{tiny}}}{D}_{h\widetilde{X}}v\widetilde{Y}+\stackrel{\text{\begin{tiny}C\end{tiny}}}{D}_{h\widetilde{X}}h\widetilde{Y},
\]
where
\begin{align}
\stackrel{\text{\begin{tiny}C\end{tiny}}}{D}_{h\widetilde{X}}h\widetilde{Y}&=hF[h\widetilde{X}, J\widetilde{Y}]_\pounds+F\widetilde{\mathcal{C}}(\widetilde{X}, \widetilde{Y}),\\
\stackrel{\text{\begin{tiny}C\end{tiny}}}{D}_{v\widetilde{X}}v\widetilde{Y}&=J[v\widetilde{X}, F\widetilde{Y}]_\pounds
+\mathcal{C}(F\widetilde{X}, F\widetilde{Y}),\\
\stackrel{\text{\begin{tiny}C\end{tiny}}}{D}_{v\widetilde{X}}h\widetilde{Y}&=h[v\widetilde{X}, \widetilde{Y}]_\pounds+F\mathcal{C}(F\widetilde{X}, \widetilde{Y}),\\
\stackrel{\text{\begin{tiny}C\end{tiny}}}{D}_{h\widetilde{X}}v\widetilde{Y}&=v[h\widetilde{X}, v\widetilde{Y}]_\pounds+\widetilde{\mathcal{C}}(\widetilde{X}, F\widetilde{Y}).
\end{align}
\begin{theorem}
Let $(E, \mathcal{F})$ be a Finsler algebroid, $h$ be a torsion free and conservative horizontal endomorphism on $\pounds^\pi E$, $\widetilde{\mathcal{G}}$ be the prolongation of $\mathcal{G}$ along $h$. Then there exists a unique d-connection
$\stackrel{\text{\begin{tiny}CR\end{tiny}}}{D}$ on $(E, \mathcal{F})$ such that $\stackrel{\text{\begin{tiny}CR\end{tiny}}}{D}$ is $h$-metrical, (i.e., $\forall\widetilde{X}\in\Gamma(\stackrel{\circ}{\pounds^\pi E})$, $\stackrel{\text{\begin{tiny}CR\end{tiny}}}{D}_{h\widetilde{X}}\widetilde{\mathcal{G}}=0$), $J^*\!\!\!\stackrel{\text{\begin{tiny}CR\end{tiny}}}{D}=J^*\!\!\!\stackrel{\text{\begin{tiny}BF\end{tiny}}}{D}$ and the $h$-horizontal torsion of $\stackrel{\text{\begin{tiny}CR\end{tiny}}}{D}$ is zero. Moreover, if the $h$-deflection of $\stackrel{\text{\begin{tiny}CR\end{tiny}}}{D}$ is zero, then $h$ is the Barthel endomorphism.
\end{theorem}
\begin{proof}
Let there exists a d-connection $\stackrel{\text{\begin{tiny}CR\end{tiny}}}{D}$ on $(E, \mathcal{F})$ such that $\stackrel{\text{\begin{tiny}CR\end{tiny}}}{D}$ is $h$-metrical, $J^*\!\!\!\stackrel{\text{\begin{tiny}CR\end{tiny}}}{D}=J^*\!\!\!\stackrel{\text{\begin{tiny}B\end{tiny}}}{D}$ and the $h$-horizontal torsions of $\stackrel{\text{\begin{tiny}CR\end{tiny}}}{D}$ is zero. Since $\stackrel{\text{\begin{tiny}CR\end{tiny}}}{D}$ is $h$-metrical and the $h$-horizontal torsion of $\stackrel{\text{\begin{tiny}CR\end{tiny}}}{D}$ is zero, then similar to the proof of theorem \ref{TCartan} we can deduce
\begin{equation}\label{CR}
\stackrel{\text{\begin{tiny}CR\end{tiny}}}{D}_{\delta_\alpha}\!\!\delta_\beta=\frac{1}{2}\mathcal{G}^{\mu\gamma}
\Big((\rho^i_\alpha\circ\pi)\frac{\partial\mathcal{G}_{\beta\gamma}}{\partial\textbf{x}^i}+\mathcal{B}^\lambda_\alpha
\frac{\partial\mathcal{G}_{\beta\gamma}}{\partial\textbf{y}^\lambda}
-\frac{\partial \mathcal{B}^\lambda_\alpha}{\partial\textbf{y}^\beta}\mathcal{G}_{\lambda\gamma}+\frac{\partial \mathcal{B}^\lambda_\alpha}{\partial\textbf{y}^\gamma}\mathcal{G}_{\lambda\beta}\Big)\delta_\mu.
\end{equation}
Also, since $\stackrel{\text{\begin{tiny}CR\end{tiny}}}{D}$ is d-connection, then above equation gives us
\begin{equation}\label{CR1}
\stackrel{\text{\begin{tiny}CR\end{tiny}}}{D}_{\delta_\alpha}\!\!\mathcal{V}_\beta=\frac{1}{2}\mathcal{G}^{\mu\gamma}
\Big((\rho^i_\alpha\circ\pi)\frac{\partial\mathcal{G}_{\beta\gamma}}{\partial\textbf{x}^i}+\mathcal{B}^\lambda_\alpha
\frac{\partial\mathcal{G}_{\beta\gamma}}{\partial\textbf{y}^\lambda}
-\frac{\partial \mathcal{B}^\lambda_\alpha}{\partial\textbf{y}^\beta}\mathcal{G}_{\lambda\gamma}+\frac{\partial \mathcal{B}^\lambda_\alpha}{\partial\textbf{y}^\gamma}\mathcal{G}_{\lambda\beta}\Big)\mathcal{V}_\mu.
\end{equation}
The condition $J^*\!\!\!\stackrel{\text{\begin{tiny}CR\end{tiny}}}{D}=J^*\!\!\!\stackrel{\text{\begin{tiny}BF\end{tiny}}}{D}$ and (\ref{BF5}) gives us
\begin{equation}\label{CR2}
\stackrel{\text{\begin{tiny}CR\end{tiny}}}{D}_{\mathcal{V}_\alpha}\!\!\!\mathcal{V}_\beta=\stackrel{\text{\begin{tiny}CR\end{tiny}}}{D}_{J\delta_\alpha}\!\!\!J\delta_\beta=\stackrel{\text{\begin{tiny}BF\end{tiny}}}{D}_{J\delta_\alpha}\!\!\!J\delta_\beta=\stackrel{\text{\begin{tiny}BF\end{tiny}}}{D}_{\mathcal{V}_\alpha}\!\!\!\mathcal{V}_\beta=0,
\end{equation}
and consequently
\begin{equation}\label{CR3}
\stackrel{\text{\begin{tiny}CR\end{tiny}}}{D}_{\mathcal{V}_\alpha}\!\!\!\delta_\beta=0.
\end{equation}
Relations (\ref{CR})-(\ref{CR3}) prove the existence and uniqueness of $\stackrel{\text{\begin{tiny}CR\end{tiny}}}{D}$. The proof of the second part of assertion is similar to proposition \ref{PCartan}.
\end{proof}
If $h$ is the Barthel endomorphism of Finsler algebroid $(E, \mathcal{F})$, then the d-connection $\stackrel{\text{\begin{tiny}CR\end{tiny}}}{D}$ given by (\ref{CR})-(\ref{CR3}) is called the Chern-Rand connection of  $(E, \mathcal{F})$.

Using (\ref{hh}), (\ref{hv}), (\ref{vv}) and (\ref{CR})-(\ref{CR3}) we can get
\begin{align*}
\stackrel{\text{\begin{tiny}CR\end{tiny}}}{R}_{\alpha\beta\gamma}^{\ \ \ \lambda}&=-(\rho^i_\alpha\circ\pi)\frac{\partial}{\partial\textbf{x}^i}\Big(\frac{1}{2}\frac{\partial ^2\mathcal{B}^\nu_\beta}{\partial\textbf{y}^\gamma\partial\textbf{y}^\kappa}\frac{\partial \mathcal{F}}{\partial\textbf{y}^\nu}\mathcal{G}^{\lambda\kappa}+\frac{\partial \mathcal{B}^\lambda_\beta}{\partial\textbf{y}^\gamma}\Big)\\
&\ \ \ -\mathcal{B}^\mu_\alpha\frac{\partial}{\partial\textbf{y}^\mu}\Big(\frac{1}{2}\frac{\partial ^2\mathcal{B}^\nu_\beta}{\partial\textbf{y}^\gamma\partial\textbf{y}^\kappa}\frac{\partial \mathcal{F}}{\partial\textbf{y}^\nu}\mathcal{G}^{\lambda\kappa}+\frac{\partial \mathcal{B}^\lambda_\beta}{\partial\textbf{y}^\gamma}\Big)\\
&\ \ \ +(\rho^i_\beta\circ\pi)\frac{\partial}{\partial\textbf{x}^i}\Big(\frac{1}{2}\frac{\partial ^2\mathcal{B}^\nu_\alpha}{\partial\textbf{y}^\gamma\partial\textbf{y}^\kappa}\frac{\partial \mathcal{F}}{\partial\textbf{y}^\nu}\mathcal{G}^{\lambda\kappa}+\frac{\partial \mathcal{B}^\lambda_\alpha}{\partial\textbf{y}^\gamma}\Big)\\
&\ \ \ +\mathcal{B}^\mu_\beta\frac{\partial}{\partial\textbf{y}^\mu}\Big(\frac{1}{2}\frac{\partial ^2\mathcal{B}^\nu_\alpha}{\partial\textbf{y}^\gamma\partial\textbf{y}^\kappa}\frac{\partial \mathcal{F}}{\partial\textbf{y}^\nu}\mathcal{G}^{\lambda\kappa}+\frac{\partial \mathcal{B}^\lambda_\alpha}{\partial\textbf{y}^\gamma}\Big)\\
&\ \ \ +\Big(\frac{1}{2}\frac{\partial ^2\mathcal{B}^\nu_\beta}{\partial\textbf{y}^\gamma\partial\textbf{y}^\kappa}\frac{\partial \mathcal{F}}{\partial\textbf{y}^\nu}\mathcal{G}^{\mu\kappa}+\frac{\partial \mathcal{B}^\mu_\beta}{\partial\textbf{y}^\gamma}\Big)\Big(\frac{1}{2}\frac{\partial ^2\mathcal{B}^\iota_\alpha}{\partial\textbf{y}^\mu\partial\textbf{y}^\sigma}\frac{\partial \mathcal{F}}{\partial\textbf{y}^\iota}\mathcal{G}^{\lambda\sigma}+\frac{\partial \mathcal{B}^\lambda_\alpha}{\partial\textbf{y}^\mu}\Big)\\
&\ \ \ -\Big(\frac{1}{2}\frac{\partial ^2\mathcal{B}^\nu_\alpha}{\partial\textbf{y}^\gamma\partial\textbf{y}^\kappa}\frac{\partial \mathcal{F}}{\partial\textbf{y}^\nu}\mathcal{G}^{\mu\kappa}+\frac{\partial \mathcal{B}^\mu_\alpha}{\partial\textbf{y}^\gamma}\Big)\Big(\frac{1}{2}\frac{\partial ^2\mathcal{B}^\iota_\beta}{\partial\textbf{y}^\mu\partial\textbf{y}^\sigma}\frac{\partial \mathcal{F}}{\partial\textbf{y}^\iota}\mathcal{G}^{\lambda\sigma}+\frac{\partial \mathcal{B}^\lambda_\beta}{\partial\textbf{y}^\mu}\Big)\\
&\ \ \ +(L^\mu_{\alpha\beta}\circ\pi)\Big(\frac{1}{2}\frac{\partial ^2\mathcal{B}^\nu_\mu}{\partial\textbf{y}^\gamma\partial\textbf{y}^\kappa}\frac{\partial \mathcal{F}}{\partial\textbf{y}^\nu}\mathcal{G}^{\lambda\kappa}+\frac{\partial \mathcal{B}^\lambda_\mu}{\partial\textbf{y}^\gamma}\Big),
\end{align*}
\[
\stackrel{\text{\begin{tiny}CR\end{tiny}}}{P}_{\alpha\beta\gamma}^{\ \ \ \lambda}=\frac{\partial}{\partial\textbf{y}^\beta}(\frac{1}{2}\frac{\partial ^2\mathcal{B}^\nu_\alpha}{\partial\textbf{y}^\gamma\partial\textbf{y}^\kappa}\frac{\partial \mathcal{F}}{\partial\textbf{y}^\nu}\mathcal{G}^{\lambda\kappa}+\frac{\partial \mathcal{B}^\lambda_\alpha}{\partial\textbf{y}^\gamma}),
\]
\[
\stackrel{\text{\begin{tiny}CR\end{tiny}}}{S}_{\alpha\beta\gamma}^{\ \ \ \lambda}=0.
\]
Let $\widetilde{X}$ and $\widetilde{Y}$ are sections of $\stackrel{\circ}{\pounds^\pi E}$. Then using (\ref{CR})-(\ref{CR3}) we can obtain the following formula for Chern-Rand connection:
\[
\stackrel{\text{\begin{tiny}CR\end{tiny}}}{D}_{\widetilde{X}}\widetilde{Y}=\stackrel{\text{\begin{tiny}CR\end{tiny}}}{D}_{v\widetilde{X}}v\widetilde{Y}+\stackrel{\text{\begin{tiny}CR\end{tiny}}}{D}_{v\widetilde{X}}h\widetilde{Y}+\stackrel{\text{\begin{tiny}CR\end{tiny}}}{D}_{h\widetilde{X}}v\widetilde{Y}+\stackrel{\text{\begin{tiny}CR\end{tiny}}}{D}_{h\widetilde{X}}h\widetilde{Y},
\]
where
\begin{align}
\stackrel{\text{\begin{tiny}CR\end{tiny}}}{D}_{h\widetilde{X}}h\widetilde{Y}&=hF[h\widetilde{X}, J\widetilde{Y}]_\pounds+F\widetilde{\mathcal{C}}(\widetilde{X}, \widetilde{Y}),\\
\stackrel{\text{\begin{tiny}CR\end{tiny}}}{D}_{v\widetilde{X}}v\widetilde{Y}&=J[v\widetilde{X}, F\widetilde{Y}]_\pounds
,\\
\stackrel{\text{\begin{tiny}CR\end{tiny}}}{D}_{v\widetilde{X}}h\widetilde{Y}&=h[v\widetilde{X}, \widetilde{Y}]_\pounds,\\
\stackrel{\text{\begin{tiny}CR\end{tiny}}}{D}_{h\widetilde{X}}v\widetilde{Y}&=v[h\widetilde{X}, v\widetilde{Y}]_\pounds+\widetilde{\mathcal{C}}(\widetilde{X}, F\widetilde{Y}).
\end{align}
\begin{theorem}
Let $(E, \mathcal{F})$ be a Finsler algebroid, $h$ be a conservative horizontal endomorphism on $\pounds^\pi E$, $\widetilde{\mathcal{G}}$ be the prolongation of $\mathcal{G}$ along $h$. Then there exists a unique d-connection
$\stackrel{\text{\begin{tiny}H\end{tiny}}}{D}$ on $(E, \mathcal{F})$ such that $\stackrel{\text{\begin{tiny}H\end{tiny}}}{D}$ is $v$-metrical, (i.e., $\forall\widetilde{X}\in\Gamma(\stackrel{\circ}{\pounds^\pi E})$, $\stackrel{\text{\begin{tiny}H\end{tiny}}}{D}_{v\widetilde{X}}\widetilde{\mathcal{G}}=0$) and the $v$-vertical and $v$-mixed torsions of $\stackrel{\text{\begin{tiny}H\end{tiny}}}{D}$ are zero.
\end{theorem}
\begin{proof}
Let there exists a d-connection $\stackrel{\text{\begin{tiny}H\end{tiny}}}{D}$ on $(E, \mathcal{F})$ such that $\stackrel{\text{\begin{tiny}H\end{tiny}}}{D}$ is $v$-metrical and the $v$-vertical and $v$-mixed torsions of $\stackrel{\text{\begin{tiny}H\end{tiny}}}{D}$ are zero. Since $\stackrel{\text{\begin{tiny}H\end{tiny}}}{D}$ is $v$-metrical and the $v$-vertical torsion of $\stackrel{\text{\begin{tiny}H\end{tiny}}}{D}$ is zero, then similar to the proof of theorem \ref{TCartan} we can deduce
\begin{equation}\label{H}
\stackrel{\text{\begin{tiny}H\end{tiny}}}{D}_{\mathcal{V}_\alpha}\!\!\mathcal{V}_\beta=\frac{1}{2}\frac{\partial\mathcal{G}_{\beta\gamma}}{\partial\textbf{y}^\alpha}\mathcal{G}^{\gamma\mu}\mathcal{V}_\mu.
\end{equation}
Also, since $\stackrel{\text{\begin{tiny}H\end{tiny}}}{D}$ is d-connection, then using the above equation we obtain
\begin{equation}\label{H1}
\stackrel{\text{\begin{tiny}H\end{tiny}}}{D}_{\mathcal{V}_\alpha}\!\!\delta_\beta=\frac{1}{2}\frac{\partial\mathcal{G}_{\beta\gamma}}{\partial\textbf{y}^\alpha}\mathcal{G}^{\gamma\mu}\delta_\mu.
\end{equation}
Moreover, since the $v$-mixed torsion of $\stackrel{\text{\begin{tiny}H\end{tiny}}}{D}$ is zero, then we can obtain
\begin{equation}\label{H2}
\stackrel{\text{\begin{tiny}H\end{tiny}}}{D}_{\delta_\alpha}\mathcal{V}_\beta=v[\delta_\alpha, \mathcal{V}_\beta]_\pounds=-\frac{\partial \mathcal{B}^\mu_\alpha}{\partial\textbf{y}^\beta}\mathcal{V}_\mu,
\end{equation}
and consequently
\begin{equation}\label{H3}
\stackrel{\text{\begin{tiny}H\end{tiny}}}{D}_{\delta_\alpha}\delta_\beta=-\frac{\partial \mathcal{B}^\mu_\alpha}{\partial\textbf{y}^\beta}\delta_\mu,
\end{equation}
because $\stackrel{\text{\begin{tiny}H\end{tiny}}}{D}$ is d-connection. Relations (\ref{H})-(\ref{H3}) prove the existence and uniqueness of $\stackrel{\text{\begin{tiny}H\end{tiny}}}{D}$.
\end{proof}
\begin{proposition}
Let $(E, \mathcal{F})$ be a Finsler algebroid, $h$ be a conservative horizontal endomorphism on $\pounds^\pi E$ and $\stackrel{\text{\begin{tiny}H\end{tiny}}}{D}$ be the d-connection given by the above theorem. If $h$-horizontal torsion and $h$-deflection of $\stackrel{\text{\begin{tiny}H\end{tiny}}}{D}$ are zero, then $h$ is the Barthel endomorphism.
\end{proposition}
\begin{proof}
The proof is similar to proof of proposition \ref{PBerwald}.
\end{proof}
If $h$ is the Barthel endomorphism of Finsler algebroid $(E, \mathcal{F})$, then the d-connection $\stackrel{\text{\begin{tiny}H\end{tiny}}}{D}$ given by (\ref{H})-(\ref{H3}) is called the Hashiguchi connection of  $(E, \mathcal{F})$.

using (\ref{hh}), (\ref{hv}), (\ref{vv}) and (\ref{H})-(\ref{H3}) we can obtain
\begin{align*}
\stackrel{\text{\begin{tiny}H\end{tiny}}}{R}_{\alpha\beta\gamma}^{\ \ \ \lambda}&=-(\rho^i_\alpha\circ\pi)\frac{\partial^2\mathcal{B}^\lambda_\beta}{\partial\textbf{x}^i\partial\textbf{y}^\gamma}
-\mathcal{B}^\mu_\alpha\frac{\partial^2\mathcal{B}^\lambda_\beta}{\partial\textbf{y}^\mu\partial\textbf{y}^\gamma}+(\rho^i_\beta\circ\pi)\frac{\partial^2\mathcal{B}^\lambda_\alpha}
{\partial\textbf{x}^i\partial\textbf{y}^\gamma}
+\mathcal{B}^\mu_\beta\frac{\partial^2\mathcal{B}^\lambda_\alpha}{\partial\textbf{y}^\mu\partial\textbf{y}^\gamma}\\
&\ \ \ +\frac{\partial \mathcal{B}^\lambda_\alpha}{\partial\textbf{y}^\mu}\frac{\partial \mathcal{B}^\mu_\beta}{\partial\textbf{y}^\gamma}-\frac{\partial \mathcal{B}^\mu_\alpha}{\partial\textbf{y}^\gamma}\frac{\partial \mathcal{B}^\lambda_\beta}{\partial\textbf{y}^\mu}+(L^\mu_{\alpha\beta}\circ\pi)\frac{\partial \mathcal{B}^\lambda_\mu}{\partial\textbf{y}^\gamma}-\frac{1}{2}R^\mu_{\alpha\beta}\frac{\partial g_{\gamma\kappa}}{\partial\textbf{y}^\mu}g^{\lambda\kappa},
\end{align*}
\begin{align*}
\stackrel{\text{\begin{tiny}H\end{tiny}}}{P}_{\alpha\beta\gamma}^{\ \ \ \lambda}&=(\rho^i_\alpha\circ\pi)\frac{\partial}{\partial\textbf{x}^i}(\frac{\partial \mathcal{G}_{\gamma\kappa}}{\partial\textbf{y}^\beta}\mathcal{G}^{\lambda\kappa})+\frac{1}{2}\mathcal{B}^\mu_\alpha\frac{\partial}{\partial\textbf{y}^\mu}(\frac{\partial \mathcal{G}_{\gamma\kappa}}{\partial\textbf{y}^\beta}\mathcal{G}^{\lambda\kappa})-\frac{1}{2}\frac{\partial \mathcal{G}_{\gamma\kappa}}{\partial\textbf{y}^\beta}\mathcal{G}^{\lambda\kappa}\frac{\partial \mathcal{B}^\lambda_\alpha}{\partial\textbf{y}^\mu}\\
&\ \ \ +\frac{\partial^2\mathcal{B}^\lambda_\alpha}{\partial\textbf{y}^\beta\partial\textbf{y}^\gamma}+\frac{1}{2}\frac{\partial \mathcal{G}_{\mu\kappa}}{\partial\textbf{y}^\beta}\mathcal{G}^{\lambda\kappa}\frac{\partial \mathcal{B}^\mu_\alpha}{\partial\textbf{y}^\gamma}+\frac{1}{2}\frac{\partial \mathcal{G}_{\gamma\kappa}}{\partial\textbf{y}^\mu}\mathcal{G}^{\lambda\kappa}\frac{\partial \mathcal{B}^\mu_\alpha}{\partial\textbf{y}^\beta},
\end{align*}
\begin{align*}
\stackrel{\text{\begin{tiny}H\end{tiny}}}{S}_{\alpha\beta\gamma}^{\ \ \ \lambda}&=\frac{1}{2}(\frac{\partial \mathcal{G}_{\gamma\kappa}}{\partial\textbf{y}^\beta}\frac{\partial \mathcal{G}^{\lambda\kappa}}{\partial\textbf{y}^\alpha}-\frac{\partial \mathcal{G}_{\gamma\kappa}}{\partial\textbf{y}^\alpha}\frac{\partial \mathcal{G}^{\lambda\kappa}}{\partial\textbf{y}^\beta})+\frac{1}{4}(\mathcal{G}^{\sigma\mu}\mathcal{G}^{\lambda\kappa}\frac{\partial \mathcal{G}_{\gamma\sigma}}{\partial\textbf{y}^\beta}\frac{\partial \mathcal{G}_{\mu\kappa}}{\partial\textbf{y}^\alpha}\\
&\ \ \ -\mathcal{G}^{\mu\kappa}\mathcal{G}^{\lambda\sigma}\frac{\partial \mathcal{G}_{\gamma\kappa}}{\partial\textbf{y}^\alpha}\frac{\partial \mathcal{G}_{\mu\sigma}}{\partial\textbf{y}^\beta}).
\end{align*}
Let $\widetilde{X}$ and $\widetilde{Y}$ are sections of $\stackrel{\circ}{\pounds^\pi E}$. Then using (\ref{H})-(\ref{H3}) we can obtain the following formula for Hashiguchi connection:
\[
\stackrel{\text{\begin{tiny}H\end{tiny}}}{D}_{\widetilde{X}}\widetilde{Y}=\stackrel{\text{\begin{tiny}H\end{tiny}}}{D}_{v\widetilde{X}}v\widetilde{Y}+\stackrel{\text{\begin{tiny}H\end{tiny}}}{D}_{v\widetilde{X}}h\widetilde{Y}+\stackrel{\text{\begin{tiny}H\end{tiny}}}{D}_{h\widetilde{X}}v\widetilde{Y}+\stackrel{\text{\begin{tiny}H\end{tiny}}}{D}_{h\widetilde{X}}h\widetilde{Y},
\]
where
\begin{align}
\stackrel{\text{\begin{tiny}H\end{tiny}}}{D}_{h\widetilde{X}}h\widetilde{Y}&=hF[h\widetilde{X}, J\widetilde{Y}]_\pounds,\\
\stackrel{\text{\begin{tiny}H\end{tiny}}}{D}_{v\widetilde{X}}v\widetilde{Y}&=J[v\widetilde{X}, F\widetilde{Y}]_\pounds+\mathcal{C}(F\widetilde{X}, F\widetilde{Y}),\\
\stackrel{\text{\begin{tiny}H\end{tiny}}}{D}_{v\widetilde{X}}h\widetilde{Y}&=h[v\widetilde{X}, \widetilde{Y}]_\pounds+F\mathcal{C}(F\widetilde{X}, \widetilde{Y}),\\
\stackrel{\text{\begin{tiny}H\end{tiny}}}{D}_{h\widetilde{X}}v\widetilde{Y}&=v[h\widetilde{X}, v\widetilde{Y}]_\pounds.
\end{align}
\begin{theorem}
Let $h$ be the Barthel endomorphism on Finsler algebroid $(E, \mathcal{F})$. Then the Cartan connection $\stackrel{\text{\begin{tiny}C\end{tiny}}}{D}$

(i)\ is Chern-Rund connection if $J^*\!\!\stackrel{\text{\begin{tiny}C\end{tiny}}}{D}=J^*\!\!\stackrel{\text{\begin{tiny}BF\end{tiny}}}{D}$,

(ii)\  is Hashiguchi connection if $h^*\!\!\stackrel{\text{\begin{tiny}C\end{tiny}}}{D}=\stackrel{\text{\begin{tiny}BF\end{tiny}}}{D}$,

(iii)\ is Berwald connection if it is the Chern-Rund connection and the Hashiguchi connection at the same time.
\end{theorem}
\section{Generalized Berwald Lie algebroids}
In this section, $h$-basic distinguished connections are introduced on Finsler algeboids. We have more attention to Ichijy\={o} connection. Dealing with conservative endomorphisms, generalized Berwald Lie algebroid is introduced and Wagner-Ichijy\={o} connection as a special case is studied notably.

\begin{defn}
Let $(D, h)$ be a d-connection on $\pounds^\pi E$. We call it a $h$-basic d-connection if there is a linear connection $\nabla$ on $E$ such that
\begin{equation}\label{GB}
D_{X^h}Y^V=(\nabla_XY)^V,\ \ \ \forall X, Y\in \Gamma(E).
\end{equation}
\end{defn}
Linear connection $\nabla$ in the above definition is called the \textit{basic connection} belongs to $(D, h)$. Note that the base connection of a $h$-basic d-connection is unique.
\begin{proposition}\label{suitable}
Let $(D, h)$ be a d-connection on $\pounds^\pi E$ and $(\widetilde{D}, h)$ be the d-connection associated to $(D, h)$ given by (\ref{m2}). Then $(D, h)$ is $h$-basic if and only if the mixed curvature of $(\widetilde{D}, h)$ is zero.
\end{proposition}
\begin{proof}
Let $(D, h)$ be a $h$-basic d-connection on $\pounds^\pi E$ and $\{e_\alpha\}$ be a basis of $\Gamma(E)$. Since $\nabla_{e_\alpha}e_\beta$ belongs to $\Gamma(E)$, then we can write it as $\nabla_{e_\alpha}e_\beta=\Gamma^\gamma_{\alpha\beta}e_\gamma$, where $\Gamma^\gamma_{\alpha\beta}$ are local functions on $M$. From (\ref{GB}) we can deduce
\[
D_{\delta_\alpha}\mathcal{V}_\beta=D_{e_\alpha^h}e_\beta^V=(\nabla_{e_\alpha}e_\beta)^V=(\Gamma^\gamma_{\alpha\beta}\circ\pi)\mathcal{V}_\gamma.
\]
Thus we have $F^\gamma_{\alpha\beta}=(\Gamma^\gamma_{\alpha\beta}\circ\pi)$, where $F^\gamma_{\alpha\beta}$ are the local coefficients of $D_{\delta_\alpha}\mathcal{V}_\beta$. Since $F^\gamma_{\alpha\beta}$ are functions with respect to $(x^h)$, only, then using the first part of (\ref{niaz}) we get $P_{\alpha\beta\gamma}^{\ \ \ \ \lambda}=0$, i.e., the mixed curvature of $(\widetilde{D}, h)$ is zero.

Conversely, let the mixed curvature of $(\widetilde{D}, h)$ be zero. Then from (\ref{niaz}) we derive that $F^\gamma_{\alpha\beta}$ are functions with respect to $(x^h)$, only. Now we define $\nabla:\Gamma(E)\times\Gamma(E)\rightarrow\Gamma(E)$ by $(\nabla_XY)^V:=D_{X^h}Y^V$. Since the vertical lift of a section of $E$ is unique, then $\nabla$ is well defined. Also, we have
\[
(\nabla_X(fY))^V=D_{X^h}(fY)^V=D_{X^h}(f^vY^V)=\rho_\pounds(X^h)(f^v)Y^V+f^vD_{X^h}Y^V,
\]
where $X, Y\in\Gamma(E)$ and $f\in C^\infty(M)$. It is easy to check that $\rho_\pounds(X^h)(f^v)=(\rho(X)f)^v$. Setting this in the above equation we get
\begin{align*}
(\nabla_X(fY))^V&=(\rho(X)f)^vY^V+f^vD_{X^h}Y^V=(\rho(X)f)^vY^V+f^v(\nabla_XY)^V\\
&=(\rho(X)(f) Y+f\nabla_XY)^V,
\end{align*}
which gives us $\nabla_X(fY)=\rho(X)(f) Y+f\nabla_XY$, because the vertical lift is unique. Similarly we can obtain $\nabla_{fX+gY}Z=f\nabla_XZ+g\nabla_YZ$ and $\nabla_X(Y+Z)=\nabla_XZ+\nabla_YZ$, for all $X, Y, Z\in\Gamma(E)$ and $f, g\in C^\infty(M)$. Thus $\nabla$ is a linear connection on $E$ and consequently $(D, h)$ is $h$-basic.
\end{proof}
Let $\nabla$ be a linear connection on $E$, $\{e_\alpha\}$ be a basis of $\Gamma(E)$ and $\nabla_{e_\alpha}e_\beta=\Gamma^\gamma_{\alpha\beta}e_\gamma$. Then
\begin{equation}\label{delta}
h_\nabla=(\mathcal{X}_\alpha-y^\gamma(\Gamma^\beta_{\alpha\gamma}\circ\pi)\mathcal{V}_\beta)\otimes\mathcal{X}^\alpha,
\end{equation}
is a horizontal endomorphism on $\pounds^\pi E$. Indeed we have
\[
(\nabla_XY)^V=[X^{h_\nabla}, Y^V]_\pounds,\ \ \ \forall X, Y\in\Gamma(E).
\]
We call $h_\nabla$ given by (\ref{delta}) the \textit{horizontal endomorphism generated by $\nabla$}. It is easy to see that $h_\nabla$ is homogenous and it is smooth on the whole $\pounds^\pi E$.
\begin{lemma}
Let $\nabla$ be a linear connection on $E$ and $h_\nabla$ be the horizontal endomorphism generated by $\nabla$. If $K_{\alpha\beta\gamma}^{\ \ \ \ \lambda}$ and $R^\lambda_{\alpha\beta}$ are the local coefficients of the curvature tensors of $\nabla$ and $h_\nabla$, respectively, then we have $\textbf{y}^\gamma(K_{\alpha\beta\gamma}^{\ \ \ \ \lambda}\circ\pi)=-R^\lambda_{\alpha\beta}$.
\end{lemma}
\begin{proof}
Setting $\mathcal{B}^\lambda_\alpha=-\textbf{y}^\gamma(\Gamma^\lambda_{\alpha\gamma}\circ\pi)$ in (\ref{curv0}) give us
\begin{align*}
R^\lambda_{\alpha\beta}&=\textbf{y}^\gamma\Big((\rho^i_\beta\circ\pi)\frac{\partial(\Gamma^\lambda_{\alpha\gamma}\circ\pi)}{\partial\textbf{x}^i}
+(\Gamma^\mu_{\alpha\gamma}\circ\pi)(\Gamma^\lambda_{\beta\mu}\circ\pi)
-(\rho^i_\alpha\circ\pi)\frac{\partial(\Gamma^\lambda_{\beta\gamma}\circ\pi)}{\partial\textbf{x}^i}
\\
&\ \ \ -(\Gamma^\mu_{\beta\gamma}\circ\pi)(\Gamma^\lambda_{\alpha\mu}\circ\pi)-(L^\mu_{\beta\alpha}\circ\pi)(\Gamma^\lambda_{\mu\gamma}\circ\pi)\Big)
=-\textbf{y}^\gamma(K_{\alpha\beta\gamma}^{\ \ \ \ \lambda}\circ\pi).
\end{align*}
\end{proof}
\begin{cor}\label{very im}
Let $\nabla$ be a linear connection on $E$ and $h_\nabla$ be the horizontal endomorphism generated by $\nabla$. Then the curvature of $\nabla$ is zero if and only if the curvature of $h_\nabla$ vanishes.
\end{cor}
\begin{proposition}\label{Bestth}
Let $(D, h)$ be a $h$-basic d-connection with base connection $\nabla$ and $h_\nabla$ be the horizontal endomorphism generated by $\nabla$. Then
\[
D_{X^h}C=X^h-X^{h_\nabla}.
\]
\end{proposition}
\begin{proof}
Let $F^\gamma_{\alpha\beta}$ be the local coefficients of $D_{\delta_\alpha}\mathcal{V}_\beta$ and $\Gamma^\gamma_{\alpha\beta}$ be the local coefficients of $\nabla_{e_\alpha}e_\beta$. In the above proposition we show that $F^\gamma_{\alpha\beta}=(\Gamma^\gamma_{\alpha\beta}\circ\pi)$, because $(D, h)$ be a $h$-basic d-connection with base connection $\nabla$. Thus we can obtain
\begin{equation}\label{ok3}
D_{X^h}C=(X^\alpha\circ\pi)(\mathcal{B}^\beta_\alpha+\textbf{y}^\gamma F^\beta_{\alpha\gamma})\mathcal{V}_\beta=(X^\alpha\circ\pi)(\mathcal{B}^\beta_\alpha+\textbf{y}^\gamma(\Gamma^\beta_{\alpha\gamma}\circ\pi))\mathcal{V}_\beta,
\end{equation}
where $X=X^\alpha e_\alpha$, $X^h=(X^\alpha\circ\pi)\delta_\alpha$ and $h$ is given by (\ref{horizontal end}). (\ref{horizontal end}), (\ref{delta}) and the above equation give us
\[
X^h-X^{h_\nabla}=(X^\alpha\circ\pi)(\mathcal{X}_\alpha+\mathcal{B}^\beta_\alpha\mathcal{V}_\beta)-(X^\alpha\circ\pi)(\mathcal{X}_\alpha
-\textbf{y}^\gamma(\Gamma^\beta_{\alpha\gamma}\circ\pi)\mathcal{V}_\beta)=D_{X^h}C.
\]
\end{proof}
\begin{cor}\label{ok}
Let $(D, h)$ be a $h$-basic d-connection with base connection $\nabla$ and $h_\nabla$ be the horizontal endomorphism generated by $\nabla$. Then $h_\nabla$ coincides with $h$ if and only if the $h$-deflection of $(D, h)$ is zero.
\end{cor}
\begin{proof}
If $h_\nabla=h$, then from the above proposition we have $D_{X^h}C=0$ and in particular $D_{\delta_\alpha}C=D_{e_\alpha^h}C=0$. Therefore we deduce $h^*(DC)(\delta_\alpha)=D_{\delta_\alpha}C=0$, i.e., the $h$-deflection of $(D, h)$ vanishes. Conversely, if the $h$-deflection of $(D, h)$ is zero, then we deduce $D_{\delta_\alpha}C=0$ and consequently $D_{X^h}C=0$. Thus from the above proposition we derive that $X^h=X^{h_\nabla}$ and consequently $h=h_\nabla$.
\end{proof}
\begin{cor}\label{ok1}
Let $(D, h)$ be a $h$-basic d-connection with base connection $\nabla$ and $h_\nabla$ be the horizontal endomorphism generated by $\nabla$. If the $h$-deflection of $(D, h)$ is zero, then we have
\[
(i)\ D_{h\widetilde{X}}v\widetilde{Y}=v[h\widetilde{X}, v\widetilde{Y}]_\pounds,\ \ \ \ \ \ \ (ii)\ D_{h\widetilde{X}}h\widetilde{Y}=hF[h\widetilde{X}, J\widetilde{Y}]_\pounds,
\]
where $\widetilde{X}, \widetilde{Y}\in\Gamma(\pounds^\pi E)$.
\end{cor}
\begin{proof}
Let $\widetilde{X}=\widetilde{X}^\alpha\delta_\alpha+\widetilde{X}^{\bar{\alpha}}\mathcal{V}_\alpha$ and  $\widetilde{Y}=\widetilde{Y}^\beta\delta_\beta+\widetilde{Y}^{\bar{\beta}}\mathcal{V}_\beta$ be sections of $\pounds^\pi E$. Since the $h$-deflection of $(D, h)$ is zero, then using the above corollary we have $h=h_\nabla$ and consequently $\mathcal{B}^\beta_\alpha=-\textbf{y}^\lambda(\Gamma^\beta_{\alpha\lambda}\circ\pi)$. Thus we can obtain
\begin{align*}
v[h\widetilde{X}, v\widetilde{Y}]_\pounds&=\widetilde{X}^\alpha\Big((\rho^i_\alpha\circ\pi)\frac{\partial \widetilde{Y}^{\bar\beta}}{\partial\textbf{x}^i}-\textbf{y}^\lambda(\Gamma^{\gamma}_{\lambda\alpha}\circ\pi)\frac{\partial \widetilde{Y}^{\bar\beta}}{\partial\textbf{y}^\gamma}\Big)\mathcal{V}_\beta+X^\alpha Y^{\bar\beta}(\Gamma^\gamma_{\alpha\beta}\circ\pi)\mathcal{V}_\gamma\\
&=D_{h\widetilde{X}}v\widetilde{Y},
\end{align*}
because $F^\gamma_{\alpha\beta}=(\Gamma^\gamma_{\alpha\beta}\circ\pi)$, where $F^\gamma_{\alpha\beta}$ are the local coefficients of $D_{\delta_\alpha}\mathcal{V}_\beta$. Therefore we have (i). Similar to (\ref{BF3}), using (i) we can deduce (ii).
\end{proof}
\begin{proposition}\label{ok2}
Let $(D, h)$ be a $h$-basic d-connection with base connection $\nabla$ and $h$ be a homogenous horizontal endomorphism. Then $h$-deflection of $(D, h)$ is zero if and only if the $v$-mixed torsion of $D$ is zero.
\end{proposition}
\begin{proof}
Using (\ref{TD4}) we have
\begin{equation}\label{ok4}
P^1(\delta_\alpha, \delta_\beta)=D_{\delta_\alpha}\mathcal{V}_\beta-v[\delta_\alpha, \mathcal{V}_\beta]_\pounds=((\Gamma_{\alpha\beta}^\gamma\circ\pi)+\frac{\partial \mathcal{B}^\gamma_\alpha}{\partial\textbf{y}^\beta})\mathcal{V}_\gamma.
\end{equation}
Thus $P^1=0$ if and only if $\frac{\partial \mathcal{B}^\gamma_\alpha}{\partial\textbf{y}^\beta}=-(\Gamma_{\alpha\beta}^\gamma\circ\pi)$. But since $h$ is homogenous, then we have $\textbf{y}^\beta\frac{\partial \mathcal{B}^\gamma_\alpha}{\partial\textbf{y}^\beta}=\mathcal{B}^\gamma_\alpha$. Thus we can deduce $P^1=0$ if and only if $\mathcal{B}^\gamma_\alpha=-y^\beta(\Gamma_{\alpha\beta}^\gamma\circ\pi)$ (this equation gives us $h=h_\nabla$). Therefore the vanishing of $P^1$ is equivalent to vanishing of  the $h$-deflection of $(D, h)$.
\end{proof}
\begin{rem}
Since in corollaries \ref{ok}, \ref{ok1} and proposition \ref{ok2} we work on the vanishing of $h$-deflection of $(D, h)$, then we have $h=h_\nabla$. But $h_\nabla$ is smooth on the whole $\pounds^\pi E$. Therefore the horizontal endomorphism $h$ should be smooth on the whole $\pounds^\pi E$.
\end{rem}
\begin{proposition}
Let $(D, h)$ be a $h$-basic d-connection with base connection $\nabla$ and the horizontal endomorphism $h$ be smooth on whole $\pounds^\pi E$. Then the $h$-deflection of $(D, h)$ coincides with the tension of $h$ if and only if the $v$-mixed torsion of $D$ is zero.
\end{proposition}
\begin{proof}
Let the $v$-mixed torsion of $D$ be zero. Then from (\ref{ok4}) we can deduce $(\Gamma_{\alpha\beta}^\gamma\circ\pi)=-\frac{\partial \mathcal{B}^\gamma_\alpha}{\partial\textbf{y}^\beta}$. But from (\ref{ok3}) we have
\[
D_{\delta_\alpha}C=(\mathcal{B}^\beta_\alpha+\textbf{y}^\gamma(\Gamma^\beta_{\alpha\gamma}\circ\pi))\mathcal{V}_\beta.
\]
Setting $(\Gamma_{\alpha\beta}^\gamma\circ\pi)=-\frac{\partial \mathcal{B}^\gamma_\alpha}{\partial\textbf{y}^\beta}$ in the above equation and using (\ref{tension}) we obtain
\[
h^*(DC)(\delta_\alpha)=D_{\delta_\alpha}C=(\mathcal{B}^\beta_\alpha-\textbf{y}^\gamma\frac{\partial \mathcal{B}^\alpha_\alpha}{\partial\textbf{y}^\gamma})\mathcal{V}_\beta=H(\delta_\alpha).
\]
Conversely, if $h^*(DC)=H$ and $h$ is smooth on whole $\pounds^\pi E$ then using (\ref{tension}) and (\ref{ok3}) we obtain $\frac{\partial \mathcal{B}^\gamma_\alpha}{\partial\textbf{y}^\beta}=-(\Gamma_{\alpha\beta}^\gamma\circ\pi)$. Setting this equation in (\ref{ok4}) we deduce $P^1=0$.
\end{proof}
\begin{theorem}
Let $(D, h)$ be a $h$-basic d-connection on Finsler algebroid $(E, \mathcal{F})$ and the first Cartan tensor be nonzero on $(E, \mathcal{F})$. Then $(D, h)$ is $h$-metrical if and only if $h$ is conservative and the $h$-deflection of $(D, h)$ is zero.
\end{theorem}
\begin{proof}
Let $(D, h)$ be $h$-metrical. Then we get
\begin{align*}
X^h\mathcal{F}&=\frac{1}{2}X^h(\widetilde{\mathcal{G}}(C, C))=\widetilde{\mathcal{G}}(C, D_{X^h}C)=(X^\alpha\circ\pi)(\mathcal{B}^\beta_\alpha+\textbf{y}^\gamma(\Gamma^\beta_{\alpha\gamma}\circ\pi))\frac{\partial \mathcal{F}}{\partial \textbf{y}^\beta}\\
&=(D_{X^h}C)\mathcal{F}.
\end{align*}
But from proposition \ref{Bestth} we have
\[
(D_{X^h}C)\mathcal{F}=X^h\mathcal{F}-X^{h_\nabla}\mathcal{F}.
\]
Two above equations gives us $X^{h_\nabla}\mathcal{F}=0$ and consequently $d_{h_\nabla}\mathcal{F}=0$. Thus $h_\nabla$ is conservative. Direct calculation we obtain
\begin{align}
&X^{h_\nabla}\widetilde{\mathcal{G}}(\mathcal{V}_\beta, \mathcal{V}_\lambda)-\widetilde{\mathcal{G}}(D_{X^h}\mathcal{V}_\beta, \mathcal{V}_\lambda)-\widetilde{\mathcal{G}}(\mathcal{V}_\beta, D_{X^h}\mathcal{V}_\lambda)=(X^\alpha\circ\pi)\Big((\rho^i_\alpha\circ\pi)\frac{\partial\mathcal{G}_{\beta\lambda}}{\partial \textbf{x}^i}\nonumber\\
&\ \ \ -y^\gamma(\Gamma^\mu_{\alpha\gamma}\circ\pi)\frac{\partial\mathcal{G}_{\beta\lambda}}{\partial \textbf{y}^\mu}-(\Gamma^\gamma_{\alpha\beta}\circ\pi)\mathcal{G}_{\gamma\lambda}-(\Gamma^\gamma_{\alpha\lambda}\circ\pi)\mathcal{G}_{\gamma\beta}\Big).
\end{align}
Since $h_\nabla$ is conservative, then we have (\ref{211}) with $\mathcal{B}^\lambda_\beta=-\textbf{y}^\mu(\Gamma^\lambda_{\mu\beta}\circ\pi)$. Setting this equation in (\ref{211}) we can see that the right side of the above equation vanishes. Therefore we have
\begin{equation}\label{Niaz1}
X^{h_\nabla}\widetilde{\mathcal{G}}(\mathcal{V}_\beta, \mathcal{V}_\lambda)=\widetilde{\mathcal{G}}(D_{X^h}\mathcal{V}_\beta, \mathcal{V}_\lambda)+\widetilde{\mathcal{G}}(\mathcal{V}_\beta, D_{X^h}\mathcal{V}_\lambda).
\end{equation}
In other hand, since $(D, h)$ is $h$-metrical, then we have
\[
X^h\widetilde{\mathcal{G}}(\mathcal{V}_\beta, \mathcal{V}_\lambda)=\widetilde{\mathcal{G}}(D_{X^h}\mathcal{V}_\beta, \mathcal{V}_\lambda)+\widetilde{\mathcal{G}}(\mathcal{V}_\beta, D_{X^h}\mathcal{V}_\lambda).
\]
Two above equations give us
\begin{equation}\label{Niaz}
(X^{h_\nabla}-X^h)\widetilde{\mathcal{G}}(\mathcal{V}_\beta, \mathcal{V}_\lambda)=0.
\end{equation}
For vertical metric $\mathcal{G}$, using (\ref{Cartan1}) we can obtain
\begin{align*}
\mathcal{G}(C(\delta_\alpha, \delta_\beta), X^h-X^{h_\nabla})&=(X^\sigma\circ\pi)(\mathcal{B}^\lambda_\sigma+\textbf{y}^\gamma(\Gamma^\lambda_{\sigma\gamma}\circ\pi))\mathcal{G}(C(\delta_\alpha, \delta_\beta), \mathcal{V}_\lambda)\\
&=\frac{1}{2}(X^\sigma\circ\pi)(\mathcal{B}^\lambda_\sigma+\textbf{y}^\gamma(\Gamma^\lambda_{\sigma\gamma}\circ\pi))(\pounds_{\mathcal{V}_\alpha}J^*\mathcal{G})(\delta_\beta, \delta_\lambda)\\
&=\frac{1}{2}(X^\sigma\circ\pi)(\mathcal{B}^\lambda_\sigma+\textbf{y}^\gamma(\Gamma^\lambda_{\sigma\gamma}\circ\pi))(\mathcal{V}_\alpha\mathcal{G}(\mathcal{V}_\beta, \mathcal{V}_\lambda)).
\end{align*}
Since $\mathcal{V}_\alpha\mathcal{G}(\mathcal{V}_\beta, \mathcal{V}_\lambda)=\mathcal{V}_\lambda\mathcal{G}(\mathcal{V}_\alpha, \mathcal{V}_\beta)$, then using this equation in the above equation and using (\ref{Niaz}) we deduce
\begin{align*}
\mathcal{G}(C(\delta_\alpha, \delta_\beta), X^h-X^{h_\nabla})&=\frac{1}{2}(X^\sigma\circ\pi)(\mathcal{B}^\lambda_\sigma+y^\gamma(\Gamma^\lambda_{\sigma\gamma}\circ\pi))(\mathcal{V}_\lambda\mathcal{G}(\mathcal{V}_\alpha, \mathcal{V}_\beta))\\
&=(X^{h_\nabla}-X^h)\mathcal{G}(\mathcal{V}_\alpha, \mathcal{V}_\beta)\\
&=(X^{h_\nabla}-X^h)\widetilde{\mathcal{G}}(\mathcal{V}_\alpha, \mathcal{V}_\beta)=0.
\end{align*}
From the above equation we can derive that $\mathcal{G}(C(\widetilde{Y}, \widetilde{Z}), X^h-X^{h_\nabla})=0$, for all $\widetilde{Y}, \widetilde{Z}\in\Gamma(\pounds^\pi E)$. Since $\mathcal{G}$ is non-degenerated, then this equation gives us $X^h-X^{h_\nabla}=0$ or $X^h=X^{h_\nabla}$ and consequently $h=h_\nabla$. Thus $h$ is conservative and using corollary \ref{ok}, the $h$-deflection of $(D, h)$ vanishes. Conversely, let $h$ be the conservative horizontal endomorphism and the $h$-deflection of $(D, h)$ be zero. Then from corollary \ref{ok}, $h$ coincides with $h_\nabla$ and so $h_\nabla$ is conservative. Therefore we have (\ref{Niaz1}) which gives us
\[
(D_{X^h}\widetilde{\mathcal{G}})(\mathcal{V}_\alpha, \mathcal{V}_\beta)=(X^h-X^{h_\nabla})\widetilde{\mathcal{G}}(\mathcal{V}_\alpha, \mathcal{V}_\beta)=0.
\]
Also, since $h=h_\nabla$ and $h$ is conservative, then using (ii) of corollary \ref{ok1} and (\ref{211}) we obtain
\[
X^h\widetilde{\mathcal{G}}(\mathcal{V}_\beta, \mathcal{V}_\lambda)-\widetilde{\mathcal{G}}(D_{X^h}\mathcal{V}_\beta, \mathcal{V}_\lambda)-\widetilde{\mathcal{G}}(\mathcal{V}_\beta, D_{X^h}\mathcal{V}_\lambda)=0,
\]
which gives us $(D_{X^h}\widetilde{\mathcal{G}})(\delta_\alpha, \delta_\beta)=0$. Therefore we can deduce $D_{h\widetilde{X}}\widetilde{\mathcal{G}}=0$, for all $\widetilde{X}\in\pounds^\pi E$.
\end{proof}
\subsection{Ichijy\={o} connection}
\begin{theorem}
Let $(E, \mathcal{F})$ be a Finsler algebroid, $\nabla$ be a linear connection on $E$, $h_\nabla$ be the horizontal endomorphism generated by $\nabla$ and $\mathcal{G}$ be the prolongation of vertical metric along $h_\nabla$. Then there is a unique d-connection $(\stackrel{{\begin{tiny}\nabla\end{tiny}}}{D}, h_\nabla)$ on $(E, \mathcal{F})$ such that

(i)\ $\stackrel{{\begin{tiny}\nabla\end{tiny}}}{D}$ is $v$-metrical,

(ii)\ The $v$-vertical torsion of $\stackrel{{\begin{tiny}\nabla\end{tiny}}}{D}$ is zero,

(iii)\ The $h$-deflection of $(\stackrel{{\begin{tiny}\nabla\end{tiny}}}{D}, h_\nabla)$ is zero,

(iv)\ The mixed curvature of $(\widetilde{\stackrel{{\begin{tiny}\nabla\end{tiny}}}{D}}, h_\nabla)$ is zero,\\
where $(\widetilde{\stackrel{{\begin{tiny}\nabla\end{tiny}}}{D}}, h_\nabla)$ the d-connection associated to $(\stackrel{{\begin{tiny}\nabla\end{tiny}}}{D}, h_\nabla)$ given by (\ref{m2}).
\end{theorem}
\begin{proof}
Let there exists a d-connection $\stackrel{{\begin{tiny}\nabla\end{tiny}}}{D}$ on $(E, \mathcal{F})$ such that $\stackrel{{\begin{tiny}\nabla\end{tiny}}}{D}$ satisfies in (i)-(iv). Since $\stackrel{{\begin{tiny}\nabla\end{tiny}}}{D}$ is $v$-metrical and the $v$-vertical of $\stackrel{{\begin{tiny}\nabla\end{tiny}}}{D}$ is zero, then similar to the proof of theorem \ref{TCartan} we can deduce
\begin{equation}\label{I}
\stackrel{{\begin{tiny}\nabla\end{tiny}}}{D}_{\mathcal{V}_\alpha}\!\!\mathcal{V}_\beta
=\frac{1}{2}\frac{\partial\mathcal{G}_{\beta\gamma}}{\partial\textbf{y}^\alpha}\mathcal{G}^{\gamma\mu}\mathcal{V}_\mu=\mathcal{C}_{\alpha\beta}^\mu\mathcal{V}_\mu.
\end{equation}
Also, since $\stackrel{{\begin{tiny}\nabla\end{tiny}}}{D}$ is d-connection, then using the above equation we obtain
\begin{equation}\label{I1}
\stackrel{{\begin{tiny}\nabla\end{tiny}}}{D}_{\mathcal{V}_\alpha}\!\!\delta_\beta
=\frac{1}{2}\frac{\partial\mathcal{G}_{\beta\gamma}}{\partial\textbf{y}^\alpha}\mathcal{G}^{\gamma\mu}\delta_\mu=\mathcal{C}_{\alpha\beta}^\mu\delta_\mu.
\end{equation}
Condition (iv) together proposition \ref{suitable} told us that $(\stackrel{{\begin{tiny}\nabla\end{tiny}}}{D}, h_\nabla)$ is $h$-basic. Thus there exists a unique linear connection $\widetilde{\nabla}$ on $E$ such that $(\widetilde{\nabla}_XY)^V=\stackrel{{\begin{tiny}\nabla\end{tiny}}}{D}_{X^{h_\nabla}}Y^V$. But using (iii) and corollary \ref{ok} we deduce that $\widetilde{\nabla}$ coincides with $\nabla$. Thus we have
\[
\stackrel{{\begin{tiny}\nabla\end{tiny}}}{D}_{X^{h_\nabla}}Y^V=(\nabla_XY)^V,\ \ \ \forall X, Y\in\Gamma(E).
\]
From the above equation we obtain
\begin{equation}\label{I2}
\stackrel{{\begin{tiny}\nabla\end{tiny}}}{D}_{\delta_\alpha}\mathcal{V}_\beta=(\Gamma^\gamma_{\alpha\beta}\circ\pi)\mathcal{V}_\gamma,
\end{equation}
where $\Gamma^\gamma_{\alpha\beta}$ are the local coefficients of linear connection $\nabla$. The above equation gives us
\begin{equation}\label{I3}
\stackrel{{\begin{tiny}\nabla\end{tiny}}}{D}_{\delta_\alpha}\delta_\beta=(\Gamma^\gamma_{\alpha\beta}\circ\pi)\delta_\gamma,
\end{equation}
because $\stackrel{{\begin{tiny}\nabla\end{tiny}}}{D}$ is a d-connection. Relations (\ref{I})-(\ref{I3}) prove the existence and uniqueness of $\stackrel{{\begin{tiny}\nabla\end{tiny}}}{D}$
\end{proof}
We call d-connection $(\stackrel{{\begin{tiny}\nabla\end{tiny}}}{D}, h_\nabla)$ introduced in the above theorem, \textit{Ichijy\={o} connection induced by $\nabla$} on Finsler algebroid $(E, \mathcal{F})$.

Let $\widetilde{X}$ and $\widetilde{Y}$ are sections of $\stackrel{\circ}{\pounds^\pi E}$. Then using (\ref{I})-(\ref{I3}) we can obtain the following formula for Ichijy\={o} connection:
\[
\stackrel{{\begin{tiny}\nabla\end{tiny}}}{D}_{\widetilde{X}}\widetilde{Y}=\stackrel{{\begin{tiny}\nabla\end{tiny}}}{D}_{v_\nabla\widetilde{X}}v_\nabla\widetilde{Y}
+\stackrel{{\begin{tiny}\nabla\end{tiny}}}{D}_{v_\nabla\widetilde{X}}h_\nabla\widetilde{Y}+\stackrel{{\begin{tiny}\nabla\end{tiny}}}{D}_{h_\nabla\widetilde{X}}v_\nabla\widetilde{Y}
+\stackrel{{\begin{tiny}\nabla\end{tiny}}}{D}_{h_\nabla\widetilde{X}}h_\nabla\widetilde{Y},
\]
where
\begin{align}
\stackrel{{\begin{tiny}\nabla\end{tiny}}}{D}_{h_\nabla\widetilde{X}}h_\nabla\widetilde{Y}&=h_\nabla F_\nabla[h_\nabla\widetilde{X}, J\widetilde{Y}]_\pounds,\\
\stackrel{{\begin{tiny}\nabla\end{tiny}}}{D}_{v_\nabla\widetilde{X}}v_\nabla\widetilde{Y}&=J[v_\nabla\widetilde{X}, F_\nabla\widetilde{Y}]_\pounds
+\mathcal{C}(F_\nabla\widetilde{X}, F_\nabla\widetilde{Y}),\\
\stackrel{{\begin{tiny}\nabla\end{tiny}}}{D}_{v_\nabla\widetilde{X}}h_\nabla\widetilde{Y}&=h_\nabla[v_\nabla\widetilde{X}, \widetilde{Y}]_\pounds+F_\nabla\mathcal{C}(F_\nabla\widetilde{X}, \widetilde{Y}),\\
\stackrel{{\begin{tiny}\nabla\end{tiny}}}{D}_{h_\nabla\widetilde{X}}v_\nabla\widetilde{Y}&=v_\nabla[h_\nabla\widetilde{X}, v_\nabla\widetilde{Y}]_\pounds.
\end{align}
Using the above equations we can obtain
\begin{align}
\stackrel{{\begin{tiny}\nabla\end{tiny}}}{D}_{X^{h_\nabla}}Y^{h_\nabla}&=\Big((X^\alpha\circ\pi)(\rho^i_\alpha\circ\pi)\frac{\partial (Y^\gamma\circ\pi)}{\partial\textbf{x}^i}+(X^\alpha\circ\pi)(Y^\beta\circ\pi)(\Gamma^\gamma_{\alpha\beta}\circ\pi)\Big)\delta_\gamma\nonumber\\
&=(\nabla_XY)^{h_\nabla},\\
\stackrel{{\begin{tiny}\nabla\end{tiny}}}{D}_{X^V}Y^V&=(X^\alpha\circ\pi)(Y^\beta\circ\pi)\mathcal{C}^\mu_{\alpha\beta}
\mathcal{V}_\mu=\mathcal{C}(X^{h_\nabla}, Y^{h_\nabla}),\\
\stackrel{{\begin{tiny}\nabla\end{tiny}}}{D}_{X^V}Y^{h_\nabla}&=(X^\alpha\circ\pi)(Y^\beta\circ\pi)\mathcal{C}^\mu_{\alpha\beta}
\delta_\mu=F\mathcal{C}(X^{h_\nabla}, Y^{h_\nabla}),\\
\stackrel{{\begin{tiny}\nabla\end{tiny}}}{D}_{X^{h_\nabla}}Y^V&=\Big((X^\alpha\circ\pi)(\rho^i_\alpha\circ\pi)\frac{\partial (Y^\gamma\circ\pi)}{\partial\textbf{x}^i}+(X^\alpha\circ\pi)(Y^\beta\circ\pi)(\Gamma^\gamma_{\alpha\beta}\circ\pi)\Big)\mathcal{V}_\gamma\nonumber\\
&=(\nabla_XY)^V,
\end{align}
where $X, Y\in\Gamma(E)$.
\begin{proposition}
Let $(E, \mathcal{F})$ be a Finsler algebroid, $\nabla$ a linear connection on $E$ and $(\stackrel{{\begin{tiny}\nabla\end{tiny}}}{D}, h_\nabla)$ be the d-connection induced by $\nabla$. Then
\[
(\stackrel{{\begin{tiny}\nabla\end{tiny}}}{D}_{J\widetilde{X}}\mathcal{C})(\widetilde{Y}, \widetilde{Z})=(\stackrel{{\begin{tiny}\nabla\end{tiny}}}{D}_{J\widetilde{Y}}\mathcal{C})(\widetilde{X}, \widetilde{Z}),\ \ \ \forall \widetilde{X}, \widetilde{Y}, \widetilde{Z}\in\Gamma(\stackrel{\circ}{\pounds^\pi E}).
\]
\end{proposition}
\begin{proof}
It is sufficient to show that $(\stackrel{{\begin{tiny}\nabla\end{tiny}}}{D}_{\mathcal{V}_\alpha}\mathcal{C})(\delta_\beta, \delta_\gamma)=(\stackrel{{\begin{tiny}\nabla\end{tiny}}}{D}_{\mathcal{V}_\beta}\mathcal{C})(\delta_\alpha, \delta_\gamma)$. Using the local expression of the first Cartan tensor and (\ref{I1}) we get
\begin{align}\label{jomee}
(\stackrel{{\begin{tiny}\nabla\end{tiny}}}{D}_{\mathcal{V}_\alpha}\mathcal{C})(\delta_\beta, \delta_\gamma)&=\frac{1}{4}\Big(2\frac{\partial^2\mathcal{G}_{\gamma\lambda}}{\partial\textbf{y}^\alpha\partial\textbf{y}^\beta}\mathcal{G}^{\lambda\mu}
+2\frac{\partial\mathcal{G}_{\gamma\lambda}}{\partial\textbf{y}^\beta}\frac{\partial\mathcal{G}^{\lambda\mu}}{\partial\textbf{y}^\alpha}
+\frac{\partial\mathcal{G}_{\gamma\lambda}}{\partial\textbf{y}^\beta}\mathcal{G}^{\lambda\sigma}
\frac{\partial\mathcal{G}_{\nu\sigma}}{\partial\textbf{y}^\alpha}\mathcal{G}^{\mu\nu}\nonumber\\
&\ \ \ -\frac{\partial\mathcal{G}_{\beta\lambda}}{\partial\textbf{y}^\alpha}\mathcal{G}^{\nu\lambda}
\frac{\partial\mathcal{G}_{\gamma\sigma}}{\partial\textbf{y}^\nu}\mathcal{G}^{\sigma\mu}
-\frac{\partial\mathcal{G}_{\gamma\lambda}}{\partial\textbf{y}^\alpha}\mathcal{G}^{\nu\lambda}
\frac{\partial\mathcal{G}_{\beta\sigma}}{\partial\textbf{y}^\nu}\mathcal{G}^{\sigma\mu}\Big)\mathcal{V}_\mu.
\end{align}
Since $\mathcal{G}^{\lambda\sigma}\frac{\partial\mathcal{G}_{\nu\sigma}}{\partial\textbf{y}^\alpha}=-\mathcal{G}_{\nu\sigma}\frac{\partial\mathcal{G}^{\lambda\sigma}}{\partial\textbf{y}^\alpha}$, then we get
\[
\frac{\partial\mathcal{G}_{\gamma\lambda}}{\partial\textbf{y}^\beta}\mathcal{G}^{\lambda\sigma}\frac{\partial\mathcal{G}_{\nu\sigma}}{\partial\textbf{y}^\alpha}\mathcal{G}^{\mu\nu}=-\frac{\partial\mathcal{G}_{\gamma\lambda}}{\partial\textbf{y}^\beta}\frac{\partial\mathcal{G}^{\lambda\mu}}{\partial\textbf{y}^\alpha}.
\]
Similary we obtain
\[
\frac{\partial\mathcal{G}_{\gamma\lambda}}{\partial\textbf{y}^\alpha}\mathcal{G}^{\nu\lambda}\frac{\partial\mathcal{G}_{\beta\sigma}}{\partial\textbf{y}^\nu}\mathcal{G}^{\sigma\mu}=\frac{\partial\mathcal{G}_{\gamma\lambda}}{\partial\textbf{y}^\alpha}\mathcal{G}^{\nu\lambda}\frac{\partial\mathcal{G}_{\nu\sigma}}{\partial\textbf{y}^\beta}\mathcal{G}^{\sigma\mu}=-\frac{\partial\mathcal{G}_{\gamma\lambda}}{\partial\textbf{y}^\alpha}\frac{\partial\mathcal{G}^{\lambda\mu}}{\partial\textbf{y}^\beta}.
\]
Setting two above equation in (\ref{jomee}) give us
\begin{align*}
(\stackrel{{\begin{tiny}\nabla\end{tiny}}}{D}_{\mathcal{V}_\alpha}\mathcal{C})(\delta_\beta, \delta_\gamma)&=\frac{1}{4}\Big(2\frac{\partial^2\mathcal{G}_{\gamma\lambda}}{\partial\textbf{y}^\alpha\partial\textbf{y}^\beta}\mathcal{G}^{\lambda\mu}
+\frac{\partial\mathcal{G}_{\gamma\lambda}}{\partial\textbf{y}^\beta}\frac{\partial\mathcal{G}^{\lambda\mu}}{\partial\textbf{y}^\alpha}
-\frac{\partial\mathcal{G}_{\beta\lambda}}{\partial\textbf{y}^\alpha}\mathcal{G}^{\nu\lambda}
\frac{\partial\mathcal{G}_{\gamma\sigma}}{\partial\textbf{y}^\nu}\mathcal{G}^{\sigma\mu}\\
&\ \ \ +\frac{\partial\mathcal{G}_{\gamma\lambda}}{\partial\textbf{y}^\alpha}\frac{\partial\mathcal{G}^{\lambda\mu}}{\partial\textbf{y}^\beta}\Big)\mathcal{V}_\mu.
\end{align*}
Similarly we can obtain
\begin{align*}
(\stackrel{{\begin{tiny}\nabla\end{tiny}}}{D}_{\mathcal{V}_\beta}\mathcal{C})(\delta_\alpha, \delta_\gamma)&=\frac{1}{4}\Big(2\frac{\partial^2\mathcal{G}_{\gamma\lambda}}{\partial\textbf{y}^\beta\partial\textbf{y}^\alpha}\mathcal{G}^{\lambda\mu}
+\frac{\partial\mathcal{G}_{\gamma\lambda}}{\partial\textbf{y}^\alpha}\frac{\partial\mathcal{G}^{\lambda\mu}}{\partial\textbf{y}^\beta}
-\frac{\partial\mathcal{G}_{\alpha\lambda}}{\partial\textbf{y}^\beta}\mathcal{G}^{\nu\lambda}
\frac{\partial\mathcal{G}_{\gamma\sigma}}{\partial\textbf{y}^\nu}\mathcal{G}^{\sigma\mu}\\
&\ \ \ +\frac{\partial\mathcal{G}_{\gamma\lambda}}{\partial\textbf{y}^\beta}\frac{\partial\mathcal{G}^{\lambda\mu}}{\partial\textbf{y}^\alpha}\Big)\mathcal{V}_\mu.
\end{align*}
Two above equation show that $(\stackrel{{\begin{tiny}\nabla\end{tiny}}}{D}_{\mathcal{V}_\alpha}\mathcal{C})(\delta_\beta, \delta_\gamma)=(\stackrel{{\begin{tiny}\nabla\end{tiny}}}{D}_{\mathcal{V}_\beta}\mathcal{C})(\delta_\alpha, \delta_\gamma)$.
\end{proof}
Let $t_\nabla$ be the weak torsion of $h_\nabla$ and $T_\nabla$ be the tosrion of $\nabla$. Then using (\ref{wt1}) and (\ref{delta}) we deduce
\[
t^\gamma_{\alpha\beta}=(\Gamma^\gamma_{\alpha\beta}-\Gamma^\gamma_{\beta\alpha}-L^\gamma_{\alpha\beta})\circ\pi=(T_\nabla(e_\alpha, e_\beta))^{h_\nabla},
\]
where $t^\gamma_{\alpha\beta}$ are the coefficient of $t_\nabla$. If we denote by $\stackrel{{\begin{tiny}\nabla\end{tiny}}}{T}$, the torsion of Ichijy\={o} connection $(\stackrel{{\begin{tiny}\nabla\end{tiny}}}{D}, h_\nabla)$ then we get
\begin{align*}
\stackrel{{\begin{tiny}\nabla\end{tiny}}}{T}(\delta_\alpha, \delta_\beta)&=\Big((\Gamma^\gamma_{\alpha\beta}-\Gamma^\gamma_{\beta\alpha}-L^\gamma_{\alpha\beta})\circ\pi\Big)\delta_\gamma-R^\gamma_{\alpha\beta}\mathcal{V}_\gamma\\
&=t^\gamma_{\alpha\beta}\delta_\gamma+\Omega(\delta_\alpha, \delta_\beta)=F_\nabla t_\nabla(\delta_\alpha, \delta_\beta)+\Omega_\nabla(\delta_\alpha, \delta_\beta)\\
&=(T_\nabla(e_\alpha, e_\beta))^{h_\nabla}+\Omega_\nabla(\delta_\alpha, \delta_\beta),\\
\stackrel{{\begin{tiny}\nabla\end{tiny}}}{T}(\delta_\alpha, \mathcal{V}_\beta)&=-\frac{1}{2}\frac{\partial\mathcal{G}_{\alpha\gamma}}{\partial\textbf{y}^\beta}\mathcal{G}^{\gamma\mu}\delta_\mu=-F_\nabla\mathcal{C}(\delta_\alpha, \delta_\beta)=-F_\nabla\mathcal{C}(\delta_\alpha, F_\nabla\mathcal{V}_\beta),\\
\stackrel{{\begin{tiny}\nabla\end{tiny}}}{T}(\mathcal{V}_\alpha, \mathcal{V}_\beta)&=0,
\end{align*}
where $\Omega_\nabla$ is the curvature tensor of $h_\nabla$. From the above equations we can conclude the following
\begin{proposition}
Let $(\stackrel{{\begin{tiny}\nabla\end{tiny}}}{D}, h_\nabla)$ be the Ichijy\={o} connection on Finsler algebroid $(E, \mathcal{F})$ with base connection $\nabla$. Then the torsion tensor of $\stackrel{{\begin{tiny}\nabla\end{tiny}}}{D}$ satisfies
\begin{align*}
\stackrel{{\begin{tiny}\nabla\end{tiny}}}{T}(\widetilde{X}, \widetilde{Y})&=F_\nabla t_\nabla(h_\nabla\widetilde{X}, h_\nabla\widetilde{Y})+\Omega(h_\nabla\widetilde{X}, h_\nabla\widetilde{Y})-F_\nabla\mathcal{C}(h_\nabla\widetilde{X}, F_\nabla v_\nabla\widetilde{Y})\\
&\ \ \ +F_\nabla\mathcal{C}(F_\nabla v_\nabla\widetilde{X}, h_\nabla\widetilde{Y}),\ \ \ \forall \widetilde{X}, \widetilde{Y}\in\Gamma(\stackrel{\circ}{\pounds^\pi E}).
\end{align*}
\end{proposition}
\begin{cor}
Let $(\stackrel{{\begin{tiny}\nabla\end{tiny}}}{D}, h_\nabla)$ be the Ichijy\={o} connection on Finsler algebroid $(E, \mathcal{F})$ with base connection $\nabla$. Then for all $X, Y\in\Gamma(E)$ we have
\begin{align*}
\stackrel{{\begin{tiny}\nabla\end{tiny}}}{T}(X^{h_\nabla}, Y^{h_\nabla})&=(T_\nabla(X, Y))^{h_\nabla}+\Omega_\nabla(X^{h_\nabla}, Y^{h_\nabla}),\\
\stackrel{{\begin{tiny}\nabla\end{tiny}}}{T}(X^{h_\nabla}, Y^V)&=-F_\nabla\mathcal{C}(X^{h_\nabla}, F_\nabla Y^V),\\
\stackrel{{\begin{tiny}\nabla\end{tiny}}}{T}(X^V, Y^V)&=0.
\end{align*}
\end{cor}
Let $\stackrel{{\begin{tiny}\nabla\end{tiny}}}{R}^{\ \ \ \ \lambda}_{\alpha\beta\gamma}$, $\stackrel{{\begin{tiny}\nabla\end{tiny}}}{P}^{\ \ \ \ \lambda}_{\alpha\beta\gamma}$ and $\stackrel{{\begin{tiny}\nabla\end{tiny}}}{S}^{\ \ \ \ \lambda}_{\alpha\beta\gamma}$ be the coefficients of the horizontal, mixed and vertical curvatures of Ichijy\={o} connection $(\stackrel{{\begin{tiny}\nabla\end{tiny}}}{D}, h_\nabla)$, respectively. Then using (\ref{hh})-(\ref{hv}) and (\ref{I})-(\ref{I3}) we get
\begin{align}
\stackrel{{\begin{tiny}\nabla\end{tiny}}}{R}^{\ \ \ \ \lambda}_{\alpha\beta\gamma}&=(\rho^i_\alpha\circ \pi)\frac{\partial (\Gamma^\lambda_{\beta\gamma}\circ\pi)}{\partial \textbf{x}^i}-(\rho^i_\beta\circ \pi)\frac{\partial (\Gamma^\lambda_{\alpha\gamma}\circ\pi)}{\partial \textbf{x}^i}+(\Gamma^\mu_{\beta\gamma}\circ\pi)(\Gamma^\lambda_{\alpha\mu}\circ\pi)\nonumber\\
&\ -(\Gamma^\mu_{\alpha\gamma}\circ\pi)(\Gamma^\lambda_{\beta\mu}\circ\pi)-(L^\mu_{\alpha\beta}\circ \pi)(\Gamma^\lambda_{\mu\gamma}\circ\pi)-R^{\ \ \mu}_{\alpha\beta}\mathcal{C}^\lambda_{\mu\gamma}\nonumber\\
&=-\frac{\partial R_{\alpha\beta}^\lambda}{\partial \textbf{y}^\gamma}-R^{\ \ \mu}_{\alpha\beta}\mathcal{C}^\lambda_{\mu\gamma},\label{hhI}\\
\stackrel{{\begin{tiny}\nabla\end{tiny}}}{P}^{\ \ \ \ \lambda}_{\alpha\beta\gamma}&=(\rho^i_\alpha\circ \pi)\frac{\partial \mathcal{C}^\lambda_{\beta\gamma}}{\partial \textbf{x}^i}-\textbf{y}^\nu(\Gamma^\mu_{\alpha\nu}\circ\pi)\frac{\partial \mathcal{C}^\lambda_{\beta\gamma}}{\partial \textbf{y}^\mu}+\mathcal{C}^\mu_{\beta\gamma}(\Gamma^\lambda_{\alpha\mu}\circ\pi)-(\Gamma^\mu_{\alpha\gamma}\circ\pi)\mathcal{C}^\lambda_{\beta\mu}\nonumber\\
&\ \ \ -(\Gamma^\mu_{\alpha\beta}\circ\pi)\mathcal{C}^\lambda_{\mu\gamma},\label{hvI}\\
\stackrel{{\begin{tiny}\nabla\end{tiny}}}{S}^{\ \ \ \ \lambda}_{\alpha\beta\gamma}&=\frac{\partial \mathcal{C}^\lambda_{\beta\gamma}}{\partial \textbf{y}^\alpha}+\mathcal{C}^\mu_{\beta\gamma}\mathcal{C}^\lambda_{\alpha\mu}-\frac{\partial \mathcal{C}^\lambda_{\alpha\gamma}}{\partial \textbf{y}^\beta}-\mathcal{C}^\mu_{\alpha\gamma}\mathcal{C}^\lambda_{\beta\mu}.\label{vvI}
\end{align}
Using the above equations we conclude the following proposition which gives us the global expressions of horizontal, mixed and vertical curvatures of Ichijy\={o} connection.
\begin{proposition}\label{very im1}
Let $(\stackrel{{\begin{tiny}\nabla\end{tiny}}}{D}, h_\nabla)$ be the Ichijy\={o} connection on Finsler algebroid $(E, \mathcal{F})$ with base connection $\nabla$. Then we have
\begin{align*}
\stackrel{{\begin{tiny}\nabla\end{tiny}}}{R}(\widetilde{X}, \widetilde{Y})\widetilde{Z}&=[J, \Omega_\nabla(\widetilde{X}, \widetilde{Y})]^{F-N}_\pounds(h_\nabla \widetilde{Z})+\mathcal{C}(F_\nabla\Omega_\nabla(\widetilde{X}, \widetilde{Y}), \widetilde{Z}),\\
\stackrel{{\begin{tiny}\nabla\end{tiny}}}{P}(\widetilde{X}, \widetilde{Y})\widetilde{Z}&=(\stackrel{{\begin{tiny}\nabla\end{tiny}}}{D}_{h_\nabla \widetilde{X}}\mathcal{C})(h_\nabla \widetilde{Y}, h_\nabla \widetilde{Z}),\\
\stackrel{{\begin{tiny}\nabla\end{tiny}}}{Q}(\widetilde{X}, \widetilde{Y})\widetilde{Z}&=\mathcal{C}(F_\nabla\mathcal{C}(\widetilde{X}, \widetilde{Z}), \widetilde{Y})-\mathcal{C}(\widetilde{X}, F_\nabla \mathcal{C}(\widetilde{Y}, \widetilde{Z})),
\end{align*}
where $\widetilde{X}, \widetilde{Y}, \widetilde{Z}\in\Gamma(\stackrel{\circ}{\pounds^\pi E})$.
\end{proposition}
\begin{cor}
The horizontal curvature of Ichijy\={o} connection is zero if and only if the curvature of $h_\nabla$ (or the curvature of base connection $\nabla$) is zero.
\end{cor}
\begin{proof}
If the curvature of $h_\nabla$ vanishes, then we have $R_{\alpha\beta}^\lambda=0$. Therefore from (\ref{hhI}) we deduce $\stackrel{{\begin{tiny}\nabla\end{tiny}}}{R}^{\ \ \ \ \lambda}_{\alpha\beta\gamma}=0$, i.e., the horizontal curvature of Ichijy\={o} connection is zero. Conversely, if $\stackrel{{\begin{tiny}\nabla\end{tiny}}}{R}^{\ \ \ \ \lambda}_{\alpha\beta\gamma}=0$, then from (\ref{hhI}) we derive that
\[
\frac{\partial R_{\alpha\beta}^\lambda}{\partial \textbf{y}^\gamma}+R^{\ \ \mu}_{\alpha\beta}\mathcal{C}^\lambda_{\mu\gamma}=0.
\]
Multiplying $\textbf{y}^\gamma$ in the above equation and using $\textbf{y}^\gamma\mathcal{C}^\lambda_{\mu\gamma}=0$, give us $\textbf{y}^\gamma\frac{\partial R_{\alpha\beta}^\lambda}{\partial \textbf{y}^\gamma}=0$. But it is easy to see that $\textbf{y}^\gamma\frac{\partial R_{\alpha\beta}^\lambda}{\partial \textbf{y}^\gamma}=R_{\alpha\beta}^\lambda$. Thus we deduce $R_{\alpha\beta}^\lambda=0$, i.e., the curvature of $h_\nabla$ is zero. Note that from corollary \ref{very im}, we deduce that the vanishing of the horizontal curvature of Ichijy\={o} connection is equivalent to the vanishing of the curvature of base connection $\nabla$.
\end{proof}
From the second relation of proposition (\ref{very im1}) we conclude
\begin{cor}
The mixed curvature of Ichijy\={o} connection is zero if and only if the $h$-covariant derivative of the first Cartan tensor with respect to $\stackrel{{\begin{tiny}\nabla\end{tiny}}}{D}$ (i.e., $\stackrel{{\begin{tiny}\nabla\end{tiny}}}{D}_{h_\nabla}\mathcal{C}$) vanishes.
\end{cor}
If we denote by $\stackrel{\nabla}{A}$, $\stackrel{\nabla}{B}$, $\stackrel{\nabla}{R^1}$, $\stackrel{\nabla}{P^1}$, $\stackrel{\nabla}{Q^1}$ the components of torsion of  Ichijy\={o} connection, then using (\ref{very im2}), (\ref{very im3}) and (\ref{I})-(\ref{I3}) we obtain
\begin{align}
\stackrel{\nabla}{A}(\delta_\alpha, \delta_\beta)&=\Big((\Gamma^\gamma_{\alpha\beta}-\Gamma^\gamma_{\beta\alpha}-L^\gamma_{\alpha\beta})\circ\pi\Big)\delta_\gamma=t^\gamma_{\alpha\beta}\delta_\gamma
=F_\nabla t_\nabla(\delta_\alpha, \delta_\beta)\nonumber\\
&=(T_\nabla(e_\alpha, e_\beta))^{h_\nabla},\label{Wag1}\\
\stackrel{\nabla}{B}(\delta_\alpha, \delta_\beta)&=-\mathcal{C}_{\alpha\beta}^\gamma\delta_\gamma=-F_\nabla \mathcal{C}(\delta_\alpha, \delta_\beta),\\
\stackrel{\nabla}{R^1}(\delta_\alpha, \delta_\beta)&=-R^\gamma_{\alpha\beta}\mathcal{V}_\gamma=\Omega_\nabla(\delta_\alpha, \delta_\beta),\\
\stackrel{\nabla}{P^1}=0,\ \stackrel{\nabla}{Q^1}&=0.
\end{align}
From the above equation we conclude the following
\begin{proposition}
Let $(\stackrel{{\begin{tiny}\nabla\end{tiny}}}{D}, h_\nabla)$ be the Ichijy\={o} connection on Finsler algebroid $(E, \mathcal{F})$ with base connection $\nabla$. Then for all sections $X$ and $Y$ of $E$ we have
\begin{align*}
\stackrel{\nabla}{A}(X^{h_\nabla}, Y^{h_\nabla})&=(T_\nabla(X, Y))^{h_\nabla}=F_\nabla t_\nabla(X^{h_\nabla}, Y^{h_\nabla}),\\
\stackrel{\nabla}{B}(X^{h_\nabla}, Y^{h_\nabla})&=-F_\nabla\mathcal{C}(X^{h_\nabla}, Y^{h_\nabla}),\\
\stackrel{\nabla}{R^1}(X^{h_\nabla}, Y^{h_\nabla})&=\Omega_\nabla\mathcal{C}(X^{h_\nabla}, Y^{h_\nabla}),\\
\stackrel{\nabla}{P^1}=0,\ \ \ \ \ \ \stackrel{\nabla}{Q^1}&=0.
\end{align*}
\end{proposition}
From the first equation of the above proposition we have
\begin{cor}
The $h$-horizontal torsion of the Ichijy\={o} connection is zero if and only if the torsion tensor of $\nabla$ ( or the weak torsion of $h_\nabla$) vanishes.
\end{cor}
\subsection{Generalized Berwald Lie algebroid}
\begin{defn}
Let $(E, \mathcal{F})$ be a Finsler algebroid and $\nabla$ be a linear connection on $E$. Then $(E, \mathcal{F}, \nabla)$ is called generalized Berwald Lie algebroid, if the horizontal endomorphism $h_\nabla$ is conservative.
\end{defn}
\begin{proposition}\label{khob}
Let $(E, \mathcal{F})$ be a Finsler algebroid and $\nabla$ be a linear connection on $E$. Then the following items are equivalent:

(i)\ $(E, \mathcal{F}, \nabla)$ is a generalized Berwald Lie algebroid.

(ii)\ Second Cartan tensor $\widetilde{\mathcal{C}}_\nabla$ belonging to $\nabla$ is zero.

(iii)\ Ichijy\={o} connection $(\stackrel{{\begin{tiny}\nabla\end{tiny}}}{D}, h_\nabla)$ is $h_\nabla$-metrical.
\end{proposition}
\begin{proof}
$(i)\Rightarrow (ii)$. Since $h_\nabla$ is conservative, then we have (\ref{cons}). Setting $\mathcal{B}^\lambda_\alpha=-\textbf{y}^\sigma(\Gamma^\lambda_{\alpha\sigma}\circ\pi)$ in this equation we have
\begin{equation}\label{ziad0}
(\rho^i_\alpha\circ\pi)\frac{\partial\mathcal{F}}{\partial\textbf{x}^i}-\textbf{y}^\sigma(\Gamma^\lambda_{\alpha\sigma}\circ\pi)
\frac{\partial\mathcal{F}}{\partial\textbf{y}^\lambda}=0.
\end{equation}
Differentiating the above equation with respect to $\textbf{y}^\beta$ and $\textbf{y}^\mu$ gives us
\begin{align}\label{ziad}
&(\rho^i_\alpha\circ\pi)\frac{\partial^3\mathcal{F}}{\partial\textbf{x}^i\partial\textbf{y}^\beta\partial\textbf{y}^\mu}
-(\Gamma^\lambda_{\alpha\beta}\circ\pi)
\frac{\partial^2\mathcal{F}}{\partial\textbf{y}^\mu\partial\textbf{y}^\lambda}-(\Gamma^\lambda_{\alpha\mu}\circ\pi)
\frac{\partial^2\mathcal{F}}{\partial\textbf{y}^\beta\partial\textbf{y}^\lambda}\nonumber\\
&-\textbf{y}^\sigma(\Gamma^\lambda_{\alpha\sigma}\circ\pi)
\frac{\partial^3\mathcal{F}}{\partial\textbf{y}^\mu\partial\textbf{y}^\beta\partial\textbf{y}^\lambda}=0.
\end{align}
If we multiply $g^{\gamma\mu}$ in the above equation, then we obtain $\widetilde{\mathcal{C}}^\gamma_{\alpha\beta}=0$, where $\widetilde{\mathcal{C}}^\gamma_{\alpha\beta}$ are the coefficients of second Cartan tensor $\widetilde{\mathcal{C}}_\nabla$ given by (\ref{Cartan17}).\\
$(ii)\Rightarrow (i)$. Since Second Cartan tensor $\widetilde{\mathcal{C}}_\nabla$ belonging to $\nabla$ is zero, then we have $\widetilde{\mathcal{C}}^\gamma_{\alpha\beta}=0$. Thus setting $\mathcal{B}^\lambda_\alpha=-\textbf{y}^\sigma(\Gamma^\lambda_{\alpha\sigma}\circ\pi)$ in (\ref{Cartan17}) and multiply $g_{\gamma\mu}$ in it, we deduce (\ref{ziad}). Multiplying $\textbf{y}^\beta \textbf{y}^\mu$ in (\ref{ziad}) and using (ii) of (\ref{2eq}) and (\ref{only God}) we obtain (\ref{ziad0}). Thus $h_\nabla$ is conservative.\\
$(iii)\Rightarrow (ii)$. Since $\stackrel{{\begin{tiny}\nabla\end{tiny}}}{D}$ is $h$-metrical, then we have $\stackrel{{\begin{tiny}\nabla\end{tiny}}}{D}_{h_\nabla}\widetilde{\mathcal{G}}=0$. Thus we get
\begin{align*}
0&=(\stackrel{{\begin{tiny}\nabla\end{tiny}}}{D}_{h_\nabla\delta_\alpha}\widetilde{\mathcal{G}})(\delta_\beta,\delta_\gamma)=
(\rho^i_\alpha\circ\pi)\frac{\partial\mathcal{G_{\beta\gamma}}}{\partial\textbf{x}^i}
-(\Gamma^\lambda_{\alpha\beta}\circ\pi)
{\mathcal{G_{\lambda\gamma}}}-(\Gamma^\lambda_{\alpha\gamma}\circ\pi)
{\mathcal{G_{\beta\lambda}}}\\
&\ \ \ \ -\textbf{y}^\sigma(\Gamma^\lambda_{\alpha\sigma}\circ\pi)
\frac{\partial^2\mathcal{G_{\beta\gamma}}}{\partial\textbf{y}^\lambda}.
\end{align*}
Therefore we have (\ref{ziad}), i.e., the second Cartan tensor $\widetilde{\mathcal{C}}_\nabla$ belonging to $\nabla$ is zero.\\
$(ii)\Rightarrow (iii)$. If (ii) holds, then we have (\ref{ziad}). Using this equation it is easy to check that $(\stackrel{{\begin{tiny}\nabla\end{tiny}}}{D}_{h_\nabla\delta_\alpha}\widetilde{\mathcal{G}})(\delta_\beta,\delta_\gamma)
=(\stackrel{{\begin{tiny}\nabla\end{tiny}}}{D}_{h_\nabla\delta_\alpha}\widetilde{\mathcal{G}})(\mathcal{V}_\beta, \mathcal{V}_\gamma)=0$. Also, we have $(\stackrel{{\begin{tiny}\nabla\end{tiny}}}{D}_{h_\nabla\delta_\alpha}\widetilde{\mathcal{G}})(\delta_\beta, \mathcal{V}_\gamma)=0$. Thus Ichijy\={o} connection $(\stackrel{{\begin{tiny}\nabla\end{tiny}}}{D}, h_\nabla)$ is $h_\nabla$-metrical.
\end{proof}
\begin{proposition}\label{Nour}
Let $(E, \mathcal{F}, \nabla)$ be a generalized Berwald Lie algebroid. Then the mixed curvature of Ichijy\={o} connection $(\stackrel{{\begin{tiny}\nabla\end{tiny}}}{D}, h_\nabla)$ is zero.
\end{proposition}
\begin{proof}
It is sufficient to show that $\stackrel{{\begin{tiny}\nabla\end{tiny}}}{P}^{\ \ \ \ \lambda}_{\alpha\beta\gamma}=0$. Using (\ref{hvI}) we have
\begin{align}\label{nour}
\stackrel{{\begin{tiny}\nabla\end{tiny}}}{P}^{\ \ \ \ \lambda}_{\alpha\beta\gamma}&=\frac{1}{2}(\rho^i_\alpha\circ \pi)(\frac{\partial^2 \mathcal{G}_{\beta\sigma}}{\partial \textbf{x}^i\partial\textbf{y}^\gamma}\mathcal{G}^{\sigma\lambda}+\frac{\partial \mathcal{G}_{\beta\sigma}}{\partial\textbf{y}^\gamma}\frac{\partial \mathcal{G}^{\sigma\lambda}}{\partial \textbf{x}^i})-\frac{1}{2}\textbf{y}^\nu(\Gamma^\mu_{\alpha\nu}\circ\pi)(\frac{\partial^2 \mathcal{G}_{\beta\sigma}}{\partial \textbf{y}^\mu\partial\textbf{y}^\gamma}\mathcal{G}^{\sigma\lambda}\nonumber\\
&\ \ \ +\frac{\partial \mathcal{G}_{\beta\sigma}}{\partial\textbf{y}^\gamma}\frac{\partial \mathcal{G}^{\sigma\lambda}}{\partial \textbf{y}^\mu})+\frac{1}{2}\frac{\partial\mathcal{G}_{\beta\sigma}}{\partial\textbf{y}^\gamma}\mathcal{G}^{\sigma\mu}(\Gamma^\lambda_{\alpha\mu}\circ\pi)
-\frac{1}{2}\frac{\partial\mathcal{G}_{\beta\sigma}}{\partial\textbf{y}^\mu}\mathcal{G}^{\sigma\lambda}(\Gamma^\mu_{\alpha\gamma}\circ\pi)
\nonumber\\
&\ \ \ -\frac{1}{2}\frac{\partial\mathcal{G}_{\mu\sigma}}{\partial\textbf{y}^\gamma}\mathcal{G}^{\sigma\lambda}(\Gamma^\mu_{\alpha\beta}\circ\pi).
\end{align}
Since Ichijy\={o} connection is $h$-metrical, then we have
\[
0=\stackrel{{\begin{tiny}\nabla\end{tiny}}}{D}_{h_\nabla\delta_\alpha}\mathcal{G}^{\sigma\lambda}=(\rho^i_\alpha\circ \pi)\frac{\partial \mathcal{G}^{\sigma\lambda}}{\partial \textbf{x}^i}-\textbf{y}^\nu(\Gamma^\mu_{\alpha\nu}\circ\pi)\frac{\partial \mathcal{G}^{\sigma\lambda}}{\partial \textbf{y}^\mu}+\mathcal{G}^{\sigma\mu}(\Gamma^\lambda_{\alpha\mu}\circ\pi)+\mathcal{G}^{\lambda\mu}(\Gamma^\sigma_{\alpha\mu}\circ\pi),
\]
which gives us
\[
(\rho^i_\alpha\circ \pi)\frac{\partial \mathcal{G}^{\sigma\lambda}}{\partial \textbf{x}^i}-\textbf{y}^\nu(\Gamma^\mu_{\alpha\nu}\circ\pi)\frac{\partial \mathcal{G}^{\sigma\lambda}}{\partial \textbf{y}^\mu}+\mathcal{G}^{\sigma\mu}(\Gamma^\lambda_{\alpha\mu}\circ\pi)=-\mathcal{G}^{\lambda\mu}(\Gamma^\sigma_{\alpha\mu}\circ\pi).
\]
Setting the above equation in (\ref{nour}) we get
\begin{align*}
\stackrel{{\begin{tiny}\nabla\end{tiny}}}{P}^{\ \ \ \ \lambda}_{\alpha\beta\gamma}&=\frac{1}{2}(\rho^i_\alpha\circ \pi)\frac{\partial^2 \mathcal{G}_{\beta\sigma}}{\partial \textbf{x}^i\partial\textbf{y}^\gamma}\mathcal{G}^{\sigma\lambda}-\frac{1}{2}\textbf{y}^\nu(\Gamma^\mu_{\alpha\nu}\circ\pi)\frac{\partial^2 \mathcal{G}_{\beta\sigma}}{\partial \textbf{y}^\mu\partial\textbf{y}^\gamma}\mathcal{G}^{\sigma\lambda}\nonumber\\
&\ \ \ -\frac{1}{2}\frac{\partial\mathcal{G}_{\beta\sigma}}{\partial\textbf{y}^\mu}\mathcal{G}^{\sigma\lambda}(\Gamma^\mu_{\alpha\gamma}\circ\pi)
-\frac{1}{2}\frac{\partial\mathcal{G}_{\mu\sigma}}{\partial\textbf{y}^\gamma}\mathcal{G}^{\sigma\lambda}(\Gamma^\mu_{\alpha\beta}\circ\pi)\nonumber\\
&\ \ \ -\frac{1}{2}\frac{\partial\mathcal{G}_{\beta\sigma}}{\partial\textbf{y}^\gamma}\mathcal{G}^{\lambda\mu}(\Gamma^\sigma_{\alpha\mu}\circ\pi).
\end{align*}
Since $h_\nabla$ is conservative, then using (\ref{ziad}) the right side of the above equation vanishes. Thus we have $\stackrel{{\begin{tiny}\nabla\end{tiny}}}{P}^{\ \ \ \ \lambda}_{\alpha\beta\gamma}=0$.
\end{proof}
Let $(E, \mathcal{F}, \nabla)$ be a generalized Berwald Lie algebroid and $f$ be a non-constant smooth function on $E$. We define $\bar{h}_\nabla:=h_\nabla-df^\vee\otimes C$. Since $df^\vee=(\rho^i_\alpha\circ\pi)\frac{\partial(f\circ\pi)}{\partial\textbf{x}^i}\mathcal{X}^\alpha$, then using (\ref{delta}) we can see that $\bar{h}_\nabla$ has the local expression
\begin{equation}
\bar{h}_\nabla=(\mathcal{X}_\alpha+\mathcal{B}^\beta_\alpha\mathcal{V}_\beta)\otimes\mathcal{X}^\alpha,
\end{equation}
where
\begin{equation}\label{Hii}
\mathcal{B}^\beta_\alpha=-(\textbf{y}^\beta(\rho^i_\alpha\circ\pi)\frac{\partial(f\circ\pi)}{\partial\textbf{x}^i}+\textbf{y}^\lambda (\Gamma^\beta_{\alpha\lambda}\circ\pi)).
\end{equation}
Using two above equation it is easy to check that $\bar{h}_\nabla$ is an everywhere smooth function and ${\bar{h}_\nabla}^2=\bar{h}_\nabla$, $\ker\bar{h}_\nabla=\Gamma(v\pounds^\pi E)$. Thus $\bar{h}_\nabla$ is an everywhere smooth, horizontal endomorphism on $\pounds^\pi E$. Moreover we can obtain $\textbf{y}^\gamma\frac{\partial \mathcal{B}^\beta_\alpha}{\partial \textbf{y}^\gamma}=\mathcal{B}^\beta_\alpha$, i.e., $\bar{h}_\nabla$ is a homogenous horizontal endomorphism.
\begin{lemma}
Let $(E, \mathcal{F}, \nabla)$ be a generalized Berwald Lie algebroid and $\{e_\alpha\}$ be a basis of sections of $E$. Then $\bar{h}_{\nabla}$ is conservative if and only if $\rho(e_\alpha)f=0$.
\end{lemma}
\begin{proof}
Using (\ref{cons}), $\bar{h}_{\nabla}$ is conservative, if and only if
\begin{equation}
(\rho^i_\alpha\circ\pi)\frac{\partial\mathcal{F}}{\partial\textbf{x}^i}+\mathcal{B}^\beta_\alpha\frac{\partial\mathcal{F}}{\partial\textbf{y}^\beta}=0,
\end{equation}
where $\mathcal{B}^\beta_\alpha$ are given by (\ref{Hii}). Setting (\ref{Hii}) in the above equation give us
\[
(\rho^i_\alpha\circ\pi)\frac{\partial\mathcal{F}}{\partial\textbf{x}^i}
-\textbf{y}^\beta(\rho^i_\alpha\circ\pi)\frac{\partial(f\circ\pi)}{\partial\textbf{x}^i}\frac{\partial\mathcal{F}}{\partial\textbf{y}^\beta}
-\textbf{y}^\lambda (\Gamma^\beta_{\alpha\lambda}\circ\pi)\frac{\partial\mathcal{F}}{\partial\textbf{y}^\beta}=0.
\]
In other hand, since $h_\nabla$ is conservative, then we have
\[
(\rho^i_\alpha\circ\pi)\frac{\partial\mathcal{F}}{\partial\textbf{x}^i}
-\textbf{y}^\lambda (\Gamma^\beta_{\alpha\lambda}\circ\pi)\frac{\partial\mathcal{F}}{\partial\textbf{y}^\beta}=0.
\]
Two above equations gives us
\[
\textbf{y}^\beta(\rho^i_\alpha\circ\pi)\frac{\partial(f\circ\pi)}{\partial\textbf{x}^i}\frac{\partial\mathcal{F}}{\partial\textbf{y}^\beta}=0,
\]
and consequently
\[
(\rho^i_\alpha\circ\pi)\frac{\partial(f\circ\pi)}{\partial\textbf{x}^i}\mathcal{F}=0,
\]
because $\mathcal{F}$ is homogenous of degree 2. But since $\mathcal{F}$ is non-zero, then from the above equation we deduce $(\rho^i_\alpha\circ\pi)\frac{\partial(f\circ\pi)}{\partial\textbf{x}^i}=0$ or $(\rho(e_\alpha)f)^\vee=0$. Thus $h_{\bar{\nabla}}$ is conservative if and only if $\rho(e_\alpha)f=0$.
\end{proof}
\begin{cor}
Let $(E, \mathcal{F}, \nabla)$ be a generalized Berwald Lie algebroid and the anchor map $\rho$ be injective. Then $\bar{h}_{\nabla}$ is not conservative.
\end{cor}
Now we consider the linear connection $\bar{\nabla}_{e_\alpha}e_\beta=\bar{\Gamma}^{\gamma}_{\alpha\beta}e_\gamma$, where
\[
(\bar{\Gamma}^{\gamma}_{\alpha\beta}\circ\pi)=-\frac{\partial\mathcal{B}^\gamma_\alpha}{\partial\textbf{y}^\beta}
=\delta^\gamma_\beta(\rho^i_\alpha\circ\pi)\frac{\partial(f\circ\pi)}{\partial\textbf{x}^i}+(\Gamma^\gamma_{\alpha\beta}\circ\pi),
\]
or
\begin{equation}\label{gammabar}
\bar{\Gamma}^{\gamma}_{\alpha\beta}=\delta_\beta^\gamma\rho^i_\alpha\frac{\partial f}{\partial x^i}
+\Gamma^\gamma_{\alpha\beta},
\end{equation}
and we call it \textit{the linear connection generated by $\bar{h}_\nabla$}.
\begin{proposition}
Let $(E, \mathcal{F}, \nabla)$ be a generalized Berwald Lie algebroid and $\bar{\nabla}$ be the linear connection generated by $\bar{h}_\nabla$. Then the mixed curvature of Ichijy\={o} connection $(\stackrel{{\begin{tiny}\bar{\nabla}\end{tiny}}}{D}, \bar{h}_\nabla)$ vanishes.
\end{proposition}
\begin{proof}
Using (\ref{hvI}) and (\ref{gammabar}) we get
\begin{align}\label{nour1}
\stackrel{{\begin{tiny}\bar{\nabla}\end{tiny}}}{P}^{\ \ \ \ \lambda}_{\alpha\beta\gamma}&=\stackrel{{\begin{tiny}\nabla\end{tiny}}}{P}^{\ \ \ \ \lambda}_{\alpha\beta\gamma}-\frac{1}{2}\textbf{y}^\mu(\rho^i_\alpha\circ\pi)\frac{\partial (f\circ\pi)}{\partial\textbf{x}^i}(\frac{\partial^2 \mathcal{G}_{\beta\sigma}}{\partial \textbf{y}^\mu\partial\textbf{y}^\gamma}\mathcal{G}^{\sigma\lambda}\nonumber\\
&\ \ \ +\frac{\partial \mathcal{G}_{\beta\sigma}}{\partial\textbf{y}^\gamma}\frac{\partial \mathcal{G}^{\sigma\lambda}}{\partial \textbf{y}^\mu})+\frac{1}{2}\frac{\partial\mathcal{G}_{\beta\sigma}}{\partial\textbf{y}^\gamma}\mathcal{G}^{\sigma\lambda}(\rho^i_\alpha\circ\pi)\frac{\partial (f\circ\pi)}{\partial\textbf{x}^i}
\nonumber\\
&\ \ \ -\frac{1}{2}\frac{\partial\mathcal{G}_{\beta\sigma}}{\partial\textbf{y}^\gamma}\mathcal{G}^{\sigma\lambda}(\rho^i_\alpha\circ\pi)\frac{\partial (f\circ\pi)}{\partial\textbf{x}^i}-\frac{1}{2}\frac{\partial\mathcal{G}_{\beta\sigma}}{\partial\textbf{y}^\gamma}\mathcal{G}^{\sigma\lambda}(\rho^i_\alpha\circ\pi)\frac{\partial (f\circ\pi)}{\partial\textbf{x}^i}.
\end{align}
Since $(E, \mathcal{F}, \nabla)$ is a generalized Berwald Lie algebroid, then $h_\nabla$ is conservative. Thus according to proposition \ref{Nour}, $\stackrel{{\begin{tiny}\nabla\end{tiny}}}{P}^{\ \ \ \ \lambda}_{\alpha\beta\gamma}=0$. Moreover, we have
\[
\textbf{y}^\mu\frac{\partial^2 \mathcal{G}_{\beta\sigma}}{\partial \textbf{y}^\mu\partial\textbf{y}^\gamma}=-\frac{\partial\mathcal{G}_{\beta\sigma}}{\partial\textbf{y}^\gamma},\ \ \ \textbf{y}^\mu\frac{\partial \mathcal{G}^{\sigma\lambda}}{\partial \textbf{y}^\mu}=0,
\]
because $\frac{\partial\mathcal{G}_{\beta\sigma}}{\partial\textbf{y}^\gamma}$ and $\mathcal{G}^{\sigma\lambda}$ are homogenous functions of degree -1 and 0, respectively. Therefore, (\ref{nour1}) reduce to the following
\begin{align*}
\stackrel{{\begin{tiny}\bar{\nabla}\end{tiny}}}{P}^{\ \ \ \ \lambda}_{\alpha\beta\gamma}&=\frac{1}{2}(\rho^i_\alpha\circ\pi)\frac{\partial (f\circ\pi)}{\partial\textbf{x}^i}\frac{\partial \mathcal{G}_{\beta\sigma}}{\partial\textbf{y}^\gamma}\mathcal{G}^{\sigma\lambda}+\frac{1}{2}\frac{\partial\mathcal{G}_{\beta\sigma}}{\partial\textbf{y}^\gamma}\mathcal{G}^{\sigma\lambda}(\rho^i_\alpha\circ\pi)\frac{\partial (f\circ\pi)}{\partial\textbf{x}^i}
\nonumber\\
&\ \ \ -\frac{1}{2}\frac{\partial\mathcal{G}_{\beta\sigma}}{\partial\textbf{y}^\gamma}\mathcal{G}^{\sigma\lambda}(\rho^i_\alpha\circ\pi)\frac{\partial (f\circ\pi)}{\partial\textbf{x}^i}-\frac{1}{2}\frac{\partial\mathcal{G}_{\beta\sigma}}{\partial\textbf{y}^\gamma}\mathcal{G}^{\sigma\lambda}(\rho^i_\alpha\circ\pi)\frac{\partial (f\circ\pi)}{\partial\textbf{x}^i}\\
&=0.
\end{align*}
\end{proof}
\begin{defn}
Generalized Berwald Lie algebroid $(E, \mathcal{F}, \nabla)$ is called Berwald Lie algebroid, if $\nabla$ be a torsion free linear connection on $E$.
\end{defn}
\begin{proposition}
Let $(E, \mathcal{F})$ be a Finsler Lie algebroid and $h_\circ$ be a Barthel endomorphism of it. Then $(E, \mathcal{F})$ is Berwald Lie algebroid if and only if there is a linear connection on $E$ such that
\[
(\nabla_XY)^V=[X^{h_\circ}, Y^V]_\pounds,\ \ \ \forall X, Y\in\Gamma(E).
\]
\end{proposition}
\begin{proof}
Let $(E, \mathcal{F})$ be a Finsler Lie algebroid. Then there is a torsion free linear connection $\nabla$ on $E$ such that $h_\nabla$ is conservative. From torsion freeness of $\nabla$ we conclude that $t_\nabla$ is zero and consequently $h_\nabla$ is homogenous. Thus $h_\nabla$ is the Barthel horizontal endomorphism and consequently $h_\nabla=h_\circ$, because the Barthel connection is unique. Therefore we have $(\nabla_XY)^V=[X^{h_\nabla}, Y^V]_\pounds=[X^{h_\circ}, Y^V]_\pounds$. Conversely, let there is a linear connection on $E$ such that $(\nabla_XY)^V=[X^{h_\circ}, Y^V]_\pounds$, for all $X, Y\in\Gamma(E)$. Since $(\nabla_XY)^V=[X^{h_\nabla}, Y^V]_\pounds$, then we deduce $[X^{h_\circ}, Y^V]_\pounds=[X^{h_\nabla}, Y^V]_\pounds$ and consequently $h_\nabla=h_\circ$. Thus $h_\nabla$ is conservative and $\nabla$ is torsion free, because the Barthel connection is conservative and torsion free. Therefore $(E, \mathcal{F})$ is a Berwald Lie algebroid.
\end{proof}
\begin{theorem}
A Finsler Lie algebroid is a Berwald Lie algebroid if and only if the Hashiguchi connection of it, is a Ichijy\={o} connection.
\end{theorem}
\begin{proof}
Let $(E, \mathcal{F})$ be a Berwald Lie algebroid. Then from the above proposition, $h_\nabla=h_\circ$, where $h_\nabla$ is a horizontal endomorphism generated by $\nabla$ and $h_\circ$ is the Barthel endomorphism. Thus we have $\mathcal{B}^\mu_\alpha=-\textbf{y}^\gamma(\Gamma^\mu_{\alpha\gamma}\circ\pi)$. Setting this equation in (\ref{H2}) and (\ref{H3}) we obtain
\begin{align*}
\stackrel{\text{\begin{tiny}H\end{tiny}}}{D}_{\delta_\alpha}\mathcal{V}_\beta&=(\Gamma^\mu_{\alpha\beta}\circ\pi)\mathcal{V}_\mu
=\stackrel{{\begin{tiny}\nabla\end{tiny}}}{D}_{\delta_\alpha}\mathcal{V}_\beta,\\
\stackrel{\text{\begin{tiny}H\end{tiny}}}{D}_{\delta_\alpha}\delta_\beta&=(\Gamma^\mu_{\alpha\beta}\circ\pi)\delta_\mu
=\stackrel{{\begin{tiny}\nabla\end{tiny}}}{D}_{\delta_\alpha}\delta_\beta.
\end{align*}
Also, from (\ref{H}), (\ref{H1}), (\ref{I}) and (\ref{I1}) we have
\[
\stackrel{\text{\begin{tiny}H\end{tiny}}}{D}_{\mathcal{V}_\alpha}\mathcal{V}_\beta
=\stackrel{{\begin{tiny}\nabla\end{tiny}}}{D}_{\mathcal{V}_\alpha}\mathcal{V}_\beta,\ \ \ \stackrel{\text{\begin{tiny}H\end{tiny}}}{D}_{\mathcal{V}_\alpha}\delta_\beta
=\stackrel{{\begin{tiny}\nabla\end{tiny}}}{D}_{\mathcal{V}_\alpha}\delta_\beta.
\]
Thus $\stackrel{\text{\begin{tiny}H\end{tiny}}}{D}=\stackrel{{\begin{tiny}\nabla\end{tiny}}}{D}$. Conversely, if the Hashiguchi connection of a Finsler algebroid $(E, \mathcal{F})$ is a Ichijy\={o} connection, then it is easy to see that $h_\nabla=h_\circ$. Thus according to the above proposition we conclude that $(E, \mathcal{F})$ is a Berwald Lie algebroid.
\end{proof}
Let $(E, \mathcal{F}, \nabla)$ be a Berwald Lie algebroid. If $\nabla$ is a flat connection then we call $(E, \mathcal{F}, \nabla)$, \textit{the locally Minkowski Lie algebroid}.
\begin{theorem}
A Finsler Lie algebroid $(E, \mathcal{F})$ is a locally Minkowski Lie algebroid if and only if there is a torsion free and flat linear connection on $E$ such that Ichijy\={o} connection $(\stackrel{{\begin{tiny}\nabla\end{tiny}}}{D}, h_\nabla)$ is $h_\nabla$-metrical.
\end{theorem}
\begin{proof}
Let $(E, \mathcal{F})$ be a locally Minkowski Lie algebroid. Then there exist torsion free and flat linear connection $\nabla$ on $E$ such that $(E, \mathcal{F}, \nabla)$ is a generalized Berwald Lie algebroid. Therefore, from proposition \ref{khob}, we deduce that Ichijy\={o} connection $(\stackrel{{\begin{tiny}\nabla\end{tiny}}}{D}, h_\nabla)$ is $h_\nabla$-metrical. Using proposition \ref{khob}, the proof of the converse of the theorem is obvious.
\end{proof}
\begin{proposition}\label{14.31}
Let $(E, \mathcal{F}, \nabla)$ be a generalized Berwald Lie algebroid. Then we have
\begin{align}
S_\nabla&=S_\circ+(d^\pounds_{i_{S_\nabla}t_\nabla}\mathcal{F})^\sharp,\label{very good}\\
h_\nabla&=h_\circ+\frac{1}{2}i_{S_\nabla}t_\nabla+\frac{1}{2}[J, (d^\pounds_{i_{S_\nabla}t_\nabla}\mathcal{F})^\sharp]_\pounds\label{very good20}.
\end{align}
\end{proposition}
\begin{proof}
Since $(E, \mathcal{F}, \nabla)$ be a generalized Berwald Lie algebroid, then $h_\nabla$ is conservative. Thus from propositions \ref{mainpor} and \ref{12.16} the proof is obvious.
\end{proof}
\begin{theorem}
Let $(E, \mathcal{F}, \nabla_1)$ and $(E, \mathcal{F}, \nabla_2)$ be generalized Berwald Lie algebroids. Then $\nabla_1$ is equal to $\nabla_2$ if and only if the torsion tensor fields of these are equal.
\end{theorem}
\begin{proof}
If $\nabla_1=\nabla_2$, then $T_{\nabla_1}=T_{\nabla_2}$. Conversely, if $T_{\nabla_1}=T_{\nabla_2}$ then the horizontal endomorphisms $h_{\nabla_1}$ and $h_{\nabla_2}$ have the same weak torsion and since these horizontal endomorphisms are homogenous, then they have the same strong torsion. Therefore using theorem \ref{12.18} we deduce that $h_{\nabla_1}=h_{\nabla_2}$ and consequently $\nabla_1=\nabla_2$.
\end{proof}
\begin{proposition}
Let $(E, \mathcal{F}, \nabla)$ be generalized Berwald Lie algebroids. If spray $S_\nabla$ generated by $\nabla$ is the projective change of spray $S_\circ$, then $S_\nabla=S_\circ$ and consequently $(E, \mathcal{F})$ is a Berwald manifold.
\end{proposition}
\begin{proof}
Since $S_\nabla$ is the projective change of $S_\circ$, then the exist a function $\widetilde{f}:E\rightarrow\mathbb{R}$ that is smooth on $E-\{0\}$ such that $S_\nabla=S_\circ+\widetilde{f}C$. Then using (\ref{very good}) we have $(d^\pounds_{i_{S_\nabla}t_\nabla}\mathcal{F})^\sharp=\widetilde{f}C$. Thus using (iii) of proposition \ref{Best1} we obtain
\[
i_{S_\nabla-S_\circ}\omega=i_{(d^\pounds_{i_{S_\nabla}t_\nabla}\mathcal{F})^\sharp}\omega=i_{\widetilde{f}C}\omega=\widetilde{f}i_C\omega
=\widetilde{f}d^\pounds_J\mathcal{F}.
\]
Also, we have
\[
i_{S_\nabla-S_\circ}\omega=d^\pounds_{i_{S_\nabla}t_\nabla}\mathcal{F}.
\]
Two above equation give us
\begin{equation}\label{very good1}
d^\pounds_{i_{S_\nabla}t_\nabla}\mathcal{F}=\widetilde{f}d^\pounds_J\mathcal{F}.
\end{equation}
Thus we have
\begin{align*}
d^\pounds_{i_{S_\nabla}t_\nabla}\mathcal{F}(S)&=d^\pounds\mathcal{F}(i_{S_\nabla}t_\nabla(S))
=d^\pounds\mathcal{F}(t_\nabla(S_\nabla, S))\\
&=d^\pounds\mathcal{F}(t_\nabla(S, S))=d^\pounds\mathcal{F}(0)=0.
\end{align*}
Also from (\ref{Finsler algebroid}) we have $d^\pounds_J\mathcal{F}(S)=\textbf{y}^\alpha\frac{\partial\mathcal{F}}{\partial \textbf{y}^\alpha}=2\mathcal{F}$. Setting this equation and the above equation in (\ref{very good1}) we deduce $\widetilde{f}\mathcal{F}=0$ and consequently $\widetilde{f}=0$. Therefore we have $S_\nabla=S$.
\end{proof}
\subsection{Wagner-Ichijy\={o} connection}
Let $\nabla$ be a linear connection on $E$ and $f$ be a smooth function on $M$. If $(\stackrel{{\begin{tiny}\nabla\end{tiny}}}{D}, h_\nabla)$ is a Ichijy\={o} connection such that the $h$-horizontal torsion of $\stackrel{{\begin{tiny}\nabla\end{tiny}}}{D}$ satisfies in
\begin{equation}\label{Wagner}
\stackrel{{\begin{tiny}\nabla\end{tiny}}}{A}=d^\pounds f^\vee\wedge h_\nabla=d^\pounds f^\vee\otimes h_\nabla-h_\nabla\otimes d^\pounds f^\vee,
\end{equation}
then we call $(\stackrel{{\begin{tiny}\nabla\end{tiny}}}{D}, h_\nabla, f)$ the \textit{Wagner-Ichijy\={o} connection generated by $\nabla$}.

From (\ref{Wagner}) we deduce that $\stackrel{{\begin{tiny}\nabla\end{tiny}}}{A}(\mathcal{V}_\alpha, \mathcal{V}_\beta)=\stackrel{{\begin{tiny}\nabla\end{tiny}}}{A}(\delta_\alpha, \mathcal{V}_\beta)=0$ and
\begin{align}\label{Wag}
\stackrel{{\begin{tiny}\nabla\end{tiny}}}{A}(\delta_\alpha, \delta_\beta)&=d^\pounds f^\vee(\delta_\alpha)h_\nabla(\delta_\beta)-h_\nabla(\delta_\alpha)d^\pounds f^\vee(\delta_\beta)\nonumber\\
&=\rho_\pounds(\delta_\alpha)(f\circ\pi)\delta_\beta-\rho_\pounds(\delta_\beta)(f\circ\pi)\delta_\alpha\nonumber\\
&=\Big((\rho^i_\alpha\circ\pi)\frac{\partial(f\circ\pi)}{\partial\textbf{x}^i}\delta^\gamma_\beta
-(\rho^i_\beta\circ\pi)\frac{\partial(f\circ\pi)}{\partial\textbf{x}^i}\delta^\gamma_\alpha\Big)\delta_\gamma.
\end{align}
\begin{lemma}\label{14.34}
Let $(\stackrel{{\begin{tiny}\nabla\end{tiny}}}{D}, h_\nabla, f)$ be a Wagner-Ichijy\={o} connection on Finsler algebroid $(E, \mathcal{F})$. Then we have
\begin{align*}
T_\nabla(X, Y)&=d^Ef(X)Y-d^Ef(Y)X,\ \ \ \forall X, Y\in\Gamma(E),\\
t_\nabla&=d^\pounds f^\vee\wedge J=d^\pounds f^\vee\otimes J-J\otimes d^\pounds f^\vee,\\
i_{S_\nabla}t_\nabla&=f^cJ-d^\pounds f^\vee\otimes C.
\end{align*}
\end{lemma}
\begin{proof}
Using (\ref{Wag}) we obtain
\begin{align*}
\stackrel{{\begin{tiny}\nabla\end{tiny}}}{A}(\delta_\alpha, \delta_\beta)&
=\Big((\rho^i_\alpha\frac{\partial f}{\partial x^i}\delta^\gamma_\beta
-\rho^i_\beta\frac{\partial f}{\partial x^i}\delta^\gamma_\alpha)e_\gamma\Big)^h=\Big(\rho(e_\alpha)(f)e_\beta-\rho(e_\beta)(f)e_\alpha\Big)^h\\
&=\Big(d^Ef(e_\alpha)e_\beta-d^Ef(e_\beta)e_\alpha\Big)^h.
\end{align*}
Also, from (\ref{Wag1}) we have $\stackrel{{\begin{tiny}\nabla\end{tiny}}}{A}(\delta_\alpha, \delta_\beta)=(T_\nabla(e_\alpha, e_\beta))^h$. Therefore we obtain
\[
T_\nabla(e_\alpha, e_\beta)=d^Ef(e_\alpha)e_\beta-d^Ef(e_\beta)e_\alpha,
\]
that gives us the first equation of the lemma. Also, from (\ref{Wag1}) and (\ref{Wag}) we obtain
\[
F_\nabla t_\nabla(\delta_\alpha, \delta_\beta)=\stackrel{{\begin{tiny}\nabla\end{tiny}}}{A}(\delta_\alpha, \delta_\beta)=d^\pounds f^\vee(\delta_\alpha)h_\nabla(\delta_\beta)-h_\nabla(\delta_\beta)d^\pounds f^\vee(\delta_\alpha).
\]
Applying $F_\nabla$ to the above equation and using $F_\nabla h_\nabla=-J$ and $F_\nabla F_\nabla=-1$ give us
\[
t_\nabla(\delta_\alpha, \delta_\beta)=d^\pounds f^\vee(\delta_\alpha)J(\delta_\beta)-J(\delta_\alpha)d^\pounds f^\vee(\delta_\beta),
\]
which gives us the second equation of the lemma. Using the above equation and (\ref{cf}) we get
\begin{align*}
i_{S_\nabla}t_\nabla(\delta_\beta)&=t_\nabla(S_\nabla, \delta_\beta)=\textbf{y}^\alpha t_\nabla(\delta_\alpha, \delta_\beta)=\textbf{y}^\alpha d^\pounds f^\vee(\delta_\alpha)\mathcal{V}_\beta-\textbf{y}^\alpha\mathcal{V}_\alpha d^\pounds f^\vee(\delta_\beta)\\
&=\textbf{y}^\alpha \rho_\pounds(\delta_\alpha)(f^\vee)\mathcal{V}_\beta-Cd^\pounds f^\vee(\delta_\beta)\\
&=\textbf{y}^\alpha(\rho^i_\alpha\circ\pi)\frac{\partial(f\circ\pi)}{\partial\textbf{x}^i}J(\delta_\beta)-Cd^\pounds f^\vee(\delta_\beta)\\
&=f^cJ(\delta_\beta)-d^\pounds f^\vee(\delta_\beta)C,
\end{align*}
which gives us the third equation of the lemma.
\end{proof}
\begin{defn}
Let $(E, \mathcal{F}, \nabla)$ be a generalized Berwald Lie algebroid and $f$ be a smooth function on $E$. Then $(E, \mathcal{F}, \nabla, f)$ is called Wagner Lie algebroid if the torsion of linear connection $\nabla$ satisfies in the following relation
\begin{equation}\label{enteha}
T_\nabla(X, Y)=d^E f(X)Y-d^E f(Y)X,\ \ \ \forall X, Y\in \Gamma(E).
\end{equation}
\end{defn}
\begin{theorem}
Let $(E, \mathcal{F})$ be a Lie algebroid, $f$ be a smooth function on $M$ and $\nabla$ be a linear connection on $E$. Then the following items are equivalent:

(i)\ $(E, \mathcal{F}, \nabla, f)$ is a Wagner Lie algebroid.

(ii)\ Wagner-Ichijy\={o} connection $(\stackrel{{\begin{tiny}\nabla\end{tiny}}}{D}, h_\nabla, f)$  generated by $\nabla$, is $h$-metrical.

(iii)\ Horizontal endomorphism $h_\nabla$ satisfies in the following
\begin{equation}\label{akhar3}
h_\nabla=h_\circ+f^c J-\mathcal{F}[J, \text{grad}f^\vee]_\pounds-d^\pounds_J\mathcal{F}\otimes\text{grad}f^\vee.
\end{equation}
\end{theorem}
\begin{proof}
From proposition \ref{khob} the equivalence of (i) and (ii) is obvious. Thus it is sufficient to prove that (i) is equivalent to (iii). Let (i) holds. Since  $(E, \mathcal{F}, \nabla, f)$ is a Wagner Lie algebroid, then $(E, \mathcal{F}, \nabla)$ is a generalized Berwald Lie algebroid and consequently from proposition \ref{14.31} we have the formula (\ref{very good20}) for $h_\nabla$. Using the third equation of lemma \ref{14.34} and (\ref{gradian}) we obtain
\begin{align}\label{akhar}
(d^\pounds_{i_{S_\nabla}t_\nabla}\mathcal{F})(\delta_\beta)&=(d^\pounds\mathcal{F}\circ i_{S_\nabla}t_\nabla)(\delta_\beta)=d^\pounds\mathcal{F}(t_\nabla(S_\nabla, \delta_\beta))\nonumber\\
&=d^\pounds\mathcal{F}(f^cJ(\delta_\beta)-d^\pounds f^\vee(\delta_\beta)C)\nonumber\\
&=f^cd^\pounds\mathcal{F}(J(\delta_\beta))-d^\pounds f^\vee(\delta_\beta)d^\pounds\mathcal{F}(C)\nonumber\\
&=f^cd^\pounds\mathcal{F}(J(\delta_\beta))-(i_{\text{grad}f^\vee}\omega)(\delta_\beta)d^\pounds\mathcal{F}(C).
\end{align}
Since $\mathcal{F}$ is homogenous of degree 2, then we deduce
\[
d^\pounds\mathcal{F}(C)=\rho_\pounds(C)(\mathcal{F})=\textbf{y}^\alpha\frac{\partial \mathcal{F}}{\partial \textbf{y}^\alpha}=2\mathcal{F}.
\]
Also, from (iii) of proposition \ref{Best1} we get
\[
d^\pounds\mathcal{F}(J(\delta_\beta))=(d^\pounds_J\mathcal{F})(\delta_\beta)=(i_C\omega)(\delta_\beta).
\]
Setting two above equations in (\ref{akhar}) we obtain $d^\pounds_{i_{S_\nabla}t_\nabla}\mathcal{F}=i_{f^c C-2\mathcal{F}\text{grad}f^v}\omega$, which gives us
\begin{equation}\label{mae}
(d^\pounds_{i_{S_\nabla}t_\nabla}\mathcal{F})^\sharp=f^c C-2\mathcal{F}\text{grad}f^\vee.
\end{equation}
Setting the third equation of lemma \ref{14.34} and the above equation in (\ref{very good20}) we get
\begin{equation}\label{akhar1}
h_\nabla=h_\circ+\frac{1}{2}(f^cJ-d^\pounds f^\vee\otimes C)+\frac{1}{2}[J, f^cC]_\pounds-[J, \mathcal{F}\text{grad}f^\vee]_\pounds.
\end{equation}
Direct calculation we can obtain the following equations
\begin{align*}
[J, f^c C]_\pounds&=f^c J+d^\pounds_J f^c\otimes C,\\
[J, \mathcal{F}\text{grad}f^\vee]_\pounds&=\mathcal{F}[J, \text{grad}f^\vee]_\pounds+d^\pounds_J\mathcal{F}\otimes\text{grad}f^\vee.
\end{align*}
Setting two above equations in (\ref{akhar1}) give us
 \begin{align}\label{akhar2}
h_\nabla&=h_\circ+\frac{1}{2}(f^cJ-d^\pounds f^\vee\otimes C)+\frac{1}{2}f^c J+\frac{1}{2}d^\pounds_J f^c\otimes C\nonumber\\
&\ \ \ -\mathcal{F}[J, \text{grad}f^\vee]_\pounds-d^\pounds_J\mathcal{F}\otimes\text{grad}f^\vee.
\end{align}
But we have
\[
(d_Jf^c)(\delta_\alpha)=df^c(\mathcal{V}_\alpha)=\frac{\partial f^c}{\partial \textbf{y}^\alpha}=(\rho^i_\alpha\circ\pi)\frac{\partial(f\circ\pi)}{\partial\textbf{x}^i}=(d^\pounds f^\vee)(\delta_\alpha),
\]
and $(d_Jf^c)(\mathcal{V}_\alpha)=0=(d^\pounds f^\vee)(\mathcal{V}_\alpha)$. Thus we have $d_Jf^c=d^\pounds f^\vee$. Setting this equation in (\ref{akhar2}) we obtain (\ref{akhar3}), i.e., (iii) holds. Now we let (iii) holds and we prove (i). Let
\[
h_\circ=(\mathcal{X}_\alpha+\mathcal{B}^\beta_\alpha\mathcal{V}_\beta)\otimes\mathcal{X}^\alpha,\ \ \ h_\nabla=(\mathcal{X}_\alpha+\widetilde{\mathcal{B}}^\beta_\alpha\mathcal{V}_\beta)\otimes\mathcal{X}^\alpha.
\]
Then using (\ref{gradian4}) and (\ref{akhar3}) we can obtain
\begin{align}\label{enteha1}
\widetilde{\mathcal{B}}^\beta_\alpha&=\mathcal{B}^\beta_\alpha+\textbf{y}^\gamma(\rho^i_\gamma\circ\pi)\frac{\partial(f\circ\pi)}{\partial\textbf{x}^i}\delta^\beta_\alpha
-\mathcal{F}\frac{\partial \mathcal{G}^{\beta\gamma}}{\partial \textbf{y}^\alpha}(\rho^i_\gamma\circ\pi)\frac{\partial(f\circ\pi)}{\partial\textbf{x}^i}\nonumber\\
&\ \ \ -\frac{\partial\mathcal{F}}{\partial \textbf{y}^\alpha}\mathcal{G}^{\beta\gamma}(\rho^i_\gamma\circ\pi)\frac{\partial(f\circ\pi)}{\partial\textbf{x}^i}.
\end{align}
Since $h_\circ$ is conservative, then using (\ref{cons}) we have $(\rho^i_\alpha\circ\pi)\frac{\partial\mathcal{F}}{\partial\textbf{x}^i}+\mathcal{B}^\beta_\alpha\frac{\partial\mathcal{F}}{\partial\textbf{y}^\beta}=0$. Thus using the above equation we get
\begin{align*}
(\rho^i_\alpha\circ\pi)\frac{\partial\mathcal{F}}{\partial\textbf{x}^i}+\widetilde{\mathcal{B}}^\beta_\alpha\frac{\partial\mathcal{F}}{\partial\textbf{y}^\beta}&
=\textbf{y}^\gamma(\rho^i_\gamma\circ\pi)\frac{\partial(f\circ\pi)}{\partial\textbf{x}^i}
\frac{\partial\mathcal{F}}{\partial\textbf{y}^\alpha}-\mathcal{F}\frac{\partial \mathcal{G}^{\beta\gamma}}{\partial \textbf{y}^\alpha}(\rho^i_\gamma\circ\pi)\frac{\partial(f\circ\pi)}{\partial\textbf{x}^i}\frac{\partial\mathcal{F}}{\partial\textbf{y}^\beta}\nonumber\\
&\ \ \ -\frac{\partial\mathcal{F}}{\partial \textbf{y}^\alpha}\mathcal{G}^{\beta\gamma}(\rho^i_\gamma\circ\pi)\frac{\partial(f\circ\pi)}{\partial\textbf{x}^i}
\frac{\partial\mathcal{F}}{\partial\textbf{y}^\beta}.
\end{align*}
Using (i) of (\ref{2eq}) in the above equation, the sum of the first and third sentences of the right side of the above equation vanishes. Thus the above equation reduce to
\begin{align*}
(\rho^i_\alpha\circ\pi)\frac{\partial\mathcal{F}}{\partial\textbf{x}^i}+\widetilde{B}^\beta_\alpha\frac{\partial\mathcal{F}}{\partial\textbf{y}^\beta}&
=-\mathcal{F}\frac{\partial \mathcal{G}^{\beta\gamma}}{\partial \textbf{y}^\alpha}(\rho^i_\gamma\circ\pi)\frac{\partial(f\circ\pi)}{\partial\textbf{x}^i}\frac{\partial\mathcal{F}}{\partial\textbf{y}^\beta}.
\end{align*}
But from (\ref{2eq}) we deduce
\[
\frac{\partial \mathcal{G}^{\beta\gamma}}{\partial \textbf{y}^\alpha}\frac{\partial\mathcal{F}}{\partial\textbf{y}^\beta}=\textbf{y}^\lambda\frac{\partial \mathcal{G}^{\beta\gamma}}{\partial \textbf{y}^\alpha}\mathcal{G}_{\lambda\beta}=-\textbf{y}^\lambda\frac{\partial \mathcal{G}_{\lambda\beta}}{\partial \textbf{y}^\alpha}\mathcal{G}^{\beta\gamma}=0.
\]
Two above equations give us $(\rho^i_\alpha\circ\pi)\frac{\partial\mathcal{F}}{\partial\textbf{x}^i}+\widetilde{B}^\beta_\alpha\frac{\partial\mathcal{F}}{\partial\textbf{y}^\beta}=0$. Thus $h_\nabla$ is conservative and consequently $E, \mathcal{F}, \nabla$ is a generalized Berwald Lie algebroid.
Now we show that the torsion of $\nabla$ satisfies in (\ref{enteha}).
Differentiating of (\ref{enteha1}) with respect to $\textbf{y}^\mu$ we obtain
\begin{align*}
\frac{\partial\widetilde{\mathcal{B}}^\beta_\alpha}{\partial\textbf{y}^\mu}&=\frac{\partial \mathcal{B}^\beta_\alpha}{\partial\textbf{y}^\mu}+(\rho^i_\mu\circ\pi)\frac{\partial(f\circ\pi)}{\partial\textbf{x}^i}\delta^\beta_\alpha
-\frac{\partial\mathcal{F}}{\partial\textbf{y}^\mu}\frac{\partial \mathcal{G}^{\beta\gamma}}{\partial \textbf{y}^\alpha}(\rho^i_\gamma\circ\pi)\frac{\partial(f\circ\pi)}{\partial\textbf{x}^i}\nonumber\\
&\ \ \ -\mathcal{F}\frac{\partial^2 \mathcal{G}^{\beta\gamma}}{\partial\textbf{y}^\mu\partial \textbf{y}^\alpha}(\rho^i_\gamma\circ\pi)\frac{\partial(f\circ\pi)}{\partial\textbf{x}^i}-\frac{\partial^2\mathcal{F}}{\partial\textbf{y}^\mu\partial \textbf{y}^\alpha}\mathcal{G}^{\beta\gamma}(\rho^i_\gamma\circ\pi)\frac{\partial(f\circ\pi)}{\partial\textbf{x}^i}\nonumber\\
&\ \ \ -\frac{\partial\mathcal{F}}{\partial \textbf{y}^\alpha}\frac{\partial\mathcal{G}^{\beta\gamma}}{\partial\textbf{y}^\mu}(\rho^i_\gamma\circ\pi)\frac{\partial(f\circ\pi)}{\partial\textbf{x}^i}.
\end{align*}
Rechanging $\alpha$ and $\mu$ in the above equation we can obtain $\frac{\partial\widetilde{\mathcal{B}}^\beta_\mu}{\partial\textbf{y}^\alpha}$. Therefore we can obtain
\begin{align*}
\widetilde{t}^\beta_{\mu\alpha}&=\frac{\partial\widetilde{\mathcal{B}}^\beta_\alpha}{\partial\textbf{y}^\mu}-\frac{\partial\widetilde{\mathcal{B}}^\beta_\mu}{\partial\textbf{y}^\alpha}-(L^\beta_{\mu\alpha}\circ\pi)=\frac{\partial\mathcal{B}^\beta_\alpha}{\partial\textbf{y}^\mu}-\frac{\partial\mathcal{B}^\beta_\mu}{\partial\textbf{y}^\alpha}-(L^\beta_{\mu\alpha}\circ\pi)+(\rho^i_\mu\circ\pi)\frac{\partial(f\circ\pi)}{\partial\textbf{x}^i}\delta^\beta_\alpha\nonumber\\
&\ \ \ -(\rho^i_\alpha\circ\pi)\frac{\partial(f\circ\pi)}{\partial\textbf{x}^i}\delta^\beta_\mu=t^\beta_{\mu\alpha}+(\rho^i_\mu\circ\pi)\frac{\partial(f\circ\pi)}{\partial\textbf{x}^i}\delta^\beta_\alpha-(\rho^i_\alpha\circ\pi)\frac{\partial(f\circ\pi)}{\partial\textbf{x}^i}\delta^\beta_\mu,
\end{align*}
where $\widetilde{t}^\beta_{\mu\alpha}$ are the coefficients of weak torsion $t_\nabla$ of $h_\nabla$ and $t^\beta_{\mu\alpha}$ are the coefficients of weak torsion $t_\circ$ of Barthel endomorphism $h_\circ$ given by (\ref{wt1}). But the Barthel endomorphism is torsion free. So $t^\beta_{\mu\alpha}=0$. Therefore from the above equation we obtain
\begin{equation}
t_\nabla(\delta_\mu, \delta_\alpha)=\widetilde{t}^\beta_{\mu\alpha}\mathcal{V}_\beta=(\rho^i_\mu\circ\pi)\frac{\partial(f\circ\pi)}{\partial\textbf{x}^i}\mathcal{V}_\alpha-(\rho^i_\alpha\circ\pi)\frac{\partial(f\circ\pi)}{\partial\textbf{x}^i}\mathcal{V}_\mu.
\end{equation}
But from (\ref{Wag1}) and the above equation we deduce
\begin{align*}
(T_\nabla(e_\mu, e_\alpha))^{h_\nabla}&=F_\nabla t_\nabla(\delta_\mu, \delta_\alpha)=(\rho^i_\mu\circ\pi)\frac{\partial(f\circ\pi)}{\partial\textbf{x}^i}\delta_\alpha-(\rho^i_\alpha\circ\pi)\frac{\partial(f\circ\pi)}{\partial\textbf{x}^i}
\delta_\mu\\
&=\Big(\rho(e_\mu)(f)e_\alpha-\rho(e_\alpha)(f)e_\mu\Big)^{h_\nabla}=\Big(d^Ef(e_\mu)e_\alpha-d^Ef(e_\alpha)e_\mu\Big)^{h_\nabla},
\end{align*}
which gives us $T_\nabla(e_\mu, e_\alpha)=d^Ef(e_\mu)e_\alpha-d^Ef(e_\alpha)e_\mu$. Therefore (\ref{enteha}) holds and consequently $(E, \mathcal{F}, \nabla, f)$ is a Wagner Lie algebroid.
\end{proof}
\begin{cor}
If $(E, \mathcal{F}, \nabla, f)$ is a Wagner Lie algebroid, then spray $S_\nabla$ generated by $h_\nabla$ satisfies in the following relation
\[
S_\nabla=S_\circ+f^cC-2\mathcal{F}\text{grad}f^\vee.
\]
\end{cor}
\begin{proof}
Since $(E, \mathcal{F}, \nabla, f)$ is a Wagner Lie algebroid, then we have (\ref{mae}). Setting (\ref{mae}) in (\ref{very good}) the proof completes.
\end{proof}

\bigskip

\noindent
Esmaeil Peyghan\\
Department of Mathematics, Faculty  of Science\\
Arak University\\
Arak 38156-8-8349,  Iran\\
Email: e-peyghan@araku.ac.ir
\end{document}